\documentclass[a4paper,reqno]{amsart}

\usepackage{amsmath,amssymb,amsthm,mathrsfs,dsfont,enumerate}
\usepackage{graphicx}
\usepackage{epsfig}
\graphicspath{{clyy_iob12_periodic_n512/paper/}{clyy_iob12_periodic/paper/}{./}}
\usepackage{amsfonts}
\usepackage{amssymb}
\usepackage{color}
\usepackage{xcolor}
\usepackage[colorlinks=true]{hyperref}
\hypersetup{linkcolor=violet,urlcolor=cyan,citecolor=cyan,hypertexnames=false}

\usepackage{tikz}
\usetikzlibrary{arrows, calc}
\usetikzlibrary{calc,quotes,angles}
\usepackage{pgfplots}
\usepackage{mathtools}
\usetikzlibrary{intersections}

\usepackage{booktabs}
\usepackage{placeins}
\usepackage{algorithm}
\usepackage{algpseudocode}

\usepackage{mathtools}

\usepackage{comment}

\usepackage[margin=3cm]{geometry}

\newtheorem{theo}{Theorem}[section]
\newtheorem{lemm}[theo]{Lemma}
\newtheorem{defi}[theo]{Definition}
\newtheorem{coro}[theo]{Corollary}
\newtheorem{prop}[theo]{Proposition}
\newtheorem{rema}[theo]{Remark}
\numberwithin{equation}{section}

\allowdisplaybreaks %????????

\newcommand\ep{{\epsilon}}

\newcommand{\R}{\mathbb{R}}

\newcommand{\Z}{\mathbb{Z}}

\renewcommand{\div}{\mathop{\rm div}}

\newcommand{\sgn}{\mathop{\mathrm{sgn}}}

\newcommand{\pa}{\partial}
\newcommand{\na}{\nabla}

\newcommand{\La}{\Lambda}

\renewcommand{\bar}[1]{\overline{#1}}

\newcommand{\De}{\Delta}

\newcommand{\T}{\mathbb{T}}
\newcommand{\LP}{\mathbb{P}}

\begin{document}

\title[GWP and vanishing damping limit for OB]{Global well-posedness and uniform-in-time vanishing damping limit for the inviscid Oldroyd-B model}

\author[X. Cheng]{Xinyu Cheng}
 \address{X. Cheng, Research Institute of Intelligent Complex Systems, Fudan University, Shanghai, P.R. China}
\email{xycheng@fudan.edu.cn}

\author[Z. Luo]{Zhaonan Luo}
\address{Z. Luo, School of Science, Shenzhen Campus of Sun Yat-sen University, Shenzhen, P.R. China}
\email{luozhn7@mail.sysu.edu.cn}

\author[Z. Yang]{Zhaojie Yang}
\address{Z. Yang, School of Mathematics, Georgia Institute of Technology, Atlanta, United States}
\email{zyang768@gatech.edu}

\author[C. Yuan]{Cheng Yuan}
\address{C. Yuan, School of Mathematical Sciences, Fudan University, Shanghai, P.R. China}
\email{cyuan22@m.fudan.edu.cn}

% \author{Xinyu $\mbox{Cheng}^1$\footnote{E-mail: xycheng@fudan.edu.cn}, \quad 
% Zhaonan $\mbox{Luo}^{2,3}$\footnote{E-mail: luozhn7@mail.sysu.edu.cn}, \quad Zhaojie $\mbox{Yang}^3$\footnote{E-mail: yangzj20@fudan.edu.cn}, \quad and \quad Cheng $\mbox{Yuan}^3$\footnote{E-mail:\ cyuan22@m.fudan.edu.cn}\\
% $^1\mbox{Research}$ Institute of Intelligent Complex Systems, 
% Fudan University, \\ Shanghai 200433, China \\
% $^2\mbox{School}$ of Science, Shenzhen Campus of Sun Yat-sen University,\\ Shenzhen 518107, China \\
% $^3\mbox{School}$ of Mathematical Sciences,
% Fudan University, Shanghai 200433, China}

\date{}
\maketitle
%\hrule

\begin{abstract}
In this paper, we consider global strong solutions and the uniform-in-time vanishing damping limit for the inviscid Oldroyd-B model in $\R^d$, where $d=2$ and $3$. 
The well-recognized problem of the global existence of smooth solutions for
the 2D inviscid Oldroyd-B model without smallness assumptions is open due to the complex structure of $Q(\na u,\tau)\neq 0$. Therefore, improving the smallness assumptions, especially in the lower regularity class, is the core question in the area of fluid models. On the other hand, long-time behaviors of solutions including temporal decay and uniform-in-time damping stability are also of deep significance. These problems have been widely studied; however, the existing results are not regularity critical and the (uniform) vanishing damping limit has not been discussed. To tackle these challenges, we first establish local well-posedness in the sense of Hadamard with critical regularity. Then, by virtue of the sharp commutator estimate for Calderon--Zygmund operators, we establish global well-posedness for $d=2$ with damping in the low regularity class $(L^2\cap B^1_{\infty,1})\times (L^2\cap B^0_{\infty,1})$, which, to the best of our knowledge, is novel in the literature. Furthermore, in both the 2D and 3D cases, we prove global existence of solutions to the inviscid Oldroyd-B model uniformly in the damping parameter $a\in[0,1]$. In addition, we obtain optimal temporal decay rates and time integrability by improving the existing Fourier splitting method and developing a novel temporal decomposition strategy. These results lead to the major contribution of the present paper: the uniform-in-time vanishing damping limit for the inviscid Oldroyd-B model. More specifically, we discover the correlation between the sharp vanishing damping rate and the temporal decay rate. On periodic domains, we also prove exponential relaxation toward spatially homogeneous equilibria: the nonzero Fourier modes decay at a rate uniform in $a\in[0,1]$, while the decay rate of the stress zero mode necessarily degenerates as $a\to0$. We support our findings by providing numerical evidence.\\
\vspace*{5pt}
\noindent {\it 2020 Mathematics Subject Classification}: 35Q31, 76A05, 74B20, 35B40.

\vspace*{5pt}
\noindent{\it Keywords}: The inviscid Oldroyd-B models; Global well-posedness; Uniform-in-time vanishing damping limit.
\end{abstract}

% \vspace*{10pt}

% \tableofcontents
% \newpage

\section{Introduction}
In this manuscript, we consider the general Oldroyd-B models. The Oldroyd-B model serves as a constitutive system for describing the flow behavior of viscoelastic fluids, combining both viscous and elastic properties. It extends the upper-convected Maxwell model and accounts for time-dependent stress responses. The model characterizes stress as a combination of a Newtonian viscous component and a viscoelastic component, effectively capturing phenomena including stress relaxation and creep. This makes it particularly useful for studying polymers, biological fluids, and other complex materials. 

Oldroyd-B model was first introduced systematically by Oldroyd in the 1950s \cite{O50,O58} to describe constitutive models for viscoelastic fluids with respect to material frame indifference. To be more specific, in this monograph, we consider the general Oldroyd-B models as follows:
\begin{align}\label{eq}
\left\{\begin{array}{l}
\partial_tu+u\cdot\nabla u+\nabla P ={\rm div}~\tau+\nu\Delta u,~~~~{\rm div}~u=0,\\[1ex]
\partial_t\tau+u\cdot\nabla\tau+a\tau+Q(\nabla u,\tau)=\alpha D(u)+\mu\Delta\tau,\\[1ex]
u|_{t=0}=u_0,~~\tau|_{t=0}=\tau_0. \\[1ex]
\end{array}\right.
\end{align}
We shall assume that the fluid domain is the entire space $\mathbb{R}^{d}$ for $d=2,3$. In \eqref{eq}, $u(t,x)$ denotes the velocity of the polymeric liquid, $\tau(t,x)$ represents the symmetric tensor of constraints, and $P$ is the pressure. Moreover, the bilinear term is given as follows:
\begin{equation}\label{eqQ}
    Q(\nabla u, \tau)=\tau \Omega-\Omega\tau-b(D(u)\tau+\tau D(u)),
\end{equation}
where $b\in[-1, 1]$, the vorticity tensor is given by $\Omega=\frac {\nabla u-(\nabla u)^T} {2}$ and the deformation tensor is defined by $D(u)=\frac {\nabla u+(\nabla u)^T} {2}$. The viscosity coefficient $\nu$ and the diffusion coefficient $\mu$ are non-negative. From the derivation of polymer fluids, parameter $\alpha\geq 0$ and parameter $b$ are related, see \cite{DLY21}. The constitutive term is co-rotational when $b=0$. The special case $\alpha=b=0$ additionally removes the deformation forcing $\alpha D(u)$. In this paper, we focus on the inviscid Oldroyd-B model with $\nu=0$ and $\alpha=\mu=1$. Here, $a$ represents the damping coefficient and is inversely proportional to the Weissenberg number $W_e$, the dimensionless ratio of the polymer relaxation time to a characteristic flow time. The larger the $W_e$ number, the more significant the elastic effect of the fluid has. The large Weissenberg number problem \cite{CWZZ20} is challenging both numerically and analytically.

\subsection{Introduction to the Oldroyd-B models and historical review}
To start with, we introduce some results from the existing literature of the Oldroyd-B models.

Taking $\nu>0$ and $\mu=0$ in \eqref{eq}, we obtain the classical Oldroyd-B model. In \cite{GS90}, C. Guillop\'e,  and J. C. Saut first proved that the Oldroyd-B model admits a unique global strong solution in Sobolev spaces. For case $b=0$, the weak solutions of the co-rotational Oldroyd-B model were obtained by P. L. Lions and N. Masmoudi \cite{LM00}. We also note that the problem for case $b\neq0$ remains open, as discussed in \cite{M13}. J. Y. Chemin and N. Masmoudi \cite{CM01} established the existence and uniqueness of strong solutions in homogeneous Besov spaces with critical regularity. Optimal decay rates for solutions to the 3-D Oldroyd-B model were proved by M. Hieber, H. Wen and R. Zi \cite{HWZ19}.

Taking $\nu=0$ and $\mu>0$ in \eqref{eq}, we obtain the inviscid Oldroyd-B model. We will present the relevant mathematical results in two cases. \\
\textbf{(1) The case with damping ($a>0$).} \\
 For $d=2$, considering the inviscid Oldroyd-B model with initial data $(u_0,\tau_0)\in H^s$ for $s>2$, the system \eqref{eq0} was proved to admit a unique global strong solution $(u,\tau)\in C([0,\infty); H^s)$ in \cite{ER15} if
\begin{align*}
\|(u_0,\tau_0)\|_{H^1}+\|(\omega_0,\tau_0)\|_{B^0_{\infty,1}}\leq \ep,~~~~\omega_0={\rm curl}~u_0
\end{align*}
for some sufficiently small $\epsilon$. The large data global solutions under the additional condition $Q=0$ can be obtained as well. However, the problem of the global existence of smooth solutions for the general 2D inviscid Oldroyd-B model \eqref{eq0} ($Q\neq 0$) without any smallness assumption is still open due to the complex structure of $Q(\na u ,\tau)$.

Recently, W. Deng, Z. Luo and Z. Yin \cite{DLY21,DLY23} proved global existence of strong solutions and weak solutions to the co-rotational inviscid Oldroyd-B model.\\
\textbf{(2) The case without damping ($a=0$).}\\
 P. Constantin, J. Wu, J. Zhao and Y. Zhu \cite{CWZZ20} established the global well-posedness of the inviscid Oldroyd-B models with small data. 

Sharp decay estimates for the Oldroyd-B model with
only fractional stress tensor diffusion were proved in \cite{WWXZ22}. We now describe the results for the equation \eqref{eq0} with integral stress tensor diffusion. Let $d=3$. Considering the inviscid Oldroyd-B model with initial data $(u_0,\tau_0)\in H^r\cap L^1$ with $r=\frac{25}{2}$, P. Wang, J. Wu, X. Xu and Y. Zhong \cite{WWXZ22} proved that  there exists some sufficiently small constant $\ep$ such that if
\begin{align*}
\|(u_0,\tau_0)\|_{H^r\cap L^1}\leq \ep,
\end{align*}
then the system \eqref{eq0} admits a unique global solution $(u,\tau)$ that obeys the following decay properties
\begin{align*}
    &\|u\|_{L^2}\lesssim \ep(1+t)^{-\frac{3}{4}},~~~~\|u\|_{L^\infty}\lesssim \ep(1+t)^{-\frac{3}{2}},~~~~\|\nabla u\|_{L^2}\lesssim \ep(1+t)^{-\frac{5}{4}}, \\
&\|\nabla u\|_{L^\infty}\lesssim \ep(1+t)^{-2},~~~~\|\mathbb{P}\rm{div}~\tau\|_{L^2}\lesssim \ep(1+t)^{-\frac{5}{4}},
\end{align*}
where $\mathbb{P}$ denotes the Leray projection onto divergence-free vector fields.

W. Deng, Z. Luo and Z. Yin \cite{DLY23} proved optimal decay rate of global weak solutions for $d=2$ by the improved Fourier splitting method.
\subsection{Relevant Models}
When considering inviscid fluids, the standard model is the Euler equation. The incompressible Euler equation \eqref{Euler} is useful in modeling scenarios where the fluid behaves more like an ideal fluid.  
\begin{align}\label{Euler}
\left\{\begin{array}{l}
\partial_tu+u\cdot\nabla u+\nabla P =0,~~~~{\rm div}~u=0,\\[1ex]
u|_{t=0}=u_0. \\[1ex]
\end{array}\right.
\end{align}
We now provide relevant references about the Euler equation in the following, which greatly help us understand the behaviour of inviscid fluids. V. I. Yudovich \cite{Y63} established the existence and uniqueness of global weak solutions for the Euler equations with $d=2$. J. T. Beale, T. Kato and A. Majda \cite{BKM84} proved the celebrated BKM criterion for determining the global existence of strong solutions to the Euler equations. A. Kiselev and V. Sverak \cite{KS14} obtained a lower bound estimate for the double exponential growth of solutions to the Euler equations with $d=2$. The local well-posedness of the Euler equation with critical regularity was proved in \cite{GLY19}. T. Hmidi and S. Keraani \cite{HK08} proved the rate of vanishing viscosity in borderline Besov spaces. For the surface quasi-geostrophic equations, X. Cheng, Z. Luo and S. Wang \cite{CLW26} study vanishing viscosity rates and same-topology convergence in borderline Besov spaces. J. Bourgain and D. Li \cite{BL15} established strong ill-posedness of the Euler equations in borderline sobolev spaces. T. Buckmaster, C. De Lellis, L. Sz\'ekelyhidi, Jr., V. Vicol \cite{BLJV19} and P. Isett \cite{Is18} independently proved Onsager’s conjecture by constructing non-trivial weak solutions via the convex integration scheme. Recently V.Giri, H. Kwon and M. Novack have proved a $L^3$-based Onsager's conjecture in \cite{GKN23,GKN24}.

For other relevant models,  Z. Lei, C. Liu and Y. Zhou \cite{LLZ08} proved the existence of global solutions for incompressible viscoelastic fluids, which is a simplified model of the Oldroyd-B model in a certain case. In \cite{L16}, Z. Lei proved the global well-posedness of incompressible elastodynamics in two dimensions; a similar result was obtained independently by X. Wang in \cite{W17} using a normal form method. The two-dimensional Euler-Poisson system was proved to admit global solutions by Li and Wu in \cite{LW14}; global regularity results of the 2D fractional Boussinesq equations are obtained by Stefanov, Wu, Xu and Ye in \cite{SWXY25}.

Moreover, it is also worth mentioning that the non-uniqueness of low regularity solutions to other related fluid equations has been studied widely.
For example, non-uniqueness of forced Navier--Stokes equations has been studied by D. Albritton, E. Bru\'e and M. Colombo \cite{ABC22}, and non-uniqueness of time-dependent SQG equations has been studied by T. Buckmaster, S. Shkoller and V. Vicol in \cite{BSV19}. X. Cheng, H. Kwon and D. Li have also shown the non-uniqueness of stationary solutions to the SQG equations in \cite{CKL21}.

\subsection{The main results}
Taking $\alpha=\mu=1$, $\nu=0$ in \eqref{eq}, then we have
\begin{align}\label{eq0}
\left\{\begin{array}{l}
\partial_tu+u\cdot\nabla u+\nabla P ={\rm div}~\tau,~~~~{\rm div}~u=0, \\ [1ex]
\partial_t\tau+u\cdot\nabla\tau+a\tau+Q(\nabla u,\tau)=D(u)+\Delta\tau,\\ [1ex]
u|_{t=0}=u_0,~~\tau|_{t=0}=\tau_0. \\[1ex]
\end{array}\right.
\end{align}
Our main results can be stated as follows.
\begin{theo}[low regularity GWP for the $a>0$ case]\label{6theo}
        Let $d=2$ and $a\in (0,1]$. Assume a divergence-free field $u_0\in L^2\cap B^1_{\infty,1}$ and a symmetric matrix $\tau_0\in L^2\cap B^0_{\infty,1}$. There exists some $c=c(a)>0$ small enough such that if
        \begin{align}\label{6con}
        \|(u_0,\tau_0)\|_{L^2}+\|(\nabla u_0,\tau_0)\|_{B^0_{\infty,1}} \leq c,
        \end{align}
        then \eqref{eq0} admits a global solution $(u,\tau)$ with
        $$
        (u,\tau) \in L^{\infty}(0,\infty;L^2\cap B^1_{\infty,1})\times L^{\infty}(0,\infty;L^2\cap B^0_{\infty,1}).
        $$
  \end{theo}
    \begin{theo}[GWP and vanishing damping limit in finite time]\label{theo1}
        Let $d=2,3$ and $a\in[0,1]$. Assume a divergence-free field $u_0\in H^1\cap B^1_{\infty,1}$ and a symmetric matrix $\tau_0\in H^1\cap B^0_{\infty,1}$. There exists some positive constant $\ep$ small enough such that if
        \begin{align*}
        \|(u_0,\tau_0)\|_{H^1}+\|(\nabla u_0,\tau_0)\|_{B^0_{\infty,1}} \leq \ep,
        \end{align*}
        then \eqref{eq0} admits a global solution $(u^a,\tau^a)$ with
        $$
        (u^a,\tau^a) \in L^{\infty}(0,\infty;H^1\cap B^1_{\infty,1})\times L^{\infty}(0,\infty;H^1\cap B^0_{\infty,1}).
        $$
        Moreover, for any $T>0$, there exists positive constant $C_T$ depending on $T$ such that the following hold
  $$\|(u^a,\tau^a)-(u^0,\tau^0)\|_{L^{\infty}(0,T;L^2)}\leq C_T a,
  $$ 
  and 
  $$\lim_{a\rightarrow0}\|(u^a,\tau^a)-(u^0,\tau^0)\|_{L^{\infty}(0,T;\dot{H}^1\cap B^1_{\infty,1})\times L^{\infty}(0,T;\dot{H}^1\cap B^0_{\infty,1})}=0.
  $$
  \end{theo}
  \begin{rema}
   The result of the global existence in Theorem \ref{theo1} extends the research of critical regularity to the case $a\in [0,1]$. In addition, for any $T>0$, we prove vanishing damping limit under the same topology of the initial data.   \end{rema}
   \begin{rema}
   The local well-posedness of the Euler equation in critical Besov space $B^1_{\infty,1}$ was established in \cite{GLY19}. One can see that the regularity of velocity $u^a$ in Theorem \ref{theo1} is critical, since $u^a$ satisfies the Euler equation with force. However, for $a=0$, we need the $H^1$ estimate of $\tau^a$ to close this critical estimate. For the case $a>0$, we point out that damping effect is crucial to reducing the regularity of $\tau^a$, see Theorem \ref{6theo}.   
  \end{rema}
\begin{rema}
    By the theory of the transport equation (cf. \cite{BCD11}), the velocity $u\in  C^{0,1}$ is the key to persisting regularity. Under the conditions in Theorem  \ref{theo1}, we deduce that the global existence of smooth solutions $(u^a,\tau^a)$ for the equation \eqref{eq0} with extra smooth initial data $(u_0,\tau_0)$.
\end{rema}
  \begin{rema}
   In the 3D case, the vorticity equation has an extra $\omega^a\cdot \nabla u^a$ term, therefore the main difference and challenge is to control this term.
\end{rema}
  \begin{theo}[Optimal uniform-in-time vanishing damping limit]\label{theo2}
  Under the same conditions as in Theorem \ref{theo1}, if additionally $(u_0,\tau_0) \in \dot{B}^{-\frac{d}{2}}_{2,\infty},$
        then there exists a positive constant $C_0$ depending on the initial data such that
        $$
        \|(u^a,\tau^a)\|_{L^2} +(1+t)^{\frac{1}{2}}\|\nabla(u^a,\tau^a)\|_{L^2}\leq C_0(1+t)^{-\frac{d}{4}},
        $$
  and 
  $$\int_{0}^{\infty}\|\nabla u^a\|_{B^0_{\infty,1}}dt'\leq C_0.$$
  Furthermore, the following sharp uniform-in-time vanishing damping rate holds
  \begin{align}\label{sharptwod}
  \|(u^a,\tau^a)-(u^0,\tau^0)\|_{L^{\infty}(0,\infty;L^2)}\leq C_0a^{\frac d 4},
  \end{align}
  and 
   $$\lim_{a\rightarrow0}\|(u^a,\tau^a)-(u^0,\tau^0)\|_{L^{\infty}(0,\infty;\dot{H}^1\cap B^1_{\infty,1})\times L^{\infty}(0,\infty;\dot{H}^1\cap B^0_{\infty,1})}=0.
  $$
    \end{theo}
 \begin{rema}
     For the case $d=2$ and $a=0$, optimal time decay rates of global solutions for the inviscid Oldroyd-B model \eqref{eq0} are critical, i.e. $\|\tau^a\|_{B^0_{\infty,1}}\leq C_0(1+t)^{-1}$ in the case of high regularity,  cf. \cite{DLY23}. This borderline upper bound alone does not imply time integrability of $\|\tau^a\|_{B^0_{\infty,1}}$. We introduce a new method of space-time frequency decomposition to obtain the time integrability of $\|\nabla u^a\|_{B^0_{\infty,1}}$. The key time integrability of $\|\nabla u^a\|_{B^0_{\infty,1}}$ enables us to discover a completely new phenomenon that there is a correlation between the uniform vanishing damping rate and the time decay rate.
 \end{rema}
 \begin{rema}
The uniform-in-time $L^2$ vanishing damping rate \eqref{sharptwod} is sharp in the sense that we can find initial data satisfying our condition such that 
$$\|(u^a,\tau^a)-(u^0,\tau^0)\|_{L^{\infty}(0,\infty;L^2)}\gtrsim a^{\frac d 4}.
  $$ 
We refer the readers to Propositions \ref{2dsharp} and \ref{3dsharp} for more discussion, where we will present examples satisfying the exact vanishing damping rate.
 \end{rema}
 \begin{rema}
     In the case with $d=3$ and high regularity, optimal time decay rates of global solutions for the inviscid Oldroyd-B models were proved in \cite{WWXZ22}. We generalize the results to the critical regularity case by proving the key integrability of $\|(\nabla u^a,\tau^a)\|_{B^0_{\infty,1}}$, which also helps to discover that there is a relationship between uniform vanishing damping rate and time decay rate.
 \end{rema}
\subsection{Motivations, difficulties and main ideas}
% {\color{red}T. M. Elgindi and F. Rousset \cite{ER15}} proved the global existence of strong solutions to the inviscid Oldroyd-B model \eqref{eq0} with $d=2$ and $a>0$, with initial data $(u_0,\tau_0)\in H^s$ where $s>2$. Specifically speaking, they proved that there exists some sufficiently small constant $\ep$ such that if
% \begin{align*}
% \|(u_0,\tau_0)\|_{H^1}+\|(\omega_0,\tau_0)\|_{B^0_{\infty,1}}\leq \ep,~~~~\omega_0={\rm curl}~u_0,
% \end{align*}
% then the system \eqref{eq0} admits a unique global strong solution $(u^a,\tau^a)\in C([0,\infty); H^s)$.

The well-recognized problem of the global existence of smooth solutions for the 2D inviscid Oldroyd-B model without smallness assumptions is open due to the complex structure of $Q(\na u,\tau)\neq 0$. Therefore improving the smallness assumptions, especially in lower regularity class, is the core question in the area of fluid models. On the other hand, long-time behaviors of solutions including temporal decay and uniform-in-time damping stability are also of deep significance. These problems have been widely studied, however, the existing results are not regularity critical and the vanishing damping limit has not been discussed. The goal of this paper is to dig deeper in this direction and eventually establish the global existence of solutions in the critical low regularity space $u\in B^{1}_{\infty,1}$. Moreover, we aim to prove the uniform-in-time vanishing damping limit for
the inviscid Oldroyd-B model and discover the correlation between the sharp vanishing damping rate
and the temporal decay rate.

 Assume that $d = 2, 3$ and $a\in[0,1]$, we first establish local well-posedness for \eqref{eq0} in the sense of Hadamard with initial data
$$(u_0,\tau_0)\in B^1_{\infty,1}\times B^0_{\infty,1}.$$ 

Then we divide the results of global theory into the following three cases.

\textbf{Case 1: motivation and main ideas in proving Theorem \ref{6theo} with damping $a>0$.} \\
In this case, we consider the 2D inviscid
Oldroyd-B model \eqref{eq0} with $a>0$. The motivation is to obtain global solutions for \eqref{eq0} with low regularity. We point out that damping effect is the key to reducing the regularity of $\tau$. We now establish the global existence of solutions of low regularity, where initial data belongs to $(L^2\cap B^1_{\infty,1})\times (L^2\cap B^0_{\infty,1})$, which to our best knowledge, is novel in the literature.
\begin{comment}
{\color{red}We first} present Theorem \ref{6theo} here: 

Let $d=2$ and $a\in (0,1]$. Assume a divergence-free field $u_0\in L^2\cap B^1_{\infty,1}$ and a symmetric matrix $\tau_0\in L^2\cap B^0_{\infty,1}$. There exists some positive constant $c=c(a)$ small enough such that if
        \begin{align*}
        \|(u_0,\tau_0)\|_{L^2}+\|(\nabla u_0,\tau_0)\|_{B^0_{\infty,1}} \leq c,
        \end{align*}
        then \eqref{eq0} admits a global solution $(u,\tau)$ with
        $$
        (u,\tau) \in L^{\infty}(0,\infty;L^2\cap B^1_{\infty,1})\times L^{\infty}(0,\infty;L^2\cap B^0_{\infty,1}).
        $$
\end{comment}

By virtue of the Littlewood-Paley decomposition theory, we prove a new commutator estimate between the Riesz operator $\mathcal{R}_i$ and the convection operator $u\cdot\nabla$:
\begin{align}\label{1in9}
    \|[\mathcal{R},u\cdot\nabla]\tau\|_{B^0_{\infty,1}} \leq C(\|\nabla u\|_{L^\infty}\|\tau\|_{B^{0}_{\infty,1}}+\|u\|_{L^2}\|\tau\|_{L^2}).
\end{align}
Note that the commutator estimate of the Calderon-Zygmund operator is sharp, which does not require additional regularity about $\|\tau\|_{B^\ep_{\infty,1}}$ with $\ep>0$ like Lemma \eqref{CR} does.

Since $B^0_{\infty,1}$ is a critical space and $Q$ does not have the transportation structure, product laws may require additional regularity. However, in the 2D case, we discover that $Q$ can be controlled by a new estimation.
        Using Bony's decomposition, we have 
  \begin{align}\label{1in10}
 \|Q(\nabla u,\tau)\|_{B^0_{\infty,1}}
   &\lesssim \|\nabla u\|_{B^0_{\infty,1}} \|\tau\|_{B^0_{\infty,1}}+\|\nabla u\|_{B^0_{\infty,1}}\|\tau\|_{H^1}.
  \end{align}
Combining \eqref{1in9} and \eqref{1in10}, we close the global estimates with low regularity for \eqref{eq0} with $a>0$.

Next we introduce main difficulties and ideas in the process of proving global existence of the strong solutions with $a\in[0,1]$ and uniform-in-time vanishing damping limit for the inviscid Oldroyd-B model \eqref{eq0}. 

\textbf{Case 2: main difficulties and ideas in the 2D case with and without damping $a\in[0,1]$.} 
     
\textbf{Difficulty (1):} Uniform-in-time control of some critical norms of the solutions.

In order to derive the dissipation of $u^a$, $a>0$ is assumed in the proof of the global existence with the critical small condition. To obtain the global existence of the strong solutions for the inviscid Oldroyd-B model \eqref{eq0} with $a\in[0,1]$, we have to obtain the uniformly in time boundedness of some critical norms of the solution. Besides basic energy estimates, we used the  inner product estimate to obtain the dissipation of velocity $u^a$.
Using the fact for $d=2$ that
$$\int_{\mathbb{R}^{2}}(u^a\cdot\nabla)u^a\cdot \Delta u^a dx=0,$$
we close the global estimate with small data in $H^1$ for the inviscid Oldroyd-B model \eqref{eq0}, which implies the existence of the global weak solutions for $a\in[0,1]$. 
%Since the core difficulty term $Q$ exists, we cannot elevate weak solutions to strong solutions solely by improving regularity. Specifically, we fail to obtain the global estimate of $\|\omega^a\|_{L^\infty}$ to prove uniqueness of the solutions. 

To elevate weak solutions to strong solutions, the key lies in controlling the $L^\infty$ norm of $\omega^{a}$, as indicated by the BKM principle \cite{BKM84}. By direct calculation, the vorticity satisfies the following equation: 
\begin{align*}
    \frac{d}{dt}\omega^a + u^a\cdot\nabla\omega^a = \nabla \times {\rm div}~\tau^a.
\end{align*}
However, the term on the right-hand side,
$\nabla \times {\rm div}~\tau^a$ breaks the conservation laws and, more crucially, the equation does not capture the dissipative properties of the vorticity $\omega^a$. This makes it difficult to get global estimate of $\|\omega^a\|_{L^{\infty}}$ from the above equation directly. To overcome this difficulty, we use the following structural variable $\Gamma^a$ introduced in \cite{ER15}:
  \begin{equation*}
  \Gamma^a=\omega^a-\mathcal{R}\tau^a,
  \end{equation*}
  where $\mathcal{R}:=-(-\Delta)^{-1}{\rm curl}~{\rm div}$. Note that once we aim to estimate the vorticity $\omega^{a}$ through the structural variable $\Gamma^{a}$ , the estimation of $\tau^{a}$ becomes necessary.

By direct calculation, $\Gamma^a$ satisfies the following transport equation with damping:
  \begin{align}\label{1eq1}
\partial_{t}\Gamma^a+u^a\cdot\nabla\Gamma^a+\frac{1}{2}\Gamma^a=(a-\frac{1}{2})\mathcal{R}\tau^a+\mathcal{R}Q(\nabla u^a,\tau^a)+[\mathcal{R},u^a\cdot\nabla]\tau^a.
  \end{align}
The structural unknown $\Gamma^a$ allows us to transfer dissipation $D(u^a)$ from the equation $\tau^a$ to the equation of $\Gamma^a$, which helps to obtain a closed estimate of the global solution under smallness condition in the critical Besov space $B^1_{\infty,1}\times B^0_{\infty,1}$.  In addition, it also plays an important role in studying the solution of supercritical regularity and obtaining the key time integrability. When $a>0$, the existence of global strong solutions can be obtained by proving that $\Gamma$ and $\tau$ belong to $C([0,\infty); H^s)$, as shown in \cite{ER15}. However, the proof does not guarantee uniform boundedness of the $H^s$ norm, posing an obstacle as $a \to 0$. Our strategy is to control $\|\tau^a\|_{B^{0}_{\infty,1}}$ directly.
Indeed, by Duhamel's formula, we get
  \begin{equation*}
  \tau^a(t)=e^{(\Delta-a)t}\tau_{0}+\int_{0}^{t}e^{(\Delta-a)(t-t')}\left[Du^a-u^a\cdot\nabla\tau^a-Q(\nabla u^a,\tau^a)\right]dt'.
  \end{equation*}
  Applying the Littlewood-Paley decomposition theory and utilizing the smoothing effect of the heat operator, we prove that $\|\tau^a\|_{B^{0}_{\infty,1}}\leq C.$
Hence, $\|\omega^{a}\|_{B^{0}_{\infty,1}}$ is uniformly bounded provided $\|\Gamma^{a}\|_{B^{0}_{\infty,1}}$ is uniformly. By virtue of the estimates on \eqref{1eq1}, we conclude that 
\begin{align}\label{1in1}
\|(u^a,\tau^a)\|_{L^{\infty}(0,\infty;H^1)}+\|(\nabla u^a,\tau^a)\|_{L^{\infty}(0,\infty;B^0_{\infty,1})}
\leq C(\|(u_0,\tau_0)\|_{H^1}+\|(\nabla u_0,\tau_0)\|_{B^0_{\infty,1}}). 
\end{align}

\textbf{Difficulty (2):} Extra regularity due to the convection term when considering vanishing damping limit.

 When considering vanishing damping limit of the inviscid Oldroyd-B model \eqref{eq0} for any $T>0$, controlling the transport term $u^a\cdot\nabla u^a$ requires extra regularity. To overcome it, notice that global solution is uniformly bounded in time, see \eqref{1in1}, we can elevate regularity of global solutions solely by improving regularity of initial data. Moreover, through precise estimates and the structural trick for the unknown $\Gamma^a$, we deduce that
\begin{align}\label{1in2}
\|(u^a,\tau^a)\|_{L^{\infty}(0,\infty;H^2)}+\|(\nabla u^a,\tau^a)\|_{L^{\infty}(0,\infty;B^1_{\infty,1})}
\leq C(\|(u_0,\tau_0)\|_{H^2}+\|(\nabla u_0,\tau_0)\|_{B^1_{\infty,1}}). 
\end{align}
The solution of supercritical regularity is also uniformly bounded in time and the bound depends linearly on the initial data, which is useful for considering vanishing damping limit for any $T>0$ under the same topology. 

With the results showing that the solutions are uniformly bounded in time, we first prove vanishing damping rate for low frequencies of the solutions. For any $T>0$, there holds 
  $$\|(u^a,\tau^a)-(u^0,\tau^0)\|_{L^{\infty}(0,T;L^2)}\leq C_Ta.
  $$ 
In the general analysis of stability and the vanishing limit, a phenomenon of supercritical regularity arises when estimating the transport terms $u^a\cdot \nabla u^a$. To handle this, we need to introduce a low-frequency truncated smoothing operator $S_N$ to quantify the regularity, see \cite{BS75}. Let $(u_N^a,\tau_N^a)$ be a solution of the inviscid Oldroyd-B model $\eqref{eq0}$ with the initial data $(S_N u_0,S_N\tau_0)$ and $a\in[0,1]$. Using \eqref{1in1} and \eqref{1in2} we deduce that for any $a\in[0,1]$ and $T>0$, there holds
\begin{align*}
      \|\nabla(u^a-u_N^a,\tau^a-\tau_N^a)\|_{L^2}\leq C_T\|(Id-S_N)(u_0,\tau_0)\|_{H^1},\quad
      \|\nabla(u_N^a-u_N^0,\tau_N^a-\tau_N^0)\|_{L^2}\leq C_T a 2^{N}.
  \end{align*}
  Combining the above inequalities, we conclude that
  $$\lim_{a\rightarrow0}\|(u^a,\tau^a)-(u^0,\tau^0)\|_{L^{\infty}(0,T;\dot{H}^1)}=0.
  $$
By virtue of the equation \eqref{1eq1} of $\Gamma^a$ and the time uniform bound of solutions with supercritical regularity, we also infer that
\begin{align*}
\|\Gamma^{a}-\Gamma^{0}\|_{L^{\infty}(0,T;B^{0}_{\infty,1})}\leq C_{T}(\|(Id-S_N)(u_0,\tau_0)\|_{(H^1\cap B^1_{\infty,1})\times (H^1\cap B^0_{\infty,1})}+a^{\frac{1}{2}}2^{N}).
\end{align*}
This implies that
 $$\lim_{a\rightarrow0}\|u^a-u^0\|_{L^{\infty}(0,T; B^1_{\infty,1})}=0.$$

\textbf{Difficulty (3):}  Establishing the time integrability of $\|\nabla u^a\|_{B^0_{\infty,1}}$ for initial data under critical regularity.

In this paper, we aim to explore and solve a newly observed problem: the uniform in time vanishing damping limit of the solution. Indeed, we first analyze the following quantity:
$$\|(u^a,\tau^a)-(u^0,\tau^0)\|_{L^{\infty}(0,\infty;L^2)}
  $$ 
as $a$ approaches zero.
To tackle this problem, we find that the key to getting a uniform in time control is to obtain the time integrability of $\|\nabla u^a\|_{B^0_{\infty,1}}$. 

To do this, it is necessary to derive the decay rates for the low-frequency quantities, specifically $\|(u,\tau)\|_{H^{1}}$.
      By virtue of the improved Fourier splitting method \cite{DLY23}, we obtain, for any $a\in[0,1]$, 
      \begin{align}\label{1in3}
        \|(u^a,\tau^a)\|_{L^2} +(1+t)^{\frac{1}{2}}\|\nabla(u^a,\tau^a)\|_{L^2}\leq C_0(1+t)^{-\frac{1}{2}}.
    \end{align}
 However, we fail to derive the decay rates for $\|(\nabla u^a,\tau^a)\|_{B^0_{\infty,1}}$ directly due to the lack of regularity of initial data. A natural approach is to prove the time integrability of $\|\Gamma\|_{B^0_{\infty,1}}$ and $\|\tau^a\|_{B^0_{\infty,1}}$.  Unfortunately, this approach fails for $d=2$, even under high regularity, as we can only have
$$\|\tau^a\|_{B^0_{\infty,1}}\leq C_0(1+t)^{-1},$$ 
 which is critically non-integrable. The main challenge lies in the low-frequency decay rate of $\tau^a$. To address this, we introduce a novel method of high-low frequency decomposition to get the following key integrability:
\begin{align}\label{1in4}
    \int_{0}^{\infty}\|\nabla u^a\|_{B^0_{\infty,1}}dt'\leq C_0.
\end{align}

For the high-frequency component, we encounter no significant challenges. We have
 \begin{align*}
    \int_{0}^{t}\|(Id-\Delta_{-1})\tau^a\|_{B^0_{\infty,1}}dt' \lesssim\int_{0}^{t}\|\nabla^2 \tau^a\|_{L^2} dt'\lesssim C_0,
    \end{align*}
With the help of the equation \eqref{1eq1} of $\Gamma^a$ again, we derive that
\begin{align*}
\int_0^t\|(Id-\Delta_{-1})\Gamma^a\|_{B^0_{\infty,1}}dt' &\lesssim C_0+ \int_0^t \int_0^{t'} e^{-\frac 1 4(t'-s)} \|(Id-\Delta_{-1})\mathcal{R}\tau^a\| 
_{B^0_{\infty,1}} dsdt' \lesssim C_0.
\end{align*}
Thus we get the integrability of the high-frequency part: 
\begin{align*}
\int_0^t\|(Id-\Delta_{-1})\nabla u^a\|_{B^0_{\infty,1}}dt'\lesssim C_0.
\end{align*}

Now it remains to control $\int_{1}^{\infty}\left\|\Delta_{-1}\nabla u^a\right\|_{L^{\infty}}dt.$ Due to the smoothing effect of $\Delta_{-1}$, we can control it using energy norms. Applying time decay rates \eqref{1in3} together with the basic energy estimates, we first derive an integral form characterization of the decay effects: 
 \begin{align*}
    \int_{0}^{t}(1+t')^{1-\delta}(\|\nabla u^a\|^2_{L^{2}}+\|\nabla \tau^a\|^2_{H^{1}}) dt'
 \lesssim C_0\frac {1}{\delta},\quad
    \int_{0}^{t}(1+t')^{2-\delta}\|\nabla^2 \tau^a\|^2_{L^{2}}dt'
\lesssim C_0\frac {1}{\delta},
    \end{align*}
 where $\delta\in(0,1)$. 
Furthermore, we can establish the time integrability of higher-order derivatives of the solutions to \eqref{eq0} as follows:
\begin{align*}
\int_{0}^{t}(1+t')^3\|\Delta_{-1}\nabla^2 u^a\|_{L^{2}}^2dt' \lesssim C_0(1+t)^{\frac 32}.
\end{align*}
Note that there is a time growth on the right side of this inequality. For more details, one can see the proof of \eqref{4in10}. We introduce a novel inequality to control this growth. By Lebesgue's monotone convergence theorem, we deduce that 
\begin{align*}
\int_{1}^{\infty}\left\|\Delta_{-1}\nabla u^a\right\|_{L^{\infty}}dt
\lesssim&\sum_{j=0}^{\infty}\int_{2^{j}}^{2^{j+1}}\left\|\nabla u^a\right\|_{L^{2}}^{\frac{1}{2}}\left\|\Delta_{-1}\nabla^{2}u^a\right\|_{L^{2}}^{\frac{1}{2}}dt\lesssim C_0\sum_{j=0}^{\infty}2^{-\frac{j}{8}}\lesssim C_0.
 \end{align*}
Consequently, we conclude that \eqref{1in4} holds true.

\textbf{Difficulty (4):} Demonstrating the optimality of the uniform vanishing damping rate.

The discovery of the key integrability enables us to establish the uniform vanishing damping limit in $L^2$:
$$\|(u^a,\tau^a)-(u^0,\tau^0)\|_{L^{\infty}(0,\infty;L^2(\R^{2}))}\lesssim a^{\frac 1 2},
  $$ 
Our goal is to show that the approximation rate $a^{\frac 1 2}$ is indeed optimal. Heuristically speaking, one can consider the uniform vanishing damping limit for the damped heat equation:
\begin{align*}
    \partial_t f^a-\Delta f^a+af^a=0,~~~~f^a|_{t=0}=f_0,
\end{align*}
where $a\geq 0$, $\hat{f_{0}}(\xi)\in C^{\infty}(\R^{2})$ with
\begin{align*}
\hat{f_{0}}(\xi)=
\begin{cases}
1,\quad |\xi|\leq 1,\\
0,\quad |\xi|\geq 2.
\end{cases}
\end{align*}

Then we infer that
\begin{align*}
    \|f^a-f^0\|_{L^2}
    =(1-e^{-at})\|e^{\Delta t}f_0\|_{L^2} \lesssim C_0(at)^{\frac{1}{2}}(1+t)^{-\frac{1}{2}} \lesssim C_0 a^{\frac{1}{2}}.
\end{align*}
The reverse bound holds at $t=a^{-1}$. This scalar heat mode
is also an exact solution of the full nonlinear system:
take $u^a=0$, $\tau^a=f^a I_d$, and $P^a=f^a$.
The pressure balances $\operatorname{div}\tau^a=\nabla f^a$, and all
terms involving $\nabla u^a$ vanish for every $b\in[-1,1]$.
Thus the heat scaling gives the sharp lower bound $a^{d/4}$
in both dimensions, without a nonlinear perturbation argument;
see Propositions \ref{2dsharp} and \ref{3dsharp}.

%By virtue of the optimal decay rates \eqref{1in3} and key integrability \eqref{1in4}, we obtain the uniform vanishing damping limit and discover the new phenomenon for \eqref{eq0} that the rate of uniform vanishing damping limit in $L^2$ is related to the time decay rate in $L^2$. Under the conditions in Theorem \ref{theo2}, there holds 

\textbf{Difficulty (5):} Obtaining the uniform vanishing damping limit in the same topology.

Furthermore, we also aim to obtain the uniform vanishing damping limit for the solutions in spaces of higher regularity. Specifically, we need to show that for any $0<\delta<1$, there exists a positive number $a_{0}$ such that
\begin{align*}
\|(u^a,\tau^a)-(u^0,\tau^0)\|_{L^{\infty}(0,\infty;\dot{H}^1\cap B^1_{\infty,1})\times L^{\infty}(0,\infty;\dot{H}^1\cap B^0_{\infty,1})}\lesssim \delta,
\end{align*}
whenever $a\leq a_{0}$. 

By virtue of the mean-value theorem, along with the optimal decay rates and key
integrability, we deduce that for above $\delta>0$, there exists $T^{*}\in (0,\infty)$ independent with $a\in[0,1]$, such that
\begin{align}\label{1in6}
\|(u^a,\tau^a)(t)\|_{(H^1\cap B^1_{\infty,1})\times (H^1\cap B^0_{\infty,1})}\lesssim \delta, \quad\text{for all}\ a\in[0,1], ~t\geq T^{*}.
\end{align}
According to the result of vanishing damping limit in finite time interval, there exists $a_{0}>0$ such that
\begin{align*}
\|(u^a,\tau^a)-(u^0,\tau^0)\|_{L^{\infty}(0,T^{*};\dot{H}^1\cap B^1_{\infty,1})\times L^{\infty}(0,T^{*};\dot{H}^1\cap B^0_{\infty,1})}\leq \delta,
\end{align*}
whenever $a\leq a_{0}$. 
This together with \eqref{1in6} ensures that
   $$\lim_{a\rightarrow0}\|(u^a,\tau^a)-(u^0,\tau^0)\|_{L^{\infty}(0,\infty;\dot{H}^1\cap B^1_{\infty,1})\times L^{\infty}(0,\infty;\dot{H}^1\cap B^0_{\infty,1})}=0.
  $$

\textbf{Case 3: main difficulties and ideas in the 3D case with and without damping $a\in[0,1]$.} \\
For $d=3$, the convection term $u^a\cdot \nabla u^a$ becomes the main term, we encounter additional difficulty.

\textbf{Difficulty (6):} Handling the transport term $u^a\cdot \nabla u^a$ in establishing the global estimates of solutions.

Since $\int_{\mathbb{R}^{3}}(u^a\cdot\nabla)u^a\cdot \Delta u^a dx\neq 0,$
we fail to close the global estimate for \eqref{eq0} in $H^1$ with small data. By direct estimate, we obtain 
\begin{align*}
    \left|\int_{\mathbb{R}^{3}}(u^a\cdot\nabla)u^a\cdot \Delta u^a dx\right|
    \lesssim \|\nabla u^a\|_{L^\infty}\|\nabla u^a\|^2_{L^2}.
\end{align*}
Hence we have to combine the estimate of $\|\nabla u^a\|_{L^\infty}$ to close the energy estimates, which leads to the estimates of $\|\Gamma^a\|_{B^0_{\infty,1}}$ and  $\|\tau^a\|_{B^0_{\infty,1}}$. 
%The key to closing estimate for the global solutions also lies in the estimate of vorticity. For $d=3$, we derive that vorticity satisfies the following equation: 
We have
  \begin{align*}
  \partial_{t}\Gamma^a+u^a\cdot\nabla\Gamma^a+\frac{1}{2}\Gamma^a=(a-\frac{1}{2})\mathcal{R}\tau^a+\mathcal{R}Q(\nabla u^a,\tau^a)+[\mathcal{R},u^a\cdot\nabla]\tau^a+\omega^a\cdot\nabla u^a.
  \end{align*}
  For $d=3$, one can see that there is an external high-order nonlinear term $\omega^a\cdot\nabla u^a$ in the equation of $\Gamma^a$. Note that though $B^0_{\infty,1}$ embeds into $L^\infty$, which satisfies the following algebraic property:
 $$\|fg\|_{L^\infty}\leq \|f\|_{L^\infty}\|g\|_{L^\infty},$$ the critical Besov space $B^0_{\infty,1}$ itself is not an algebra. However, by virtue of the divergence structure of $\omega^a\cdot\nabla u^a$, we obtain the estimate: 
\begin{align*}
    \|\omega^a\cdot\nabla u^a\|_{B^0_{\infty,1}}\lesssim\|\omega^a\|_{B^0_{\infty,1}}\|u^a\|_{B^1_{\infty,1}}.
\end{align*}
  For any $a\in[0,1]$ and $d=3$, fortunately, we close the global estimate for vorticity $\omega^a$ through smallness condition in the critical Besov space $B^0_{\infty,1}$ using this observation. This observation also plays a crucial role in considering vanishing damping limit in the same topology for any $T>0$. 

\textbf{Difficulty (7):} Obtaining optimal decay rates and time integrability of $\|\nabla u^a\|_{B^0_{\infty,1}}$ with critical regularity in the case $d=3$.

From now on, we investigate optimal decay rate of global solutions and prove the time integrability of $\|\nabla u^a\|_{B^0_{\infty,1}}$ for the inviscid Oldroyd-B equation \eqref{eq0} with critical regularity, which will be the key to obtain the uniform vanishing damping limit. For $d=3$, optimal time decay rates of global strong solutions with high regularity for the inviscid Oldroyd-B models were proved in \cite{WWXZ22}. However, due to low regularity of the solutions, the case $a\in[0,1]$ is much more challenging. We will address the issues of consistent damping and low regularity by employing an improved Fourier splitting method alongside iteration techniques. We describe the iteration ideas in the following brief chart (The cubic term $\mathcal{N}(u,\tau)$ is defined in \eqref{5de3}).
\begin{equation*}
\begin{aligned}
 \text{Energy and Fourier splitting}
 &\longrightarrow \|(u,\tau)\|_{H^1}
   \lesssim (1+t)^{-1/8},\ (1+t)^{-1/4},\ (1+t)^{-3/8},\\
 \|(u,\tau)\|_{H^1}\lesssim(1+t)^{-3/8}
 &\longrightarrow \sup_t\|(u,\tau)\|_{\dot B^{-3/2}_{2,\infty}}<\infty
 \longrightarrow \|(u,\tau)\|_{H^1}\lesssim(1+t)^{-3/4},\\
 \|(u,\tau)\|_{H^1}\lesssim(1+t)^{-3/4}
 &\longrightarrow \text{key integrability}
 \longrightarrow \|\nabla(u,\tau)\|_2\lesssim(1+t)^{-5/4}.
\end{aligned}
\end{equation*}
\begin{rema}
Indeed, the following iteration techniques are essential in the mutual promotion between the decay rates and the low-frequency bounds in Fourier side. In the recent work of H. Jia, Z. Lei and the fourth author \cite{JLY24}, they have also discovered the mutual promotion between the decay rates and the spatial-weighted bounds on the physical side.
\end{rema}

To conclude the proof in the 3D case, we utilize the established optimal decay rates and time integrability to obtain the uniform vanishing damping limit for $d=3$. Furthermore, we discover that the rate of the uniform vanishing damping limit in $L^2$ is closely related to the time decay rate in  $L^2$, similar to the observations made in the 2D case.

\subsection{Organization of the presenting paper }
%{\rm\textbf{Structure of the paper:}}~~
In Section 2, we list the notation and preliminaries which will be used in the sequel. In Section 3, we establish the local well-posedness in the sense of Hadamard with critical regularity. In Section 4, by virtue of the sharp commutator estimate for the Calderon-Zygmund operator, we establish global existence
for low regularity with damping.
In Section 5, we prove the global existence and uniform vanishing damping limit for the 2D
inviscid Oldroyd-B model. In Section 6, we study the global existence and uniform vanishing damping limit for the inviscid Oldroyd-B model with $d=3$. In Section 7, we provide some numerical results regarding the vanishing damping limit in periodic domains.

\section{Notations and preliminaries}
In this section, we present our notations and introduce useful lemmas which will be used in the sequel.

\subsection{Notations} In this article, we adopt the following notation conventions:
\begin{itemize}
\item[(1)] We use $ A\lesssim B $ means $ A\leq C B $ for some absolute constant $ C $. The notation $ A\sim B $ means $ A\lesssim B $ and $ B\lesssim A $. Additionally, $ A \lesssim_{\star} B $ implies $ A\leq C B $ with $ C $ dependent on a specific quantity $ \star $, i.e., $ C=C(\star) $. 
\item[(2)] For any two quantities $X$ and $Y$, we shall denote $X\ll Y$ if
$X \le c Y$ for some sufficiently small constant $c$. The smallness of the constant $c$ is
usually clear from the context. The notation $X\gg Y$ is similarly defined. Note that
our use of $\ll$ and $\gg$ here is \emph{different} from the usual Vinogradov notation
in number theory or asymptotic analysis.
\item[(3)]For two vectors $u,v$, we denote the inner product of $u,v$ by $ u\cdot v $. More precisely,
\begin{equation*}
u\cdot v=\sum_{j}u_{j}v_{j}.
\end{equation*}
For two matrices $A,B$ we denote the inner product of $A,B$ by $A:B$, which is expressed by
\begin{equation*}
A:B=\sum_{i,j}A_{ij}B_{ij}.
\end{equation*}
\item[(4)] For a real-valued function $u:\Omega \to \R$ we denote its usual Lebesgue $L^p$-norm by
\begin{align*}
    \|u\|_{p}=\|u\|_{L^p(\Omega)}=\begin{cases}
       & \left(\int_{\Omega} |u|^p\ dx\right)^{\frac{1}{p}},\quad  1\le p<\infty;\\
       & \operatorname{esssup}_{x\in\Omega}|u(x)|,\quad p=\infty.
    \end{cases}
\end{align*}
\item[(5)] Similarly, we use the weak derivative in the following sense: For  $u$, $v\in L^1_{loc}(\Omega)$, (i.e they are locally integrable); $\forall\phi\in C^{\infty}_0(\Omega)$, i.e $\phi$ is infinitely differentiable (smooth) and compactly supported; and 
$$\int_{\Omega}u(x)\ \partial^{\alpha} \phi(x)\ dx=(-1)^{\alpha_1+\cdots+\alpha_n}\int_{\Omega} v(x)\ \phi(x)\ dx ,$$
then $v$ is defined to be the weak partial derivative of $u$, denoted by $\partial^\alpha u$. 
Suppose $u\in L^p(\Omega)$ and all weak derivatives $\partial^\alpha u$ exist for $|\alpha|=\alpha_1+\cdots+\alpha_n \leq k$ , such that $\partial^\alpha u\in L^p(\Omega)$ for $|\alpha|\leq k$, then we denote $u\in W^{k,p}(\Omega)$ to be the standard Sobolev space. The corresponding norm of $W^{k,p}(\Omega)$ is :
$$\| u\|_{W^{k,p}(\Omega)}=\left(\sum_{|\alpha|\leq k}\int_{\Omega}|\partial^\alpha u|^p\ dx\right)^{\frac{1}{p}}\ .$$
\noindent  For $p=2$ case, we use the convention $H^k(\Omega)$ to denote the space $W^{k,2}(\Omega)$. We often use $D^m u$ to denote any differential operator $D^\alpha u$ for any $|\alpha|=m$: $D^2$ denotes $\partial_{x_i}\partial_{x_j}u$ for $1\leq i , j\leq d$ in particular.  
\item[(6)] In this paper we use the following convention for Fourier transform on $\mathbb{R}^d$: 
$$f(x)=\frac{1}{(2\pi)^d}\int_{\xi\in \R^d}\hat{f}(\xi)e^{i\xi\cdot x}\ d\xi\ ,\ \widehat{f}(\xi)=\int_{\R^d}f(x)e^{-i\xi\cdot x}\ dx\ .$$
Taking advantage of the Fourier expansion, we use the well-known equivalent $H^s$-norm and $\dot{H}^s$-semi-norm of function $f$ by $$\| f\|_{H^s}=\frac{1}{(2\pi)^{d/2}}\|(1+|\xi|^2)^{\frac{s}{2}}\hat{f}(\xi)\|_{L^2}\ ,\ \| f\|_{\dot{H}^s}=\frac{1}{(2\pi)^{d/2}}\||\xi|^{s}\hat{f}(\xi)\|_{L^2}. $$
\item[(7)] We sometimes adopt the notation $\La =(-\Delta)^{\frac12}$, which can be understood from the Fourier side:
$$\widehat{\La f}(k)=|k|\widehat{f}(k).$$
Therefore, by Plancherel's identity, $\|f\|_{\dot H^s}=\|\La^s f\|_{L^2}$.
\item[(8)] We use the $\langle \cdot, \cdot \rangle$ to denote the $L^2$ inner product in $\mathbb{R}^d$, $\langle f, g\rangle := \int_{\mathbb{R}^d} f(x)g(x)dx$.  
\end{itemize}

\subsection{Preliminaries}
We now recall the Littlewood-Paley decomposition theory and Besov spaces.
\begin{lemm}\label{LPD}
Let $\mathcal{C}$ be the annulus $\{\xi\in\mathbb{R}^d:\frac 3 4\leq|\xi|\leq\frac 8 3\}$. There exist radial functions $\chi$ and $\varphi$, valued in the interval $[0,1]$, belonging respectively to $\mathscr{D}(B(0,\frac 4 3))$ and $\mathscr{D}(\mathcal{C})$, and such that
$$ \forall\xi\in\mathbb{R}^d,\ \chi(\xi)+\sum_{j\geq 0}\varphi(2^{-j}\xi)=1, $$
$$ \forall\xi\in\mathbb{R}^d\backslash\{0\},\ \sum_{j\in\mathbb{Z}}\varphi(2^{-j}\xi)=1, $$
$$ |j-j'|\geq 2\Rightarrow\mathrm{Supp}\ \varphi(2^{-j}\cdot)\cap \mathrm{Supp}\ \varphi(2^{-j'}\cdot)=\emptyset, $$
$$ j\geq 1\Rightarrow\mathrm{Supp}\ \chi(\cdot)\cap \mathrm{Supp}\ \varphi(2^{-j}\cdot)=\emptyset. $$
The set $\widetilde{\mathcal{C}}=B(0,\frac 2 3)+\mathcal{C}$ is an annulus, and we have
$$ |j-j'|\geq 5\Rightarrow 2^{j}\mathcal{C}\cap 2^{j'}\widetilde{\mathcal{C}}=\emptyset. $$
Further, we have
$$ \forall\xi\in\mathbb{R}^d,\ \frac 1 2\leq\chi^2(\xi)+\sum_{j\geq 0}\varphi^2(2^{-j}\xi)\leq 1, $$
$$ \forall\xi\in\mathbb{R}^d\backslash\{0\},\ \frac 1 2\leq\sum_{j\in\mathbb{Z}}\varphi^2(2^{-j}\xi)\leq 1. $$
\end{lemm}

$\mathscr{F}$ represents the Fourier transform and  its inverse is denoted by $\mathscr{F}^{-1}$.
Let $u$ be a tempered distribution in $\mathcal{S}'(\mathbb{R}^d)$. For all $j\in\mathbb{Z}$, define
$$
\Delta_j u=0\,\ \text{if}\,\ j\leq -2,~~
\Delta_{-1} u=\mathscr{F}^{-1}(\chi\mathscr{F}u),~~
\Delta_j u=\mathscr{F}^{-1}(\varphi(2^{-j}\cdot)\mathscr{F}u)\,\ \text{if}\,\ j\geq 0,~~
S_j u=\sum_{j'<j}\Delta_{j'}u.
$$
Then the Littlewood-Paley decomposition is given as follows:
$$ u=\sum_{j\in\mathbb{Z}}\Delta_j u \quad \text{in}\ \mathcal{S}'(\mathbb{R}^d). $$
Let $s\in\mathbb{R},\ 1\leq p,r\leq\infty.$ The non-homogeneous Besov spaces $B^s_{p,r}$ and $B^s_{p,r}(\mathcal{L}^{p'})$ are defined by
$$ B^s_{p,r}=\{u\in S':\|u\|_{B^s_{p,r}}=\Big\|(2^{js}\|\Delta_j u\|_{L^p})_j \Big\|_{l^r(\mathbb{Z})}<\infty\}.$$
For any positive time $T$, the Time-Space Besov Spaces are defined by
\begin{align*}
L^{\rho}_T(B^s_{p,r}) = \{u\in S':\|u\|_{L^{\rho}_T(B^s_{p,r})} = \left\|\left\|2^{js}\|\Delta_ju\|_{L^p}\right\|_{l^r({\mathbb{Z}})}\right\|_{L^{\rho}_T}<\infty\},
\end{align*}
and
\begin{align*}
\tilde{L}^{\rho}_T(B^s_{p,r}) = \{u\in S':\|u\|_{\tilde{L}^{\rho}_T(B^s_{p,r})} = \left\|2^{js}\|\Delta_ju\|_{L^{\rho}_T(L^p)}\right\|_{l^r({\mathbb{Z}})}<\infty\}.
\end{align*}
Moreover, the following embedding relationships hold:
$$
L^{\rho}_T(B^s_{p,r})\hookrightarrow \tilde{L}^{\rho}_T(B^s_{p,r}) ~~~~\text{if}~~~~r\geq\rho~~~~\text{and}~~~~
\tilde{L}^{\rho}_T(B^s_{p,r})\hookrightarrow L^{\rho}_T(B^s_{p,r})  ~~~~\text{if}~~~~r\leq\rho.$$

Let $u$ and $v$ be tempered distributions in $\mathscr{S}'$, then the non-homogeneous $\rm paraproduct$ of $v$ by $u$ is defined as follows:
\begin{align*}
    T_u v = \sum_{j} S_{j-1}u\Delta_j v,
\end{align*}
and the non-homogeneous $\rm remainder$ of $u$ and $v$ is defined by
\begin{align*}
    R(u,v) = \sum_{|k-j|\leq1} \Delta_{k} u \Delta_{j} v\triangleq\sum_{k\geq-1} \Delta_{k}u\tilde{\Delta}_k v.
\end{align*}
At least formally, we obtain the so-called Bony's decomposition:
\begin{align*}
    uv=T_u v + T_v u + R(u,v).
\end{align*}
\begin{lemm}\label{T}
    For any $(s,t)\in\mathbb{R}\times(-\infty,0)$ and $(p,p_1,p_2,r,r_1,r_2)\in[1,\infty]^6$, there exists a constant $C$ such that
    \begin{align*}
        \|T_u v\|_{B^s_{p,r}} \leq C^{1+|s|} \|u\|_{L^{p_1}}\|v\|_{B^s_{p_2,r}},
    \end{align*}
    with $(u,v)\in L^{p_1}\times B^s_{p_2,r}$ and $\frac{1}{p}=\frac{1}{p_1}+\frac{1}{p_2}$. Moreover, we have
    \begin{align*}
        \|T_u v\|_{B^{s+t}_{p,r}} \leq \frac{C^{1+|s+t|}}{-t} \|u\|_{B^{t}_{\infty,r_1}}\|v\|_{B^s_{p,r_2}},
    \end{align*}
    with $(u,v) \in B^{t}_{\infty,r_1}\times B^s_{p,r_2}$ and $\frac{1}{r}=\min\{1,\frac{1}{r_1}+\frac{1}{r_2}\}$.
\end{lemm}
\begin{lemm}\label{R}
    A constant $C$ exists which satisfies the following inequalities. Let $(s_1,s_2)\in\mathbb{R}^2$ and $(p_1,p_2,r_1,r_2) \in [1,\infty]^4$. Assume that $$\frac{1}{p}=\frac{1}{p_1}+\frac{1}{p_2}\leq1~~~~~\text{and}~~~~~\frac{1}{r}=\frac{1}{r_1}+\frac{1}{r_2}\leq1.$$
    If $s_1+s_2>0$, for any $(u,v)\in B^{s_1}_{p_1,r_1}\times B^{s_2}_{p_2,r_2}$, then we have
    \begin{align*}
        \|R(u,v)\|_{B^{s_1+s_2}_{p,r}} \leq \frac{C^{1+|s_1+s_2|}}{s_1+s_2} \|u\|_{B^{s_1}_{p_1,r_1}}\|v\|_{B^{s_2}_{p_2,r_2}}.
    \end{align*}
    If $r=1$ and $s_1+s_2=0$, for any $(u,v)\in B^{s_1}_{p_1,r_1}\times B^{s_2}_{p_2,r_2}$, then we have
    \begin{align*}
        \|R(u,v)\|_{B^{0}_{p,\infty}} \leq C \|u\|_{B^{s_1}_{p_1,r_1}}\|v\|_{B^{s_2}_{p_2,r_2}}.
    \end{align*}
\end{lemm}

\begin{lemm}\label{CI}
Let $1\leq p\leq\infty,\ 1\leq r\leq\infty,\ \sigma>-1-d\min(\frac 1 {p}, \frac 1 {p'})$ and $\rm div~u=0$. Define $I_j=[u\cdot\nabla, \Delta_j]f$. There exists a constant $C$ such that for $\sigma>0$, we have
$$\Big\|(2^{j\sigma}\|I_j\|_{L^p_{x}})_j\Big\|_{l^r(\mathbb{Z})}\leq C(\|\nabla u\|_{L^{\infty}}\|f\|_{B^{\sigma}_{p,r}}+\|\nabla u\|_{B^{\sigma-1}_{p,r}}\|\nabla_{x}f\|_{L^{\infty}_{x}}).$$
If $\sigma>1+\frac d p$ or $\sigma=1+\frac d p,~r=1$, we have
$$\Big\|(2^{j\sigma}\|I_j\|_{L^p_{x}})_j\Big\|_{l^r(\mathbb{Z})}\leq C\|\nabla u\|_{B^{\sigma-1}_{p,r}}\|f\|_{B^\sigma_{p,r}}.$$
If $\sigma<1+\frac d p$, we have
$$\Big\|(2^{j\sigma}\|I_j\|_{L^p_{x}})_j\Big\|_{l^r(\mathbb{Z})}\leq C\|\nabla u\|_{B^{\frac {d} {p}}_{p,r}\cap L^{\infty}}\|f\|_{B^\sigma_{p,r}}.$$
\end{lemm}

We have the following product rules.
\begin{lemm}\label{PL}
For any $\epsilon>0$, there exists $C>0$  such that
$$ \|uv\|_{B^0_{\infty,1}}\leq C(\|u\|_{L^{\infty}}\|v\|_{B^0_{\infty,1}}+\|u\|_{B^0_{\infty,1}}\|v\|_{B^\epsilon_{\infty,\infty}}).$$
Let ${\rm div}~u=0$. For nonlinear terms that satisfy the transportation structure, the accurate estimate hold 
$$ \|u\cdot \nabla v\|_{B^0_{\infty,1}}\leq C\|u\|_{B^0_{\infty,1}}\|v\|_{B^1_{\infty,1}}. $$
\end{lemm}
We introduce the following lemma to describe the action of the heat kernel.
\begin{lemm}\label{H}
Let $\mathcal{C}$ be an annulus. Positive constants $c$ and $C$ exist such that for any $p\in[1,+\infty]$ and any couple $(t,\lambda)$ of positive real numbers, we have
\begin{align*}
\operatorname{Supp}~\hat{u} \subset \lambda\mathcal{C} \Rightarrow \|e^{t\Delta}u\|_{L^p} \leq Ce^{-ct\lambda^2}\|u\|_{L^p}.
\end{align*}
\end{lemm}
The following commutator lemma is useful to estimate the structural unknown $\Gamma$.
\begin{lemm}\label{CR}
Let ${\rm div}~u=0$ and $\mathcal{R}=-(-\Delta)^{-1}{\rm curl}~{\rm div}$. For every $(r,p)\in[1,\infty]\times(1,\infty)$ and $\epsilon>0$, there exists a constant $C=C(r,p,\epsilon)$ such that
\begin{align*}
\|[\mathcal{R},u\cdot\nabla]\tau\|_{B^{0}_{\infty,r}}\leq C(\|\omega\|_{L^{\infty}}+\|\omega\|_{L^p})(\|\tau\|_{B^\epsilon_{\infty,r}}+\|\tau\|_{L^p}).
\end{align*}
\end{lemm}

The following facts can be proved by the standard properties of Calderon-Zygmund operator.
\begin{lemm}\label{CZ}
(1) Let $\operatorname{div}v=0$ and $\omega=\operatorname{curl}v$. For any $a\in[1,\infty)$ and $b\in[1,\infty]$, there exists positive constant $C$ such that
\begin{align*}
\|\Delta_{-1}\nabla v\|_{L^{\infty}} \leq C\min\{\|\omega\|_{L^a},\|v\|_{L^b}\}.
\end{align*}
(2) For all $s\in\mathbb{R}$ and $1\leq p,r\leq\infty$, there exists a constant $C$ such that
\begin{align*}
\|(Id-\Delta_{-1}) \mathcal{R}f\|_{B^s_{p,r}} \leq C\|f\|_{B^s_{p,r}}.
\end{align*}
\end{lemm}

We also present following estimates for the transport equations, which will be frequently used in the subsequent chapters.
\begin{lemm}\label{TP}
Suppose $f$ satisfies the following transport equation:
\begin{align*}
\left\{
\begin{array}{l}
 \partial_t f + v\cdot\nabla f = g\\
 f|_{t=0}=f_0
\end{array}
\right.
\end{align*}
with $v$ divergence free. Then the following estimates hold:
\begin{align*}
  \|f\|_{B^\sigma_{\infty,1}} \leq C \|f_0\|_{B^\sigma_{\infty,1}} e^{C\int_0^t \|\nabla v\|_{B^0_{\infty,1}}ds } + C\int_0^t e^{C\int_s^{t}\|\nabla v\|_{B^0_{\infty,1}}dt'} \|g\|   
_{B^\sigma_{\infty,1}} ds, 
\end{align*}
where $\sigma\in [0,1]$.
\end{lemm}

The proof to the lemmas above are standard and we refer the readers to Bahouri-Chemin-Danchin \cite{BCD11}, Guo-Li-Yin \cite{GLY19}, Hmidi-Keraani-Rousset \cite{HKR11}, Li \cite{L19} and Li-Sire \cite{LS23} for the details.

\section{Local well-posedness with critical regularity}
In this section we show the local well-posedness of \eqref{eq} in both 2D and 3D cases. In fact as will be clear, our proof does not depend on the dimension. To start with we introduce the following lemma which will be useful in dealing with the pressure term $\nabla P$.
\begin{lemm}\cite{BCD11}\label{P}
For any $-1<\sigma\leq 1$, there exists a constant $C$ such that if $\rm div~u=\rm div~v=0$, then
\begin{align*}
\|\nabla(-\Delta)^{-1}\rm div(u\cdot\nabla v)\|_{B^\sigma_{\infty,1}}\leq C\min(\|u\|_{B^\sigma_{\infty,1}}\|v\|_{B^1_{\infty,1}},\|u\|_{B^1_{\infty,1}}\|v\|_{B^\sigma_{\infty,1}}).
\end{align*}
\end{lemm}
We are interested in the well-posedness of $\eqref{eq0}$ in the sense of Hadamard. Since the local estimates are uniformly in $a$, we may drop the upper index $a$ of the solutions for the inviscid Oldroyd-B model \eqref{eq0} without any confusions. 
For the entire-space problem in the inhomogeneous Besov classes below,
we fix the pressure gradient by
\begin{equation}\label{3pressure-normalization}
 \nabla P=\nabla\Delta^{-1}\operatorname{div}\operatorname{div}
          (\tau-u\otimes u).
\end{equation}
The multiplier is defined by the usual low--high frequency localization
and annihilates spatial constants; no additional harmonic pressure
gradient is allowed. All uniqueness statements in these classes refer
to this normalization. The low-frequency part is well defined on bounded
inputs by the integrable kernel used in \eqref{3ineq6} below.
\begin{defi}[Well-posedness]
    The Cauchy problem \eqref{eq0} is said to be locally well-posed
in a metric space $\mathbb{X}$ if for any $(u_0,\tau_0)\in \mathbb{X}$, there exists a neighbourhood $U$ of $(u_0,\tau_0)$
and $T = T(U) > 0$ such that for any $(v_0,\sigma_0)\in U$, we have \\
(i) (Existence): There exists a distributional solution $(u,\tau)$ in $E$ for some $E \subset C([0, T): \mathbb{X})$ to \eqref{eq0}. \\
(ii) (Uniqueness): The solution $(u,\tau)$ is unique in $E$. \\
(iii) (Continuous dependence): The solution map $(v_0,\sigma_0)\rightarrow(v, \sigma)$ is continuous from
$U$ to $C([0, T): \mathbb{X})$.
\end{defi}
We now establish local well-posedness for \eqref{eq0} in the sense of Hadamard with critical regularity.
\begin{prop}\label{3prop1}
Let $d\geq 2$ and $a\in[0,1]$. For any divergence-free $u_0\in B^1_{\infty,1}$ and symmetric $\tau_0\in B^0_{\infty,1}$, there exists a time $T>0$, which is independent of damping coefficient $a$, such that $(\ref{eq0})$ has a  unique solution $(u,\tau)\in C([0,T);B^1_{\infty,1})\times \bigl(C([0,T);B^0_{\infty,1})\cap L^{1}([0,T);B^2_{\infty,1})\bigr)$. Moreover, the
solution depends continuously on the initial data.
\end{prop}
\begin{rema}
    For the transport equation, the velocity $u\in C^{0,1}$ is the key to persisting regularity. Notice that $B^1_{\infty,1}\hookrightarrow C^{0,1}$. Under the conditions in Proposition \ref{3prop1}, if additionally $(u_0,\tau_0)\in H^k$ with $k\geq 0$, one can deduce that the solution $(u,\tau)\in C([0,T); H^k)$. This means that we establish local well-posedness for \eqref{eq0} in $(B^1_{\infty,1}\cap H^k)\times(B^0_{\infty,1}\cap H^k)$.
\end{rema}
\begin{proof}
Let $(u^{0},\tau^{0})=(S_0u_0,S_0\tau_0)$. We define a sequence $(u^{n},\tau^{n})_{n\in \mathbb{N}}$ by solving the following
linear equations:
\begin{align}\label{3eq1}
\left\{\begin{array}{l}
\partial_tu^{n+1}+u^{n}\cdot\nabla u^{n+1}+\nabla P^{n+1} ={\rm div}~\tau^{n+1},~~~~{\rm div}~u^{n+1}=0,\\[1ex]
\partial_t\tau^{n+1}+u^{n}\cdot\nabla\tau^{n+1}+a\tau^{n+1}+Q(\nabla u^{n},\tau^{n+1})=D(u^{n})+\Delta\tau^{n+1},\\[1ex]
u^{n+1}|_{t=0}=S_{n+1}u_0,~~\tau^{n+1}|_{t=0}=S_{n+1}\tau_0. \\[1ex]
\end{array}\right.
\end{align}
\textbf{Step 1: Uniform bounds.} \\
Applying $\Delta_j$ to $\eqref{3eq1}_1$, for any $j\geq -1$, we have
\begin{align}\label{3eq2}
\partial_t\Delta_ju^{n+1} + u^{n}\cdot\nabla\Delta_ju^{n+1} = -[\Delta_j,u^{n}\cdot\nabla]u^{n+1} + \Delta_j{\rm div}~\tau^{n+1} - \Delta_j\nabla P^{n+1}.
\end{align}
Integrating $(\ref{3eq2})$ on $[0,T]$,  we infer that
\begin{align}\label{3ineq1}
\|\Delta_ju^{n+1}\|_{L^{\infty}} &\leq \|\Delta_ju_0\|_{L^{\infty}}+ C\int_0^T \|[\Delta_j,u^{n}\cdot\nabla]u^{n+1}\|_{L^{\infty}}dt \\ \notag
&~~~+C\int_0^T \|\Delta_j{\rm div}\tau^{n+1}\|_{L^{\infty}} + \|\Delta_j\nabla P^{n+1}\|_{L^{\infty}}dt.
\end{align}
Multiplying both sides of \eqref{3ineq1} by $2^j$ and summing up $j$, we obtain
\begin{align}\label{3ineq2}
\|u^{n+1}\|_{B^1_{\infty,1}} &\leq \|u_0\|_{B^1_{\infty,1}}+ C\int_0^T \sum_{j\geq -1}2^j\|[\Delta_j,u^{n}\cdot\nabla]u^{n+1}\|_{L^{\infty}}dt \\ \notag
&~~~+C\int_0^T \|\tau^{n+1}\|_{B^2_{\infty,1}} + \|\nabla P^{n+1}\|_{B^1_{\infty,1}}dt.
\end{align}
By virtue of Lemma \ref{CI}, we have
\begin{align*}
\sum_{j\geq -1}2^j\|[\Delta_j,u^{n}\cdot\nabla]u^{n+1}\|_{L^{\infty}}\lesssim\|u^{n}\|_{B^1_{\infty,1}}\|u^{n+1}\|_{B^1_{\infty,1}}.
\end{align*}
Since ${\rm div}~u^{n}={\rm div}~u^{n+1}=0$, we obtain
\begin{align*}
\nabla P^{n+1}= \nabla\Delta^{-1}{\rm div}~{\rm div}~(\tau^{n+1}-u^{n}\otimes u^{n+1}).
\end{align*}
By virtue of Lemma \ref{P}, we have
\begin{align*}
\|\nabla\Delta^{-1}{\rm div}~{\rm div}~(u^{n}\otimes u^{n+1})\|_{B^1_{\infty,1}}\lesssim\|u^{n}\|_{B^1_{\infty,1}}\|u^{n+1}\|_{B^1_{\infty,1}}.
\end{align*}
By Lemma \ref{CZ}, we deduce that
\begin{align*}
\|(Id-\Delta_{-1})\nabla(-\Delta)^{-1}{\rm div}~{\rm div}~\tau^{n+1}\|_{B^1_{\infty,1}}\lesssim\|\tau^{n+1}\|_{B^2_{\infty,1}},
\end{align*}
and
\begin{align}\label{3ineq6}
\|\Delta_{-1}\nabla(-\Delta)^{-1}{\rm div}~{\rm div}~\tau^{n+1}\|_{B^1_{\infty,1}}
&\lesssim\|\Delta_{-1}\nabla(-\Delta)^{-1}{\rm div}~{\rm div}~\tau^{n+1}\|_{L^\infty} \\ \notag
&\lesssim\|\theta E_d\ast\Delta_{-1}\nabla{\rm div}~{\rm div}~\tau^{n+1}\|_{L^\infty}\\ \notag
&~~~+\|\nabla{\rm div}~{\rm div}~(1-\theta) E_d\ast\Delta_{-1}\tau^{n+1}\|_{L^\infty}\\ \notag
&\lesssim\|\Delta_{-1}\tau^{n+1}\|_{L^\infty}\\ \notag
&\lesssim\|\tau^{n+1}\|_{B^2_{\infty,1}}.
\end{align}
Hence, we deduce from \eqref{3ineq2}-\eqref{3ineq6} that
\begin{align}\label{3ineq7}
\|u^{n+1}\|_{L^{\infty}_T(B^1_{\infty,1})} \leq \|u_0\|_{B^1_{\infty,1}} + CT\|u^{n}\|_{L^{\infty}_T(B^1_{\infty,1})}\|u^{n+1}\|_{L^{\infty}_T(B^1_{\infty,1})}  + C\int_0^T \|\tau^{n+1}\|_{B^2_{\infty,1}} dt.
\end{align}
Applying $\Delta_j$ to $\eqref{3eq1}_2$, we obtain
\begin{align*}
\Delta_j\tau^{n+1} = e^{\Delta t-at}\Delta_j\tau^{n+1}_0 + \int_0^t e^{\Delta(t-s)-a(t-s)}\Delta_j(-Q(\nabla u^{n},\tau^{n+1})-u^{n}\cdot\nabla\tau^{n+1}+D(u^{n}))ds.
\end{align*}
By virtue of Lemmas \ref{H} and \ref{PL}, for any $a\in[0,1]$, we infer that
\begin{align*}
&\|\tau^{n+1}\|_{L^{\infty}_T(B^0_{\infty,1})}+ \|\tau^{n+1}\|_{L^{1}_T(B^2_{\infty,1})}\\
&\leq \|\tau^{n+1}\|_{L^{\infty}_T(B^0_{\infty,1})}+\int_0^T \sum_{j\geq 0}2^{2j}\|\Delta_j \tau^{n+1}\|_{L^{\infty}}dt\\ \notag
&~~~+\int_0^T 2^{-2}\|\Delta_{-1} \tau^{n+1}\|_{L^{\infty}} dt \\ \notag
&\leq (1+T)\|\tau^{n+1}\|_{L^{\infty}_T(B^0_{\infty,1})}+ \int_0^T \sum_{j\geq 0}2^{2j}\|\Delta_j \tau^{n+1}\|_{L^{\infty}}dt \\ \notag
&\leq C(1+T)\big(\|\tau_0\|_{B^0_{\infty,1}} + \int_0^T \|Q(\nabla u^{n},\tau^{n+1})\|_{B^0_{\infty,1}}\\ \notag
&~~~+\|u^{n}\cdot\nabla\tau^{n+1}\|_{B^0_{\infty,1}}+\|\nabla u^{n}\|_{B^0_{\infty,1}}dt\big)\\ \notag
&\leq C(1+T)\big(\|\tau_0\|_{B^0_{\infty,1}} + C\|u^n\|_{L^{\infty}_T(B^1_{\infty,1})}(T+T^{\frac{1}{2}}\|\tau^{n+1}\|_{L^2_T(B^1_{\infty,1})})\big).
\end{align*}
Note that
\begin{align*}
\|\tau^{n+1}\|_{L^2_T(B^1_{\infty,1})} \lesssim \|\tau^{n+1}\|^{\frac{1}{2}}_{L^{\infty}_T(B^0_{\infty,1})}\|\tau^{n+1}\|^{\frac{1}{2}}_{L^1_T(B^2_{\infty,1})}.
\end{align*}
Assuming that
\begin{align*}
\|u^{n}\|_{L^{\infty}_T(B^1_{\infty,1})}+\|\tau^{n}\|_{L^{\infty}_T(B^0_{\infty,1})} + \|\tau^{n}\|_{L^1_T(B^2_{\infty,1})} \leq 4C(\|u_0\|_{B^1_{\infty,1}} + \|\tau_0\|_{B^0_{\infty,1}}),
\end{align*}
and
\begin{align}\label{3ineq11}
T=\min\left\{\frac 1 {100C^2},\frac{1}{100C^4(1+\|u_0\|_{B^1_{\infty,1}} + \|\tau_0\|_{B^0_{\infty,1}})^4}\right\},
\end{align}
where $C>1$. Then we conclude from \eqref{3ineq7}-\eqref{3ineq11} that
\begin{align*}
\|u^{n+1}\|_{L^{\infty}_T(B^1_{\infty,1})}+\|\tau^{n+1}\|_{L^{\infty}_T(B^0_{\infty,1})} + \|\tau^{n+1}\|_{L^1_T(B^2_{\infty,1})} \leq 3C(\|u_0\|_{B^1_{\infty,1}} + \|\tau_0\|_{B^0_{\infty,1}}).
\end{align*}
We thus finish the proof of the uniform bounds for sequence $(u^{n},\tau^{n})_{n\in \mathbb{N}}$.  \\
\textbf{Step 2: Convergence.} \\
Denote $\bar{f}^n\triangleq f^{n+1}-f^{n}$. For $n\geq1$, subtracting two consecutive systems in \eqref{3eq1} gives
\begin{align*}
\left\{\begin{array}{l}
\partial_t\bar{u}^n+u^{n}\cdot\nabla \bar{u}^n+\nabla \bar{P}^n ={\rm div}~\bar{\tau}^n-\bar{u}^{n-1}\cdot\nabla u^n,~~~~{\rm div}~\bar{u}^n=0,\\[1ex]
\partial_t\bar{\tau}^n+u^{n}\cdot\nabla\bar{\tau}^n+a\bar{\tau}^n+Q(\nabla u^{n},\bar{\tau}^n)+Q(\nabla \bar{u}^{n-1},\tau^n)=D(\bar{u}^{n-1})+\Delta\bar{\tau}^n-\bar{u}^{n-1}\cdot\nabla\tau^n,\\[1ex]
\bar{u}^n|_{t=0}=\Delta_{n}u_0,~~\bar{\tau}^n|_{t=0}=\Delta_{n}\tau_0. \\[1ex]
\end{array}\right.
\end{align*}
Using Lemmas \ref{CI} and \ref{P}, we obtain
\begin{align}\label{3ineq13}
\|\bar{u}^n\|_{L^{\infty}_T(B^0_{\infty,1})}&\leq \|\Delta_{n} u_0\|_{B^0_{\infty,1}} + CT\|u^{n}\|_{L^{\infty}_T(B^1_{\infty,1})}(\|\bar{u}^n\|_{L^{\infty}_T(B^0_{\infty,1})}+\|\bar{u}^{n-1}\|_{L^{\infty}_T(B^0_{\infty,1})}) \\ \notag
&~~~+C\int_0^T \|\bar{\tau}^n\|_{B^1_{\infty,1}}dt.
\end{align}
By virtue of Lemmas \ref{H} and \ref{PL}, for any $a\in[0,1]$, we infer that
\begin{align}\label{3ineq14}
\|\bar{\tau}^n\|_{L^{\infty}_T(B^{-1}_{\infty,1})}+ \|\bar{\tau}^n\|_{L^{1}_T(B^1_{\infty,1})} &\leq \|\bar{\tau}^n\|_{L^{\infty}_T(B^{-1}_{\infty,1})}+\int_0^T \sum_{j\geq 0}2^{j}\|\Delta_j \bar{\tau}^n\|_{L^{\infty}}dt\\ \notag
&~~~+\int_0^T 2^{-1}\|\Delta_{-1} \bar{\tau}^n\|_{L^{\infty}} dt \\ \notag
&\leq (1+T)\|\bar{\tau}^n\|_{L^{\infty}_T(B^{-1}_{\infty,1})}+ \int_0^T \sum_{j\geq 0}2^{j}\|\Delta_j \bar{\tau}^n\|_{L^{\infty}}dt \\ \notag
&\leq C\big(\|\Delta_{n}\tau_0\|_{B^{-1}_{\infty,1}} + \int_0^T \|Q(\nabla u^{n},\bar{\tau}^n)\|_{B^{-1}_{\infty,1}}+\|Q(\nabla \bar{u}^{n-1},\tau^n)\|_{B^{-1}_{\infty,1}}\\ \notag
&~~~+\|u^{n}\cdot\nabla\bar{\tau}^n\|_{B^{-1}_{\infty,1}}+\|\bar{u}^{n-1}\cdot\nabla\tau^n\|_{B^{-1}_{\infty,1}}+\| \bar{u}^{n-1}\|_{B^0_{\infty,1}}dt\big)\\ \notag
&\leq C\|\Delta_{n}\tau_0\|_{B^{-1}_{\infty,1}} +C\|\bar{u}^{n-1}\|_{L^{\infty}_T(B^0_{\infty,1})}(T+\|\tau^{n}\|_{L^1_T(B^{\frac 1 2}_{\infty,1})})\\ \notag
&~~~+C\int_0^T \|Q(\nabla u^{n},\bar{\tau}^n)\|_{B^{0}_{\infty,1}}+\|Q(\nabla \bar{u}^{n-1},\tau^n)\|_{B^{-1}_{\infty,1}}dt\\ \notag
&~~~+C\|u^{n}\|_{L^{\infty}_T(B^{1}_{\infty,1})}\|\bar{\tau}^n\|_{L^{1}_T(B^{0}_{\infty,1})}\\ \notag
&\leq C\|\Delta_{n}\tau_0\|_{B^{-1}_{\infty,1}}+C\|\bar{u}^{n-1}\|_{L^{\infty}_T(B^0_{\infty,1})}(T+\|\tau^{n}\|_{L^1_T(B^{\frac 1 2}_{\infty,1})})\\ \notag
&~~~+C\|u^{n}\|_{L^{\infty}_T(B^{1}_{\infty,1})}\|\bar{\tau}^n\|_{L^1_T(B^{\frac 1 2}_{\infty,1})}+C\int_0^T \|Q(\nabla \bar{u}^{n-1},\tau^n)\|_{B^{-1}_{\infty,1}}dt.
\end{align}
Using Lemma \ref{PL}, we have
\begin{align}\label{3ineq15}
\|Q(\nabla \bar{u}^{n-1},\tau^n)\|_{B^{-1}_{\infty,1}}&\lesssim \|\partial_kQ(\bar{u}^{n-1},\tau^n)\|_{B^{-1}_{\infty,1}}+\|Q(\bar{u}^{n-1},\partial_k\tau^n)\|_{B^{-1}_{\infty,1}}  \\ \notag
&\lesssim \|Q(\bar{u}^{n-1},\tau^n)\|_{B^{0}_{\infty,1}}+\|Q(\bar{u}^{n-1},\nabla\tau^n)\|_{B^{0}_{\infty,1}} \\ \notag
&\lesssim \|\bar{u}^{n-1}\|_{B^{0}_{\infty,1}}\|\tau^n\|_{B^{\frac 3 2}_{\infty,1}}.
\end{align}
Set
\[
 M=\|u_0\|_{B^1_{\infty,1}}+\|\tau_0\|_{B^0_{\infty,1}},\qquad
 U_n=\|\bar u^n\|_{L^\infty_T(B^0_{\infty,1})},
\]
\[
 Z_n=\|\bar\tau^n\|_{L^\infty_T(B^{-1}_{\infty,1})}
       +\|\bar\tau^n\|_{L^1_T(B^1_{\infty,1})},\qquad
 \rho(T)=C\bigl(T+MT^{1/4}\bigr).
\]
Interpolation and Step~1 imply
\[
 \|\tau^n\|_{L^1_T(B^{3/2}_{\infty,1})}
 \lesssim T^{1/4}
 \|\tau^n\|_{L^\infty_T(B^0_{\infty,1})}^{1/4}
 \|\tau^n\|_{L^1_T(B^2_{\infty,1})}^{3/4}
 \lesssim MT^{1/4}.
\]
Combining \eqref{3ineq14}--\eqref{3ineq15} and absorbing the terms
containing $Z_n$, we obtain, for $n\ge1$,
\begin{align}\label{3ineq16}
 Z_n\le CM2^{-n}+\rho(T)U_{n-1}.
\end{align}
The term $CT$ in $\rho(T)$ accounts for the linear source
$D(\bar u^{n-1})$. Together with \eqref{3ineq13}, this gives, after
increasing $C$ in the definition of $\rho(T)$,
\[
 U_n\le CM2^{-n}+\rho(T)U_{n-1}.
\]
Choose the constant in \eqref{3ineq11} sufficiently large so that
$\rho(T)\le1/2$. The zeroth difference satisfies $U_0+Z_0\le CM$
directly by Step~1. Consequently,
\[
 \sum_{n=0}^mU_n\le CM+\rho(T)\sum_{n=0}^{m-1}U_n,
 \qquad
 \sum_{n=0}^{\infty}(U_n+Z_n)\le CM.
\]
Hence, we deduce that $(u^{n},\tau^{n})_{n\in \mathbb{N}}$ is a Cauchy sequence in  $L^{\infty}([0,T);B^0_{\infty,1})\times L^{\infty}([0,T);B^{-1}_{\infty,1})\cap L^{1}([0,T);B^1_{\infty,1})$.
Uniform bounds for $(u^{n},\tau^{n})_{n\in \mathbb{N}}$ in Besov spaces ensure that there exists $(u,\tau)\in L^{\infty}([0,T);B^1_{\infty,1})\times L^{\infty}([0,T);B^{0}_{\infty,1})\cap L^{1}([0,T);B^2_{\infty,1})$ satisfying \eqref{eq0} such that
\begin{align*}
u^n\rightarrow u~~~~\text{in}~~~~L^{\infty}([0,T);B^{s'}_{\infty,1})~~~~\text{and}~~~~\tau^n\rightarrow \tau~~~~\text{in}~~~~L^{\infty}([0,T);B^{s'-1}_{\infty,1})\cap L^{1}([0,T);B^{s'+1}_{\infty,1}),
\end{align*}
where $s'<1$. Note that the solutions $(u,\tau)$ also belong to $C([0,T);B^1_{\infty,1})\times C([0,T);B^{0}_{\infty,1})$. Moreover, if $(u_0,\tau_0)\in B^s_{\infty,1}\times B^{s-1}_{\infty,1}$ for any $s>1$, similar to \textbf{step 1},
then there exists $C>0$, depending on $s,d$ and the bounded critical norms but not on the higher initial norms, such that
\begin{align}\label{3ineq18}
\|u\|_{L^{\infty}_T(B^s_{\infty,1})}+\|\tau\|_{L^{\infty}_T(B^{s-1}_{\infty,1})} + \|\tau\|_{L^1_T(B^{s+1}_{\infty,1})} \leq C(\|u_0\|_{B^s_{\infty,1}} + \|\tau_0\|_{B^{s-1}_{\infty,1}}).
\end{align}
\textbf{Step 3: Uniqueness.} \\
Let $(u,\tau)$ and $(v,\sigma)$ be two solutions of $\eqref{eq0}$ with the initial data $(u_0,\tau_0)$ and $(v_0,\sigma_0)$. Then we get 
\begin{align}\label{3eq5}
\left\{\begin{array}{l}
\partial_t(u-v)+u\cdot\nabla (u-v)+\nabla (P_u-P_v) ={\rm div}~(\tau-\sigma)-(u-v)\cdot\nabla v,~~~~{\rm div}~(u-v)=0,\\[1ex]
\partial_t(\tau-\sigma)+u\cdot\nabla(\tau-\sigma)-\Delta(\tau-\sigma)+a(\tau-\sigma)+Q(\nabla u,\tau-\sigma)\\[1ex]
=D(u-v)-(u-v)\cdot\nabla\sigma-Q(\nabla (u-v),\sigma),\\[1ex]
(u-v)|_{t=0}=u_0-v_0,~~(\tau-\sigma)|_{t=0}=\tau_0-\sigma_0. \\[1ex]
\end{array}\right.
\end{align}
Using Lemmas \ref{CI} and \ref{P}, we obtain
\begin{align*}
\|u-v\|_{L^{\infty}_T(B^0_{\infty,1})}&\leq \|u_0-v_0\|_{B^0_{\infty,1}}+C\int_0^T \|u-v\|_{B^0_{\infty,1}}(\|u\|_{B^1_{\infty,1}}+\|v\|_{B^1_{\infty,1}})+\|\tau-\sigma\|_{B^1_{\infty,1}}dt \\ \notag
&\leq\|u_0-v_0\|_{B^0_{\infty,1}}+\frac 1 2\|u-v\|_{L^{\infty}_T(B^0_{\infty,1})}+C\|\tau-\sigma\|_{L^{1}_T(B^1_{\infty,1})},
\end{align*}
which implies that
\begin{align}\label{3ineq19}
\|u-v\|_{L^{\infty}_T(B^0_{\infty,1})}\leq C\|u_0-v_0\|_{B^0_{\infty,1}}+C\|\tau-\sigma\|_{L^{1}_T(B^1_{\infty,1})}.
\end{align}
By virtue of Lemmas \ref{H} and \ref{PL}, for any $a\in[0,1]$, we infer from \eqref{3ineq15} that
\begin{align}\label{3ineq20}
\|\tau-\sigma\|_{L^{\infty}_T(B^{-1}_{\infty,1})}+\|\tau-\sigma\|_{L^{1}_T(B^1_{\infty,1})} 
&\leq (1+T)\|\tau-\sigma\|_{L^{\infty}_T(B^{-1}_{\infty,1})}+ \int_0^T \sum_{j\geq 0}2^{j}\|\Delta_j (\tau-\sigma)\|_{L^{\infty}}dt \\ \notag
&\leq C\big(\|\tau_0-\sigma_0\|_{B^{-1}_{\infty,1}} + \int_0^T \|Q(\nabla u,\tau-\sigma)\|_{B^{-1}_{\infty,1}}\\ \notag
&~~~+\|Q(\nabla (u-v),\sigma)\|_{B^{-1}_{\infty,1}}+\|u\cdot\nabla(\tau-\sigma)\|_{B^{-1}_{\infty,1}}\\ \notag
&~~~+\|(u-v)\cdot\nabla\sigma\|_{B^{-1}_{\infty,1}}+\|u-v\|_{B^0_{\infty,1}}dt\big)\\ \notag
&\leq C\|\tau_0-\sigma_0\|_{B^{-1}_{\infty,1}}+C\int_0^T \|Q(\nabla (u-v),\sigma)\|_{B^{-1}_{\infty,1}}dt\\ \notag
&~~~+C\int_0^T \|u-v\|_{B^{0}_{\infty,1}}(1+\|\sigma\|_{B^{1}_{\infty,1}})+\|u\|_{B^{1}_{\infty,1}}\|\tau-\sigma\|_{B^{0}_{\infty,1}}dt\\ \notag
&\leq C\|\tau_0-\sigma_0\|_{B^{-1}_{\infty,1}}+C\int_0^T \|u-v\|_{B^{0}_{\infty,1}}(1+\|\sigma\|_{B^{\frac 3 2}_{\infty,1}})dt\\ \notag
&~~~+\int_0^T\|u\|_{B^{1}_{\infty,1}}\|\tau-\sigma\|_{B^{\frac 1 2}_{\infty,1}}dt\\ \notag
&\leq C\|\tau_0-\sigma_0\|_{B^{-1}_{\infty,1}}+\frac 1 {4C}\|u-v\|_{L^{\infty}_T(B^0_{\infty,1})}\\ \notag
&~~~+\frac 1 4(\|\tau-\sigma\|_{L^{\infty}_T(B^{-1}_{\infty,1})}+\|\tau-\sigma\|_{L^{1}_T(B^1_{\infty,1})}).
\end{align}
Combining \eqref{3ineq19} and \eqref{3ineq20}, we conclude that
\begin{align}\label{3ineq21}
\|u-v\|_{L^{\infty}_T(B^0_{\infty,1})}+\|\tau-\sigma\|_{L^{\infty}_T(B^{-1}_{\infty,1})\cap L^{1}_T(B^1_{\infty,1})}\leq C\|u_0-v_0\|_{B^0_{\infty,1}}+C\|\tau_0-\sigma_0\|_{B^{-1}_{\infty,1}}.
\end{align}
We thus prove the uniqueness for \eqref{eq0} on $[0,T]$. \\
\textbf{Step 4: Continuous dependence.} \\
Assume that $(u,\tau)$ and $(v,\sigma)$ are two solutions of $\eqref{eq0}$ with the initial data $(u_0,\tau_0)$ and $(v_0,\sigma_0)$. Due to the difficulty of exceeding regularity, we introduce a smoothing operator $S_N$ to quantify regularity, see \cite{BS75}. Let $(u_N,\tau_N)$ and $(v_N,\sigma_N)$ be two solutions of $\eqref{eq0}$ with the initial data $(S_N u_0,S_N\tau_0)$ and $(S_N v_0,S_N\sigma_0)$. Similar to \textbf{Step 3}, we deduce from \eqref{3ineq21} that 
\begin{align}\label{3ineq22}
&\|u-u_N\|_{L^{\infty}_T(B^0_{\infty,1})}+\|\tau-\tau_N\|_{L^{\infty}_T(B^{-1}_{\infty,1})\cap L^{1}_T(B^1_{\infty,1})}\\ \notag
\leq&C\|(Id-S_N)u_0\|_{B^0_{\infty,1}}+C\|(Id-S_N)\tau_0\|_{B^{-1}_{\infty,1}} \\ \notag
\leq&C2^{-N}(\|(Id-S_N)u_0\|_{B^1_{\infty,1}}+\|(Id-S_N)\tau_0\|_{B^{0}_{\infty,1}}).
\end{align}
and 
\begin{align*}
&\|v-v_N\|_{L^{\infty}_T(B^0_{\infty,1})}+\|\sigma-\sigma_N\|_{L^{\infty}_T(B^{-1}_{\infty,1})\cap L^{1}_T(B^1_{\infty,1})}\\ \notag
\leq&C\|(Id-S_N)v_0\|_{B^0_{\infty,1}}+C\|(Id-S_N)\sigma_0\|_{B^{-1}_{\infty,1}} \\ \notag
\leq&C2^{-N}(\|(Id-S_N)v_0\|_{B^1_{\infty,1}}+\|(Id-S_N)\sigma_0\|_{B^{0}_{\infty,1}}).
\end{align*}
Replacing $(v,\sigma)$ with $(u_N,\tau_N)$ in \eqref{3eq5}, by Lemmas \ref{CI} and \ref{P}, we obtain
\begin{align*}
\|(u-u_N)\|_{L^{\infty}_T(B^1_{\infty,1})}&\leq \|(Id-S_N)u_0\|_{B^1_{\infty,1}}+C\int_0^T \|u-u_N\|_{B^1_{\infty,1}}(\|u\|_{B^1_{\infty,1}}+\|u_N\|_{B^1_{\infty,1}})dt\\ \notag
&~~~+C\int_0^T\|(u-u_N)\cdot\nabla u_N\|_{B^1_{\infty,1}}+\|\tau-\tau_N\|_{B^2_{\infty,1}}dt.
\end{align*}
According to Lemmas \ref{T} and \ref{R}, we have
\begin{align*}
\|(u-u_N)\cdot\nabla u_N\|_{B^1_{\infty,1}}&\lesssim \|u-u_N\|_{B^1_{\infty,1}}\|u_N\|_{B^1_{\infty,1}}+\|u-u_N\|_{B^0_{\infty,1}}\|u_N\|_{B^2_{\infty,1}}.
\end{align*}
According to \eqref{3ineq22}, \eqref{3ineq18}, we infer that
\begin{align*}
\|u-u_N\|_{B^0_{\infty,1}}\|u_N\|_{B^2_{\infty,1}}
&\lesssim(\|(Id-S_N)u_0\|_{B^1_{\infty,1}}+\|(Id-S_N)\tau_0\|_{B^{0}_{\infty,1}})(\|u_0\|_{B^1_{\infty,1}}+\|\tau_0\|_{B^{0}_{\infty,1}}).
\end{align*}
The above inequalities ensure that
\begin{align}\label{3ineq24}
\|u-u_N\|_{L^{\infty}_T(B^1_{\infty,1})}\leq C(\|(Id-S_N)u_0\|_{B^1_{\infty,1}}+\|(Id-S_N)\tau_0\|_{B^{0}_{\infty,1}})+C\|\tau-\tau_N\|_{L^{1}_T(B^2_{\infty,1})}.
\end{align}
By virtue of Lemmas \ref{H} and \ref{PL}, for any $a\in[0,1]$, we infer that
\begin{align}\label{3ineq25}
\|\tau-\tau_N\|_{L^{\infty}_T(B^{0}_{\infty,1})\cap L^{1}_T(B^2_{\infty,1})} 
&\leq C\bigg(\|(Id-S_N)\tau_0\|_{B^{0}_{\infty,1}} + \int_0^T \|Q(\nabla u,\tau-\tau_N)\|_{B^{0}_{\infty,1}}\\ \notag
&~~~+\|Q(\nabla (u-u_N),\tau_N)\|_{B^{0}_{\infty,1}}+\|u\cdot\nabla(\tau-\tau_N)\|_{B^{0}_{\infty,1}}\\ \notag
&~~~+\|(u-u_N)\cdot\nabla\tau_N\|_{B^{0}_{\infty,1}}+\|u-u_N\|_{B^1_{\infty,1}} dt\bigg)\\ \notag
&\leq C\|(Id-S_N)\tau_0\|_{B^{0}_{\infty,1}}+C\int_0^T \|u-u_N\|_{B^{1}_{\infty,1}}(1+\|\tau_N\|_{B^{1}_{\infty,1}})dt \\ \notag
&~~~+C\int_0^T \|u\|_{B^{1}_{\infty,1}}\|\tau-\tau_N\|_{B^{1}_{\infty,1}} dt\\ \notag
&\leq C\|(Id-S_N)\tau_0\|_{B^{0}_{\infty,1}}+\frac 1 {4C}\|u-u_N\|_{L^{\infty}_T(B^1_{\infty,1})}\\ \notag
&~~~+\frac 1 4\|\tau-\tau_N\|_{L^{\infty}_T(B^{0}_{\infty,1})\cap L^{1}_T(B^2_{\infty,1})}.
\end{align}
According to \eqref{3ineq24} and \eqref{3ineq25}, we obtain
\begin{align}\label{3ineq26}
&\|u-u_N\|_{L^{\infty}_T(B^1_{\infty,1})}+\|\tau-\tau_N\|_{L^{\infty}_T(B^{0}_{\infty,1})\cap L^{1}_T(B^2_{\infty,1})} \\ \notag
\leq&C(\|(Id-S_N)u_0\|_{B^1_{\infty,1}}+\|(Id-S_N)\tau_0\|_{B^{0}_{\infty,1}}).
\end{align}
Analogously, we have
\begin{align*}
&\|v-v_N\|_{L^{\infty}_T(B^1_{\infty,1})}+\|\sigma-\sigma_N\|_{L^{\infty}_T(B^{0}_{\infty,1})\cap L^{1}_T(B^2_{\infty,1})} \\ \notag
\leq&C(\|(Id-S_N)v_0\|_{B^1_{\infty,1}}+\|(Id-S_N)\sigma_0\|_{B^{0}_{\infty,1}}).
\end{align*}
According to the interpolation inequality, \eqref{3ineq18} and \eqref{3ineq21}, we know
\begin{align}\label{3ineq28}
&\|u_N-v_N\|_{L^{\infty}_T(B^1_{\infty,1})}+\|\tau_N-\sigma_N\|_{L^{\infty}_T(B^{0}_{\infty,1})\cap L^{1}_T(B^2_{\infty,1})} \\ \notag
\leq&C(\|u_N-v_N\|_{L^{\infty}_T(B^0_{\infty,1})}+\|\tau_N-\sigma_N\|_{L^{\infty}_T(B^{-1}_{\infty,1})\cap L^{1}_T(B^1_{\infty,1})})^{\frac 12}(\|u_N-v_N\|_{L^{\infty}_T(B^2_{\infty,1})}\\ \notag
&+\|\tau_N-\sigma_N\|_{L^{\infty}_T(B^{1}_{\infty,1})\cap L^{1}_T(B^3_{\infty,1})})^{\frac 12}\\ \notag
\leq &C2^{\frac N 2}(\|u_0-v_0\|_{B^0_{\infty,1}}+\|\tau_0-\sigma_0\|_{B^{-1}_{\infty,1}})^{\frac 1 2}(\|u_0\|_{B^1_{\infty,1}}+\|v_0\|_{B^1_{\infty,1}}+\|\tau_0\|_{B^{0}_{\infty,1}}+\|\sigma_0\|_{B^{0}_{\infty,1}})^{\frac 1 2}.
\end{align}
Combining \eqref{3ineq26}-\eqref{3ineq28}, we conclude that
\begin{align}\label{3ineq29}
&\|u-v\|_{L^{\infty}_T(B^1_{\infty,1})}+\|\tau-\sigma\|_{L^{\infty}_T(B^{0}_{\infty,1})\cap L^{1}_T(B^2_{\infty,1})} \\ \notag
\leq& C(\|(Id-S_N)u_0\|_{B^1_{\infty,1}}+\|(Id-S_N)\tau_0\|_{B^{0}_{\infty,1}}+\|(Id-S_N)v_0\|_{B^1_{\infty,1}}+\|(Id-S_N)\sigma_0\|_{B^{0}_{\infty,1}})\\ \notag
&+C2^{\frac N 2}(\|u_0-v_0\|_{B^0_{\infty,1}}+\|\tau_0-\sigma_0\|_{B^{-1}_{\infty,1}})^{\frac 1 2}(\|u_0\|_{B^1_{\infty,1}}+\|v_0\|_{B^1_{\infty,1}}+\|\tau_0\|_{B^{0}_{\infty,1}}+\|\sigma_0\|_{B^{0}_{\infty,1}})^{\frac 1 2}.
\end{align}
For any $R>0$, we consider $(u_0,\tau_0)\in B_R=\{(f,g)\in B^1_{\infty,1}\times B^0_{\infty,1}: \|(f,g)\|_{B^1_{\infty,1}\times B^0_{\infty,1}}\leq R\}$. Then we will prove that the solution map $(u_0,\tau_0)$ to $(u,\tau)$ is continuous from $B_R$ to $C([0,T];B^1_{\infty,1}\times B^0_{\infty,1})$ uniformly with respect to $a$. Note that $\lim_{N\rightarrow\infty}\|(Id-S_N)f\|_{B^1_{\infty,1}}=0$. For any $\ep>0$, there exists a $M(\ep,(u_0,\tau_0),R)>0$, for any $N\geq M$, there holds  
$$C(\|(Id-S_N)u_0\|_{B^1_{\infty,1}}+\|(Id-S_N)\tau_0\|_{B^{0}_{\infty,1}})\leq \frac{\ep}{4}.$$
Then, there exists a $\delta(\ep,N,(u_0,\tau_0),R)>0$ such that for any $\|u_0-v_0\|_{B^1_{\infty,1}}+\|\tau_0-\sigma_0\|_{B^{0}_{\infty,1}}\leq \delta$, there holds
\begin{align*}
C(\|(Id-S_N)v_0\|_{B^1_{\infty,1}}+\|(Id-S_N)\sigma_0\|_{B^{0}_{\infty,1}})
&\leq C(\|(Id-S_N)u_0\|_{B^1_{\infty,1}}+\|(Id-S_N)\tau_0\|_{B^{0}_{\infty,1}}) \\ 
&~~~+C(\|u_0-v_0\|_{B^1_{\infty,1}}+\|\tau_0-\sigma_0\|_{B^{0}_{\infty,1}}) \\
&\leq \frac{\ep}{4}+C\delta\leq \frac{\ep}{2},   
\end{align*}
and
$$C2^{\frac N 2}\left(\|u_0-v_0\|_{B^0_{\infty,1}}+\|\tau_0-\sigma_0\|_{B^{-1}_{\infty,1}}\right)^{\frac 1 2}\left(\|u_0\|_{B^1_{\infty,1}}+\|v_0\|_{B^1_{\infty,1}}+\|\tau_0\|_{B^{0}_{\infty,1}}+\|\sigma_0\|_{B^{0}_{\infty,1}}\right)^{\frac 1 2}\leq \frac{\ep}{4}.$$
One can infer from \eqref{3ineq29} that 
$$\|u-v\|_{L^{\infty}_T(B^1_{\infty,1})}+\|\tau-\sigma\|_{L^{\infty}_T(B^{0}_{\infty,1})\cap L^{1}_T(B^2_{\infty,1})}\leq \ep.$$
This completes the proof of Proposition \ref{3prop1}.
\end{proof}

\section{Global existence of solutions of low regularity in $\R^2$ with damping}
In this section, we consider the 2D inviscid
Oldroyd-B model \eqref{eq0} with $a>0$ and prove Theorem~\ref{6theo}. We point out that damping effect is the key to reducing the regularity of $\tau$; indeed we then proved the global well-posedness of lowered regularity with help of our new observation.

%{\color{red}The result of global well-posedness with low regularity improves the work by T. M. Elgindi and F. Rousset \cite{ER15}. }

Denote $A(D) f=\mathcal{F}^{-1}(A(\xi) \widehat{f})$. In the next proposition, we recall some properties of the Riesz operator.
    \begin{lemm}\label{CR1}
        Let $\mathcal{R}_i$ be the Riesz operator $\mathcal{R}_i=\frac{\partial_i}{|D|}$. Then the following hold true.\\
        (1) For any $p\in(1,\infty)$, there exists a positive constant $C(p)$ such that
        \begin{align*}
            \|\mathcal{R}_i\|_{\mathcal{L}(L^p)}\leq C.
        \end{align*}
        (2) Let $\chi\in\mathcal{D}(\mathbb{R}^d)$ be radial and fix $s>0$.
        There exists $C=C(d,s,\chi)$ such that
        \begin{align*}
            \||D|^s\chi(2^{-q}|D|)\mathcal{R}_i\|_{\mathcal{L}(L^p)}\leq C2^{qs},
        \end{align*}
        ~~~~~for every $p\in[1,\infty]$ and $q\in\mathbb Z$.\\
        (3) Let $\mathcal{C}$ be a fixed ring. Then, there exists $g\in\mathcal{S}$ whose spectrum does not meet the origin such that
        \begin{align*}
            \mathcal{R}_if = 2^{qd}g(2^q \dot) \ast f,
        \end{align*}
        ~~~~~for any $f$ with Fourier transform supported in $2^q\mathcal{C}$.
    \end{lemm}
    As explained in the introduction, we give the proof of a new commutator lemma between the Riesz operator $\mathcal{R}_i$ and the convection operator $u\cdot\nabla$, which plays a crucial role in reducing the regularity. Note that the following commutator lemma of Calderon-Zygmund operator is sharp, which does not require additional regularity of $\|\tau\|_{B^\ep_{\infty,1}}$ with $\ep>0$ as in Lemma \ref{CR}. For more details, one can refer to \cite{HKR11,L19,LS23}.
    \begin{lemm}\label{CR2}
        Let $\mathcal{R}_i$ be the Riesz operator $\mathcal{R}_i=\frac{\partial_i}{|D|}$ and assume $\operatorname{div}u=0$, $u,\tau\in L^2$, $\nabla u\in L^\infty$, and $\tau\in B^0_{\infty,1}$. Then there exists a constant $C>0$ such that
        \begin{align*}
            \|[\mathcal{R}_i,u\cdot\nabla]\tau\|_{B^0_{\infty,1}} \leq C(\|\nabla u\|_{L^\infty}\|\tau\|_{B^{0}_{\infty,1}}+\|u\|_{L^2}\|\tau\|_{L^2}).
        \end{align*}
    \end{lemm}
    \begin{proof}
        Firstly, by Bony's decomposition, we split the commutator into three parts: 
        \begin{align*}
            [\mathcal{R}_i,u\cdot\nabla]\tau &= \sum_{j\geq0} [\mathcal{R}_i,S_{j-1}u\cdot\nabla]\Delta_j\tau + \sum_{j\geq0} [\mathcal{R}_i,\Delta_{j}u\cdot\nabla]S_{j-1}\tau + \sum_{j\geq-1} [\mathcal{R}_i,\Delta_{j}u\cdot\nabla]\tilde{\Delta}_j \tau \\ \notag
            &\triangleq\sum_{j\geq0} A^1_j + \sum_{j\geq0} A^2_j + \sum_{j\geq-1} A^3_j,
        \end{align*}
  where $\tilde{\Delta}_j \tau=\sum_{k=j-1}^{j+1}\Delta_k \tau$.
        By Lemma \ref{CR1}, we know that there exists $g \in \mathcal{S}$ whose spectrum does not meet the origin such that
        \begin{align*}
            A^1_j = g_j \ast \left(S_{j-1}u\cdot\nabla\Delta_j\tau\right) - S_{j-1}u\cdot  \left(g_j \ast\nabla\Delta_j\tau\right),
        \end{align*}
        where $g_j(x) = 2^{dj} g(2^jx)$. According to Bernstein's inequality and $\|xg_j\|_{L^1} = 2^{-j}\|xg\|_{L^1}$, we obtain
        \begin{align*}
            \|A^1_j\|_{L^\infty} &=\left\|g_j \ast \left(S_{j-1}u\cdot\nabla\Delta_j\tau\right) - S_{j-1}u\cdot  \left(g_j \ast\nabla\Delta_j\tau\right)\right\|_{L^\infty}\\ \notag
   &=\left\|\int_{\mathbb{R}^d}g_j(x-y)(S_{j-1}u(y)-S_{j-1}u(x))\cdot\nabla\Delta_j\tau(y)dy\right\|_{L^\infty}\\ \notag
            &=\left\|\int_{0}^{1}\int_{\mathbb{R}^d}g_j(x-y)[(y-x)\cdot\nabla S_{j-1}u(x+t(y-x))]\cdot\nabla\Delta_j\tau(y)dy\right\|_{L^\infty}\\ \notag
   &\lesssim \|x g_j\|_{L^1} \|\nabla S_{j-1}u\|_{L^\infty}\|\nabla\Delta_j\tau\|_{L^{\infty}} \\ \notag
            &\lesssim \|\nabla u\|_{L^\infty}\|\Delta_j\tau\|_{L^\infty}.
        \end{align*}
        Then we obtain
        \begin{align*}
             \sum_{k\geq -1}\|\Delta_k\sum_{j\geq0}A^1_j\|_{L^\infty} \lesssim \sum_{j\geq0} \|A^1_j\|_{L^\infty} \lesssim  \|\nabla u\|_{L^\infty}\|\tau\|_{B^0_{\infty,1}}.
        \end{align*}
        For the second paraproduct, retain the Riesz multiplier on the
        low-frequency factor:
        \begin{align*}
        A_j^2={}&\mathcal R_i(\Delta_j u\cdot\nabla S_{j-1}\tau)
        -\Delta_j u\cdot\mathcal R_i\nabla S_{j-1}\tau.
        \end{align*}
        The first product has Fourier support in an annulus of size $2^j$,
        whereas $S_{j-1}\tau$ need not be supported away from the origin.
        We therefore estimate the two terms separately. Annular multiplier
        bounds and a dyadic decomposition of the low-frequency factor give
        \begin{align*}
        \|A_j^2\|_{L^\infty}
        &\lesssim \|\Delta_j u\|_{L^\infty}
          \sum_{-1\leq q\leq j-2}2^q\|\Delta_q\tau\|_{L^\infty}\\
        &\lesssim 2^{-j}\|\nabla u\|_{L^\infty}
          \sum_{-1\leq q\leq j-2}2^q\|\Delta_q\tau\|_{L^\infty}.
        \end{align*}
        For $q=-1$, the multiplier $\mathcal R_i\nabla\Delta_{-1}$
        has an integrable convolution kernel, so the same bound applies.
        Both terms in $A_j^2$ have annular output support. Consequently,
        \begin{align*}
        \sum_{k\geq-1}\left\|\Delta_k\sum_{j\geq0}A_j^2\right\|_{L^\infty}
        &\lesssim \|\nabla u\|_{L^\infty}
          \sum_{j\geq0}\sum_{-1\leq q\leq j-2}
             2^{q-j}\|\Delta_q\tau\|_{L^\infty}\\
        &\lesssim \|\nabla u\|_{L^\infty}\|\tau\|_{B^0_{\infty,1}}.
        \end{align*}
        Using ${\rm div}~u=0$, we rewrite $A^3_j$ as
        \begin{align}\label{6eq2}
            \sum_{j\geq-1} A^3_j &= \sum_{j\geq2} \mathcal{R}_i{\rm div}~(\Delta_ju\tilde{\Delta}_j\tau)-\sum_{j\geq2} {\rm div}~(\Delta_ju\mathcal{R}_i\tilde{\Delta}_j\tau) + \sum_{j\leq1} [\mathcal{R}_i,\Delta_ju\cdot\nabla]\tilde{\Delta}_j\tau\\ \notag
            &\triangleq\sum_{j\geq2} B^1_j+\sum_{j\geq2} B^2_j + \sum_{j\leq1} A^3_j.
        \end{align}
        Using \eqref{6eq2} and Lemma \ref{CR1}, we have
        \begin{align*}
            \sum_{k\geq -1}\|\Delta_{k}\sum_{j\geq2}B^1_j\|_{L^\infty} &\lesssim \sum_{k\geq -1}\sum_{j\geq \max\{2,k-4\}}2^k\|\Delta_j u\|_{L^\infty}\|\tilde{\Delta}_j\tau\|_{L^\infty} \\
   &\lesssim \sum_{k\geq -1}\sum_{j\geq \max\{2,k-4\}}2^{k-j}\|\nabla\Delta_j u\|_{L^\infty}\|\tilde{\Delta}_j\tau\|_{L^\infty}\\
   &\lesssim  \|\nabla u\|_{L^\infty}\|\tau\|_{B^0_{\infty,1}}.
        \end{align*}
        For $j\geq2$, $\tilde\Delta_j\tau$ is supported in an annulus.
Choose a fixed annular multiplier kernel $h\in\mathcal S$ and put
$h_j(x)=2^{dj}h(2^jx)$. Then
$\mathcal R_i\tilde\Delta_j\tau=h_j*\tilde\Delta_j\tau$ and
$\|h_j\|_{L^1}=\|h\|_{L^1}$. 
        Then we obtain
        \begin{align*}
            \sum_{k\geq -1}\|\Delta_{k}\sum_{j\geq2}B^2_j\|_{L^\infty} \lesssim \|\nabla u\|_{L^\infty}\|\tau\|_{B^0_{\infty,1}}.
        \end{align*}
         Finally, we deduce that
        \begin{align*} 
            \sum_{k\geq -1}\|\Delta_{k}\sum_{j\leq 1}A^3_j\|_{L^\infty} 
   &\lesssim\sum_{k\leq 4}\sum_{j\leq 1}\|\Delta_{k}A^3_j\|_{L^\infty} \\
   &\lesssim\sum_{j\leq 1}\|[\mathcal{R}_i,\Delta_ju\cdot\nabla]\tilde{\Delta}_j\tau\|_{L^\infty} \\
   &\lesssim\sum_{j\leq 1}\|\Delta_ju\cdot\nabla\tilde{\Delta}_j\tau\|_{L^2}+\|\Delta_ju\cdot\nabla\mathcal{R}_i\tilde{\Delta}_j\tau\|_{L^2} \\
            &\lesssim \|u\|_{L^2}\|\tau\|_{L^2}.
        \end{align*}    
    We thus complete the proof of Lemma \ref{CR2}.
    \end{proof}
    \begin{coro}\label{CR3}
        Let ${\rm div}~u=0$ and $\mathcal{R}=-(-\Delta)^{-1}{\rm curl}~{\rm div}$. There exists a constant $C>0$ such that
        \begin{align*}
            \|[\mathcal{R},u\cdot\nabla]\tau\|_{B^0_{\infty,1}}\leq C(\|\nabla u\|_{L^\infty}+ \|u\|_{L^2})(\|\tau\|_{B^0_{\infty,1}}+\|\tau\|_{L^2}).
        \end{align*}
    \end{coro}
    \begin{proof}
        Since each component of $\mathcal R$ is a finite linear combination of double Riesz transforms, it suffices to use the identity
        \begin{align*}
            [\mathcal{R}_i\mathcal{R}_j,u\cdot\nabla] \tau &= \mathcal{R}_i\mathcal{R}_j(u\cdot\nabla\tau) - u\cdot\nabla\left(\mathcal{R}_i\mathcal{R}_j(\tau)\right) \\ \notag
&=\mathcal{R}_i\left([\mathcal{R}_j,u\cdot\nabla]\tau\right) + [\mathcal{R}_i,u\cdot\nabla](\mathcal{R}_j\tau).
        \end{align*}
        By virtue of Lemmas \ref{CZ} and \ref{CR2}, we infer that
        \begin{align*}  \|\mathcal{R}_i\left([\mathcal{R}_j,u\cdot\nabla]\tau\right)\|_{B^0_{\infty,1}} 
  &\lesssim \|(Id-\Delta_{-1})\mathcal{R}_i\left([\mathcal{R}_j,u\cdot\nabla]\tau\right)\|_{B^0_{\infty,1}}+\|\Delta_{-1}[\mathcal{R}_j,u\cdot\nabla]\tau\|_{L^2}\\
  &\leq\|[\mathcal{R}_j,u\cdot\nabla]\tau\|_{B^0_{\infty,1}}+ \|u\|_{L^2}\|\tau\|_{L^2}\\ \notag
            &\lesssim \|\nabla u\|_{L^\infty}\|\tau\|_{B^0_{\infty,1}}+\|u\|_{L^2}\|\tau\|_{L^2}.
        \end{align*}
        Similarly, we obtain
        \begin{align*}
            \|[\mathcal{R}_i,u\cdot\nabla](\mathcal{R}_j\tau)\|_{B^0_{\infty,1}}
            &\lesssim \|\nabla u\|_{L^\infty}\left(\|\Delta_{-1}\mathcal{R}_j\tau\|_{B^0_{\infty,1}}+\|(Id-\Delta_{-1})\mathcal{R}_j\tau\|_{B^0_{\infty,1}}\right) + \|u\|_{L^2}\|\tau\|_{L^2}\\ \notag
            &\lesssim \|\nabla u\|_{L^\infty}\left(\|\mathcal{R}_j\tau\|_{L^2}+\|\tau\|_{B^0_{\infty,1}}\right)+ \|u\|_{L^2}\|\tau\|_{L^2} \\ \notag
            &\lesssim  (\|\nabla u\|_{L^\infty}+ \|u\|_{L^2})(\|\tau\|_{B^0_{\infty,1}}+\|\tau\|_{L^2}).
        \end{align*}
        Thus we conclude that
        \begin{align*}
            \|[\mathcal{R},u\cdot\nabla]\tau\|_{B^0_{\infty,1}}\lesssim (\|\nabla u\|_{L^\infty}+ \|u\|_{L^2})(\|\tau\|_{B^0_{\infty,1}}+\|\tau\|_{L^2}).
        \end{align*}
        This completes the proof of Corollary \ref{CR3}.
    \end{proof}
 Now, we establish global existence for \eqref{eq0} with low regularity.

{\bf Proof of Theorem \ref{6theo} :}  \\
Firstly, for any $T>0$, we assume that $$\|(u,\tau)\|_{L_T^{\infty}(L^2\cap B^1_{\infty,1})\times L_T^{\infty}(L^2\cap B^0_{\infty,1})}\leq 4c^{\frac 1 2}.$$
Taking $L^2$ inner product for $\eqref{eq0}_1$ with $u$, we obtain
        \begin{align}\label{6ineq1}
        \frac 1 2\frac{d}{dt}\|u\|^2_{L^2} = \langle{\rm div}~\tau,u\rangle.
        \end{align}
        Taking $L^2$ inner product for $\eqref{eq0}_2$ with $\tau$, we get
        \begin{align}\label{6ineq2}
        \frac 1 2\frac{d}{dt}\|\tau\|^2_{L^2}+\|\nabla\tau\|^2_{L^2} + a\|\tau\|^2_{L^2} - \langle D u,\tau\rangle &= - \langle Q(\nabla u,\tau),\tau\rangle \\ \notag
        &\leq C\|\nabla u\|_{L^\infty}\|\tau\|^2_{L^2} \\ \notag
        &\leq \frac a 2\|\tau\|^2_{L^2}
        \end{align}
        where we take $c^{\frac 1 2}\leq\frac{a}{100C}$.
        By virtue of cancellation $\langle{\rm div}~\tau,u\rangle+\langle D u,\tau\rangle=0$, we infer from \eqref{6ineq1} and \eqref{6ineq2} that
        \begin{align}\label{6ineq3}
        \|(u,\tau)\|^2_{L^\infty_T(L^2)} + \|\nabla\tau\|^2_{L^2_T(L^2)}+a\|\tau\|^2_{L^2_T(L^2)}
        \leq 2\|(u_0,\tau_0)\|^2_{L^2}.
        \end{align}
  
    From now on, we focus on the estimate of $\|\tau\|_{L^\infty_T(B^0_{\infty,1})}$. Applying Duhamel's principle to $\eqref{eq0}_2$ and taking $\Delta_j$ with $j\geq 0$, we have
        \begin{align}\label{6ineq4}
            \Delta_j\tau = e^{\Delta t}e^{-at}\Delta_j\tau_0 - \int_0^t e^{\Delta(t-s)}e^{-a(t-s)}\Delta_j(Q(\nabla u,\tau)+u\cdot\nabla\tau-Du)(s)\,ds.
        \end{align}
        Estimating \eqref{6ineq4} at the endpoint $T$, for $j\geq 0$, we infer from Lemma \ref{H} that
        \begin{align}\label{6ineq5}
            \|\Delta_j\tau(T)\|_{L^{\infty}} &\lesssim \|\Delta_j\tau_0\|_{L^{\infty}} + \int_0^T e^{-c_H2^{2j}(T-t)}\|\Delta_jQ(\nabla u,\tau)\|_{L^{\infty}}dt \\ \notag
            &~~~+\int_0^T e^{-c_H2^{2j}(T-t)}\|\Delta_j(u\cdot\nabla\tau)\|_{L^{\infty}}dt+\int_0^T e^{-c_H2^{2j}(T-t)}\|\Delta_j Du\|_{L^{\infty}}dt.
        \end{align}
        Here $c_H>0$ is the heat-decay constant from Lemma~\ref{H},
        independent of the initial-data smallness parameter $c$.
        Applying Bernstein's inequality, we infer from \eqref{6ineq3} and \eqref{6con} that
        \begin{align*}
            \int_0^T e^{-c_H2^{2j}(T-t)}\|\Delta_j Du\|_{L^{\infty}}dt \lesssim \int_0^T e^{-c_H2^{2j}(T-t)}2^j\|\Delta_j u\|_{L^{\infty}}dt
            \lesssim 2^{-j}\|u\|_{L_T^{\infty}(L^{\infty})}\lesssim 2^{-j}c^{\frac 3 4}.
        \end{align*}
         According to Young's inequality, we have
        \begin{align*}
            \int_0^T e^{-c_H2^{2j}(T-t)}\|\Delta_jQ(\nabla u,\tau)\|_{L^{\infty}}dt &\lesssim \int_0^T e^{-c_H2^{2j}(T-t)}2^j\|\Delta_jQ(\nabla u,\tau)\|_{L^{2}}dt \\ \notag
            &\lesssim  2^{-j}\|\nabla u\|_{L_T^{\infty}(L^{\infty})}\|\tau\|_{L_T^{\infty}(L^{2})}\lesssim  2^{-j}c
        \end{align*}
        and
        \begin{align}\label{6ineq8}
            \int_0^T e^{-c_H2^{2j}(T-t)}\|\Delta_j(u\cdot\nabla\tau)\|_{L^{\infty}}dt &\lesssim \int_0^T e^{-c_H2^{2j}(T-t)}2^j\|\Delta_j(u\tau)\|_{L^{\infty}}dt\\ \notag
            &\lesssim  2^{-j}\|u\|_{L_T^{\infty}(L^{\infty})}\|\tau\|_{L_T^{\infty}(L^{\infty})}\lesssim  2^{-j}c.
        \end{align}
        Combining the estimates \eqref{6ineq5}-\eqref{6ineq8}, we obtain
        \begin{align}\label{6ineq9}
            \|\tau\|_{L^\infty_T(B^0_{\infty,1})} &\lesssim \|\Delta_{-1}\tau\|_{L^\infty_T(L^{\infty})}+\sum_{j\geq 0}\|\Delta_j\tau\|_{L^{\infty}_T(L^{\infty})} \\ \notag
            &\lesssim \|(u_0,\tau_0)\|_{L^2}+\|\tau_0\|_{B^0_{\infty,1}}+c^{\frac34}\lesssim c^{\frac 3 4}.
        \end{align}
        
    Moreover, we can cancel ${\rm div}~\tau$ and $\Delta\tau$ in \eqref{eq0} by virtue of the structural trick
        \begin{align}\label{6eq4}
        \Gamma = \omega-\mathcal{R}\tau,
        \end{align}
        where $\mathcal{R} =-(-\Delta)^{-1}{\rm curl}~{\rm div}$. Using \eqref{6eq4} and \eqref{eq0}, we infer that
        \begin{align}\label{6eq5}
            \partial_t\Gamma + u\cdot\nabla\Gamma +\frac 12\Gamma=(a-\frac{1}{2})\mathcal{R}\tau + \mathcal{R}Q(\nabla u,\tau) + [\mathcal{R},u\cdot\nabla]\tau.
        \end{align}
        Apply Lemma~\ref{TP} to $e^{t/2}\Gamma$. Choose $c$ smaller,
        if necessary, so that $C\|\nabla u\|_{L^\infty_T(B^0_{\infty,1})}
        \leq1/4$ under the bootstrap bound. Then \eqref{6eq5} gives
        \begin{align}\label{6ineq10}
        \|\Gamma(T)\|_{B^0_{\infty,1}}
        &\lesssim e^{-T/2+V(0,T)}\|\Gamma_0\|_{B^0_{\infty,1}}
        +\int_0^T e^{-(T-t)/2+V(t,T)}\mathcal F_\Gamma(t)\,dt\notag\\
        &\lesssim \|\Gamma_0\|_{B^0_{\infty,1}}
        +\int_0^T e^{-(T-t)/4}\mathcal F_\Gamma(t)\,dt,
        \end{align}
        where
        \begin{align*}
        V(s,t)&=C\int_s^t\|\nabla u(r)\|_{B^0_{\infty,1}}\,dr,\\
        \mathcal F_\Gamma(t)&=\|\mathcal R\tau(t)\|_{B^0_{\infty,1}}
          +\|\mathcal RQ(\nabla u,\tau)(t)\|_{B^0_{\infty,1}}
          +\|[\mathcal R,u\cdot\nabla]\tau(t)\|_{B^0_{\infty,1}}.
        \end{align*}
        By virtue of Lemmas \ref{CZ} and Corollary \ref{CR3}, we have
        \begin{align*}
            \|\mathcal{R}\tau\|_{B^0_{\infty,1}} + \|[\mathcal{R},u\cdot\nabla]\tau\|_{B^0_{\infty,1}}&\lesssim \|\tau\|_{L^2} + \|\tau\|_{B^0_{\infty,1}} + (\|\nabla u\|_{L^\infty}+ \|u\|_{L^2})(\|\tau\|_{B^0_{\infty,1}}+\|\tau\|_{L^2})\lesssim  c^{\frac 3 4}.
        \end{align*}
  Since $B^0_{\infty,1}$ is a critical space and $Q$ does not have transportation structure, product laws may require additional regularity. However, in the 2D case, we discover that $Q$ can be controlled by a new estimation.
        Using Lemmas \ref{CZ}, \ref{T} and \ref{R}, we have \begin{align}\label{6ineq12}
  \|\mathcal{R}Q(\nabla u,\tau)\|_{B^0_{\infty,1}} &\lesssim \|Q(\nabla u,\tau)\|_{B^0_{\infty,1}}+ \|Q(\nabla u,\tau)\|_{L^2}\\ \notag
            &\lesssim \|T_{\nabla u}\tau\|_{B^0_{\infty,1}}+ \|T_{\tau}\nabla u\|_{B^0_{\infty,1}}+\|R(\nabla u,\tau)\|_{B^0_{\infty,1}}+\|\nabla u\|_{L^\infty}\|\tau\|_{L^2} \\ \notag
            &\lesssim \|\nabla u\|_{B^0_{\infty,1}} \|\tau\|_{B^0_{\infty,1}}+\|R(\nabla u,\tau)\|_{B^1_{2,1}}+\|\nabla u\|_{L^\infty}\|\tau\|_{L^2} \\ \notag
   &\lesssim \|\nabla u\|_{B^0_{\infty,1}} \|\tau\|_{B^0_{\infty,1}}+\|\nabla u\|_{B^0_{\infty,1}}\|\tau\|_{H^1}+\|\nabla u\|_{L^\infty}\|\tau\|_{L^2}\\ \notag
   &\lesssim c+c^{\frac 12}\|\nabla \tau\|_{L^2}.
  \end{align}
    According to \eqref{6ineq10}-\eqref{6ineq12}, we conclude from \eqref{6ineq3} that
        \begin{align*}
            \|\Gamma\|_{L^\infty_T(B^0_{\infty,1})} &\lesssim c^{\frac 3 4}.
        \end{align*}
The convolution estimates above are uniform at every endpoint
in $[0,T]$, so taking the supremum gives the stated time norms.
This together with \eqref{6ineq3} and \eqref{6ineq9} ensures that
$$\|(u,\tau)\|_{L_T^{\infty}(L^2\cap B^1_{\infty,1})\times L_T^{\infty}(L^2\cap B^0_{\infty,1})}\leq 2c^{\frac 1 2},$$
after further choosing $Cc^{1/4}$ sufficiently small. This
strictly improves the bootstrap bound.
According to Proposition \ref{3prop1} and the continuity method, we complete the proof of Theorem \ref{6theo}. \qedsymbol{}
%\hfill$\Box$

\section{Global existence and uniform vanishing damping limit in $\R^2$}
In this section, we first look for a global estimate for \eqref{eq0} in $H^1$. Then we can obtain the global strong solution through smallness condition in the critical Besov space. Indeed our global well-posedness results hold for all damping $a\in[0,1]$ and therefore extend the work in the existing literature.

%{\color{red}The result of the global existence extends the work of T. M. Elgindi and F. Rousset \cite{ER15} to the case $a\in[0,1]$.} 

Moreover, global solutions is uniformly bounded in time, which is useful to prove that the solution of supercritical regularity is also uniformly bounded in time. The results of global existence play a crucial role in considering vanishing damping limit in the same topology for any $T>0$. 

Then, by virtue of the improved Fourier splitting method, for any $a\in[0,1]$, we obtain optimal time decay rates of global solutions for the inviscid Oldroyd-B model \eqref{eq0}. The high-regularity bound $\|\tau\|_{B^0_{\infty,1}}\leq C(1+t)^{-1}$ for $d=2$ does not by itself give time integrability of $\|\tau\|_{B^0_{\infty,1}}$. We introduce a novel method of high-low frequency decomposition to obtain time integrability of $\|\nabla u\|_{B^0_{\infty,1}}$. By virtue of the optimal decay rates and key integrability, we obtain the uniform vanishing damping limit and discover a new phenomenon for \eqref{eq0} that the sharp rate of uniform vanishing damping limit in $L^2$ is related to the time decay rate in $L^2$.
\subsection{Global existence and vanishing damping limit}
Since the global estimates are uniformly in $a$, we may drop the upper index $a$ of the solutions for the inviscid Oldroyd-B model 
\eqref{eq0} without any confusions. We first provide the following global estimate for \eqref{eq0} in $H^1$.
\begin{prop}\label{4prop1}
Let $d=2$ and $a\in[0,1]$. Assume $(u,\tau)$ is a smooth solution of \eqref{eq0} with divergence-free field $u_0\in H^1$ and a symmetric matrix $\tau_0\in H^1$.  There exists some sufficiently small constant $\epsilon>0$ such that if
        \begin{align*}
        \|(u_0,\tau_0)\|_{H^1}\leq \ep,
        \end{align*}
        then there exist a fixed constant $\eta$, for any $t>0$, we have
        \begin{align}\label{4ineq1}
        \frac{d}{dt}\left(\frac{1}{2}\|(u,\tau)\|^2_{H^1}-\eta\langle\tau,\nabla u\rangle\right) + \frac \eta 4\|\nabla u\|^2_{L^2} + \frac{1}{4}\|\nabla\tau\|^2_{H^1}\leq 0,
        \end{align}
        and
        \begin{align}\label{4ineq2}
        \frac{d}{dt}\|\nabla(u,\tau)\|^2_{L^2} + \|\nabla^2\tau\|^2_{L^2}\lesssim \|\nabla u\|^2_{L^2}\|\tau\|^2_{H^1}.
        \end{align}
\end{prop}
\begin{proof}
We start with a \emph{bootstrap assumption}
\begin{equation*}
\|(u,\tau)(t)\|_{H^{1}}\leq 2\ep, \quad t\geq 0.
\end{equation*}
Testing the first equation of \eqref{eq0} by $ u $ and the second equation of \eqref{eq0} by $ \tau $, we obtain the following energy identity:
\begin{align*}
\frac{d}{dt}\frac{1}{2}\left\|(u,\tau)\right\|_{L^{2}}^{2}+\|\nabla \tau\|^2_{L^{2}}+a\|\tau\|_{L^{2}}^{2}=-\langle Q(\nabla u,\tau),\tau \rangle.
\end{align*}
In the above, we have employed the identity:
\begin{equation*}
\langle u, \mathrm{div}\tau \rangle + \langle D(u), \tau \rangle =0.
\end{equation*}
Applying Sobolev's inequality, we have
\begin{align*}
\left|\langle Q(\nabla u,\tau),\tau \rangle\right|&\leq \|\nabla u\|_{L^{2}}\|\tau\|_{L^{4}}^{2}\\ \notag
&\lesssim \|\nabla u\|_{L^{2}}\|\tau\|_{L^{2}}\|\nabla \tau\|_{L^{2}}\\ \notag
&\leq \frac{1}{4}\|\nabla \tau\|^2_{L^{2}}+C\|\nabla u\|_{L^{2}}^{2}\|\tau\|^2_{L^{2}} \notag.
\end{align*}
Therefore,
\begin{align}\label{4eq3}
\frac{d}{dt}\frac{1}{2}\left\|(u,\tau)\right\|_{L^{2}}^{2}+\frac34\|\nabla \tau\|^2_{L^{2}}+a\|\tau\|_{L^{2}}^{2}\leq C\|\nabla u\|_{L^{2}}^{2}\|\tau\|^2_{L^{2}}.
\end{align}
To absorb the right-hand side, we choose $ \eta>0 $, which will be determined later. By direct computation, we have 
\begin{align*}
\frac{d}{dt} \langle -\eta\tau, \nabla u \ \rangle &=\eta\langle\partial_{t}u,  \mathrm{div}\tau\ \rangle 
-\eta\langle\partial_{t}\tau, \nabla u\ \rangle\\
&=-\frac{\eta}{2}\|\nabla u\|_{L^2}^{2}+ \notag
\eta \langle\mathbb{P}\left(\mathrm{div}\tau-u\cdot\nabla u\right), \mathrm{div}\tau \rangle\\ \notag
&\quad+\eta \langle u\cdot\nabla\tau+a\tau+Q(\nabla u,\tau)-\Delta \tau, \nabla u \rangle. \notag
\end{align*}
Applying Sobolev's inequality, we obtain 
\begin{align*}
&\frac{d}{dt}\langle -\eta\tau, \nabla u \ \rangle+\frac{\eta}{2}\|\nabla u\|_{L^2}^{2}\\ \notag
&=\eta \langle\mathbb{P}\left(\mathrm{div}\tau-u\cdot\nabla u\right),\mathrm{div}\tau \rangle+\eta \langle u\cdot\nabla\tau+a\tau+Q(\nabla u,\tau)-\Delta \tau, \nabla u \rangle\\ \notag
&\leq \eta\Bigg(\|\nabla \tau\|_{L^{2}}^{2}+\|u\|_{L^{4}}\|\nabla u\|_{L^{2}}\|\nabla \tau\|_{L^{4}}+a\|\tau\|_{L^{2}}\|\nabla u\|_{L^{2}}+\|\tau\|_{L^{\infty}}\|\nabla u\|^2_{L^{2}}+\|\Delta \tau\|_{L^{2}}\|\nabla u\|_{L^{2}}\Bigg)\\ \notag
&\leq \frac{\eta}{8}\|\nabla u\|_{L^{2}}^{2}+C\eta\|\nabla \tau\|_{H^{1}}^{2}+\frac{a}{8}\|\tau\|^2_{L^{2}}, \notag
\end{align*}
provided $\eta$ is smaller than an absolute constant, which depends only on the coefficients in Sobolev's inequalities.

Finally, testing the first equation of \eqref{eq0} by $-\Delta u$, and the second equation of \eqref{eq0} by $ -\Delta \tau $, we obtain the following energy identity:
\begin{align}\label{4eq6}
\frac{1}{2}\frac{d}{dt}\left\|(\nabla u,\nabla \tau)\right\|^2_{L^{2}}+a\|\nabla\tau\|^2_{L^{2}}+\|\nabla^{2}\tau\|^2_{L^{2}}=\langle (u\cdot\nabla)\tau, \Delta \tau \rangle +\langle Q(\nabla u,\tau), \Delta\tau \rangle,
\end{align}
where we have used the fact that for $d=2$ 
$$\langle (u\cdot\nabla)u, \Delta u \rangle = 0.$$
Applying Sobolev's inequality, we get
\begin{align}\label{4eq7}
&\frac{1}{2}\frac{d}{dt}\left\|(\nabla u,\nabla \tau)\right\|^2_{L^{2}}+a\|\nabla\tau\|^2_{L^{2}}+\|\nabla^{2}\tau\|^2_{L^{2}}\\ \notag
\leq& \|\nabla u\|_{L^{2}}\|\tau\|_{L^{\infty}}\|\nabla^{2}\tau\|_{L^{2}}+\|u\|_{L^{4}}\|\nabla\tau\|_{L^{4}}\|\nabla^{2}\tau\|_{L^{2}}\\ \notag
\leq& \frac{1}{8}\|\nabla^{2} \tau\|_{L^{2}}^{2}+C\left(\|\nabla u\|_{L^{2}}^{2}\|\tau\|_{L^{\infty}}^{2}+\|u\|_{L^{4}}^{2}\|\nabla\tau\|_{L^{4}}^{2}\right)\\ \notag
\leq& \frac{1}{8}\|\nabla^{2}\tau\|_{L^{2}}^{2} + C\left(\|\nabla u\|_{L^{2}}^{2}\|\tau\|_{L^{2}}\|\nabla^{2}\tau\|_{L^{2}}+\|u\|_{H^{1}}^{2}\|\nabla\tau\|_{L^{2}}\|\nabla^{2}\tau\|_{L^{2}}\right)\\ \notag
\leq&\frac{1}{4}\|\nabla^{2} \tau\|_{L^{2}}^{2}+C\left(\|\nabla u\|_{L^{2}}^{4}\|\tau\|_{L^{2}}^{2}+\|u\|_{H^{1}}^{4}\|\nabla\tau\|_{L^{2}}^{2}\right)\\ \notag
\leq& \frac{1}{4}\|\nabla^{2} \tau\|_{L^{2}}^{2}+C\ep^{4}\|\nabla u\|_{L^{2}}^{2}+C\ep^{4}\|\nabla\tau\|_{L^{2}}^{2},
\end{align}
where we have employed the bootstrap assumption in the last inequality. 

Choose $0<\eta\leq1/10$ and then $\ep>0$ so that
$C\eta\leq1/8$ and $C(\ep^2+\ep^4)\leq\eta/8$.
Combining \eqref{4eq3}-\eqref{4eq7}, we obtain the energy inequality
\begin{align}\label{4eq8}
\frac{d}{dt}\left(\frac{1}{2}\left\|(u,\tau)\right\|_{H^{1}}^{2}-\eta\langle \tau,\nabla u\rangle \right)+\frac{\eta}{4}\|\nabla u\|_{L^{2}}^{2}+\frac{1}{2}\|\nabla \tau\|_{H^{1}}^{2}\leq 0.
\end{align}
When $ \eta \leq \frac{1}{10} $, by \eqref{4eq8}, we obtain
\begin{align*}
\|(u,\tau)\|_{H^{1}}\leq\frac{3}{2}\ep,
\end{align*}
which is stronger than our bootstrap assumption at the start. Hence, \eqref{4ineq1} follows from \eqref{4eq8} and the standard bootstrap argument.
For completeness, the sharper inequality \eqref{4ineq2} uses the
transport cancellation rather than the rough estimate in \eqref{4eq7}.
Indeed, integration by parts and the two-dimensional
Gagliardo--Nirenberg inequalities give
\begin{align*}
 |\langle u\cdot\nabla\tau,\Delta\tau\rangle|
 &\lesssim \|\nabla u\|_2\|\nabla\tau\|_4^2
 \lesssim \|\nabla u\|_2\|\nabla\tau\|_2\|\nabla^2\tau\|_2,\\
 |\langle Q,\Delta\tau\rangle|
 &\lesssim \|\nabla u\|_2\|\tau\|_2^{1/2}
                  \|\nabla^2\tau\|_2^{3/2}.
\end{align*}
Young's inequality bounds their sum by
$\frac12\|\nabla^2\tau\|_2^2+
C\|\nabla u\|_2^2\|\nabla\tau\|_2^2+
C\|\nabla u\|_2^4\|\tau\|_2^2$.
The bootstrap bound $\|\nabla u\|_2\leq2\ep$ and
\eqref{4eq6} now yield \eqref{4ineq2}, uniformly for $a\in[0,1]$.
We finish the proof of Proposition~\ref{4prop1}.
\end{proof}
\begin{rema}
Through the global estimates in Proposition \ref{4prop1} and compactness methods, one can infer that the equation \eqref{eq0} has global $H^{1}$ weak solutions when $d=2$.
\end{rema}
 Since the core difficulty term $Q$ exists, we cannot elevate weak solutions to strong solutions solely by improving regularity. Specifically, we fail to obtain the global estimate of $\|\omega\|_{L^\infty}$ to prove uniqueness of the solutions, see \cite{Y63}. However, we can obtain the global strong solution for $a\in [0,1]$ through smallness condition in the critical Besov space. The following result of the global existence do not rely on damping and is original in the literature to our knowledge.
\begin{prop}\label{4prop2}
        Let $d=2$ and $a\in [0,1]$. Assume a divergence-free field $u_0\in H^1\cap B^1_{\infty,1}$ and a symmetric matrix $\tau_0\in H^1\cap B^0_{\infty,1}$. There exists some positive constant $\ep$ small enough such that if
        \begin{align*}
        \|(u_0,\tau_0)\|_{H^1}+\|(\nabla u_0,\tau_0)\|_{B^0_{\infty,1}} \leq \ep,
        \end{align*}
        then \eqref{eq0} admits a global solution $(u,\tau)$ with
        $(u,\tau) \in L^{\infty}(0,\infty;H^1\cap B^1_{\infty,1})\times L^{\infty}(0,\infty;H^1\cap B^0_{\infty,1})$. Moreover, there exists a $C_2>1$ such that
  $$\|(u,\tau)\|_{L^{\infty}(0,\infty;H^1\cap B^1_{\infty,1})\times L^{\infty}(0,\infty;H^1\cap B^0_{\infty,1})}\leq C_2\ep.
        $$
  \end{prop}
 \begin{proof} 
 By Proposition \ref{4prop1}, we have already obtained
  \begin{align}\label{4eq10}
  \left\|(u(t),\tau(t))\right\|_{H^{1}}\leq \frac{3}{2}\ep,\quad t\geq 0.
  \end{align}
  Now we need to estimate
  \begin{equation*}
  \left\|(\nabla u(t),\tau(t))\right\|_{B^{0}_{\infty,1}}.
  \end{equation*}
\textbf{Estimate of $\|\tau(t)\|_{B^{0}_{\infty,1}}$.} \\
 By Duhamel's formula, we get
  \begin{equation*}
  \tau(t)=e^{(\Delta-a)t}\tau_{0}+\int_{0}^{t}e^{(\Delta-a)(t-t')}\left[Du-u\cdot\nabla\tau-Q(\nabla u,\tau)\right]dt'.
  \end{equation*}
  Applying Sobolev's inequality and Bernstein's inequality, we obtain
  \begin{align*}
\|\tau\|_{B^{0}_{\infty,1}}&=\|\Delta_{-1}\tau\|_{L^{\infty}}+\sum_{j\geq 0}\|\Delta_{j}\tau\|_{L^{\infty}}\\
  &\lesssim \|\tau_{0}\|_{B^{0}_{\infty,1}}+\|\tau\|_{L^{2}}\\ &\quad+\sum_{j\geq 0}\int_{0}^{t}\left\|e^{(\Delta-a)(t-t')}\Delta_{j}\left[Du-u\cdot\nabla\tau-Q(\nabla u,\tau)\right]\right\|_{L^{\infty}}dt'\\
  &\lesssim \|\tau_{0}\|_{B^{0}_{\infty,1}}+\|\tau\|_{L^{2}}\\ &\quad+\sum_{j\geq 0}\int_{0}^{t}e^{-c_H2^{2j}(t-t')}2^{\frac{3}{2}j}\left(\|Du\|_{L^{2}}+\|u\|_{L^{8}}\|\tau\|_{L^{8}}+\|\nabla u\|_{L^{2}}\|\tau\|_{L^{4}}\right)dt'.
  \end{align*}
  This together with \eqref{4eq10} ensures that there exist a $C_1>1$ such that
  \begin{align}\label{4eq11}
  \|\tau(t)\|_{B^{0}_{\infty,1}}\leq C_{1}\ep.
  \end{align}
  The estimate for $\|\tau(t)\|_{B^{0}_{\infty,1}}$ is crucial for obtaining global strong solutions, which are time uniformly bounded. \\
\textbf{Estimate of $\|\nabla u(t)\|_{B^{0}_{\infty,1}}$.} \\
 We still start with a \emph{bootstrap assumption}
  \begin{align}\label{boot2}
  \|\nabla u(t)\|_{B^{0}_{\infty,1}}\leq C_{2}\ep,
  \end{align}
  where $ C_{2}\geq 10CC_1 $ and $ \ep\leq\frac{1}{CC_{2}} $.
  
  We now introduce the following structural variable $\Gamma$:
  \begin{equation*}
  \Gamma=\omega-\mathcal{R}\tau.
  \end{equation*}
By direct calculation, $\Gamma$ satisfies the following transport equation with damping:
  \begin{align}\label{4eq12}
  \partial_{t}\Gamma+u\cdot\nabla\Gamma+\frac{1}{2}\Gamma=(a-\frac{1}{2})\mathcal{R}\tau+\mathcal{R}Q+[\mathcal{R},u\cdot\nabla]\tau:=F.
  \end{align}
The structural trick $\Gamma$ allows us to transfer dissipation from $\tau$ to $u$, which is crucial for obtaining a closed estimate of the global solution. 

By Lemma \ref{PL}, Lemma \ref{CR}, and \eqref{4eq10}, \eqref{4eq11},
\begin{align*}
\|F\|_{B^{0}_{\infty,1}}&\leq \left\|\mathcal{R}(\tau)\right\|_{B^{0}_{\infty,1}}+\left\|\mathcal{R}Q\right\|_{B^{0}_{\infty,1}}+\left\|[\mathcal{R},u\cdot\nabla]\tau\right\|_{B^{0}_{\infty,1}}\\
&\lesssim \|\tau\|_{L^{2}}+\|\tau\|_{B^{0}_{\infty,1}}+\|Q\|_{L^{2}}+\|Q\|_{B^{0}_{\infty,1}}+\left\|[\mathcal{R},u\cdot\nabla]\tau\right\|_{B^{0}_{\infty,1}}\\
&\lesssim \|\tau\|_{L^{2}}+\|\tau\|_{B^{0}_{\infty,1}}+\|\nabla u\|_{L^{2}\cap B^{0}_{\infty,1}}\|\tau\|_{H^{2}}\\
&\leq 2CC_1\ep+C_{2}\ep^{2}+CC_{2}\ep\|\nabla^{2} \tau\|_{L^{2}}.
\end{align*}
By the equation \eqref{4eq12}, and Proposition \ref{4prop1}, we infer that
\begin{align*}
\|\Gamma\|_{B^{0}_{\infty,1}}&\leq \|\Gamma_{0}\|_{B^{0}_{\infty,1}}+\int_{0}^{t}e^{-\frac{1}{4}(t-t')}\|F\|_{B^{0}_{\infty,1}}dt'\\
&\lesssim 3CC_1\ep+CC_{2}\ep^{2}.
\end{align*}
Hence, we obtain
\begin{align*}
\|\nabla u\|_{B^{0}_{\infty,1}}&\lesssim \|\nabla u\|_{L^{2}}+\|\omega\|_{B^{0}_{\infty,1}}\\
&\leq \|\nabla u\|_{L^{2}}+\|\mathcal{R}\tau\|_{B^{0}_{\infty,1}}+\|\Gamma\|_{B^{0}_{\infty,1}}\\
&\leq 4CC_1\ep+CC_{2}\ep^{2}\leq \frac{1}{2}C_{2}\ep,
\end{align*}
which is stronger than our bootstrap assumption \eqref{boot2}. Hence, we obtain the uniformly bound
\begin{equation*}
\|\nabla u(t)\|_{B^{0}_{\infty,1}}\leq C_{2}\ep.
\end{equation*}
We finish the proof of Proposition \ref{4prop2}.
\end{proof}
According to Proposition \ref{4prop2}, one can see that global solutions is uniformly bounded in time, which is useful to improve the regularity of the solution under additional initial conditions. The following corollary show that the solution of supercritical regularity is also uniformly bounded in time, which plays a crucial role in considering vanishing damping limit for any $T>0$.
  \begin{coro}\label{4cor1}
     Under the conditions in Proposition \ref{4prop2}, if additionally $u_0\in H^2\cap B^2_{\infty,1}$ and $\tau_0\in H^2\cap B^1_{\infty,1}$, then 
        $(u,\tau) \in L^{\infty}(0,\infty;H^2\cap B^2_{\infty,1})\times L^{\infty}(0,\infty;H^2\cap B^1_{\infty,1})$. Moreover, 
        there holds
     \begin{align*}
\|(u,\tau)\|_{L^{\infty}(0,\infty;H^2)}+\|(\nabla u,\tau)\|_{L^{\infty}(0,\infty;B^1_{\infty,1})}
\leq C(\|(u_0,\tau_0)\|_{H^2}+\|(\nabla u_0,\tau_0)\|_{B^1_{\infty,1}}). 
\end{align*}
  \end{coro}
  \begin{rema}
    According to the theory for the transport equation, the velocity $u\in B^1_{\infty,1}\hookrightarrow C^{0,1}$ is the key to persisting regularity. Under the conditions in Proposition \ref{4prop2}, one can prove the global existence of smooth solutions $(u,\tau)$ for the inviscid Oldroyd-B equation \eqref{eq0} with extra smooth initial data $(u_0,\tau_0)$.
\end{rema}
  \begin{proof}
By energy estimates for \eqref{eq0}, we have
\begin{align*}
&\frac{1}{2}\frac{d}{dt} \left\|\nabla^2(u,\tau)\right\|_{L^{2}}^{2}+\|\nabla^3 \tau\|_{L^{2}}^2+a\|\nabla^2\tau\|_{L^{2}}^{2} \\ \notag
=&-\langle \nabla^2(u\cdot\nabla u), \nabla^2 u \rangle - \langle \nabla^2(u\cdot\nabla\tau), \nabla^2\tau \rangle - \langle \nabla^2 Q,\nabla^2\tau \rangle.
\end{align*}
By virtue of integration by parts, we infer from Proposition \ref{4prop1} that
\begin{align*}
&|\langle -\nabla^2(u\cdot\nabla u), \nabla^2 u \rangle - \langle \nabla^2(u\cdot\nabla\tau), \nabla^2\tau \rangle - \langle \nabla^2 Q,\nabla^2\tau \rangle|\\
=&|\langle -[\nabla^2,u\cdot\nabla] u, \nabla^2 u \rangle + \langle \nabla(u\cdot\nabla\tau), \nabla^3\tau \rangle + \langle \nabla Q,\nabla^3\tau \rangle| \\
\lesssim & \|\nabla^2 u\|_{L^{2}}^2\|\nabla u\|_{L^{\infty}}
+ \|\nabla^3 \tau\|_{L^{2}}\|\nabla \tau\|_{H^{1}} \|u\|_{B^1_{\infty,1}}
+ \|\nabla^3 \tau\|_{L^{2}}\|\nabla^2 u\|_{L^{2}}\|\tau\|_{L^{\infty}}\\
\lesssim & \ep \|\nabla^2 u\|_{L^{2}}^2 
+ \ep\|\nabla^3 \tau\|_{L^{2}}^2 
+ \ep\|\nabla \tau\|_{H^{1}}^2,
\end{align*}
Similarly, we deduce that
\begin{align*}
\frac{d}{dt} \langle \eta\nabla\tau,-\nabla^2 u \rangle + \frac{\eta}{2}\|\nabla^2 u\|_{L^{2}}^2 &=
-\eta \langle\mathbb{P}\left(\mathrm{div}~\tau-u\cdot\nabla u\right),\Delta\mathrm{div}~\tau \rangle \\ \notag
&~~~+\eta \langle \nabla(u\cdot\nabla\tau+a\tau+Q-\Delta \tau),\nabla^2 u\rangle.
\end{align*}
Then we have
\begin{align*}
|\eta \langle\mathbb{P}\left(\mathrm{div}\tau-u\cdot\nabla u\right),\Delta\mathrm{div}\tau \rangle|
&\lesssim  \eta \|\nabla^2 \tau\|_{L^{2}}^2 
+ \eta \|\nabla^3 \tau\|_{L^{2}}\|\nabla u\|_{L^{2}}\|u\|_{L^{\infty}} \\
&\lesssim  \eta \|\nabla^2 \tau\|_{L^{2}}^2 
+ \ep\|\nabla^3 \tau\|_{L^{2}}^2 + \ep\|\nabla^2 u\|_{L^{2}}^2.
\end{align*}
and
\begin{align*}
&|\eta \langle \nabla(u\cdot\nabla\tau+a\tau+Q-\Delta \tau),\nabla^2 u\rangle|\\
\lesssim & \eta \|\nabla^2 u\|_{L^{2}}\|\nabla \tau\|_{H^{1}} \|u\|_{B^1_{\infty,1}} 
+ \eta a \|\nabla^2 \tau\|_{L^2}\|\nabla u\|_{L^{2}} + \eta \|\nabla^2 u\|_{L^{2}}^2\|\tau\|_{L^{\infty}}\\
&+ \eta \|\nabla^2 u\|_{L^{2}}\|\nabla \tau\|_{L^2}\|u\|_{B^1_{\infty,1}}
+ \eta \|\nabla^3 \tau\|_{L^{2}}\|\nabla^2u\|_{L^{2}}\\
\lesssim & \ep\|\nabla \tau\|_{H^{2}}^2 
+ \ep\|\nabla u\|_{H^{1}}^2
+ \eta^{\frac{1}{2}}\|\nabla^2 \tau\|_{H^{1}}^2
+ \eta^{\frac{3}{2}}\|\nabla u\|_{H^{1}}^2.
\end{align*}
Combining the above inequalities and \eqref{4ineq1}, we obtain 
\begin{equation*}
\frac{d}{dt}\left(\frac{1}{2}\left\|(u,\tau)\right\|_{H^{2}}^{2}+\eta\langle \tau,-\nabla u \rangle + \eta\langle \nabla\tau,-\nabla^2 u \rangle\right)
+\frac{1}{8}\|\nabla \tau\|_{H^{2}}^2
+ \frac{\eta}{8}\|\nabla u\|_{H^{1}}^2
\leq 0,
\end{equation*}
which implies that for any $T>0$, we have
\begin{equation}\label{4su4}
\|(u,\tau)\|_{L_T^{\infty}(H^2)}^2+  \int_0^T \|\nabla \tau\|_{H^{2}}^2
+ \|\nabla u\|_{H^{1}}^2 dt \leq C\|(u_0,\tau_0)\|_{H^2}^2.
\end{equation}

We now focus on the estimate of $\|\tau\|_{L_T^{\infty}(B^1_{\infty,1})}$. Using Lemma \ref{H}, we get
  \begin{equation}\label{4su5}\begin{aligned}
\|\tau(t)\|_{B^1_{\infty,1}}&=\frac 1 2\|\Delta_{-1}\tau(t)\|_{L^\infty}+\sum_{j\geq 0}2^j\|\Delta_{j}\tau(t)\|_{L^\infty}\\ 
  &\lesssim \|\tau_{0}\|_{B^{1}_{\infty,1}}+\|(u_0,\tau_0)\|_{H^2}\\ &\quad+\sum_{j\geq 0}\int_{0}^{t}2^j\left\|e^{(\Delta-a)(t-t')}\Delta_{j}\left[Du-u\cdot\nabla\tau-Q(\nabla u,\tau)\right]\right\|_{L^{\infty}}dt'\\ 
  &\lesssim \|\tau_0\|_{B^1_{\infty,1}}+\|(u_0,\tau_0)\|_{H^2}\\ &\quad+\sum_{j\geq 0}\int_{0}^{t}e^{-c_H2^{2j}(t-t')}2^{\frac{3}{2}j}\left(\|u\|_{H^{2}}+\|u\|_{L^{\infty}}\|\nabla\tau\|_{L^{4}}+\|\nabla u\|_{L^{\infty}}\|\tau\|_{L^{4}}\right)dt'\\ 
  &\lesssim \|\tau_{0}\|_{B^{1}_{\infty,1}}+\|(u_0,\tau_0)\|_{H^2}.
  \end{aligned}\end{equation}
Using Lemma \ref{CI}, we have
\begin{align*}
\Big\|(2^{j}\|[\Delta_j,u\cdot \nabla] \Gamma\|_{L^{\infty}_{x}})_j\Big\|_{l^1(\mathbb{Z})}\leq C\|\nabla u\|_{B^{0}_{\infty,1}}\|\Gamma\|_{B^1_{\infty,1}}.
\end{align*}
This together with Proposition \ref{4prop2} implies that
\begin{align}\label{4su6}
\|\Gamma\|_{L_T^{\infty}(B^1_{\infty,1})} 
\lesssim &\|\Gamma_0\|_{B^1_{\infty,1}}+\sup_{0\leq r\leq T}\int_0^r e^{-\frac 1 4(r-t)} \|F(t)\|_{B^1_{\infty,1}} dt
\end{align}
By Lemmas \ref{PL} and \ref{CR}, we infer from \eqref{4su4} and  \eqref{4su5} that
\begin{align*}
\|F\|_{B^{1}_{\infty,1}}&\lesssim \|F\|_{B^{0}_{\infty,1}}+\|\nabla F\|_{B^{0}_{\infty,1}} \\ 
&\lesssim \|(u_0,\tau_0)\|_{H^1}+\|(\nabla u_0,\tau_0)\|_{B^0_{\infty,1}}+\left\|\tau\right\|_{B^{1}_{\infty,1}}+\left\|Q\right\|_{B^{1}_{\infty,1}}+\left\|\nabla([\mathcal{R},u\cdot\nabla]\tau)\right\|_{B^{0}_{\infty,1}}\\
&\lesssim \|(u_0,\tau_0)\|_{H^2}+\|(u_0,\tau_0)\|_{B^1_{\infty,1}}+\left\|\Gamma\right\|_{B^{1}_{\infty,1}}\left\|\tau\right\|_{B^{0}_{\infty,1}}+\left\|u\right\|_{B^{1}_{\infty,1}}\left\|\tau\right\|_{B^{1}_{\infty,1}} \\
&+\left\|[\mathcal{R}, \nabla u\cdot\nabla]\tau\right\|_{B^{0}_{\infty,1}}+\left\|[\mathcal{R}, u\cdot\nabla]\nabla\tau\right\|_{B^{0}_{\infty,1}} \\
&\lesssim \|(u_0,\tau_0)\|_{H^2}+\|(u_0,\tau_0)\|_{B^1_{\infty,1}}+\left\|\Gamma\right\|_{B^{1}_{\infty,1}}\left\|\tau\right\|_{B^{\frac 1 2}_{\infty,1}}+\left\|u\right\|_{B^{1}_{\infty,1}}\left\|\tau\right\|_{B^{\frac 3 2}_{\infty,1}} \\
&\lesssim \|(u_0,\tau_0)\|_{H^2}+\|(u_0,\tau_0)\|_{B^1_{\infty,1}}+\left\|\Gamma\right\|_{B^{1}_{\infty,1}}\left\|\tau\right\|_{H^2}+\left\|u\right\|_{B^{1}_{\infty,1}}\left\|\tau\right\|_{H^3}.
\end{align*}
This together with \eqref{4su6} ensures that
the factor $\|\tau\|_{H^2}$ can be absorbed after convolution,
not by assuming it is pointwise small. Specifically,
\eqref{4ineq1}, \eqref{4su4}, and Cauchy--Schwarz imply
\begin{align*}
 \sup_{r\geq0}\int_0^r e^{-(r-t)/4}\|\tau(t)\|_{H^2}\,dt
 &\leq C\|\tau\|_{L^\infty H^1}
       +C\|\nabla^2\tau\|_{L^2_tL^2_x}\leq C\ep,\\
 \sup_{r\geq0}\int_0^r e^{-(r-t)/4}\|\tau(t)\|_{H^3}\,dt
 &\leq C\|(u_0,\tau_0)\|_{H^2}.
\end{align*}
Only $H^2$ regularity of the initial data is used in these
convolution bounds; the additional Besov norms in the
corollary need only be finite.
Consequently,
\begin{align*}
\|\Gamma\|_{L_T^{\infty}(B^1_{\infty,1})} 
&\lesssim \|\Gamma_0\|_{B^1_{\infty,1}}+\|(u_0,\tau_0)\|_{H^2}+\|(u_0,\tau_0)\|_{B^1_{\infty,1}} \\ \notag
&~~~+\sup_{0\leq r\leq T}\int_0^r e^{-\frac 1 4(r-t)} (\left\|\Gamma\right\|_{B^{1}_{\infty,1}}\left\|\tau\right\|_{H^2}+\left\|u\right\|_{B^{1}_{\infty,1}}\left\|\tau\right\|_{H^3}) dt \\ \notag
&\lesssim \|(u_0,\tau_0)\|_{H^2}+\|(\nabla u_0,\tau_0)\|_{B^1_{\infty,1}}+\ep\|\Gamma\|_{L_T^{\infty}(B^1_{\infty,1})}. 
\end{align*}
This implies that
\begin{align*}
\|\Gamma\|_{L_T^{\infty}(B^1_{\infty,1})} 
\lesssim \|(u_0,\tau_0)\|_{H^2}+\|(\nabla u_0,\tau_0)\|_{B^1_{\infty,1}}. 
\end{align*}
We thus complete the proof of Corollary \ref{4cor1}.
\end{proof}
With the results of the solutions are uniformly bounded in time, we can consider vanishing damping limit for any $T>0$ under the same topology and prove vanishing damping rate for low frequency of the solutions. Since the solutions $(u,\tau)$ of the inviscid Oldroyd-B model \eqref{eq0} depends on damping $a\in[0,1]$, when considering the problems of vanishing damping limit, we denote the solutions of \eqref{eq0} as $(u^a,\tau^a)$.
\begin{prop}\label{4prop3} 
Under the conditions in Proposition \ref{4prop2}, for any $T>0$, there holds 
  $$\|(u^a,\tau^a)-(u^0,\tau^0)\|_{L^{\infty}(0,T;L^2)}\leq C_Ta,
  $$ 
  and 
  $$\lim_{a\rightarrow0}\|(u^a,\tau^a)-(u^0,\tau^0)\|_{L^{\infty}(0,T;\dot{H}^1\cap B^1_{\infty,1})\times L^{\infty}(0,T;\dot{H}^1\cap B^0_{\infty,1})}=0.
  $$
  \end{prop}
    \begin{proof} Let $(u^a,\tau^a)$ be a solution of $\eqref{eq0}$ with the initial data $(u_0,\tau_0)$ and $a\in[0,1]$. Then we get 
\begin{align}\label{4lo1}
\left\{\begin{array}{l}
\partial_t(u^a-u^0)+u^a\cdot\nabla (u^a-u^0)+\nabla (P_{u^a}-P_{u^0}) \\[1ex]
={\rm div}~(\tau^a-\tau^0)-(u^a-u^0)\cdot\nabla u^0,~~~~{\rm div}~(u^a-u^0)=0,\\[1ex]
\partial_t(\tau^a-\tau^0)+u^a\cdot\nabla(\tau^a-\tau^0)-\Delta(\tau^a-\tau^0)+Q(\nabla u^a,\tau^a-\tau^0)\\[1ex]
=-a\tau^a+D(u^a-u^0)-(u^a-u^0)\cdot\nabla\tau^0-Q(\nabla (u^a-u^0),\tau^0),\\[1ex]
(u^a-u^0)|_{t=0}=0,~~(\tau^a-\tau^0)|_{t=0}=0. \\[1ex]
\end{array}\right.
\end{align}
We divide the proof into three parts. \\
\textbf{Step 1: Vanishing damping limit for low frequency.} \\
By \eqref{4lo1}, we obtain the following energy identity:
\begin{align}\label{4lo2}
&\frac{1}{2}\frac{d}{dt}\left\|(u^{a}-u^{0},\tau^{a}-\tau^{0})\right\|_{L^{2}}^{2}+\|\nabla (\tau^{a}-\tau^{0})\|_{L^{2}}^{2}\\
=
& -\langle a\tau^{a}+Q(\nabla u^{a},\tau^{a}-\tau^{0})+Q(\nabla(u^{a}-u^{0}),\tau^{0})+(u^{a}-u^{0})\cdot\nabla \tau^{0}, (\tau^{a}-\tau^{0})\rangle \notag\\
& -\langle \left[(u^{a}-u^{0})\cdot\nabla u^{0}\right], (u^{a}-u^{0}) \rangle. \notag
\end{align}
Define
\begin{align*}
\phi(t):=\left\|(u^{a}-u^{0},\tau^{a}-\tau^{0})\right\|_{L^{2}}.
\end{align*}
Then we have the following inequalities:
\begin{align*}
\left| \langle (u^{a}-u^{0})\cdot\nabla u^{0}, u^{a}-u^{0}\rangle\right|\leq \|\nabla u^{0}\|_{L^{\infty}}\phi^{2}(t),
\end{align*}
and
\begin{align*}
\left|\langle a\tau^{a}, \tau^{a}-\tau^{0}\rangle\right|\leq a^{2}+\|\tau^{a}\|_{L^{2}}^{2}\phi^{2}(t).
\end{align*}
One can infer that
\begin{align*}
\left|\langle Q(\nabla u^{a},\tau^{a}-\tau^{0}), \tau^{a}-\tau^{0}\rangle\right|\leq \|\nabla u^{a}\|_{L^{\infty}}\phi^{2}(t),
\end{align*}
and
\begin{align*}
&\left|\langle Q(\nabla(u^{a}-u^{0}),\tau^{0}), \tau^{a}-\tau^{0}\rangle \right|\\\notag
\leq &\|u^{a}-u^{0}\|_{L^{2}}\left(\|\nabla\tau^{0}\|_{L^{4}}\|\tau^{a}-\tau^{0}\|_{L^{4}}+\|\tau^{0}\|_{L^{\infty}}\|\nabla(\tau^{a}-\tau^{0})\|_{L^{2}}\right)\\\notag
\leq &\frac{1}{4}\|\nabla(\tau^{a}-\tau^{0})\|_{L^{2}}^{2}+C\phi^{2}(t)\left(\|\nabla\tau^{0}\|_{L^{4}}^{4/3}+\|\tau^{0}\|_{L^{\infty}}^{2}\right).
\end{align*} 
Moreover, we obtain
\begin{align}\label{4lo7}
\left|\langle (u^{a}-u^{0})\cdot\nabla \tau^{0}, (\tau^{a}-\tau^{0})\rangle \right|&\leq\|u^{a}-u^{0}\|_{L^{2}}\|\nabla\tau^{0}\|_{L^{4}}\|\tau^{a}-\tau^{0}\|_{L^{4}}\\ \notag
&\leq\frac{1}{4}\|\nabla(\tau^{a}-\tau^{0})\|_{L^{2}}^{2}+C\phi^{2}(t)\|\nabla\tau^{0}\|_{L^{4}}^{4/3}. \notag
\end{align}
Here $\|\Sigma\|_4\lesssim\|\Sigma\|_2^{1/2}
\|\nabla\Sigma\|_2^{1/2}$, for $\Sigma=\tau^a-\tau^0$.
No $L^4$ assumption on $\nabla\tau_0$ is needed:
$\|\nabla\tau^0\|_4^2\lesssim
\|\nabla\tau^0\|_2\|\nabla^2\tau^0\|_2$, so its $4/3$ power is
integrable on every finite interval by \eqref{4ineq1}.
Applying Gronwall's inequality, we infer from \eqref{4lo2}-\eqref{4lo7} that
\begin{equation}\label{4lo8}
\sup_{0\leq t\leq T}\phi^{2}(t)\lesssim_{T} a^{2}.
\end{equation}

Next we estimate $\|\tau^{a}-\tau^{0}\|_{B^{0}_{\infty,1}}$. Note that
\begin{align}\label{4lo9}
\|\tau^{a}-\tau^{0}\|_{B^{0}_{\infty,1}}&=\|\Delta_{-1}(\tau^{a}-\tau^{0})\|_{L^{\infty}}+\sum_{j\geq 0}\|\Delta_{j}(\tau^{a}-\tau^{0})\|_{L^{\infty}}\\ \notag
&\lesssim \|\tau^{a}-\tau^{0}\|_{L^{2}}+\sum_{j\geq 0}\|\Delta_{j}(\tau^{a}-\tau^{0})\|_{L^{\infty}}\\ \notag
&\lesssim_{T}a+\sum_{j\geq 0}\|\Delta_{j}(\tau^{a}-\tau^{0})\|_{L^{\infty}}. \notag
\end{align}
By Duhamel's formula
\begin{align*}
\Delta_{j}(\tau^{a}-\tau^{0})(t)&=\int_{0}^{t}e^{\Delta(t-s)}\Delta_{j}\bigg[D(u^{a}-u^{0})-u^{a}\cdot\nabla(\tau^{a}-\tau^{0})\\
&\quad -(u^{a}-u^{0})\cdot\nabla\tau^{0}-Q(\nabla u^{a},\tau^{a}-\tau^{0})-Q(\nabla(u^{a}-u^{0}),\tau^{0})-a\tau^{a}\bigg]ds.
\end{align*}
Applying Lemma \ref{H}, we have
\begin{align*}
\|\Delta_{j}(\tau^{a}-\tau^{0})(t)\|_{L^{\infty}}&\lesssim \int_{0}^{t}e^{-c(t-s)2^{2j}}\Bigg[\|\Delta_{j}D(u^{a}-u^{0})\|_{L^{\infty}}+\|\Delta_{j}\left[u^{a}\cdot\nabla(\tau^{a}-\tau^{0})\right]\|_{L^{\infty}}\\
&\quad+\|\Delta_{j}\left[(u^{a}-u^{0})\cdot\nabla\tau^{0}\right]\|_{L^{\infty}}+\|\Delta_{j}Q(\nabla u^{a},(\tau^{a}-\tau^{0}))\|_{L^{\infty}}\\
&\quad+\|\Delta_{j}Q(\nabla(u^{a}-u^{0}),\tau^{0})\|_{L^{\infty}}+a\|\Delta_{j}\tau^{a}\|_{L^{\infty}}\Bigg]ds\\
&=A_{1}+A_{2}+A_{3}+A_{4}+A_{5}+A_{6}.
\end{align*}
Applying Bernstein's inequality and Proposition \ref{4prop2}, we have the following inequalities:
\begin{align*}
A_{1}&\lesssim \int_{0}^{t}e^{-c(t-s)2^{2j}}2^{\frac{3}{2}j}\|u^{a}-u^{0}\|_{L^{4}}ds\\
&\lesssim \int_{0}^{t}e^{-c(t-s)2^{2j}}2^{\frac{3}{2}j}\|u^{a}-u^{0}\|_{L^{2}}^{\frac{3}{4}}\|u^{a}-u^{0}\|_{B^{1}_{\infty,1}}^{\frac{1}{4}}ds\lesssim _{T}2^{-\frac{j}{2}}a^{\frac{3}{4}},
\end{align*}
and
\begin{align*}
A_{2}&\lesssim \int_{0}^{t}e^{-c(t-s)2^{2j}}2^{\frac{7}{4}j}\|u^{a}\|_{L^{\infty}}\|\tau^{a}-\tau^{0}\|_{L^{\frac{8}{3}}}ds\\
&\lesssim \int_{0}^{t}e^{-c(t-s)2^{2j}}2^{\frac{7}{4}j}\|u^{a}\|_{L^{\infty}}\|\tau^{a}-\tau^{0}\|_{L^{2}}^{\frac{3}{4}}\|\tau^{a}-\tau^{0}\|_{L^{\infty}}^{\frac14}ds\lesssim_{T} 2^{-\frac{j}{4}}a^{\frac{3}{4}}.
\end{align*}
One can deduce that
\begin{align*}
A_{3}&\lesssim \int_{0}^{t}e^{-c(t-s)2^{2j}}2^{\frac{3}{2}j}\|u^{a}-u^{0}\|_{L^{4}}\|\tau^{0}\|_{L^{\infty}}ds\\
&\lesssim \int_{0}^{t}e^{-c(t-s)2^{2j}}2^{\frac{3}{2}j}\|u^{a}-u^{0}\|_{L^{2}}^{\frac{3}{4}}\|u^{a}-u^{0}\|_{B^{1}_{\infty,1}}^{\frac{1}{4}}\|\tau^{0}\|_{L^{\infty}}ds\lesssim_{T}a^{\frac{3}{4}}2^{-\frac{j}{2}},
\end{align*}
and
\begin{align*}
A_{4}&\lesssim \int_{0}^{t}e^{-c(t-s)2^{2j}}2^{j}\|\nabla u^{a}\|_{L^{\infty}}\|\tau^{a}-\tau^{0}\|_{L^{2}}ds\lesssim_{T} 2^{-j}a.
\end{align*}
For $A_5$, put $U=u^a-u^0$ and integrate the derivative in
$Q(\nabla U,\tau^0)$ by parts: each component is a sum of
$\partial_k(U_\ell\tau^0_{mn})$ and
$U_\ell\partial_k\tau^0_{mn}$. Thus
\begin{align*}
A_{5}&\lesssim \int_{0}^{t}e^{-c(t-s)2^{2j}}
 \left[2^{3j/2}\|U\|_4\|\tau^0\|_\infty
       +2^j\|U\|_4\|\nabla\tau^0\|_4\right]ds\\
 &\lesssim_T 2^{-j/2}a^{3/4}.
\end{align*}
In the second term use
$\|\nabla\tau^0\|_{L^4(0,T;L^4)}\leq C_T$ and the
$L^{4/3}$ norm $C2^{-3j/2}$ of the heat kernel in time.
The interpolation $\|U\|_4\lesssim\|U\|_2^{3/4}
\|U\|_{B^1_{\infty,1}}^{1/4}\lesssim_T a^{3/4}$
controls both terms.
and
\begin{align*}
A_{6}&\lesssim \int_{0}^{t}e^{-c(t-s)2^{2j}}a\|\tau^{a}\|_{L^{\infty}}ds\lesssim 2^{-2j}a.
\end{align*}
The above inequalities together with \eqref{4lo9} ensure that
\begin{equation}\label{4lo10}
\sup_{0\leq t\leq T}\|\tau^{a}-\tau^{0}\|_{B^{0}_{\infty,1}}\lesssim_{T}a^{\frac{3}{4}}.
\end{equation}
\textbf{Step 2: Vanishing damping limit in $\dot{H}^1$.} \\
In general analysis for stability and vanishing limit, there is a phenomenon of supercritical regularity in the estimate of transport terms $u^a\cdot \nabla u^a$. We introduce a low-frequency truncated smoothing operator $S_N$ to quantify the regularity.

Let $(u_N^a,\tau_N^a)$ be a solution of $\eqref{eq0}$ with the initial data $(S_N u_0,S_N\tau_0)$ and $a\in[0,1]$. Then we get 
\begin{align}\label{4lo11}
\left\{\begin{array}{l}
\partial_t(u^a-u_N^a)+u^a\cdot\nabla (u^a-u_N^a)+\nabla (P_{u^a}-P_{u_N^a}) \\[1ex]
={\rm div}~(\tau^a-\tau_N^a)-(u^a-u_N^a)\cdot\nabla u_N^a,~~~~{\rm div}~(u^a-u_N^a)=0,\\[1ex]
\partial_t(\tau^a-\tau_N^a)+u^a\cdot\nabla(\tau^a-\tau_N^a)-\Delta(\tau^a-\tau_N^a)+a(\tau^a-\tau_N^a)+Q(\nabla u^a,\tau^a-\tau_N^a)\\[1ex]
=D(u^a-u_N^a)-(u^a-u_N^a)\cdot\nabla\tau_N^a-Q(\nabla (u^a-u_N^a),\tau_N^a),\\[1ex]
(u^a-u_N^a)|_{t=0}=(Id-S_N)u_0,~~(\tau^a-\tau_N^a)|_{t=0}=(Id-S_N)\tau_0. \\[1ex]
\end{array}\right.
\end{align}
Proceeding $L^{2}$ energy estimate again, we obtain
\begin{align}\label{4lo12}
\sup_{0\leq t\leq T}\|(u^{a}-u_{N}^{a},\tau^{a}-\tau^{a}_{N})\|_{L^{2}}&\lesssim_{T}\|(Id-S_{N})u_{0},(Id-S_{N})\tau_{0}\|_{L^{2}}\\\notag
&\lesssim_{T} 2^{-N}\|(Id-S_{N})u_{0},(Id-S_{N})\tau_{0}\|_{H^{1}}.\notag
\end{align}
For these fixed initial data, define the tail
\begin{equation*}
c_N=\|(Id-S_N)(u_0,\tau_0)\|_{H^1}
 +\|(Id-S_N)u_0\|_{B^1_{\infty,1}}
 +\|(Id-S_N)\tau_0\|_{B^0_{\infty,1}}\longrightarrow0.
\end{equation*}
All constants below are independent of $a$ and $N$; the notation
$c_N$ in bounds may absorb a fixed multiple of this tail.
Then \eqref{4lo12} can be simplified by
\begin{align*}
\sup_{0\leq t\leq T}\|(u^{a}-u_{N}^{a},\tau^{a}-\tau^{a}_{N})\|_{L^{2}}\lesssim_{T} 2^{-N}c_{N}.
\end{align*}
Next, we test the first equation of \eqref{4lo11} by $-\Delta(u^{a}-u^{a}_{N})$ and the second equation of \eqref{4lo11} by $-\Delta(\tau^{a}-\tau^{a}_{N})$, we obtain the following energy inequality:
\begin{align}\label{4lo14}
&\frac{1}{2}\frac{d}{dt}\left\|\nabla(u^{a}-u^{a}_{N}),\nabla(\tau^{a}-\tau^{a}_{N})\right\|_{L^{2}}^{2}+\left\|\nabla^{2}(\tau^{a}-\tau^{a}_{N})\right\|_{L^{2}}^{2}\\\notag
\leq&\left|\langle u^{a}\cdot\nabla(u^{a}-u^{a}_{N}), \Delta(u^{a}-u^{a}_{N})\rangle\right|+\left|\langle(u^{a}-u^{a}_{N})\cdot\nabla u^{a}_{N}, \Delta (u^{a}-u^{a}_{N})\rangle\right|\\\notag
&+\left|\langle u^{a}\cdot\nabla(\tau^{a}-\tau^{a}_{N}), \Delta(\tau^{a}-\tau^{a}_{N})\rangle \right|+\left|\langle(u^{a}-u^{a}_{N})\cdot\nabla\tau^{a}_{N}, \Delta(\tau^{a}-\tau^{a}_{N})\rangle \right|\\\notag
&+\left|\langle Q(\nabla u^{a},\tau^{a}-\tau^{a}_{N}), \Delta(\tau^{a}-\tau^{a}_{N})\rangle \right|+\left|\langle Q(\nabla(u^{a}-u^{a}_{N}),\tau^{a}_{N}), \Delta(\tau^{a}-\tau^{a}_{N})\rangle\right|\\\notag
=:&B_{1}+B_{2}+B_{3}+B_{4}+B_{5}+B_{6}.\notag
\end{align}
By virtue of Corollary \ref{4cor1} and integrating by parts, we obtain
\begin{align*}
B_{1}+B_{2}&\leq \left(\|\nabla u^{a}\|_{L^{\infty}}+\|\nabla u^{a}_{N}\|_{L^{\infty}}\right)\|\nabla(u^{a}-u^{a}_{N})\|_{L^{2}}^{2}+\|u^{a}-u^{a}_{N}\|_{L^{2}}\|\nabla^{2}u^{a}_{N}\|_{L^{\infty}}\|\nabla(u^{a}-u^{a}_{N})\|_{L^{2}}\\
&\lesssim_{T}  \left(\|\nabla u^{a}\|_{L^{\infty}}+\|\nabla u^{a}_{N}\|_{L^{\infty}}\right)\|\nabla(u^{a}-u^{a}_{N})\|_{L^{2}}^{2}+\|\nabla(u^{a}-u^{a}_{N})\|_{L^{2}}c_{N}\\
&\lesssim_{T}  \left(\|\nabla u^{a}\|_{L^{\infty}}+\|\nabla u^{a}_{N}\|_{L^{\infty}}+1\right)\|\nabla(u^{a}-u^{a}_{N})\|_{L^{2}}^{2}+c^2_{N}.
\end{align*}
By H\"older's inequality and Corollary \ref{4cor1},
\begin{align*}
&B_{3}+B_{4}+B_{5}+B_{6}\\
\lesssim&\|u^{a}\|_{L^{\infty}}^{2}\|\nabla(\tau^{a}-\tau^{a}_{N})\|_{L^{2}}^{2}+\|u^{a}-u^{a}_{N}\|_{L^{2}}^{2}\|\nabla\tau^{a}_{N}\|_{L^{\infty}}^{2}\\
&+\|\nabla u^{a}\|_{L^{\infty}}^{2}\|\tau^{a}-\tau^{a}_{N}\|_{L^{2}}^{2}+\|\nabla(u^{a}-u^{a}_{N})\|_{L^{2}}^{2}\|\tau^{a}_{N}\|_{L^{\infty}}^{2}+\frac{1}{2}\|\nabla^{2}(\tau^{a}-\tau^{a}_{N})\|_{L^{2}}^{2}\\
\lesssim _{T}&\|u^{a}\|_{L^{\infty}}^{2}\|\nabla(\tau^{a}-\tau^{a}_{N})\|_{L^{2}}^{2}+\|\tau^{a}_{N}\|_{L^{\infty}}^{2}\|\nabla(u^{a}-u^{a}_{N})\|_{L^{2}}^{2}+\frac{1}{2}\|\nabla^{2}(\tau^{a}-\tau^{a}_{N})\|_{L^{2}}^{2}+c^2_{N}.
\end{align*}
By the above inequalities, applying Gronwall's inequality to \eqref{4lo14}, we have
\begin{align}\label{4lo15}
\sup_{0\leq t\leq T}\left\|\nabla(u^{a}-u^{a}_{N}),\nabla(\tau^{a}-\tau^{a}_{N})\right\|_{L^{2}}\lesssim_{T} c_{N}.
\end{align}
According to  \eqref{eq0}, we get
\begin{align}\label{4lo16}
\left\{\begin{array}{l}
\partial_t(u_N^a-u_N^0)+u_N^a\cdot\nabla (u_N^a-u_N^0)+\nabla (P_{u_N^a}-P_{u_N^0}) \\[1ex]
={\rm div}~(\tau_N^a-\tau_N^0)-(u_N^a-u_N^0)\cdot\nabla u_N^0,~~~~{\rm div}~(u_N^a-u_N^0)=0,\\[1ex]
\partial_t(\tau_N^a-\tau_N^0)+u_N^a\cdot\nabla(\tau_N^a-\tau_N^0)-\Delta(\tau_N^a-\tau_N^0)+Q(\nabla u_N^a,\tau_N^a-\tau_N^0)\\[1ex]
=-a\tau_N^a+D(u_N^a-u_N^0)-(u_N^a-u_N^0)\cdot\nabla\tau_N^0-Q(\nabla (u_N^a-u_N^0),\tau_N^0),\\[1ex]
(u_N^a-u_N^0)|_{t=0}=0,~~(\tau_N^a-\tau_N^0)|_{t=0}=0. \\[1ex]
\end{array}\right.
\end{align}
Testing the first equation of \eqref{4lo16} by $-\Delta(u^{a}_{N}-u^{0}_{N})$, the second equation by $-\Delta(\tau^{a}_{N}-\tau^{0}_{N})$, and integrating by parts, we have the following energy inequality:
\begin{align*}
&\frac{1}{2}\frac{d}{dt}{\underbrace{\left\|\nabla(u^{a}_{N}-u^{0}_{N}),\nabla(\tau^{a}_{N}-\tau^{0}_{N})\right\|_{L^{2}}^{2}}_{\coloneqq M_{1}(t)}}+\left\|\nabla^{2}(\tau^{a}_{N}-\tau^{0}_{N})\right\|_{L^{2}}^{2}\\\notag
\leq&\left|\langle u^{a}_{N}\cdot\nabla(u^{a}_{N}-u^{0}_{N}), \Delta(u^{a}_{N}-u^{0}_{N})\rangle\right|+\left|\langle (u^{a}_{N}-u^{0}_{N})\cdot\nabla u^{0}_{N}, \Delta (u^{a}_{N}-u^{0}_{N})\rangle \right|\\\notag
&+\left|\langle a\tau^{a}_{N}, \Delta(\tau^{a}_{N}-\tau^{0}_{N})\rangle\right|+\left|\langle u^{a}_{N}\cdot\nabla(\tau^{a}_{N}-\tau^{0}_{N}), \Delta(\tau^{a}_{N}-\tau^{0}_{N})\rangle\right|\\\notag
&+\left|\langle(u^{a}_{N}-u^{0}_{N})\cdot\nabla\tau^{0}_{N}, \Delta(\tau^{a}_{N}-\tau^{0}_{N})\rangle\right|+\left|\langle Q(\nabla u^{a}_{N},\tau^{a}_{N}-\tau^{0}_{N}), \Delta(\tau^{a}_{N}-\tau^{0}_{N})\rangle\right|\\\notag
&+\left|\langle Q(\nabla(u^{a}_{N}-u^{0}_{N}),\tau^{0}_{N}), \Delta(\tau^{a}_{N}-\tau^{0}_{N})\rangle\right|\\\notag
\leq&\left(\|\nabla u^{a}_{N}\|_{L^{\infty}}+\|\nabla u^{0}_{N}\|_{L^{\infty}}\right)M_{1}(t)+\|u^{a}_{N}-u^{0}_{N}\|_{L^{2}}\|\nabla^{2}u^{0}_{N}\|_{L^{\infty}}\|\nabla(u^{a}_{N}-u^{0}_{N})\|_{L^{2}}\\\notag
&+\frac{1}{2}\|\nabla^{2}(\tau^{a}_{N}-\tau^{0}_{N})\|_{L^{2}}^{2}+a^{2}\|\tau^{a}_{N}\|_{L^{2}}^{2}+\|u^{a}_{N}\|_{L^{\infty}}^{2}M_{1}(t)\\\notag
&+\|u^{a}_{N}-u^{0}_{N}\|_{L^{2}}^{2}\|\nabla\tau^{0}_{N}\|_{L^{\infty}}^{2}+\|\nabla u^{a}_{N}\|_{L^{\infty}}^{2}\|\tau^{a}_{N}-\tau^{0}_{N}\|_{L^{2}}^{2}+\|\tau^{0}_{N}\|_{L^{\infty}}^{2}M_{1}(t).\notag
\end{align*}
By the same argument in \textbf{Step 1}, we have
\begin{align*}
\|(u^{a}_{N}-u^{0}_{N},\tau^{a}_{N}-\tau^{0}_{N})\|_{L^{2}}\lesssim _{T} a.
\end{align*}
By the virtue of Corollary \ref{4cor1}, we get
\begin{align*}
\|\nabla^{2}u^{0}_{N}\|_{L^{\infty}}+\|\nabla \tau^{0}_{N}\|_{L^{\infty}}\lesssim 2^{N}.
\end{align*}
Applying Gronwall's inequality, we have
\begin{align}\label{4lo18}
\sup_{0\leq t\leq T}M_1(t)\lesssim _{T}a^{2}2^{2N}.
\end{align}
Combining with \eqref{4lo15} and \eqref{4lo18}, we obtain
\begin{align*}
\sup_{0\leq t\leq T}\|\nabla(u^{a}-u^{0}),\nabla(\tau^{a}-\tau^{0})\|_{L^{2}}\lesssim_{T} c_{N}+a2^{N},
\end{align*}
which implies that
  $$\lim_{a\rightarrow0}\|(u^a,\tau^a)-(u^0,\tau^0)\|_{L^{\infty}(0,T;\dot{H}^1)\times L^{\infty}(0,T;\dot{H}^1)}=0.$$
    \textbf{Step 3: Vanishing damping limit in $B^1_{\infty,1}$.} \\
    We also need a low-frequency truncated smoothing operator $S_N$ to quantify the regularity. In addition, we also need to establish a low-frequency analysis of the global stability for $(u^a,\tau^a)$ regarding the smoothing operator $S_N$.
    Define
    \begin{equation*}
    \begin{cases}
    M_{2}(t)\coloneqq\left\|(u^{a}-u^{a}_{N})(t)\right\|_{B^{0}_{\infty,1}},\\
    M_{3}(t)\coloneqq\left\|(\tau^{a}-\tau^{a}_{N})(t)\right\|_{B^{-1}_{\infty,1}}+\displaystyle\int_{0}^{t}\left\|\tau^{a}-\tau^{a}_{N}\right\|_{B^{1}_{\infty,1}}ds.
    \end{cases}
    \end{equation*}
    By the first equation of \eqref{4lo11}, we infer from Proposition \ref{4prop2} that
    \begin{align}\label{4lo20}
    M_{2}(t)\lesssim _{T}2^{-N}c_{N}+\int_{0}^{t}\left\|(u^{a}-u^{a}_{N})\cdot\nabla u^{a}_{N}\right\|_{B^{0}_{\infty,1}}ds+\int_{0}^{t}\left\|\nabla(p_{u^{a}}-p_{u^{a}_{N}})\right\|_{B^{0}_{\infty,1}}ds+M_{3}(t).
    \end{align}
    Applying Lemma \ref{PL}, we have
    \begin{align*}
    \int_{0}^{t}\left\|(u^{a}-u^{a}_{N})\cdot \nabla u^{a}_{N}\right\|_{B^{0}_{\infty,1}}ds\lesssim \int_{0}^{t}M_{2}(s)ds.
    \end{align*}
    By the virtue of Lemma \ref{P}, we deduce that
    \begin{align*}
    \int_{0}^{t}\left\|\nabla(p_{u^{a}}-p_{u^{a}_{N}})\right\|_{B^{0}_{\infty,1}}ds\lesssim &\int_{0}^{t}\left\|\nabla(-\Delta)^{-1}{\rm div}\Big[u^{a}\cdot\nabla(u^{a}-u^{a}_{N})+(u^{a}-u^{a}_{N})\cdot\nabla u^{a}_{N}\Big]\right\|_{B^{0}_{\infty,1}}ds\\
    &+\int_{0}^{t}\left\|\nabla(-\Delta)^{-1}\mathrm{div}\mathrm{div}(\tau^{a}-\tau^{a}_{N})\right\|_{B^{0}_{\infty,1}}ds\\
    \lesssim & \int_{0}^{t}M_{2}(s)ds+M_{3}(t)+2^{-N}c_{N}.
    \end{align*}
    These together with \eqref{4lo20} ensure that
    \begin{align}\label{4lo21}
    M_{2}(t)\lesssim _{T}2^{-N}c_{N}+\int_{0}^{t}M_{2}(s)ds+M_{3}(t).
    \end{align}
Write $\Theta_N=\tau^a-\tau_N^a$. By \eqref{4lo12},
\[
 \|\Delta_{-1}\Theta_N\|_{L^\infty(0,T;L^\infty)}
 +\|\Delta_{-1}\Theta_N\|_{L^1(0,T;L^\infty)}
 \leq C_T2^{-N}c_N.
\]
For the blocks $j\geq0$, the heat estimate has a constant
independent of $T$. Consequently, for $0\leq t\leq T$,
\begin{align*}
M_3(t)\leq{}&C_T2^{-N}c_N
 +C\int_0^t\|u^a\cdot\nabla\Theta_N\|_{B^{-1}_{\infty,1}}\,ds\\
&+C\int_0^t\|(u^a-u_N^a)\cdot\nabla\tau_N^a
                     \|_{B^{-1}_{\infty,1}}\,ds\\
&+C\int_0^t\|Q(\nabla u^a,\Theta_N)
                     \|_{B^{-1}_{\infty,1}}\,ds\\
&+C\int_0^t\|Q(\nabla(u^a-u_N^a),\tau_N^a)
                     \|_{B^{-1}_{\infty,1}}\,ds
 +C\int_0^tM_2(s)\,ds,
\end{align*}
where $C$ is independent of $a$, $N$, and $T$.

    By the transport structure and Lemma \ref{PL},
    \begin{align*}
    \int_{0}^{t}\|u^{a}\cdot\nabla(\tau^{a}-\tau^{a}_{N})\|_{B^{-1}_{\infty,1}}ds&\lesssim \int_{0}^{t}\|u^{a}\|_{B^{0}_{\infty,1}}\cdot \|\tau^{a}-\tau^{a}_{N}\|_{B^{1}_{\infty,1}}ds\lesssim \epsilon M_{3}(t).
    \end{align*}
    By Lemma \ref{PL}, we have
    \begin{align*}
    \int_{0}^{t}\left\|(u^{a}-u^{a}_{N})\cdot\nabla\tau^{a}_{N}\right\|_{B^{-1}_{\infty,1}}ds&\lesssim \int_{0}^{t}\left\|u^{a}-u^{a}_{N}\right\|_{B^{0}_{\infty,1}}\|\tau^{a}_{N}\|_{H^{2}}ds.
    \end{align*}
    By virtue of Bony's decomposition, we infer that
    \begin{align*}
    \int_{0}^{t}\|Q(\nabla u^{a},\tau^{a}-\tau^{a}_{N})\|_{B^{-1}_{\infty,1}} ds&\lesssim \int_{0}^{t}\|\nabla u^{a}\|_{B^{0}_{\infty,1}}\|\tau^{a}-\tau^{a}_{N}\|_{B^{1}_{\infty,1}} ds\\
    &\lesssim  \int_{0}^{t}\|\nabla u^{a}\|_{B^{0}_{\infty,1}}\|\tau^{a}-\tau^{a}_{N}\|_{B^{1}_{\infty,1}}ds\lesssim \ep M_{3}(t),
    \end{align*}
    and
    \begin{align*}
    \int_{0}^{t}\left\|Q(\nabla(u^{a}-u^{a}_{N}),\tau^{a}_{N})\right\|_{B^{-1}_{\infty,1}}ds\lesssim \int_{0}^{t}M_{2}(s)\|\tau^{a}_{N}\|_{H^{2}}ds.
    \end{align*}
Then we get
\begin{align}\label{4lo23}
M_3(t)\leq C_T2^{-N}c_N
 +C\int_0^tM_2(s)(1+\|\tau_N^a(s)\|_{H^2})\,ds.
\end{align}
Combining \eqref{4lo21} and \eqref{4lo23}, and applying Gronwall's inequality, we obtain
\begin{align}\label{4lo24}
\sup_{0\leq t\leq T}\left[M_{2}(t)+M_{3}(t)\right]\lesssim _{T}2^{-N}c_{N}.
\end{align}
The coefficient in this Gronwall estimate is uniform in $N$:
\eqref{4ineq1} gives
$\|\tau_N^a\|_{L^2(0,T;H^2)}\leq C_T\ep$.
Moreover, writing $\Theta_N=\tau^a-\tau_N^a$, interpolation yields
\begin{equation}\label{4lo24-interpolation}
 \|\Theta_N\|_{L^2(0,T;B^0_{\infty,1})}
 \lesssim
 \|\Theta_N\|_{L^\infty(0,T;B^{-1}_{\infty,1})}^{1/2}
 \|\Theta_N\|_{L^1(0,T;B^1_{\infty,1})}^{1/2}
 \leq C_T2^{-N}c_N.
\end{equation}
Consequently products of this error with $\|\tau_N^a\|_{H^2}$
are integrable with bound $C_T2^{-N}c_N$, without any
smallness assumption on the high norms of the truncated data.
Recall that
\begin{align}\label{4lo25}
\begin{cases}
\Gamma^{a}=\omega^{a}-\mathcal{R}\tau^{a},\\
\Gamma^{a}_{N}=\omega^{a}_{N}-\mathcal{R}\tau^{a}_{N}.
\end{cases}
\end{align}
According to \eqref{eq0} and \eqref{4lo25}, we infer that
\begin{align*}
&\partial_t (\Gamma^a-\Gamma_N^a) + u^a\cdot\nabla(\Gamma^a-\Gamma_N^a) + \frac 1 2(\Gamma^a-\Gamma_N^a) \\ \notag
= &-(u^a-u_N^a)\cdot\nabla\Gamma_N^a+(a-\frac{1}{2})\mathcal{R}(\tau^a-\tau_N^a)  + [\mathcal{R}, u^a\cdot\nabla](\tau^a-\tau_N^a)+ [\mathcal{R}, (u^a-u_N^a)\cdot\nabla]\tau_N^a\\ \notag
&+\mathcal{R}Q(\nabla u^a,\tau^a-\tau_N^a)+\mathcal{R}Q(\nabla (u^a-u_N^a),\tau_N^a).
\end{align*}
Define $M_{4}(t):=\sup_{0\leq s\leq t}\|\Gamma^{a}-\Gamma^{a}_{N}\|_{B^{0}_{\infty,1}}.$
By the standard estimate for the transport equation, we have
\begin{align}\label{4lo27}
&\|\Gamma^{a}-\Gamma^{a}_{N}\|_{B^{0}_{\infty,1}}\\\notag
\lesssim&e^{-\frac{1}{4}t}c_{N}+\int_{0}^{t}e^{-\frac{1}{4}(t-s)}\left\|(u^{a}-u^{a}_{N})\cdot\nabla\Gamma^{a}_{N}\right\|_{B^{0}_{\infty,1}}ds\\\notag
&+\int_{0}^{t}e^{-\frac{1}{4}(t-s)}\left\|\mathcal{R}(\tau^{a}-\tau^{a}_{N})\right\|_{B^{0}_{\infty,1}}ds+\int_{0}^{t}e^{-\frac{1}{4}(t-s)}\left\| [\mathcal{R}, u^a\cdot\nabla](\tau^a-\tau_N^a)\right\|_{B^{0}_{\infty,1}}ds\\\notag
&+\int_{0}^{t}e^{-\frac{1}{4}(t-s)}\left\| [\mathcal{R}, (u^a-u_N^a)\cdot\nabla]\tau_N^a\right\|_{B^{0}_{\infty,1}}ds+\int_{0}^{t}e^{-\frac{1}{4}(t-s)}\left\|\mathcal{R}Q(\nabla u^a,\tau^a-\tau_N^a)\right\|_{B^{0}_{\infty,1}}ds\\\notag
&+\int_{0}^{t}e^{-\frac{1}{4}(t-s)}\left\|\mathcal{R}Q(\nabla (u^a-u_N^a),\tau_N^a)\right\|_{B^{0}_{\infty,1}}ds=: e^{-\frac{1}{4}t}c_{N}+C_{1}+C_{2}+C_{3}+C_{4}+C_{5}+C_{6}.\notag
\end{align}
By Corollary \ref{4cor1} and \eqref{4lo24}, we deduce that
\begin{align*}
C_{1}\lesssim \int_{0}^{t}\|u^{a}-u^{a}_{N}\|_{B^{0}_{\infty,1}}\|\Gamma^{a}_{N}\|_{B^{1}_{\infty,1}}ds\lesssim_{T} c_{N}.
\end{align*}
By the property of the Calderon-Zygmund operator, we have
\begin{align*}
C_{2}\lesssim \int_{0}^{t}e^{-\frac{1}{4}(t-s)}\left(\|\tau^{a}-\tau^{a}_{N}\|_{L^{2}}+\|\tau^{a}-\tau^{a}_{N}\|_{B^{0}_{\infty,1}}\right)ds\lesssim_{T} 2^{-N}c_{N}.
\end{align*}
By Corollary \ref{CR3}, there holds
\begin{align*}
C_{3}\lesssim\int_{0}^{t}e^{-\frac{1}{4}(t-s)}\left(\|\nabla u^{a}\|_{L^{\infty}}\|\tau^{a}-\tau^{a}_{N}\|_{B^{0}_{\infty,1}}+\|u^{a}\|_{L^{2}}\|\tau^{a}-\tau^{a}_{N}\|_{L^{2}}\right)ds\lesssim_{T} 2^{-N}c_{N}.
\end{align*}
Applying Corollary \ref{CR3} again, we get
\begin{align*}
C_{4}&\lesssim \int_{0}^{t}e^{-\frac{1}{4}(t-s)}\left(\|\omega^{a}-\omega^{a}_{N}\|_{B^{0}_{\infty,1}}+\|u^a-u_N^a\|_{L^2}\right)\left(\|\tau^{a}_{N}\|_{B^{0}_{\infty,1}}+\|\tau^{a}_{N}\|_{L^{2}}\right)ds\lesssim \ep M_{4}(t)+C_Tc_{N},
\end{align*}
where we have used the fact that
\begin{align*}
\|\omega^{a}-\omega^{a}_{N}\|_{B^{0}_{\infty,1}}\leq \|\Gamma^{a}-\Gamma^{a}_{N}\|_{B^{0}_{\infty,1}}+\|\mathcal{R}(\tau^{a}-\tau^{a}_{N})\|_{B^{0}_{\infty,1}}.
\end{align*}
By Lemma \ref{PL} and \eqref{4lo24}, we obtain
\begin{align*}
C_{5}&\lesssim \int_{0}^{t}e^{-\frac{1}{4}(t-s)}\|\nabla u^{a}\|_{B^{0}_{\infty,1}}\left(\|\tau^{a}-\tau^{a}_{N}\|_{L^{2}}+\|\tau^{a}-\tau^{a}_{N}\|_{B^{1}_{\infty,1}}\right)ds\lesssim_{T} 2^{-N}c_{N}.
\end{align*}
Now we deal with $C_6$:
\begin{align*}
C_{6}\lesssim& \int_{0}^{t}e^{-\frac{1}{4}(t-s)}\left\|\mathcal{R}Q(\nabla (u^a-u_N^a),\tau_N^a)\right\|_{L^{2}}ds\\
&+\int_{0}^{t}e^{-\frac{1}{4}(t-s)}\left\|Q(\nabla (u^a-u_N^a),\tau_N^a)\right\|_{B^{0}_{\infty,1}}ds
=:C_{61}+C_{62}.
\end{align*}
$C_{61}$ can be estimated directly by virtue of \eqref{4lo15},
\begin{align*}
C_{61}\leq \int_{0}^{t}e^{-\frac{1}{4}(t-s)}\|\nabla(u^{a}-u^{a}_{N})\|_{L^{2}}\|\tau^{a}_{N}\|_{L^{\infty}}ds\lesssim _{T}c_{N}.
\end{align*}
For the term $C_{62}$, applying Lemma \ref{PL},
\begin{align*}
C_{62}&\lesssim \int_{0}^{t}e^{-\frac{1}{4}(t-s)}\|\nabla(u^{a}-u^{a}_{N})\|_{B^{0}_{\infty,1}}\|\tau^{a}_{N}\|_{H^{2}}ds\\
&\lesssim \int_{0}^{t}e^{-\frac{1}{4}(t-s)}\left(\|\Gamma^{a}-\Gamma^{a}_{N}\|_{B^{0}_{\infty,1}}+\|u^{a}-u^{a}_{N}\|_{L^{2}}+\|\mathcal{R}(\tau^{a}-\tau^{a}_{N})\|_{B^{0}_{\infty,1}}\right)\|\tau^{a}_{N}\|_{H^{2}}ds
\\
&\lesssim \ep M_{4}(t)+C_T2^{-N}c_{N},
\end{align*}
where we have used Proposition \ref{4prop1}. Due to \eqref{4lo27}, we infer that
\begin{align*}
\sup_{0\leq t\leq T}\|\Gamma^{a}-\Gamma^{a}_{N}\|_{B^{0}_{\infty,1}}\lesssim _{T} c_{N}.
\end{align*}
Finally, we estimate $\|\Gamma^{a}_{N}-\Gamma^{0}_{N}\|_{B^{0}_{\infty,1}}$. According to \eqref{eq0}, we infer that
\begin{align*}
&\partial_t (\Gamma_N^a-\Gamma_N^0) + u_N^a\cdot\nabla(\Gamma_N^a-\Gamma_N^0) + \frac 1 2(\Gamma_N^a-\Gamma_N^0) \\ \notag
=&-(u_N^a-u_N^0)\cdot\nabla\Gamma_N^0-\frac{1}{2}\mathcal{R}(\tau_N^a-\tau_N^0)+a\mathcal{R}\tau_N^a  \\ \notag
&+ [\mathcal{R}, u_N^a\cdot\nabla](\tau_N^a-\tau_N^0)+ [\mathcal{R}, (u_N^a-u_N^0)\cdot\nabla]\tau_N^0\\ \notag
&+\mathcal{R}Q(\nabla u_N^a,\tau_N^a-\tau_N^0)+\mathcal{R}Q(\nabla (u_N^a-u_N^0),\tau_N^0).
\end{align*}
Define $ M_{5}(t):=\sup_{0\leq s\leq t}\|(\Gamma^{a}_{N}-\Gamma^{0}_{N})(s)\|_{B^{0}_{\infty,1}}. $
By the estimate for the transport equations, we have
\begin{align}\label{4lo30}
&\|(\Gamma^{a}_{N}-\Gamma^{0}_{N})(t)\|_{B^{0}_{\infty,1}}\\\notag
\lesssim&\int_{0}^{t}e^{-\frac{1}{4}(t-s)}\left\|(u_N^a-u_N^0)\cdot\nabla\Gamma_N^0\right\|_{B^{0}_{\infty,1}}ds+ \int_{0}^{t}e^{-\frac{1}{4}(t-s)}\left\|\mathcal{R}(\tau_N^a-\tau_N^0)\right\|_{B^{0}_{\infty,1}}ds\\\notag
&+ \int_{0}^{t}e^{-\frac{1}{4}(t-s)}a\left\|\mathcal{R}\tau_N^a\right\|_{B^{0}_{\infty,1}}ds+ \int_{0}^{t}e^{-\frac{1}{4}(t-s)}\left\| [\mathcal{R}, u_N^a\cdot\nabla](\tau_N^a-\tau_N^0)\right\|_{B^{0}_{\infty,1}}ds\\\notag
&+\int_{0}^{t}e^{-\frac{1}{4}(t-s)}\left\| [\mathcal{R}, (u_N^a-u_N^0)\cdot\nabla]\tau_N^0\right\|_{B^{0}_{\infty,1}}ds+\int_{0}^{t}e^{-\frac{1}{4}(t-s)}\left\|\mathcal{R}Q(\nabla u_N^a,\tau_N^a-\tau_N^0)\right\|_{B^{0}_{\infty,1}}ds\\\notag
&+\int_{0}^{t}e^{-\frac{1}{4}(t-s)}\left\|\mathcal{R}Q(\nabla (u_N^a-u_N^0),\tau_N^0)\right\|_{B^{0}_{\infty,1}}ds=:\sum_{j=1}^{7}D_{j}.
\end{align}
By \eqref{4lo8}, Lemma \ref{PL} and Corollary \ref{4cor1}, we infer that
\begin{align*}
D_{1}&\lesssim \int_{0}^{t}e^{-\frac{1}{4}(t-s)}\|u^{a}_{N}-u^{0}_{N}\|_{B^{0}_{\infty,1}}\|\Gamma^{0}_{N}\|_{B^{1}_{\infty,1}}ds\\
&\lesssim  \int_{0}^{t}e^{-\frac{1}{4}(t-s)}\|u^{a}_{N}-u^{0}_{N}\|_{L^{2}}^{\frac{1}{2}}\|u^{a}_{N}-u^{0}_{N}\|_{B^{1}_{\infty,1}}^{\frac{1}{2}}\|\Gamma^{0}_{N}\|_{B^{1}_{\infty,1}}ds\lesssim_{T}a^{\frac{1}{2}}2^{N}.
\end{align*}
By the property of the Calderon-Zygmund operator, \eqref{4lo8} and \eqref{4lo10}, we get
\begin{align*}
D_{2}\lesssim \int_{0}^{t}e^{-\frac{1}{4}(t-s)}\left(\|\tau^{a}_{N}-\tau^{0}_{N}\|_{L^{2}}+\|\tau^{a}_{N}-\tau^{0}_{N}\|_{B^{0}_{\infty,1}}\right)ds\lesssim _{T}a^{\frac{3}{4}}.
\end{align*}
By the embedding $H^{2}\hookrightarrow B^{0}_{\infty,1}$ and Proposition \ref{4prop1},
\begin{equation*}
D_{3}\lesssim \int_{0}^{t}e^{-\frac{1}{4}(t-s)}a\|\tau^{a}_{N}\|_{H^{2}}ds\lesssim a.
\end{equation*}
Applying Corollary \ref{CR3}, we have
\begin{align*}
D_{4}\lesssim \int_{0}^{t}e^{-\frac{1}{4}(t-s)}\left(\|\nabla u^{a}_{N}\|_{L^{\infty}}+\| u^{a}_{N}\|_{L^{2}}\right)\left(\|\tau^{a}_{N}-\tau^{0}_{N}\|_{B^{0}_{\infty,1}}+\|\tau^{a}_{N}-\tau^{0}_{N}\|_{L^{2}}\right)ds\lesssim_T a^{\frac{3}{4}}.
\end{align*}
Applying Lemma \ref{CR}, we deduce that
\begin{align*}
D_{5}&\lesssim \int_{0}^{t}e^{-\frac{1}{4}(t-s)}\|\omega^{a}_{N}-\omega^{0}_{N}\|_{L^{\infty}\cap L^{2}}\|\tau^{0}_{N}\|_{H^{2}}ds\\
&\leq \int_{0}^{t}e^{-\frac{1}{4}(t-s)}\left(\|\Gamma^{a}_{N}-\Gamma^{0}_{N}\|_{L^{\infty}}+\|\mathcal{R}(\tau^{a}_{N}-\tau^{0}_{N})\|_{L^{\infty}}+\|\omega^{a}_{N}-\omega^{0}_{N}\|_{L^{2}}\right)\|\tau^{0}_{N}\|_{H^{2}}ds\\
&\lesssim \ep M_{5}(t)+C_T\left(a^{\frac{3}{4}}+a2^{N}\right).
\end{align*}
We divide the estimate of $D_{6}$ into two parts: 
\begin{align*}
D_{6}\lesssim &\int_{0}^{t}e^{-\frac{1}{4}(t-s)}\left\|Q(\nabla u_N^a,\tau_N^a-\tau_N^0)\right\|_{L^{2}}ds\\
&+\int_{0}^{t}e^{-\frac{1}{4}(t-s)}\left\|Q(\nabla u_N^a,\tau_N^a-\tau_N^0)\right\|_{B^{0}_{\infty,1}}ds
=: D_{61}+D_{62}.
\end{align*}
By \eqref{4lo8}, we get
\begin{align*}
D_{61}\leq \int_{0}^{t}e^{-\frac{1}{4}(t-s)}\|\nabla u_N^a\|_{L^{\infty}}\|\tau_N^a-\tau_N^0\|_{L^{2}}ds\lesssim_{T} a.
\end{align*}
By virtue of Bony's decomposition, we have
\begin{align*}
D_{62}&\lesssim \int_{0}^{t}e^{-\frac{1}{4}(t-s)}\|\nabla u_N^a\|_{B^{0}_{\infty,1}}\left(\|\tau_N^a-\tau_N^0\|_{B^{0}_{\infty,1}}+\|\tau_N^a-\tau_N^0\|_{H^{1}}\right)ds\lesssim _{T}a^{\frac{3}{4}}+a2^{N}.
\end{align*}
We also divide the estimate of $D_{7}$ into two parts:
\begin{align*}
D_{7}\lesssim &\int_{0}^{t}e^{-\frac{1}{4}(t-s)}\left\|\mathcal{R}Q(\nabla (u_N^a-u_N^0),\tau_N^0)\right\|_{L^{2}}ds\\
&+\int_{0}^{t}e^{-\frac{1}{4}(t-s)}\left\|Q(\nabla (u_N^a-u_N^0),\tau_N^0)\right\|_{B^{0}_{\infty,1}}ds
=: D_{71}+D_{72}.
\end{align*}
By \eqref{4lo18}, we obtain
\begin{align*}
D_{71}\lesssim\int_{0}^{t}e^{-\frac{1}{4}(t-s)}\|\nabla (u_N^a-u_N^0)\|_{L^{2}}\|\tau_N^0\|_{L^{\infty}}ds\lesssim_{T}a2^{N}.
\end{align*}
For the term $D_{72}$, applying Lemma \ref{PL},
\begin{align*}
D_{72}&\lesssim \int_{0}^{t}e^{-\frac{1}{4}(t-s)}\|\nabla (u_N^a-u_N^0)\|_{B^{0}_{\infty,1}}\|\tau_{N}^{0}\|_{H^{2}}ds\\
&\lesssim \int_{0}^{t}e^{-\frac{1}{4}(t-s)}\left[\|\Gamma_{N}^{a}-\Gamma_{N}^{0}\|_{B^{0}_{\infty,1}}+\|\mathcal{R}(\tau_{N}^{a}-\tau_{N}^{0})\|_{B^{0}_{\infty,1}}+\|u^{a}_{N}-u^{0}_{N}\|_{L^{2}}\right]\|\tau_{N}^{0}\|_{H^{2}}ds\\
&\lesssim \ep M_{5}(t)+C_{T}a^{\frac{3}{4}}.
\end{align*}
Recall \eqref{4lo30}, we infer that
\begin{align*}
\sup_{0\leq t\leq T}M_{5}(t)\lesssim_{T}a^{\frac{1}{2}}2^{N}.
\end{align*}
Finally, we obtain
\begin{align*}
\sup_{0\leq t\leq T}\|\Gamma^{a}-\Gamma^{0}\|_{B^{0}_{\infty,1}}\lesssim_{T}c_{N}+a^{\frac{1}{2}}2^{N}.
\end{align*}
This implies that
 $$\lim_{a\rightarrow0}\|u^a-u^0\|_{L^{\infty}(0,T; B^1_{\infty,1})}=0.$$
 We thus complete the proof of Proposition \ref{4prop3}.
 \end{proof}

\subsection{The key integrability and uniform vanishing damping limit} 
By virtue of the improved Fourier splitting method, for any $a\in[0,1]$, we obtain optimal time decay rates of global solutions for the inviscid Oldroyd-B model. However, we fail to obtain the decay rates of $\|(\nabla u,\tau)\|_{B^0_{\infty,1}}$ in the case of critical regularity. The high-regularity bound $\|\tau\|_{B^0_{\infty,1}}\leq C(1+t)^{-1}$ for $d=2$ is borderline and alone does not imply time integrability of this norm. We introduce a different method of high-low frequency decomposition to obtain time integrability of $\|\nabla u\|_{B^0_{\infty,1}}$. 
  \begin{prop}\label{4prop4}
  Under the conditions in Proposition \ref{4prop2}, if additionally $(u_0,\tau_0) \in \dot{B}^{-1}_{2,\infty},$
        then there holds
        $$
        \|(u,\tau)\|_{L^2} +(1+t)^{\frac{1}{2}}\|\nabla(u,\tau)\|_{L^2}\leq C_0(1+t)^{-\frac{1}{2}},
        $$
  and 
  $$\int_{0}^{\infty}\|\nabla u\|_{B^0_{\infty,1}}dt\leq C_0,$$
  where $C_0$ depends on $\|(u_0,\tau_0)\|_{H^1\cap \dot{B}^{-1}_{2,\infty}}$ and $\|(\nabla u_0,\tau_0)\|_{B^0_{\infty,1}}.$
  \end{prop}
  \begin{proof}
We first prove the decay estimate with constants independent of $a$.
Write
$E=\|(u,\tau)\|_{H^1}^2$,
$Y=\|\nabla(u,\tau)\|_2^2$,
$D=\|\nabla u\|_2^2+\|\nabla\tau\|_{H^1}^2$, and
$\mathscr E=E-2\eta\langle\tau,\nabla u\rangle\sim E$.
By \eqref{4ineq1},
$\mathscr E'+cD\leq0$.
Let $L(r,t)=\int_{|\xi|\leq r}(|\hat u|^2+|\hat\tau|^2)\,d\xi$
and $A=\|(u_0,\tau_0)\|_{\dot B^{-1}_{2,\infty}}^2$.
The Fourier energy identity, symmetry of $\tau$, and
$\operatorname{div}u=0$ give, for $0<r\leq1$,
\begin{equation}\label{4in-low-mass}
 L(r,t)\leq Cr^2\left(A+\int_0^t E(s)^{3/2}\,ds\right)
              +Cr\int_0^t E(s)D(s)^{1/2}\,ds.
\end{equation}
For example the transport terms are divergences, hence their
contribution is bounded by
$Cr^2(\|u\|_2+\|\tau\|_2)^3$.
For $Q$ use $\|\widehat Q\|_\infty\leq
C\|\nabla u\|_2\|\tau\|_2$ and
$\|\hat\tau\|_{L^1(|\xi|\leq r)}\leq Cr\|\tau\|_2$.
The initial term is bounded by $Cr^2A$ by a homogeneous dyadic sum.
The damping term has favorable sign and is discarded.
Fourier splitting therefore yields
\begin{equation}\label{4in-splitting}
 \mathscr E'+cr(t)^2\mathscr E+\tfrac c2D
                  \leq Cr(t)^2L(r(t),t).
\end{equation}
The constants and the choice of $\eta$ are independent of $a$;
any initial interval on which $r(t)>1$ is absorbed into the constants.

We spell out the borderline bootstrap underlying this estimate
(cf.\ the Fourier-splitting approach in \cite{DLY23}).
First take $r(t)^2=K/((e+t)\log(e+t))$.
Since $\int_0^\infty D<\infty$, \eqref{4in-low-mass} gives
$L\leq C\log(e+t)^{-1/2}$ and thus
$E\leq C\log(e+t)^{-1/2}$ by \eqref{4in-splitting}, choosing $K$ large.
If $E\leq C_k\log(e+t)^{-k}$, the two integrals in
\eqref{4in-low-mass} are bounded by
$C_k(1+t)\log(e+t)^{-3k/2}$ and
$C_k(1+t)^{1/2}\log(e+t)^{-k}$, respectively.
The same splitting improves $k$ to $k+1/2$.
Consequently, for every $m>0$,
\begin{equation}\label{4in-log-bootstrap}
 E(t)\leq C_m\log(e+t)^{-m},\qquad
 \int_0^\infty\log(e+t)^mD(t)\,dt\leq C_m.
\end{equation}
The second assertion follows by weighting
$\mathscr E'+cD\leq0$ and using the first assertion with exponent $m+1$.

Next take $r(t)^2=K/(1+t)$ and integrate \eqref{4in-splitting}.
Using \eqref{4in-low-mass}, one obtains
\[
 E(t)\leq \frac{C}{1+t}\left(A+\int_0^tE^{3/2}\,ds\right)
       +\frac{C}{(1+t)^{1/2}}\int_0^t ED^{1/2}\,ds.
\]
For $Z(t)=\sup_{s\leq t}(1+s)^{1/2}E(s)$, this implies
\[
 Z(t)\leq C+C\int_0^t
 \left(\frac{\log(e+s)^{-2}}{1+s}
       +\frac{D(s)^{1/2}}{(1+s)^{1/2}}\right)Z(s)\,ds.
\]
Both coefficients are integrable by \eqref{4in-log-bootstrap}
and Cauchy--Schwarz. Thus $E\leq C(1+t)^{-1/2}$.
Weighting the basic energy and \eqref{4ineq2} now gives
\[
 \int_0^t(1+s)^{3/2}D(s)\,ds\leq C(1+t),\qquad
 Y(t)\leq C(1+t)^{-1},\qquad
 \int_0^\infty(1+s)^\beta D(s)\,ds<\infty\quad(\beta<1/2).
\]
Indeed $( (1+t)^2Y)' \leq C(1+t)^{3/2}Y$ after discarding
the stress dissipation, so the middle bound follows from the first one.

An intermediate negative index removes the remaining borderline.
Define
\[
 M(t)=\sup_{s\leq t}\|(u,\tau)(s)\|_{\dot B^{-1/2}_{2,\infty}}.
\]
The block energy inequality, integrating the stress transport
term by parts and absorbing it into the heat dissipation, yields
\[
 M(t)^2\leq C M(0)^2
 +C\int_0^t\|u\otimes\tau\|_{\dot B^{-1/2}_{2,\infty}}^2\,ds
 +CM(t)\int_0^t\|(u\cdot\nabla u,Q)\|_{\dot B^{-1/2}_{2,\infty}}\,ds.
\]
Using the embedding of $L^{4/3}$ into
$\dot B^{-1/2}_{2,\infty}$ and two-dimensional interpolation,
the two integrands have the bounds
\[
 CE^{3/2}Y^{1/2},\qquad CE^{1/4}Y^{1/4}D^{1/2}.
\]
The first is $O((1+t)^{-5/4})$; the second is integrable by
Cauchy--Schwarz with weight $(1+t)^{3/8}D$.
Thus $M(t)\leq C$.
In the $Q$ contribution to \eqref{4in-low-mass} we may now replace
$r\|\tau\|_2$ by
$Cr\,L_\tau(r,s)^{1/2}\leq Cr^{3/2}M(s)$.
It follows, using $\beta=3/8$, that
$L(K^{1/2}(1+t)^{-1/2},t)\leq C(1+t)^{-11/16}$.
Another application of \eqref{4in-splitting} gives
$E\leq C(1+t)^{-11/16}$ and
$\int_0^\infty(1+t)^{5/8}D\,dt<\infty$.
Hence
\[
 \int_0^\infty\|(-u\cdot\nabla u,-u\cdot\nabla\tau-Q)\|_{L^1}\,dt
 \leq C\int_0^\infty E^{1/2}D^{1/2}\,dt<\infty,
\]
since $(11/16)+(5/8)>1$.
The dyadic energy inequality and
$L^1\hookrightarrow\dot B^{-1}_{2,\infty}$ then give a uniform
bound on $\|(u,\tau)\|_{\dot B^{-1}_{2,\infty}}$.
Now $L(r,t)\leq Cr^2$; \eqref{4in-splitting} gives
$E\leq C(1+t)^{-1}$.
Finally
$\int_0^t(1+s)^2D(s)\,ds\leq C(1+t)$ and
$((1+t)^3Y)'\leq C(1+t)^2Y$ from \eqref{4ineq2}
imply $Y\leq C(1+t)^{-2}$.
We have proved, uniformly for $a\in[0,1]$,
\begin{align}\label{4in1}
 \|(u,\tau)\|_{L^2}+(1+t)^{1/2}\|\nabla(u,\tau)\|_{L^2}
                     \leq C_0(1+t)^{-1/2}.
\end{align}

Multiplying $(1+t)^{1-\delta}$ with $\delta\in(0,1)$ to \eqref{4ineq1} and integrating, we infer from \eqref{4in1} that
 \begin{align*}
    \int_{0}^{t}(1+t')^{1-\delta}(\|\nabla u\|^2_{L^{2}}+\|\nabla \tau\|^2_{H^{1}}) dt'
 &\lesssim C_0+(1+t)^{1-\delta}\|(u,\tau)\|^2_{H^1} \\ \notag
 &~~~+\int_{0}^{t}(1+t')^{-\delta}\|(u,\tau)\|^2_{H^1} dt' 
 \\  \notag
  &\lesssim C_0+C_0\int_{0}^{t}(1+t')^{-\delta-1} dt' \lesssim C_0\frac {1}{\delta}.
\end{align*}
        Combining \eqref{4in1} and \eqref{4ineq2}, we get
  \begin{align}\label{4in3}
    \int_{0}^{t}(1+t')^{2-\delta}\|\nabla^2 \tau\|^2_{L^{2}}dt'
 &\lesssim C_0+(1+t)^{2-\delta}\|\nabla(u,\tau)\|^2_{L^2} \\ \notag
 &~~~+\int_{0}^{t}(1+t')^{1-\delta}\|\nabla(u,\tau)\|^2_{L^2} dt' \\ \notag
 &~~~+\int_{0}^{t}(1+t')^{2-\delta}\|\nabla u\|^2_{L^2}\|\tau\|^2_{H^1} dt'
 \\  \notag
  &\lesssim C_0+C_0\int_{0}^{t}(1+t')^{-\delta-1} dt' \lesssim C_0\frac {1}{\delta}.
    \end{align}
The high-regularity bound $\|\tau\|_{B^0_{\infty,1}}\leq C(1+t)^{-1}$ for $d=2$ does not by itself give time integrability of $\|\tau\|_{B^0_{\infty,1}}$. The core difficulty lies in the low-frequency decay rate of $\tau$. We now introduce a new method to get the following key integrability:
 $$\int_{0}^{\infty}\|\nabla u\|_{B^0_{\infty,1}}dt'\leq C_0.$$

 We first deal with high-frequency case.
 According to \eqref{4in3}, for any $t>0$, we infer that
 \begin{align}\label{4in4}
    \int_{0}^{t}\|(Id-\Delta_{-1})\mathcal{R} \tau\|_{B^0_{\infty,1}}dt'
 &\lesssim\int_{0}^{t}\|(Id-\Delta_{-1})\tau\|_{B^0_{\infty,1}}dt' \\ \notag
 &\lesssim\int_{0}^{t}\sum_{j\geq 0}2^{-j}\|\Delta_{j}\nabla^2\tau\|_{L^2} dt' \\ \notag
 &\lesssim\int_{0}^{t}\|\nabla^2 \tau\|_{L^2} dt'\lesssim\left(\int_{0}^{t}(1+t')^{\frac 5 4}\|\nabla^2 \tau\|^2_{L^{2}} dt'\right)^{\frac 1 2}\lesssim C_0,
    \end{align}
This means that the high frequency of $\tau$ is integrable. With this information, we can study the integrability of high frequency for $\nabla u$. We deduce from the equation of $\Gamma$ that
\begin{align*}
 &\partial_t (Id-\Delta_{-1})\Gamma + (u\cdot\nabla)(Id-\Delta_{-1})\Gamma + \frac 1 2(Id-\Delta_{-1})\Gamma \\ \notag
 = &(a-\frac{1}{2})(Id-\Delta_{-1})\mathcal{R}\tau + (Id-\Delta_{-1})[\mathcal{R}, u\cdot\nabla]\tau  \\ \notag
&+ (Id-\Delta_{-1})\mathcal{R}Q-[(Id-\Delta_{-1}),(u\cdot\nabla)]\Gamma.   
\end{align*}
where $\Gamma = \omega - \mathcal{R}\tau$. 
Then by Proposition \ref{4prop2} and Lemma \ref{TP}, we have
\begin{align*}
\|(Id-\Delta_{-1})\Gamma\|_{B^0_{\infty,1}} &\lesssim \|\Gamma_0\|_{B^0_{\infty,1}} e^{-\frac {1}{4}t} + \int_0^t e^{-\frac 1 4(t-t')} \|(Id-\Delta_{-1})\mathcal{R}\tau\|   
_{B^0_{\infty,1}} dt' \\ \notag
&~~~+ \int_0^t e^{-\frac 1 4(t-t')} \|[\mathcal{R}, u\cdot\nabla]\tau\|   
_{B^0_{\infty,1}} +\| \mathcal{R}Q\|   
_{B^0_{\infty,1}} + \|[\Delta_{-1},(u\cdot\nabla)]\Gamma\|   
_{B^0_{\infty,1}} dt'.
\end{align*}
This together with \eqref{4in1}-\eqref{4in4}, Lemmas \ref{CZ}, \ref{PL} and \ref{CR} ensures that
\begin{align*}
\int_0^t\|(Id-\Delta_{-1})\Gamma\|_{B^0_{\infty,1}}dt' &\lesssim C_0+ \int_0^t \|[\mathcal{R}, u\cdot\nabla]\tau\|   
_{B^0_{\infty,1}} +\| \mathcal{R}Q\|   
_{B^0_{\infty,1}} dt'  \\ \notag
&~~~+ \int_0^t\|[\Delta_{-1},(u\cdot\nabla)]\Gamma\|   
_{B^0_{\infty,1}} dt' \\ \notag
&\lesssim C_0+ \int_0^t (\|\nabla u\|_{B^0_{\infty,1}}+\|\omega\|_{L^2})(\|\tau\|_{B^{\frac 1 4}_{\infty,1}}+\|\tau\|_{L^2})dt' \\ \notag
&~~~+\int_0^t \|\Delta_{-1}(u\cdot\nabla\Gamma)\|   
_{B^0_{\infty,1}}+\|(u\cdot\nabla)\Delta_{-1}\Gamma\|   
_{B^0_{\infty,1}} dt' \\ \notag
&\lesssim C_0+ \int_0^t \|\nabla u\|_{B^0_{\infty,1}}(\|\tau\|_{L^2} + \|u\|   
_{L^2})dt' \\ \notag
&\lesssim C_0+\ep\int_0^t\|(Id-\Delta_{-1})\Gamma\|_{B^0_{\infty,1}}dt'.
\end{align*}
For the last absorption, separate the low-frequency part:
$\|\Delta_{-1}\nabla u\|_{B^0_{\infty,1}}
\lesssim\|\nabla u\|_2\lesssim C_0(1+t)^{-1}$.
Its product with $\|u\|_2+\|\tau\|_2$ is integrable.
The remaining part is bounded by
$C\|(Id-\Delta_{-1})\Gamma\|_{B^0_{\infty,1}}+
C\|(Id-\Delta_{-1})\mathcal R\tau\|_{B^0_{\infty,1}}$;
use \eqref{4in4} and the small critical bound to absorb the
first term. Thus this argument does not use the low-frequency
integrability that will be proved next.
Then we have 
\begin{align*}
\int_0^t\|(Id-\Delta_{-1})\nabla u\|_{B^0_{\infty,1}}dt'\lesssim \int_0^t\|(Id-\Delta_{-1})\Gamma\|_{B^0_{\infty,1}}+\|(Id-\Delta_{-1})\mathcal{R} \tau\|_{B^0_{\infty,1}}dt'\lesssim C_0.
\end{align*}

Then we deal with low-frequency case. Since low-frequency operator $\Delta_{-1}$ have a smoothing effect, we can consider the time weighted integrability of solutions with high-order regularity. From \eqref{eq0}, we deduce that
\begin{align}\label{4in8}
\frac{d}{dt} \langle\Delta_{-1}\nabla\tau,-\Delta_{-1}\nabla^2 u \rangle + \frac{1}{2}\|\Delta_{-1}\nabla^2 u\|_{L^{2}}^2 &=
-\langle\mathbb{P}\Delta_{-1}\left(\mathrm{div}~\tau-u\cdot\nabla u\right),\Delta_{-1}\Delta\mathrm{div}~\tau \rangle \\ \notag
&~~~+\langle \Delta_{-1}\nabla(u\cdot\nabla\tau+a\tau+Q-\Delta \tau),\Delta_{-1}\nabla^2 u\rangle.
\end{align}
Then we have
\begin{align*}
|\langle\mathbb{P}\Delta_{-1}\left(\mathrm{div}~\tau-u\cdot\nabla u\right),\Delta_{-1}\Delta\mathrm{div}~\tau \rangle|
&\lesssim \|\Delta_{-1}\nabla^2 \tau\|_{L^{2}}^2 
+ \|\Delta_{-1}\nabla^3 \tau\|_{L^{2}}\|\nabla u\|_{L^{2}}\|u\|_{L^{2}} \\
&\lesssim  \|\nabla^2 \tau\|_{L^{2}}^2 
+\|\nabla u\|_{L^{2}}^2\|u\|^2_{L^{2}}.
\end{align*}
and
\begin{align*}
&|\langle \Delta_{-1}\nabla(u\cdot\nabla\tau+a\tau+Q-\Delta \tau),\Delta_{-1}\nabla^2 u\rangle|\\
\lesssim & \|\Delta_{-1}\nabla^2 u\|_{L^{2}}\|\nabla \tau\|_{L^{2}} \|u\|_{L^2} 
+ a \|\Delta_{-1}\nabla^2 \tau\|_{L^2}\|\Delta_{-1}\nabla u\|_{L^{2}}\\
&+ \|\Delta_{-1}\nabla^2 u\|_{L^{2}}\| \tau\|_{L^2}\|\nabla u\|_{L^2}
+ \|\Delta_{-1}\nabla^3 \tau\|_{L^{2}}\|\Delta_{-1}\nabla^2u\|_{L^{2}}\\
\leq & \frac{1}{4}\|\Delta_{-1}\nabla^2 u\|_{L^{2}}^2+C(\|\nabla \tau\|^2_{L^{2}} \|u\|^2_{L^2} +\|\nabla^2 \tau\|^2_{L^2}+\| \tau\|^2_{L^2}\|\nabla u\|^2_{L^2}+\|\nabla^2 \tau\|_{L^2}\|\nabla u\|_{L^{2}}).
\end{align*}
Multiplying $(1+t)^3$ to \eqref{4in8}, then we infer from \eqref{4in1} that
\begin{align*}
&\frac{d}{dt} [(1+t)^3\langle\Delta_{-1}\nabla\tau,-\Delta_{-1}\nabla^2 u \rangle] + \frac{1}{4}(1+t)^3\|\Delta_{-1}\nabla^2 u\|_{L^{2}}^2\\ \notag
&\lesssim C_0+(1+t)^3\|\nabla^2 \tau\|^2_{L^2}+ (1+t)^2\|\nabla^2 \tau\|_{L^2}+(1+t)^2|\langle\Delta_{-1}\nabla\tau,-\Delta_{-1}\nabla^2 u \rangle|.
\end{align*}
This together with \eqref{4in8} and \eqref{4ineq2} ensures that
\begin{align}\label{4in10}
\int_{0}^{t}(1+t')^3\|\Delta_{-1}\nabla^2 u\|_{L^{2}}^2dt' &\lesssim C_0(1+t)+\int_{0}^{t}(1+t')^3\|\nabla^2 \tau\|^2_{L^2}+C_0(1+t')^2\|\nabla^2 \tau\|_{L^2}dt'\\ \notag
&~~~+(1+t)^3|\langle\Delta_{-1}\nabla\tau,-\Delta_{-1}\nabla^2 u \rangle| \\ \notag
&~~~+\int_{0}^{t}(1+t')^2|\langle\Delta_{-1}\nabla\tau,-\Delta_{-1}\nabla^2 u \rangle|dt' \\ \notag
&\lesssim C_0(1+t)+C_0(1+t)\left(\int_{0}^{t}(1+t')^3\|\nabla^2 \tau\|^2_{L^2}dt'\right)^{\frac 12}\\ \notag
&\lesssim C_0(1+t)^{\frac 32}.
\end{align}
Our goal is to prove that
\begin{equation}\label{4in11}
\int_{0}^{\infty}\left\|\Delta_{-1}\nabla u\right\|_{L^{\infty}}dt\lesssim C_{0}.
\end{equation}
By Proposition \ref{4prop2}, we only need to show that
\begin{equation*}
\int_{1}^{\infty}\left\|\Delta_{-1}\nabla u\right\|_{L^{\infty}}dt\lesssim C_{0}.
\end{equation*}
Applying Bernstein's inequality and Sobolev's inequality, we infer that
\begin{align*}
\int_{1}^{\infty}\left\|\Delta_{-1}\nabla u\right\|_{L^{\infty}}dt&\lesssim \int_{1}^{\infty}\left\|\Delta_{-1}\nabla u\right\|_{L^4}dt\lesssim \int_{1}^{\infty}\left\|\nabla u\right\|_{L^{2}}^{\frac{1}{2}}\left\|\Delta_{-1}\nabla^{2}u\right\|_{L^{2}}^{\frac{1}{2}}dt.
\end{align*}
Note that there is time growth on the right side of the inequality \eqref{4in10}. We now introduce a novel inequality to control time growth. By Lebesgue's monotone convergence theorem and H\"older's inequality, we deduce from \eqref{4in1} and \eqref{4in10} that 
\begin{align*}
 &\int_{1}^{\infty}\left\|\nabla u\right\|_{L^{2}}^{\frac{1}{2}}\left\|\Delta_{-1}\nabla^{2}u\right\|_{L^{2}}^{\frac{1}{2}}dt \\ \notag
 =&\sum_{j=0}^{\infty}\int_{2^{j}}^{2^{j+1}}\left\|\nabla u\right\|_{L^{2}}^{\frac{1}{2}}\left\|\Delta_{-1}\nabla^{2}u\right\|_{L^{2}}^{\frac{1}{2}}dt\\ \notag
 \lesssim&C_0\sum_{j=0}^{\infty}2^{-\frac{j}{2}}\int_{2^{j}}^{2^{j+1}}\left\|\Delta_{-1}\nabla^{2}u\right\|_{L^{2}}^{\frac{1}{2}}dt\\ \notag
 \leq&C_0\sum_{j=0}^{\infty}2^{-\frac{j}{2}}\left(\int_{2^{j}}^{2^{j+1}}(1+t)^{3}\left\|\Delta_{-1}\nabla^{2}u\right\|_{L^{2}}^{2}dt\right)^{\frac{1}{4}}\left(\int_{2^{j}}^{2^{j+1}}(1+t)^{-1}dt\right)^{\frac{3}{4}}\\ \notag
 \lesssim&C_0\sum_{j=0}^{\infty}2^{-\frac{j}{8}}\lesssim C_0.
 \end{align*}
This implies that \eqref{4in11} is valid. We thus finish the proof of Proposition \ref{4prop4}.
  \end{proof}
By virtue of the optimal decay rates and key integrability in the Proposition \ref{4prop4}, for $d=2$, we finally obtain the uniform vanishing damping limit in the following proposition. In addition, we discover a new phenomenon for \eqref{eq0} that the rate of uniform vanishing damping limit in $L^2$ is related to the time decay rate in $L^2$.
  \begin{prop}\label{4prop5}
  Under the conditions in Proposition \ref{4prop4}, there holds 
  $$\|(u^a,\tau^a)-(u^0,\tau^0)\|_{L^{\infty}(0,\infty;L^2)}\leq C_0a^{\frac 1 2},
  $$ 
  and 
   $$\lim_{a\rightarrow0}\|(u^a,\tau^a)-(u^0,\tau^0)\|_{L^{\infty}(0,\infty;\dot{H}^1\cap B^1_{\infty,1})\times L^{\infty}(0,\infty;\dot{H}^1\cap B^0_{\infty,1})}=0.
  $$
    \end{prop}
 \begin{proof}   Let $(u^a,\tau^a)$ be a solution of $\eqref{eq0}$ with the initial data $(u_0,\tau_0)$ and $a\in[0,1]$. Then we get 
\begin{align}\label{4gl1}
\left\{\begin{array}{l}
\partial_t(u^a-u^0)+u^a\cdot\nabla (u^a-u^0)+\nabla (P_{u^a}-P_{u^0}) \\[1ex]
={\rm div}~(\tau^a-\tau^0)-(u^a-u^0)\cdot\nabla u^0,~~~~{\rm div}~(u^a-u^0)=0,\\[1ex]
\partial_t(\tau^a-\tau^0)+u^a\cdot\nabla(\tau^a-\tau^0)-\Delta(\tau^a-\tau^0)+Q(\nabla u^a,\tau^a-\tau^0)\\[1ex]
=-a\tau^a+D(u^a-u^0)-(u^a-u^0)\cdot\nabla\tau^0-Q(\nabla (u^a-u^0),\tau^0),\\[1ex]
(u^a-u^0)|_{t=0}=0,~~(\tau^a-\tau^0)|_{t=0}=0. \\[1ex]
\end{array}\right.
\end{align}
\textbf{Step 1: $L^{2}$ convergence.} \\
Testing the first equation of \eqref{4gl1} by $(u^{a}-u^{0})$ and the second equation of \eqref{4gl1} by $(\tau^{a}-\tau^{0})$, we obtain the following energy identity:
\begin{align*}
&\frac{1}{2}\frac{d}{dt}\left\|(u^{a}-u^{0}, \tau^{a}-\tau^{0})\right\|_{L^{2}}^{2}+\|\nabla (\tau^{a}-\tau^{0})\|_{L^{2}}^{2}= -\langle (u^{a}-u^{0})\cdot\nabla u^{0}, u^{a}-u^{0}\rangle\\ \notag
&\quad -\langle a\tau^{a}+Q(\nabla u^{a},\tau^{a}-\tau^{0})+Q(\nabla(u^{a}-u^{0}),\tau^{0})+(u^{a}-u^{0})\cdot\nabla \tau^{0}, \tau^{a}-\tau^{0}\rangle.\notag
\end{align*}
Define $\phi(t):=\left\|(u^{a}-u^{0},\tau^{a}-\tau^{0})\right\|_{L^{2}}$. We have the following inequalities:
\begin{align*}
\left |\langle (u^{a}-u^{0})\cdot\nabla u^{0}, (u^{a}-u^{0})\rangle\right|\leq \|\nabla u^{0}\|_{L^{\infty}}\phi^{2}(t),
\end{align*}
and
\begin{align*}
\left| \int_{\mathbb{R}^{2}} a\tau^{a}:(\tau^{a}-\tau^{0})dx\right|\leq a\|\tau^{a}\|_{L^{2}}\phi(t)\lesssim aC_0(1+t)^{-\frac{1}{2}}\phi(t).
\end{align*}
One can deduce that
\begin{align*}
\left|\langle Q(\nabla u^{a},\tau^{a}-\tau^{0}), \tau^{a}-\tau^{0}\rangle\right|\leq \|\nabla u^{a}\|_{L^{\infty}}\phi^{2}(t),
\end{align*}
and
\begin{align*}
&\left|\langle Q(\nabla(u^{a}-u^{0}),\tau^{0}), \tau^{a}-\tau^{0}\rangle\right|\\\notag
\leq &\|u^{a}-u^{0}\|_{L^{2}}\|\nabla \tau^{0}\|_{L^{4}}\|\tau^{a}-\tau^{0}\|_{L^{4}}+\|u^{a}-u^{0}\|_{L^{2}}\|\tau^{0}\|_{L^{\infty}}\|\nabla(\tau^{a}-\tau^{0})\|_{L^{2}}\\\notag
\leq &C\|u^{a}-u^{0}\|_{L^{2}}\|\nabla \tau^{0}\|_{L^{4}}\left(\|\tau^{a}-\tau^{0}\|_{L^{2}}+\|\nabla(\tau^{a}-\tau^{0})\|_{L^{2}}\right)\\\notag
&+C\|u^{a}-u^{0}\|_{L^{2}}^{2}\|\tau^{0}\|_{L^{\infty}}^{2}+\frac{1}{8}\|\nabla(\tau^{a}-\tau^{0})\|_{L^{2}}^{2}\\\notag
\leq&C\left(\|\nabla\tau^{0}\|_{L^{4}}+\|\nabla\tau^{0}\|_{L^{4}}^{2}+\|\tau^{0}\|_{L^{\infty}}^{2}\right)\phi^{2}(t)+\frac{1}{4}\|\nabla(\tau^{a}-\tau^{0})\|_{L^{2}}^{2}.
\end{align*}
Moreover, we obtain
\begin{align*}
\left|\langle (u^{a}-u^{0})\cdot\nabla \tau^{0}, \tau^{a}-\tau^{0}\rangle \right|&\leq \|u^{a}-u^{0}\|_{L^{2}}\|\nabla \tau^{0}\|_{L^{4}}\|\tau^{a}-\tau^{0}\|_{L^{4}}\\\notag
&\lesssim  \|u^{a}-u^{0}\|_{L^{2}}\|\nabla \tau^{0}\|_{L^{4}}\|\tau^{a}-\tau^{0}\|_{H^{1}}\\\notag
&\leq C\left(\|\nabla \tau^{0}\|_{L^{4}}+\|\nabla \tau^{0}\|_{L^{4}}^{2}\right)\phi^{2}(t)+\frac{1}{4}\|\nabla(\tau^{a}-\tau^{0})\|_{L^{2}}^{2}.\notag
\end{align*}
Then, we deduce that
\begin{align}\label{4gl8}
\phi'(t)\lesssim \psi(t)\phi(t)+aC_0(1+t)^{-\frac{1}{2}},
\end{align}
where we define $\psi(t)$ as follows:
\begin{align*}
\psi(t)=\|\nabla u^{0}\|_{L^{\infty}}+\|\nabla u^{a}\|_{L^{\infty}}+\|\nabla\tau^{0}\|_{L^{4}}+\|\nabla\tau^{0}\|_{L^{4}}^{2}+\|\tau^{0}\|_{L^{\infty}}^{2}.
\end{align*}
Applying \eqref{4in3}, Proposition \ref{4prop4} and the following interpolation inequalities for $d=2$:
$$\|\nabla \tau^{0}\|_{L^{4}(\mathbb{R}^{2})}^{2}\lesssim \|\nabla \tau^{0}\|_{L^{2}(\mathbb{R}^{2})}\|\nabla^{2}\tau^{0}\|_{L^{2}(\mathbb{R}^{2})},\quad \|\tau^{0}\|_{L^{\infty}(\mathbb{R}^{2})}^{2}\lesssim \|\tau^{0}\|_{L^{2}(\mathbb{R}^{2})}\|\nabla^{2} \tau^{0}\|_{L^{2}(\mathbb{R}^{2})},$$
we have
\begin{align*}
\int_{0}^{\infty} \left(\|\nabla u^{0}\|_{L^{\infty}}+\|\nabla u^{a}\|_{L^{\infty}}+\|\nabla\tau^{0}\|_{L^{4}}+\|\nabla\tau^{0}\|_{L^{4}}^{2}+\|\tau^{0}\|_{L^{\infty}}^{2}\right)dt\lesssim C_0.
\end{align*}
Applying Gronwall's inequality, we infer from \eqref{4gl8} that
\begin{align*}
\phi(t)\lesssim aC_0(1+t)^{\frac{1}{2}}.
\end{align*}
If $at\leq 1$, then we get $\phi(t)\leq C_0 a^{\frac{1}{2}}$.
If $at\geq 1$, applying Proposition \ref{4prop4}, we also deduce that
\begin{align*}
\phi(t)\lesssim C_0(1+t)^{-\frac{1}{2}}\leq C_0a^{\frac{1}{2}}.
\end{align*}
\textbf{Step 2: High-order convergence.} \\
We need to show that: For any $0<\delta<1$, there exists a positive number $a_{0}$ such that
\begin{align*}
\|(u^a,\tau^a)-(u^0,\tau^0)\|_{L^{\infty}(0,\infty;\dot{H}^1\cap B^1_{\infty,1})\times L^{\infty}(0,\infty;\dot{H}^1\cap B^0_{\infty,1})}\lesssim \delta,
\end{align*}
whenever $a\leq a_{0}$. 

Firstly, we claim that for above $\delta>0$, by Proposition \ref{4prop4}, there exists $T^{*}\in (0,\infty)$ independent with $a\in[0,1]$, such that
\begin{align}\label{4gl9}
\|(u^a,\tau^a)(t)\|_{H^1\cap B^1_{\infty,1}\times H^1\cap B^0_{\infty,1}}\lesssim \delta, \quad\text{for all}\ a\in[0,1], ~t\geq T^{*}.
\end{align}
Indeed, by \eqref{4in1}, there exists $T>0$, which is independent with $a\in[0,1]$, such that
\begin{align*}
\|(u^a,\tau^a)(t)\|_{H^{1}\times H^{1}}\leq \frac{\delta}{2},\quad \text{for all}\ a\in[0,1],t\geq T.
\end{align*}
By Proposition \ref{4prop4} and \eqref{4in3}, there exists $M>0$ such that
\begin{equation*}
\int_{0}^{\infty}\left(\|\nabla u^{a}\|_{B^{0}_{\infty,1}}+\|\tau^{a}\|_{B^{0}_{\infty,1}}^{2}\right)dt\leq M.
\end{equation*}
To justify the stress contribution, the low-frequency Bernstein
and interpolation inequalities give
$\|\Delta_{-1}\tau\|_\infty
\lesssim\|\tau\|_4\lesssim\|\tau\|_2^{1/2}
\|\nabla\tau\|_2^{1/2}\lesssim C_0(1+t)^{-3/4}$,
whereas
$\|(Id-\Delta_{-1})\tau\|_{B^0_{\infty,1}}
\lesssim\|\nabla^2\tau\|_2$.
Their squares are integrable by \eqref{4in1} and \eqref{4in3}.
Choose $\delta$ so small that the fixed multiple $C\delta$
is below the small-data threshold.
By mean-value theorem, for any $a\in[0,1]$, there exists $T_{a}\in [T,T+4M\delta^{-2}]$ such that
\begin{equation*}
\|\nabla u^{a}(T_{a})\|_{B^{0}_{\infty,1}}+\|\tau^{a}(T_{a})\|_{B^{0}_{\infty,1}}^{2}\leq \frac{\delta^{2}}{4}.
\end{equation*}
Then we have
\begin{equation*}
\|(u^a,\tau^a)(T_{a})\|_{H^1\cap B^1_{\infty,1}\times H^1\cap B^0_{\infty,1}}\lesssim \delta.
\end{equation*}
Applying the result of global well-posedness in Proposition \ref{4prop2},
\begin{equation*}
\|(u^a,\tau^a)(t)\|_{H^1\cap B^1_{\infty,1}\times H^1\cap B^0_{\infty,1}}\lesssim \delta,\quad \text{for all}\ t\geq T_{a}.
\end{equation*}
Taking $T^{*}=T+4M\delta^{-2}$, we finish the proof of the claim \eqref{4gl9}.

By Proposition \ref{4prop3}, there exists $a_{0}>0$ such that
\begin{align*}
\|(u^a,\tau^a)-(u^0,\tau^0)\|_{L^{\infty}(0,T^{*};\dot{H}^1\cap B^1_{\infty,1})\times L^{\infty}(0,T^{*};\dot{H}^1\cap B^0_{\infty,1})}\leq \delta,
\end{align*}
whenever $a\leq a_{0}$. 
This together with the claim \eqref{4gl9} and triangle inequality ensures that
\begin{align*}
\|(u^a,\tau^a)-(u^0,\tau^0)\|_{L^{\infty}(0,\infty;\dot{H}^1\cap B^1_{\infty,1})\times L^{\infty}(0,\infty;\dot{H}^1\cap B^0_{\infty,1})}\lesssim \delta.
\end{align*}
We finish the proof of Proposition \ref{4prop5}.
\end{proof}

We show that the vanishing damping rate in Proposition \ref{4prop5} is sharp in the following sense. 
\begin{prop}[Sharpness of the vanishing damping rate]\label{2dsharp}
There exists initial data $(u_0, \tau_0)$ satisfying the conditions of Proposition \ref{4prop5} such that the following estimates hold:
$$\|(u^a,\tau^a)-(u^0,\tau^0)\|_{L^{\infty}(0,\infty;L^2)}\gtrsim a^{\frac 1 2}.
  $$ 
\end{prop}
\begin{proof}
We give an exact solution, which also works in dimension $d=3$.
Choose a real Schwartz function $\phi$ with
$\widehat\phi\in C^\infty_0(\R^d)$ and
$\widehat\phi(\xi)=1$ for $|\xi|\leq1$, and set
$u_0=0$, $\tau_0=\ep\phi I_d$.
These data are symmetric, satisfy all the required Besov and
Sobolev conditions, and satisfy the small-data hypothesis when
$\ep>0$ is sufficiently small. In particular
$2^{-dj/2}\|\dot\Delta_j\phi\|_2\leq C$ for $j\leq0$,
which verifies the endpoint negative Besov condition.

For every $a\in[0,1]$ the solution is
\begin{equation}\label{whole-space-isotropic-exact}
 u^a=0,\qquad
 \tau^a=f^a I_d,\qquad P^a=f^a,\qquad
 f^a=\ep e^{-at}e^{t\Delta}\phi.
\end{equation}
Indeed, $\operatorname{div}(f^a I_d)=\nabla f^a$ is balanced by
the pressure, while $D(u^a)=0$ and $Q(\nabla u^a,\tau^a)=0$
for every $b\in[-1,1]$. Thus no nonlinear remainder occurs.
By Plancherel's identity, for $t\geq1$,
\begin{align*}
 \|\tau^a(t)-\tau^0(t)\|_2
 &=\sqrt d\,\ep(1-e^{-at})\|e^{t\Delta}\phi\|_2\\
 &\geq c_d\ep(1-e^{-at})
       \left(\int_{|\xi|\leq t^{-1/2}}e^{-2t|\xi|^2}\,d\xi\right)^{1/2}\\
 &\geq c_d\ep(1-e^{-at})t^{-d/4}.
\end{align*}
Taking $t=a^{-1}$ gives the lower bound
$c_d\ep a^{d/4}$ for $0<a\leq1$.
In dimension two this proves the asserted sharp rate $a^{1/2}$.
The initial data are fixed independently of $a$.
\end{proof}

\section{Global existence and uniform vanishing damping limit in $\R^3$}
In this section, we first consider global existence for the inviscid Oldroyd-B model \eqref{eq0}. Since 
$$\int_{\mathbb{R}^{3}}(u\cdot\nabla)u\cdot \Delta u dx\neq 0,$$
for $d=3$, we fail to close the global estimate for \eqref{eq0} in $H^1$. Considering the equation of vorticity $\omega$ for $d=3$, one can see that there is an external high-order nonlinear term $\omega\cdot\nabla u$. By virtue of transportation structure of $\omega\cdot\nabla u$, for any $a\in[0,1]$, we fortunately close the global estimate for vorticity $\omega$ through smallness condition in the critical Besov space $B^0_{\infty,1}$, which is not an algebra. The results of global existence play a crucial role in considering vanishing damping limit in the same topology for any $T>0$. 

Then, by virtue of the improved Fourier splitting method and time weighted estimates, for any $a\in[0,1]$, we obtain optimal time decay rates and the time integrability of global solutions for the inviscid Oldroyd-B model \eqref{eq0}. Finally, we obtain the uniform vanishing damping limit and discover a new phenomenon for \eqref{eq0} that the sharp rate of uniform vanishing damping limit in $L^2$ is related to the time decay rate in $L^2$.
\subsection{Global existence and vanishing damping limit}
We first prove global existence for $a\in[0,1]$ under smallness of
both the $H^1$ norm and the critical Besov norms stated below.
As in the two-dimensional case, the estimates are uniform in $a$
and include zero damping.
\begin{prop}\label{5prop1}
        Let $d=3$. Assume a divergence-free field $u_0\in H^1\cap B^1_{\infty,1}$ and a symmetric matrix $\tau_0\in H^1\cap B^0_{\infty,1}$. There exists some positive constant $\ep$ small enough such that if
        \begin{align*}
        \|(u_0,\tau_0)\|_{H^1}+\|(\nabla u_0,\tau_0)\|_{B^0_{\infty,1}} \leq \ep,
        \end{align*}
        then \eqref{eq0} admits a global solution $(u,\tau)$ with
        $$
        (u,\tau) \in L^{\infty}(0,\infty;H^1\cap B^1_{\infty,1})\times L^{\infty}(0,\infty;H^1\cap B^0_{\infty,1}).
        $$
  \end{prop}
  \begin{proof}
We start with a bootstrap assumption. Assume that for $T>0$, it holds that
\begin{equation}\label{5EI}
    \sup_{0\le t\le T}\left( \|u\|_{H^1} + \|u\|_{B^1_{\infty,1}} + \|\tau\|_{H^1} + \|\tau\|_{B^0_{\infty,1}} \right) \le C_1\ep,
\end{equation}
where $C_1$ is a fixed constant to be chosen later.
Firstly, we obtain the following energy identities:
\begin{align}\label{5g1}
    \frac{d}{dt}\frac{1}{2}\left\|(u,\tau)\right\|_{L^{2}}^{2}+\|\nabla \tau\|_{L^{2}}^2+a\|\tau\|_{L^{2}}^{2}=-\langle Q(\nabla u,\tau),\tau \rangle,
    \end{align}
    and
    \begin{align*}
    \frac{d}{dt}\frac{1}{2}\left\|\nabla(u,\tau)\right\|_{L^{2}}^{2}+\|\nabla^2 \tau\|_{L^{2}}^2+a\|\nabla\tau\|_{L^{2}}^{2}=\langle u\cdot\nabla u, \Delta u \rangle + \langle u\cdot\nabla\tau, \Delta\tau \rangle + \langle Q(\nabla u,\tau),\Delta\tau \rangle.
\end{align*}
By virtue of the estimate for the inner product, we obtain
    \begin{align}\label{5g3}
    \frac{d}{dt} \langle \eta\tau,-\nabla u \rangle + \frac{\eta}{2}\|\nabla u\|_{L^{2}}^2 
    &=\eta \langle\mathbb{P}\left(\mathrm{div}\tau
-u\cdot\nabla u\right),\mathrm{div}\tau \rangle \\ \notag
&~~~+\eta \langle u\cdot\nabla\tau+a\tau+Q(\nabla u,\tau)-\Delta \tau,\nabla u\rangle.
\end{align}
where we choose $\eta$ to be a constant such that $0\ll \ep \ll \eta \ll 1.$
We deduce that
\begin{align*}
    |\langle Q(\nabla u,\tau),\tau \rangle|&\lesssim  \|\nabla u\|_{L^{2}} \|\nabla \tau\|_{L^{2}}\|\tau\|_{L^{3}} \\
    &\lesssim C_1\ep \left(\|\nabla u\|_{L^{2}}^2 + \|\nabla \tau\|_{L^{2}}^2\right),
\end{align*} 
and
\begin{align*}
 &\eta |\langle\mathbb{P}\left(\mathrm{div}\tau-u\cdot\nabla u\right),\mathrm{div}\tau \rangle | + \eta |\langle u\cdot\nabla\tau+a\tau+Q(\nabla u,\tau)-\Delta \tau,\nabla u\rangle|\\
 \lesssim&  \eta \|\nabla \tau\|_{L^{2}}^2 + \eta\|\nabla u\|_{L^{2}} \|\nabla \tau\|_{L^{2}} \|u\|_{L^{\infty}} + \eta a \|\tau\|_{L^{2}}\|\nabla u\|_{L^{2}} \\
 &+\eta \|\nabla u\|_{L^{2}}^2\|\tau\|_{L^{\infty}} + \eta \|\nabla^2 \tau\|_{L^{2}}\|\nabla u\|_{L^{2}}\\
 \le& C\eta \|\nabla \tau\|_{L^{2}}^2 + \frac{\eta}{4} \|\nabla u\|_{L^{2}}^2 + \frac{a}{4} \|\tau\|_{L^{2}}^2 + \frac{1}{4} \|\nabla^2 \tau\|_{L^{2}}^2.
\end{align*}
Moreover, one can infer that
\begin{align*}
&|\langle u\cdot\nabla u, \Delta u \rangle| + |\langle u\cdot\nabla\tau, \Delta\tau \rangle| + |\langle Q(\nabla u,\tau),\Delta\tau \rangle|\\
\lesssim&  \|\nabla u\|_{L^{2}}^2\|\nabla u\|_{L^{\infty}} 
+ \|\nabla \tau\|_{L^{2}}\|\nabla^2 \tau\|_{L^{2}}\|u\|_{L^{\infty}} 
+ \|\nabla u\|_{L^{2}}\|\nabla^2 \tau\|_{L^{2}}\|\tau\|_{L^{\infty}}\\
\lesssim& \ep \|\nabla u\|_{L^{2}}^2 
+ \ep\|\nabla \tau\|_{H^{1}}^2.
\end{align*}
Combining these three inequalities, we deduce from \eqref{5g1}-\eqref{5g3} that
\begin{equation*}
\frac{d}{dt}\left(\frac{1}{2}\left\|(u,\tau)\right\|_{H^{1}}^{2}+\eta\langle \tau,-\nabla u \rangle\right)
+\frac{1}{2}\|\nabla \tau\|_{H^{1}}^2
+ \frac{\eta}{8}\|\nabla u\|_{L^{2}}^2
+\frac{a}{2}\|\tau\|_{H^{1}}^{2}\le 0,
\end{equation*}
which implies that
\begin{equation*}
\sup_{0\le t<\infty}\left( \|u\|_{H^1}^2 + \|\tau\|_{H^1}^2 \right) +  \int_0^\infty \left(\|\nabla \tau\|_{H^{1}}^2
+ \frac{\eta}{4}\|\nabla u\|_{L^{2}}^2\right)dt\le 4\ep^2.
\end{equation*}
To estimate $\|\tau\|_{B^0_{\infty,1}}$, we infer that
\begin{align*}
\|\tau\|_{B^0_{\infty,1}} \le& \|\tau\|_{L^{2}} + \sum_{j\ge 0} \|\Delta_j \tau\|_{L^{\infty}}\\
\le & 2\ep + \sum_{j\ge 0}\int_0^t \left\|e^{(t-s)\Delta}\Delta_j (u\cdot\nabla\tau + Q - Du)\right\|_{L^{\infty}} ds\\
\le & 2\ep + \sum_{j\ge 0}\int_0^t e^{-(t-s)2^{2j}} 2^{\frac{7}{4}j}\left( \|u\cdot \nabla \tau\|_{L^{\frac{12}{7}}} + \|Q\|_{L^{\frac{12}{7}}} + \|\nabla u\|_{L^{2}}\right) ds\\
\le & 2\ep + \sum_{j\ge 0}\int_0^t e^{-(t-s)2^{2j}} 2^{\frac{7}{4}j}\left( \|\nabla \tau\|_{L^{2}}\| u\|_{L^{12}} + \|\nabla u\|_{L^{2}}\|\tau\|_{L^{12}} + \|\nabla u\|_{L^{2}}\right) ds\\
\le & 2\ep + 3\ep \sum_{j\ge 0}\int_0^t e^{-(t-s)2^{2j}} 2^{\frac{7}{4}j}ds
\le  30\ep.
\end{align*}
To estimate $\|u\|_{B^1_{\infty,1}}$, we also define $\Gamma := \omega - \mathcal{R}\tau$. Then we have
\begin{equation*}
\partial_t \Gamma + (u\cdot\nabla)\Gamma + \frac{1}{2}\Gamma = (a-\frac{1}{2})\mathcal{R}\tau + [\mathcal{R}, u\cdot\nabla]\tau + \mathcal{R}Q + \omega\cdot\nabla u := F.
\end{equation*}
Then by Lemma \ref{TP}, we get
\begin{align}\label{5g4}
\|\Gamma\|_{B^0_{\infty,1}} &\lesssim \|\Gamma_0\|_{B^0_{\infty,1}} e^{-\frac12t+C\int_0^t \|\nabla u\|_{B^0_{\infty,1}}ds } + \int_0^t e^{-\frac12(t-s)+C\int_s^{t}\|\nabla u\|_{B^0_{\infty,1}}dt'} \|F\|   
_{B^0_{\infty,1}} ds \\ \notag
&\lesssim \ep + \int_0^t e^{-\frac{1}{4}(t-s)} \|F\|  
_{B^0_{\infty,1}} ds.
\end{align}
By Lemmas \ref{T}, \ref{R}, \ref{PL} and \ref{CR}, we deduce that 
\begin{align*}
\|(a-\tfrac12)\mathcal{R}\tau \|_{B^0_{\infty,1}} \lesssim  \|\tau\|_{L^{2}} + \|\tau\|_{B^0_{\infty,1}}\lesssim \ep,
\end{align*}
and for any $0<\delta<\frac{1}{2}$, 
\begin{align*}
\|[\mathcal{R}, u\cdot\nabla]\tau\|_{B^0_{\infty,1}}
&\lesssim (\|\omega\|_{L^{2}} + \|\omega\|_{L^{\infty}})(\|\tau\|_{L^{2}} + \|\tau\|_{B^{\delta}_{\infty,1}})\\
&\lesssim C_1 \ep (\|\tau\|_{L^{2}} + \|\nabla^2\tau\|_{L^{2}})\\
&\lesssim \ep + \|\nabla^2\tau\|_{L^{2}}.
\end{align*}
Similarly, we infer that
\begin{align*}
\|\mathcal{R}Q\|_{B^0_{\infty,1}}
&\lesssim \|Q\|_{L^{2}} + \|Q\|_{B^0_{\infty,1}} \\
&\lesssim \|\nabla u\|_{L^{2}}\|\tau\|_{L^{\infty}} + \|\nabla u\|_{B^0_{\infty,1}}\|\tau\|_{B^{\delta}_{\infty,1}}\\
&\lesssim \ep^2 + C_1 \epsilon (\|\tau\|_{L^{2}} + \|\nabla^2\tau\|_{L^{2}})\\
&\lesssim \ep + \|\nabla^2\tau\|_{L^{2}},
\end{align*}
and 
\begin{align*}
\|\omega\cdot\nabla u\|_{B^0_{\infty,1}}
&\lesssim \|T_{\omega}\nabla u\|_{B^0_{\infty,1}} + \|T_{\nabla u}\omega\|_{B^0_{\infty,1}} + \|\rm div~\mathcal{R}(\omega, u)\|_{B^0_{\infty,1}}\\
&\lesssim \|\omega\|_{L^{\infty}} \|\nabla u\|_{B^0_{\infty,1}} + \|\nabla u\|_{L^{\infty}} \|\omega\|_{B^0_{\infty,1}} + \|\omega\|_{B^0_{\infty,1}}\|u\|_{B^1_{\infty,1}}\lesssim \ep.
\end{align*} 
Combining these estimates, we deduce from \eqref{5g4} that
\begin{align*}
\|\Gamma\|_{B^0_{\infty,1}}
&\lesssim \ep + \int_0^t e^{-\frac{1}{4}(t-s)} \|F\|  
_{B^0_{\infty,1}} ds\\
&\lesssim \ep + \int_0^t e^{-\frac{1}{4}(t-s)} (\ep +  \|\nabla^2\tau\|_{L^{2}}) ds\lesssim \ep,
\end{align*}
which implies that
\begin{align*}
\|u\|_{B^1_{\infty,1}} \le C\|u\|_{L^{2}} 
+ C\|\Gamma\|_{B^0_{\infty,1}} 
+ C\|\tau\|_{B^0_{\infty,1}}\le C \ep.
\end{align*}
In these bounds the stretching contribution is
$C(C_1\ep)^2$, rather than a new linear term.
Choose $C_1=200C$ first and then choose $\ep$ so small that
$CC_1\ep\leq\eta/32$ (for the energy and damped transport estimates) and
$C C_1^2\ep\leq1$.
The preceding estimates then improve \eqref{5EI} to
$(C_1/2)\ep$ on $[0,T]$.
The critical local continuation criterion and the uniform
finite-time heat estimate extend the solution globally.
Therefore, we complete the proof of Proposition \ref{5prop1}.
  \end{proof}
Now, we consider vanishing damping limit for any $T>0$ under the same topology. Since the solutions $(u,\tau)$ of the inviscid Oldroyd-B model \eqref{eq0} depend on damping $a\in[0,1]$, when considering the problems of vanishing damping limit, we denote the solutions of \eqref{eq0} as $(u^a,\tau^a)$.
\begin{prop}\label{5prop2} 
Under the conditions in Proposition \ref{5prop1}, for any $T>0$, there holds 
  $$\|(u^a,\tau^a)-(u^0,\tau^0)\|_{L^{\infty}(0,T;L^2)}\leq C_Ta,
  $$ 
  and 
  $$\lim_{a\rightarrow0}\|(u^a,\tau^a)-(u^0,\tau^0)\|_{L^{\infty}(0,T;\dot{H}^1\cap B^1_{\infty,1})\times L^{\infty}(0,T;\dot{H}^1\cap B^0_{\infty,1})}=0.
  $$
  \end{prop}
\begin{proof}
We record the dimension-dependent estimates needed to adapt
Proposition \ref{4prop3}. Put $U=u^a-u^0$ and
$\Sigma=\tau^a-\tau^0$. Integration by parts in
$Q(\nabla U,\tau^0)$ and $\|\Sigma\|_6\lesssim\|\nabla\Sigma\|_2$
give
\[
 (\|U\|_2^2+\|\Sigma\|_2^2)'
 +\|\nabla\Sigma\|_2^2
 \leq C_T(t)(\|U\|_2^2+\|\Sigma\|_2^2)+Ca^2,
\]
where
$C_T(t)=C(1+\|\nabla(u^a,u^0)\|_\infty+
\|\tau^0\|_\infty^2+\|\nabla\tau^0\|_3^2)$.
Its integral over $[0,T]$ is bounded uniformly in $a$, since
$\|\nabla\tau^0\|_3^2\lesssim
\|\nabla\tau^0\|_2\|\nabla^2\tau^0\|_2$.
Thus $\|(U,\Sigma)\|_{L^\infty_TL^2}\leq C_Ta$.

For later recovery of the velocity endpoint norm, the heat-block
argument in Step 1 of Proposition \ref{4prop3} gives
\begin{equation}\label{5finite-stress-endpoint}
 \|\Sigma\|_{L^\infty_T B^0_{\infty,1}}\leq C_Ta^{1/2}.
\end{equation}
Here use $\|U\|_{12}\lesssim\|U\|_2^{1/2}
\|U\|_{B^1_{\infty,1}}^{1/2}$ and
$\|\Sigma\|_4\leq\|\Sigma\|_2^{1/2}\|\Sigma\|_\infty^{1/2}$.
After writing $Q(\nabla U,\tau^0)$ as derivatives of $U\tau^0$
plus $U\nabla\tau^0$, its latter term belongs to $L^{12/5}$.
Since $\|\nabla\tau^0\|_{L^4_TL^3}\leq C_T$, Bernstein and
the $L^{4/3}$ time norm of the heat kernel give
$C_T2^{-j/4}a^{1/2}$ for that term.
The divergence term $u^a\Sigma$ gives the same bound; all
other heat-block bounds have summable factors
$2^{-j/2}$ or $2^{-3j/4}$. The low block is controlled by the $L^2$ estimate.

For the strong norms take truncated solutions
$(u_N^a,\tau_N^a)$ with data $(S_Nu_0,S_N\tau_0)$.
Let $c_N$ denote the fixed initial $H^1$ and critical Besov tails
defined in Proposition \ref{4prop3}. On each finite interval,
the positive-index tame estimates give
\begin{align}\label{5finite-tame}
 \sup_{a\in[0,1]}\bigl(
 \|(u_N^a,\tau_N^a)\|_{L^\infty_TH^2}
 +\|u_N^a\|_{L^\infty_TB^2_{\infty,1}}
 +\|\tau_N^a\|_{\widetilde L^\infty_TB^1_{\infty,1}}
 +\|\tau_N^a\|_{L^1_TB^3_{\infty,1}}\bigr)
 \leq C_T2^N.
\end{align}
To see the uniformity of $C_T$, apply the differentiated energy
and heat-block estimates: the highest derivative in $Q$ is
integrated by parts, leaving the coefficient
$\|\tau_N^a\|_\infty^2+\|\nabla\tau_N^a\|_\infty^2+
\|\nabla u_N^a\|_\infty$.
The first and third are bounded, and
$\int_0^T\|\nabla\tau_N^a\|_\infty^2
\lesssim\|\tau_N^a\|_{L^\infty_TB^0_{\infty,1}}
\|\tau_N^a\|_{L^1_TB^2_{\infty,1}}\leq C_T$.
At the positive Besov index, use
$\|Q\|_{B^1}\lesssim\|\nabla u_N^a\|_\infty\|\tau_N^a\|_{B^1}
+\|\tau_N^a\|_\infty\|u_N^a\|_{B^2}$,
absorb the transport derivative into the heat norm, and apply
Gronwall. No high-norm smallness is imposed.

Write $W_N=u^a-u_N^a$, $\Theta_N=\tau^a-\tau_N^a$.
The lower Besov estimate \eqref{4lo24} remains valid in dimension three:
the mixed Bony remainder is in $B^1_{2,1}\hookrightarrow
B^{-1/2}_{\infty,1}\hookrightarrow B^{-1}_{\infty,1}$.
Its coefficient $\|\tau_N^a\|_{H^2}$ has a uniform $L^2_T$ bound
from Proposition \ref{5prop1}. The differentiated difference
energy and \eqref{5finite-tame} also give, uniformly in $a,N$,
\[
 \|(W_N,\Theta_N)\|_{L^\infty_T\dot H^1}
 +\|\nabla^2\Theta_N\|_{L^2_TL^2}\leq C_Tc_N.
\]
The only potentially large transport term is
$\|W_N\|_2\|\nabla^2u_N^a\|_\infty$; the factors
$2^{-N}c_N$ and $C_T2^N$ cancel.

The estimates for $\Gamma^a-\Gamma_N^a$ are then as in
\eqref{4lo27}, with two changes. In the $Q$ term involving
$\Theta_N$, use its $L^2_TH^2$ bound:
$\|R(\nabla u_N^a,\Theta_N)\|_{B^0_{\infty,1}}
\lesssim\|\nabla u_N^a\|_{B^0_{\infty,1}}\|\Theta_N\|_{H^2}$.
For the additional stretching term use
$\operatorname{div}\omega^a=\operatorname{div}(\omega^a-\omega_N^a)=0$
and Lemma \ref{PL}:
\[
 \|\omega^a\cdot\nabla W_N+
       (\omega^a-\omega_N^a)\cdot\nabla u_N^a\|_{B^0}
 \leq C\ep\bigl(\|\Gamma^a-\Gamma_N^a\|_{B^0}
          +\|\Theta_N\|_{B^0}+\|(W_N,\Theta_N)\|_{H^1}\bigr).
\]
The integrated stress endpoint error is small by
\eqref{4lo24-interpolation}. Finally the coefficient
$\|\tau_N^a\|_{H^2}$ in the other $Q$ term is absorbed using
$\sup_r\int_0^r e^{-(r-s)/4}\|\tau_N^a(s)\|_{H^2}\,ds
\leq C\ep$, not a pointwise high-norm bound. Hence
$\|\Gamma^a-\Gamma_N^a\|_{L^\infty_TB^0}\leq C_Tc_N$.

For each fixed $N$, the smooth difference energy and endpoint
transport estimates yield convergence of
$(u_N^a,\tau_N^a)$ to $(u_N^0,\tau_N^0)$ in the claimed norms,
with constants allowed to depend on $N$.
First send $a\to0$ and then $N\to\infty$.
Recover $\|\nabla U\|_{B^0}$ from
$\Gamma^a-\Gamma^0$, \eqref{5finite-stress-endpoint}, and the low
$L^2$ block. Together with the differentiated energy estimates
this proves both assertions.
\end{proof}

\subsection{Optimal decay rates and uniform vanishing damping limit} 
From now on, we investigate optimal decay rate of global solutions for the inviscid Oldroyd-B equation \eqref{eq0} with critical regularity. Due to low regularity of the solutions, the case $a\in[0,1]$ is extremely challenging. We will solve the difficulties of consistent damping and low regularity by virtue of the improved Fourier splitting method. Then we obtain optimal decay rate for the solutions in $L^2$. Since 
$$\int_{\mathbb{R}^{3}}(u\cdot\nabla)u\cdot \Delta u dx\neq 0,$$
we cannot obtain directly time decay rate of the $\dot{H}^{1}$ norm for lack of external higher order regularity for the solutions. However, we find that the time integrability of global solutions for $d=3$ can be obtained and helps to overcome this difficulty.   
  \begin{prop}\label{5prop3}
  Under the same assumptions as in Proposition \ref{5prop1}, if additionally $(u_0,\tau_0) \in \dot{B}^{-\frac 3 2}_{2,\infty},$
        then there holds
        $$
        \|(u,\tau)\|_{L^2} +(1+t)^{\frac{1}{2}}\|\nabla(u,\tau)\|_{L^2}\leq C_0(1+t)^{-\frac{3}{4}}
        $$
  and 
  $$\int_{0}^{\infty}\|(\tau, \nabla u)\|_{B^0_{\infty,1}}dt'\leq C_0,$$
  where $C_0$ depends on $\|(u_0,\tau_0)\|_{H^1\cap \dot{B}^{-\frac 3 2}_{2,\infty}}$ and $\|(\nabla u_0,\tau_0)\|_{B^0_{\infty,1}}.$
  \end{prop}
\begin{proof} The proof of Proposition \ref{5prop3} relies on a series of \emph{iteration technique}. For sake of clarity, we first present the main ideas in the following chart. \begin{equation*}
\begin{aligned}
 \text{Energy and Fourier splitting}
 &\longrightarrow \|(u,\tau)\|_{H^1}
   \lesssim (1+t)^{-1/8},\ (1+t)^{-1/4},\ (1+t)^{-3/8},\\
 \|(u,\tau)\|_{H^1}\lesssim(1+t)^{-3/8}
 &\longrightarrow \sup_t\|(u,\tau)\|_{\dot B^{-3/2}_{2,\infty}}<\infty
 \longrightarrow \|(u,\tau)\|_{H^1}\lesssim(1+t)^{-3/4},\\
 \|(u,\tau)\|_{H^1}\lesssim(1+t)^{-3/4}
 &\longrightarrow \text{key integrability}
 \longrightarrow \|\nabla(u,\tau)\|_2\lesssim(1+t)^{-5/4}.
\end{aligned}
\end{equation*}

We divide the proof into the following five steps. \\
\textbf{Step 1: Initial decay rate.} \\
According to Proposition \ref{5prop1}, for any $t>0$, we have
     \begin{equation}\label{5de1}
\frac{d}{dt}\left(\left\|(u,\tau)\right\|_{H^{1}}^{2}+2\eta\langle \tau,-\nabla u \rangle\right)
+\|\nabla \tau\|_{H^{1}}^2
+ \frac{\eta}{4}\|\nabla u\|_{L^{2}}^2
+a\|\tau\|_{H^{1}}^{2}\le 0.
\end{equation}
  Define $S(t)=\{\xi:|\xi|^2\leq C_1(1+t)^{-1}\}$ with $C_1$ large enough. Considering the issue of time decay rates, we can assume that $t$ is sufficiently large relative to $C_1$. Applying Schonbek's \cite{S85} strategy to \eqref{5de1}, we have
$$\|\nabla \tau\|^2_{H^1}=\int_{S(t)}(1+|\xi|^{2})|\xi|^2|\hat{\tau}(\xi)|^2 d\xi+\int_{S(t)^c}(1+|\xi|^{2})|\xi|^2|\hat{\tau}(\xi)|^2 d\xi.$$
and
$$\|\nabla u\|^2_{L^2}=\int_{S(t)}|\xi|^2|\hat{u}(\xi)|^2 d\xi+\int_{S(t)^c}|\xi|^2|\hat{u}(\xi)|^2 d\xi.$$
One can easily deduce that
$$\frac {C_1} {1+t}\int_{S(t)^c}(1+|\xi|^{2})|\hat{\tau}(\xi)|^2 d\xi\leq\|\nabla \tau\|^2_{H^1},$$
and
$$\frac {C_1} {1+t}\int_{S(t)^c}|\hat{u}(\xi)|^2 d\xi\leq\|\nabla u\|^2_{L^{2}}.$$
    By \eqref{5de1}, we obtain
    \begin{align}\label{5de2}
    \frac d {dt} \left(\left\|(u,\tau)\right\|_{H^{1}}^{2}+2\eta\langle \tau,-\nabla u \rangle\right)+\frac {\eta C_1} {8(1+t)}\|(u,\tau)\|^2_{H^1}
    \leq \frac {C} {1+t}\int_{S(t)}|\hat{u}(\xi)|^2+|\hat{\tau}(\xi)|^2 d\xi.
    \end{align}
 The $L^2$ estimate to the low frequency of $(u,\tau)$ is the key to studying time decay rates. Applying Fourier transform to \eqref{eq0}, we deduce that
    \begin{align}\label{eqF}
    \left\{
    \begin{array}{ll}
    \hat{u}^{j}_t+i\xi_{j} \hat{P}-i\xi_{k} \hat{\tau}^{jk}=\hat{F}^j,  \\[1ex]
    \hat{\tau}^{jk}_t+|\xi|^2\hat{\tau}^{jk}+a\hat{\tau}^{jk}-\frac i 2(\xi_{k} \hat{u}^j+\xi_{j} \hat{u}^k)=\hat{G}^{jk},\\[1ex]
    \end{array}
    \right.
    \end{align}
    where $F=-u\cdot\nabla u$ and $G=-u\cdot\nabla \tau-Q(\nabla u, \tau)$.
    According to \eqref{eqF}, we get
    \begin{align}\label{5de3}
    \int_{S(t)}|\hat{u}|^2+|\hat{\tau}|^2d\xi
    \lesssim\int_{S(t)} |\hat{u}_0|^2+|\hat{\tau}_0|^2d\xi  +{\underbrace{\int_{S(t)}\int_{0}^{t}|\hat{F}\cdot\bar{\hat{u}}|+|\hat{G}\cdot\bar{\hat{\tau}}|dt'd\xi}_{\coloneqq \mathcal{N}(u,\tau)}}.
    \end{align}
    According to $(u_0,\tau_0)\in L^2\cap\dot{B}^{-\frac 3 2}_{2,\infty}$ and applying Lemma \ref{LPD}, we have
    \begin{align}\label{5de4}
    \int_{S(t)}|\hat{u}_0|^2+|\hat{\tau}_0|^2 d\xi
    &\lesssim\sum_{j\leq \log_2[\frac {4} {3}C_1^{\frac 1 2 }(1+t)^{-\frac 1 2}]}\int_{\mathbb{R}^{3}} 2\varphi^2(2^{-j}\xi)(|\hat{u}_0|^2+|\hat{\tau}_0|^2)d\xi \\ \notag
    &\lesssim\sum_{j\leq \log_2[\frac {4} {3}C_1^{\frac 1 2 }(1+t)^{-\frac 1 2}]}2(\|\dot{\Delta}_j u_0\|^2_{L^2}+\|\dot{\Delta}_j \tau_0\|^2_{L^2}) \\ \notag
    &\lesssim(1+t)^{-\frac 3 2}\|(u_0,\tau_0)\|^2_{\dot{B}^{-\frac 3 2}_{2,\infty}}.
    \end{align}
Moreover, we infer from \eqref{5de1} that
\begin{align}\label{5de5}   \int_{S(t)}\int_{0}^{t}|\hat{F}\cdot\bar{\hat{u}}|+|\hat{G}\cdot\bar{\hat{\tau}}|dt'd\xi
 &\lesssim(1+t)^{-\frac 3 4} \int_{0}^{t} \|\hat{F}\cdot\bar{\hat{u}}\|_{L^2_\xi} + \|\hat{G}\cdot\bar{\hat{\tau}}\|_{L^2_\xi} dt'\\ \notag
 &\lesssim(1+t)^{-\frac 3 4} \int_{0}^{t}(\|u\|^2_{L^{2}}+\|\tau\|^2_{L^{2}})(\|\nabla u\|_{L^{2}}+\|\nabla \tau\|_{L^{2}})dt' \\ \notag
 &\lesssim C_0(1+t)^{-\frac 1 4}.
    \end{align}
    Combining \eqref{5de2}-\eqref{5de5}, we have
    \begin{align*}
    \frac d {dt} \left(\left\|(u,\tau)\right\|_{H^{1}}^{2}+2\eta\langle \tau,-\nabla u \rangle\right)+\frac {\eta C_1} {8(1+t)}\|(u,\tau)\|^2_{H^1}
     \lesssim C_0(1+t)^{-\frac 5 4}
    \end{align*}
    which implies that
    \begin{align}\label{5de7}
    \left\|(u,\tau)\right\|_{H^{1}}^{2}\lesssim C_0(1+t)^{-\frac 1 4}.
    \end{align}
 Plugging \eqref{5de7} into \eqref{5de5}, we obtain
    \begin{align}\label{5de8}   \int_{S(t)}\int_{0}^{t}|\hat{F}\cdot\bar{\hat{u}}|+|\hat{G}\cdot\bar{\hat{\tau}}|dt'd\xi
 &\lesssim C_0(1+t)^{-\frac 3 4} \int_{0}^{t}(1+t')^{-\frac 1 4}(\|\nabla u\|_{L^{2}}+\|\nabla \tau\|_{L^{2}})dt' \\ \notag
 &\lesssim C_0(1+t)^{-\frac 1 2}.
    \end{align}
    Combining \eqref{5de2}-\eqref{5de4} and \eqref{5de8}, we have
    \begin{align*}
    \frac d {dt} \left(\left\|(u,\tau)\right\|_{H^{1}}^{2}+2\eta\langle \tau,-\nabla u \rangle\right)+\frac {\eta C_1} {8(1+t)}\|(u,\tau)\|^2_{H^1}
     \lesssim C_0(1+t)^{-\frac 3 2}
    \end{align*}
    which implies that
    \begin{align}\label{5de10}
    \left\|(u,\tau)\right\|_{H^{1}}^{2}\lesssim C_0(1+t)^{-\frac 1 2}.
    \end{align}
 This together with \eqref{5de1} ensure that
 \begin{align}\label{5de11}
    \int_{0}^{t}(1+t')^{\frac 1 2-\delta}(\|\nabla u\|^2_{L^{2}}+\|\nabla \tau\|^2_{L^{2}})dt'\lesssim C_0\frac {1}{\delta},
    \end{align}
 where $\delta\in(0,\frac 1 2)$.
    Plugging \eqref{5de10} into \eqref{5de5}, we infer from \eqref{5de11} that 
    \begin{align}\label{5de12}  \int_{S(t)}\int_{0}^{t}|\hat{F}\cdot\bar{\hat{u}}|+|\hat{G}\cdot\bar{\hat{\tau}}|dt'd\xi
 &\lesssim C_0(1+t)^{-\frac 3 4} \int_{0}^{t}(1+t')^{-\frac 1 2}(\|\nabla u\|_{L^{2}}+\|\nabla \tau\|_{L^{2}})dt' \\ \notag
 &\lesssim C_0(1+t)^{-\frac 3 4}.
    \end{align}
    Combining \eqref{5de2}-\eqref{5de4} and \eqref{5de12}, we have
    \begin{align*}
    \frac d {dt} \left(\left\|(u,\tau)\right\|_{H^{1}}^{2}+2\eta\langle \tau,-\nabla u \rangle\right)+\frac {\eta C_1} {8(1+t)}\|(u,\tau)\|^2_{H^1}
     \lesssim C_0(1+t)^{-\frac 7 4}
    \end{align*}
    which implies that
    \begin{align}\label{5de14}
    \left\|(u,\tau)\right\|_{H^{1}}^{2}\lesssim C_0(1+t)^{-\frac 3 4}.
    \end{align}
    This together with \eqref{5de1} ensure that
 \begin{align}\label{5de15}
    \int_{0}^{t}(1+t')^{\frac 3 4-\delta}(\|\nabla u\|^2_{L^{2}}+\|\nabla \tau\|^2_{L^{2}})dt'\lesssim C_0\frac {1}{\delta},
    \end{align} 
 where $\delta\in(0,\frac 3 4)$. \\
\textbf{Step 2: Uniform bounds in negative Besov space.} \\
Applying $\dot{\Delta}_j$ to \eqref{eq0}, we have
    \begin{align}\label{eqB}
    \left\{
    \begin{array}{ll}
    \dot{\Delta}_j u_t+\nabla\dot{\Delta}_j P-{\rm div}~\dot{\Delta}_j\tau=\dot{\Delta}_j F,  \\[1ex]
    \dot{\Delta}_j \tau_t-\Delta\dot{\Delta}_j \tau+a\dot{\Delta}_j \tau-\dot{\Delta}_j D(u)=\dot{\Delta}_j G. \\[1ex]
    \end{array}
    \right.
    \end{align}
    We firstly infer from \eqref{eqB} that
    \begin{align}\label{5de16}
    \frac d {dt}(\|\dot{\Delta}_j u\|^2_{L^2}+\|\dot{\Delta}_j \tau\|^2_{L^2})+2\|\nabla\dot{\Delta}_j \tau\|^2_{L^2}
    \lesssim\|\dot{\Delta}_j F\|_{L^2}\|\dot{\Delta}_j u\|_{L^2}+\|\dot{\Delta}_j G\|_{L^2}\|\dot{\Delta}_j \tau\|_{L^2}.
    \end{align}
Set
\[
 H(t)=\sup_{0\leq r\leq t}
 \left(\|u(r)\|_{\dot B^{-3/2}_{2,\infty}}
       +\|\tau(r)\|_{\dot B^{-3/2}_{2,\infty}}\right).
\]
Multiplying \eqref{5de16} by $2^{-3j}$ and integrating before
 taking the supremum gives
\begin{align}\label{5de17}
 &\sup_{0\leq r\leq t}\sup_{j\in\mathbb Z}2^{-3j}
 \left(\|\dot\Delta_j u(r)\|_2^2
       +\|\dot\Delta_j\tau(r)\|_2^2\right)\notag\\
 &\qquad\lesssim H^2(0)+H(t)\int_0^t
 \left(\|F(s)\|_{\dot B^{-3/2}_{2,\infty}}
       +\|G(s)\|_{\dot B^{-3/2}_{2,\infty}}\right)\,ds.
\end{align}
The paired-block supremum on the left is equivalent to $H^2(t)$.
One may first restrict the dyadic supremum to a finite index range;
the bounds below are independent of that range, so passage to the limit
also justifies finiteness of $H(t)$. Consequently,

    \begin{align}\label{5de18}
    H^2(t)&\lesssim H^2(0)+H(t)\int_0^{t}\|F\|_{\dot{B}^{-\frac 3 2}_{2,\infty}}+\|G\|_{\dot{B}^{-\frac 3 2}_{2,\infty}}dt'.
    \end{align}
    Using \eqref{5de14}, \eqref{5de15} and inclusion between Lesbesgue and Besov space, we infer that
    \begin{align}\label{5de19}
    \int_0^{t}\|(F,G)\|_{\dot{B}^{-\frac 3 2}_{2,\infty}}ds'
    &\lesssim\int_0^{t}\|(F,G)\|_{L^{1}}dt'  \\ \notag
    &\lesssim\int_0^{t}(\|u\|_{L^{2}}+\|\tau\|_{L^{2}})(\|\nabla u\|_{L^2}+\|\nabla\tau\|_{L^2})dt'   \\ \notag
    &\lesssim C_0\int_0^{t}(1+t')^{-\frac 3 8}(\|\nabla u\|_{L^2}+\|\nabla\tau\|_{L^2})dt'   \\ \notag
    &\lesssim C_0.
    \end{align}
    Combining \eqref{5de18} and \eqref{5de19}, for any $t>0$, we obtain 
 \begin{align*}
 H(t)\lesssim C_0.
 \end{align*}
\textbf{Step 3: Optimal decay rate in $L^2$.} \\
    According to \eqref{5de2}-\eqref{5de4}, we infer that
    \begin{align*}
    &\frac d {dt} \left(\left\|(u,\tau)\right\|_{H^{1}}^{2}+2\eta\langle \tau,-\nabla u \rangle\right)+\frac {\eta C_1} {8(1+t)}\|(u,\tau)\|^2_{H^1} \\ \notag
   \lesssim&  \frac {1} {1+t}\int_{S(t)}|\hat{u}(\xi)|^2+|\hat{\tau}(\xi)|^2 d\xi \notag\\
   \lesssim& H^2(t)(1+t)^{-\frac 5 2}\lesssim C_0^2(1+t)^{-\frac 5 2}. \notag
    \end{align*}
    As a result, we conclude that
 \begin{align*}
    \left\|(u,\tau)\right\|_{H^{1}}^{2}\lesssim C_0(1+t)^{-\frac 3 2}.
    \end{align*}
    This together with \eqref{5de1} ensure that
      \begin{align}\label{5de24}
    \int_{0}^{t}(1+t')^{\frac 5 2}(\|\nabla u\|^2_{L^{2}}+\|\nabla \tau\|^2_{H^{1}}) dt'\lesssim C_0(1+t),
    \end{align}
 and
      \begin{align}\label{5de25}
    \int_{0}^{t}(1+t')^{\frac 3 2-\delta}(\|\nabla u\|^2_{L^{2}}+\|\nabla \tau\|^2_{H^{1}}) dt'\lesssim C_0\frac {1}{\delta},
    \end{align}
 where $\delta\in(0,\frac 3 2)$. \\
\textbf{Step 4: The key integrability.} \\
    However, considering the decay rate for the first derivative of the solution to \eqref{eq0}, the main difficulty is the lack of control of the higher order energy. To overcome this difficulty, we introduce a new method which flexibly combines the time weighted energy estimate and the following key integrability:
 $$\int_{0}^{\infty}\|(\tau,\nabla u)\|_{B^0_{\infty,1}}dt'\leq C_0.$$
 According to \eqref{5de25}, for any $t>0$, we infer that
 \begin{align}\label{5de26}
    \int_{0}^{t}\|\mathcal{R} \tau\|_{B^0_{\infty,1}}dt' &\lesssim\int_{0}^{t}\|\Delta_{-1}\mathcal{R} \tau\|_{B^0_{\infty,1}} dt'+\int_{0}^{t}\|(Id-\Delta_{-1})\mathcal{R} \tau\|_{B^0_{\infty,1}} dt' \\ \notag
 &\lesssim\int_{0}^{t}\|\Delta_{-1}\mathcal{R} \tau\|_{L^6} dt'+\int_{0}^{t}\|(Id-\Delta_{-1})\tau\|_{B^0_{\infty,1}}dt' \\ \notag
 &\lesssim\int_{0}^{t}\|\nabla\tau\|_{H^1} dt'\lesssim\left(\int_{0}^{t}(1+t')^{\frac 5 4}\|\nabla \tau\|^2_{H^{1}} dt'\right)^{\frac 1 2}\lesssim C_0,
    \end{align}
Notice that $\int_{0}^{t}\|\tau\|_{B^0_{\infty,1}}dt'\lesssim C_0$ can be derived by the same method. Recall that
\begin{equation*}
\partial_t \Gamma + (u\cdot\nabla)\Gamma + \frac 1 2\Gamma = (a-\frac{1}{2})\mathcal{R}\tau + [\mathcal{R}, u\cdot\nabla]\tau + \mathcal{R}Q + \omega\cdot\nabla u.
\end{equation*} 
where $\Gamma = \omega - \mathcal{R}\tau$. 
Then by Proposition \ref{5prop1}, we have
\begin{align*}
\|\Gamma\|_{B^0_{\infty,1}} &\lesssim \|\Gamma_0\|_{B^0_{\infty,1}} e^{-\frac {1}{4}t} + \int_0^t e^{-\frac 1 4(t-t')} \|(a-\frac{1}{2})\mathcal{R}\tau + [\mathcal{R}, u\cdot\nabla]\tau + \mathcal{R}Q + \omega\cdot\nabla u\|   
_{B^0_{\infty,1}} dt'.
\end{align*}
This together with \eqref{5de26}, Lemmas \ref{CZ}, \ref{PL} and \ref{CR} ensures that
\begin{align*}
\int_0^t\|\Gamma\|_{B^0_{\infty,1}}dt' &\lesssim \|\Gamma_0\|_{B^0_{\infty,1}}+ \int_0^t \|(a-\tfrac12)\mathcal{R}\tau + [\mathcal{R}, u\cdot\nabla]\tau + \mathcal{R}Q + \omega\cdot\nabla u\|   
_{B^0_{\infty,1}} dt' \\ \notag
&\lesssim \|\Gamma_0\|_{B^0_{\infty,1}}+ \int_0^t \|\mathcal{R}\tau\|   
_{B^0_{\infty,1}} +(\|\nabla u\|_{B^0_{\infty,1}}+\|\omega\|_{L^2})(\|\tau\|_{B^{\frac 1 4}_{\infty,1}}+\|\tau\|_{L^2})dt' \\ \notag
&~~~+\int_0^t \|\omega\|   
_{B^0_{\infty,1}}\|u\|   
_{B^1_{\infty,1}} dt' \\ \notag
&\lesssim C_0+ \int_0^t \|\nabla u\|_{B^0_{\infty,1}}\|\tau\|_{L^2}dt'+\int_0^t \|\omega\|   
_{B^0_{\infty,1}}\|u\|   
_{B^1_{\infty,1}} dt' \\ \notag
&\lesssim C_0+\ep\int_0^t\|\Gamma\|_{B^0_{\infty,1}}dt'.
\end{align*}
Then we have 
\begin{align}\label{5de29}
\int_0^t\|\nabla u\|_{B^0_{\infty,1}}dt'\lesssim \int_0^t\|\Gamma\|_{B^0_{\infty,1}}+\|\nabla u\|_{L^2}+\|\mathcal{R} \tau\|_{B^0_{\infty,1}}dt'\lesssim C_0.
\end{align}
\textbf{Step 5: Optimal decay rate in $\dot{H}^1$.} \\
According to \eqref{eq0}, we have
\begin{align}\label{5de30}
\frac{1}{2}\frac{d}{dt}\left\|\nabla(u,\tau)\right\|_{L^{2}}^{2}+\|\nabla^2 \tau\|_{L^{2}}^2+a\|\nabla\tau\|_{L^{2}}^{2}=\langle u\cdot\nabla u, \Delta u \rangle + \langle u\cdot\nabla\tau, \Delta\tau \rangle + \langle Q(\nabla u,\tau),\Delta\tau \rangle.
\end{align}
Then we have
\begin{align*}
&|\langle u\cdot\nabla u, \Delta u \rangle| + |\langle u\cdot\nabla\tau, \Delta\tau \rangle| + |\langle Q(\nabla u,\tau),\Delta\tau \rangle| \\
\lesssim&\|\nabla u\|_{L^{2}}^2\|\nabla u\|_{L^{\infty}} 
+ \|\nabla \tau\|_{L^{6}}\|\nabla^2 \tau\|_{L^{2}}\|u\|_{L^{3}} 
+ \|\nabla u\|_{L^{2}}\|\nabla^2 \tau\|_{L^{2}}\|\tau\|_{L^{\infty}} \\
\leq& C\|\nabla u\|_{L^{2}}^2\|\nabla u\|_{L^{\infty}} 
+\frac 1 2\|\nabla^2 \tau\|^2_{L^{2}}
+ C\|\nabla u\|^2_{L^{2}}\|\tau\|^2_{L^{\infty}}.
\end{align*}
This together with \eqref{5de30} implies that
\begin{align}\label{5de31}
\frac{d}{dt}\left\|\nabla(u,\tau)\right\|_{L^{2}}^{2}\lesssim \|\nabla u\|_{L^{2}}^2\|\nabla u\|_{L^{\infty}} 
+ \|\nabla u\|^2_{L^{2}}\|\tau\|^2_{L^{\infty}}.
\end{align}
Multiplying $(1+t)^{\frac 7 2}$ to \eqref{5de31}, then we get
\begin{align}\label{5de32}
\frac{d}{dt}\left[(1+t)^{\frac 7 2}\left\|\nabla(u,\tau)\right\|_{L^{2}}^{2}\right]\lesssim (\|\nabla u\|_{L^{\infty}} 
+\|\tau\|^2_{L^{\infty}})(1+t)^{\frac 7 2}\left\|\nabla(u,\tau)\right\|_{L^{2}}^{2}+(1+t)^{\frac 5 2}\left\|\nabla(u,\tau)\right\|_{L^{2}}^{2}.
\end{align}
Integrating $t$ at both sides of \eqref{5de32}, applying Gronwall's inequality and using \eqref{5de24}, \eqref{5de29}, we obtain
\begin{align*}
(1+t)^{\frac 7 2}\left\|\nabla(u,\tau)\right\|_{L^{2}}^{2}&\lesssim e^{\int_0^t(\|\nabla u\|_{L^{\infty}} 
+\|\tau\|^2_{L^{\infty}})\,dt'}\left(\left\|\nabla(u_0,\tau_0)\right\|_{L^{2}}^{2}+\int_0^t (1+t')^{\frac 5 2}\left\|\nabla(u,\tau)\right\|_{L^{2}}^{2}dt'\right) \\ \notag
&\lesssim C_0(1+t), 
\end{align*}
which implies that
\begin{align*}
\left\|\nabla(u,\tau)\right\|_{L^{2}}^{2}\lesssim C_0(1+t)^{-\frac 5 2}. 
\end{align*}
Thus we complete the proof of Proposition \ref{5prop3}.
  \end{proof}
 By virtue of the optimal decay rates and key integrability in Proposition \ref{5prop3}, we finally obtain the uniform vanishing damping limit and prove that the rate of uniform vanishing damping limit in $L^2$ is related to the time decay rate in $L^2$. 
  \begin{prop}\label{5prop4}
  Under the conditions in Proposition \ref{5prop3}, there holds 
  $$\|(u^a,\tau^a)-(u^0,\tau^0)\|_{L^{\infty}(0,\infty;L^2)}\leq Ca^{\frac 3 4},
  $$ 
  and 
   $$\lim_{a\rightarrow0}\|(u^a,\tau^a)-(u^0,\tau^0)\|_{L^{\infty}(0,\infty;\dot{H}^1\cap B^1_{\infty,1})\times L^{\infty}(0,\infty;\dot{H}^1\cap B^0_{\infty,1})}=0.
  $$
    \end{prop}
 \begin{proof}
By direct calculation, we obtain
\begin{align*}
    \frac{d}{dt}\|(u^a-u^0,\tau^a-\tau^0)\|^2_{L^2}
\lesssim&\|(u^a-u^0,\tau^a-\tau^0)\|^2_{L^2}(\|\nabla (u^0,u^a)\|_{L^\infty}+\|\tau^0\|^2_{L^\infty}+\|\nabla\tau^0\|^2_{L^3}) \\
   &+a\|\tau^a\|_{L^2}\|\tau^a-\tau^0\|_{L^2}.
\end{align*}
Then, we deduce that
\begin{align*}
    &\frac{d}{dt}\|(u^a-u^0,\tau^a-\tau^0)\|_{L^2}  \\
    \lesssim & a(1+t)^{-\frac{3}{4}} +  \|(u^a-u^0,\tau^a-\tau^0)\|_{L^2}\left(\|\nabla (u^0,u^a)\|_{L^\infty}+\|\tau^0\|^2_{L^\infty}+\|\nabla\tau^0\|^2_{L^3}\right).
\end{align*}
We infer from Propositions \ref{5prop1} and \ref{5prop3} that
\begin{align*}
    \int_0^\infty \|\nabla (u^0,u^a)\|_{L^\infty}+\|\tau^0\|^2_{L^\infty}+\|\nabla\tau^0\|^2_{L^3} dt \lesssim 1.
\end{align*}
Using Gronwall's inequality, we have
\begin{align*}
    \|(u^a-u^0,\tau^a-\tau^0)\|_{L^2} \lesssim a(1+t)^{\frac{1}{4}}.
\end{align*}
If $at\leq 1$, we get
\begin{equation*}
\|(u^a-u^0,\tau^a-\tau^0)\|_{L^2}\lesssim a^{\frac{3}{4}}.
\end{equation*}
If $at>1$, by Proposition \ref{5prop3}, we obtain
\begin{equation*}
\|(u^a-u^0,\tau^a-\tau^0)\|_{L^2}\leq \|(u^{a},\tau^{a})\|_{L^{2}}+\|(u^{0},\tau^{0})\|_{L^{2}}\lesssim (1+t)^{-\frac{3}{4}}\leq a^{\frac{3}{4}}.
\end{equation*}
For the global high-order convergence, thanks to Propositions \ref{5prop2} and \ref{5prop3}, it is the same as the proof of Proposition \ref{4prop5}. Hence we finish the proof of Proposition \ref{5prop4}.
\end{proof}
In the following we show that the vanishing damping limit in Proposition \ref{5prop4} is sharp in $a$ in the following sense. 
\begin{prop}\label{3dsharp}
There exists initial data $(u_0, \tau_0)$ satisfying the conditions of Proposition \ref{5prop4} such that the following estimates hold:
$$\|(u^a,\tau^a)-(u^0,\tau^0)\|_{L^{\infty}(0,\infty;L^2)}\gtrsim a^{\frac 3 4}.
  $$ 
\end{prop}
\begin{proof}
Use the exact isotropic solution \eqref{whole-space-isotropic-exact}
with $d=3$. The same fixed Schwartz initial stress lies in
$\dot B^{-3/2}_{2,\infty}$; at $t=a^{-1}$ the computation in
Proposition \ref{2dsharp} gives
$\|\tau^a-\tau^0\|_2\geq c_3\ep a^{3/4}$.
\end{proof}
{\bf The proof of Theorem \ref{theo1} :}  \\
Combining Propositions \ref{4prop2}, \ref{4prop3}, \ref{5prop1} and \ref{5prop2}, we complete the proof of Theorem \ref{theo1}. 
\hfill$\Box$

{\bf The proof of Theorem \ref{theo2} :}  \\
Combining Propositions \ref{4prop4}, \ref{4prop5}, \ref{5prop3}, \ref{5prop4}, \ref{2dsharp} and \ref{3dsharp}, we finish the proof of Theorem \ref{theo2}. 
\hfill$\Box$

\section{Dynamics in the periodic domains $\T^d$}
In this section we consider Oldroyd-B models in the periodic domain
$\T^d=(\R/2\pi\Z)^d$ for $d=2,3$:
\begin{align}\label{7.1}
\left\{\begin{array}{l}
\partial_tu+u\cdot\nabla u+\nabla P ={\rm div}~\tau,~~~~{\rm div}~u=0,\\[1ex]
\partial_t\tau+u\cdot\nabla\tau+a\tau+Q(\nabla u,\tau)=2D(u)+\Delta\tau,\\[1ex]
u|_{t=0}=u_0,~~\tau|_{t=0}=\tau_0. \\[1ex]
\end{array}\right.
\end{align}
Here we choose the coupling coefficient $\alpha=2$ in order to agree with the
normalization used in the numerical models below.  The same arguments apply to
the normalization $\alpha=1$ in \eqref{eq0}; in that case the coefficient of the
good unknown $\Gamma$ below is $1/2$ instead of $1$.  We use the periodic
inhomogeneous Littlewood--Paley decomposition throughout this
section and write $B^s_{p,r}=B^s_{p,r}(\T^d)$ for simplicity.  The product,
commutator and transport estimates stated in the preliminary section remain valid
on $\T^d$.  The zero Fourier mode, which is contained in the lowest dyadic block,
will be treated separately below.  Exponential decay on the torus has also been
obtained for related incompressible Oldroyd-B systems in an analytic framework;
see, for example, \cite{CP25}.  Here we record the corresponding relaxation
mechanism for the present Sobolev--Besov setting.
 
\subsection{Well-posedness and vanishing damping limit results in the periodic case}
The spatial average of the velocity is conserved.  Moreover, the system is
Galilean invariant: if $c=\overline{u_0}$, then
\begin{equation}\label{7per-Galilean-normalization}
 \widetilde u(t,x)=u(t,x+ct)-c,\qquad
 \widetilde\tau(t,x)=\tau(t,x+ct),\qquad
 \widetilde P(t,x)=P(t,x+ct)
\end{equation}
solves the same system and satisfies $\overline{\widetilde u_0}=0$.  We shall
therefore work, without loss of generality, in the zero-momentum frame
\begin{equation}\label{7per-zero-mean-u0}
 \overline{u_0}=0.
\end{equation}
For data with nonzero velocity average, all the conclusions below follow
after \eqref{7per-Galilean-normalization}, with $u$ replaced by
$u-\overline{u_0}$ and with $\mathcal E_0$ computed from
$(u_0-\overline{u_0},\tau_0)$.  No analogous zero-mean normalization is imposed on
$\tau$: its spatial average obeys a forced ODE, damped when $a>0$, with
forcing generated by $Q(\nabla u,\tau)$, and therefore is not conserved.

Our main results for $\T^d$ with $d=2,3$ can be stated as follows.

\begin{theo}[\bf Local well-posedness]\label{thmlocsol}
Let $d=2,3$ and $a\in[0,1]$. Assume that $u_0\in H^1\cap
        B^1_{\infty,1}$ is divergence free with $\overline{u_0}=0$, and that
        $\tau_0\in H^1\cap
        B^0_{\infty,1}$ is symmetric. Set
        \begin{equation*}
        \mathcal E_0:=\|(u_0,\tau_0)\|_{H^1(\mathbb{T}^{d})}
        +\|(\nabla u_0,\tau_0)\|_{B^0_{\infty,1}(\mathbb{T}^{d})}.
        \end{equation*}
        There exist a lifespan $T=T(\mathcal E_{0})>0$, such that \eqref{7.1}
        admits a unique local solution $(u^a,\tau^a)$ satisfying
        \begin{align*}
        &u^a\in C([0,T];H^1\cap B^1_{\infty,1}),\\
        &\tau^a\in C([0,T];H^1\cap B^0_{\infty,1})
        \cap L^1(0,T;B^2_{\infty,1}),
        \end{align*}
        and
        \begin{align}\label{7per-global-boundlocal}
        &\sup_{0\leq t\leq T}\big(\|(u^a,\tau^a)(t)\|_{H^1}
        +\|(\nabla u^a,\tau^a)(t)\|_{B^0_{\infty,1}}\big)\notag\\
        &\quad+\int_{0}^{T}\|\tau^{a}\|_{B^{2}_{\infty,1}}dt+\left(\int_0^T
        \|\nabla\tau^a\|_{H^1}^2
        \,dt\right)^{\frac12}
        \lesssim \mathcal E_0.
        \end{align}
        Moreover, there exists a constant $C_T$, independent
        of $a$, such that
        \begin{equation}\label{7per-L2-limitlocal}
        \|(u^a,\tau^a)-(u^0,\tau^0)\|_{L^{\infty}(0,T;L^2)}
        \leq C_Ta.
        \end{equation}
\end{theo}
    \begin{theo}[\bf Global well-posedness]\label{thm1t}
        Let $d=2,3$ and $a\in[0,1]$. Assume that $u_0\in H^1\cap
        B^1_{\infty,1}$ is divergence free with $\overline{u_0}=0$, and that
        $\tau_0\in H^1\cap
        B^0_{\infty,1}$ is symmetric. Set
        \begin{equation*}
        \mathcal E_0:=\|(u_0,\tau_0)\|_{H^1(\mathbb{T}^{d})}
        +\|(\nabla u_0,\tau_0)\|_{B^0_{\infty,1}(\mathbb{T}^{d})}.
        \end{equation*}
        There exist constants $\ep_0>0$ and $C>0$, independent of
        $a\in[0,1]$, such that if $\mathcal E_0\leq\ep_0$, then \eqref{7.1}
        admits a unique global solution $(u^a,\tau^a)$ satisfying
        \begin{align*}
        &u^a\in C([0,\infty);H^1\cap B^1_{\infty,1}),\\
        &\tau^a\in C([0,\infty);H^1\cap B^0_{\infty,1})
        \cap L^1_{\rm loc}(0,\infty;B^2_{\infty,1}),
        \end{align*}
        and
        \begin{align}\label{7per-global-bound}
        &\sup_{t\geq0}\big(\|(u^a,\tau^a)(t)\|_{H^1}
        +\|(\nabla u^a,\tau^a)(t)\|_{B^0_{\infty,1}}\big)\notag\\
        &\quad+\left(\int_0^\infty
        \big(\|\nabla u^a\|_{L^2}^2+\|\nabla\tau^a\|_{H^1}^2
        +a\|\tau^a\|_{H^1}^2\big)\,dt\right)^{\frac12}
        \leq C\mathcal E_0.
        \end{align}
        Moreover, for every $T>0$, there exists a constant $C_T$, independent
        of $a$, such that
        \begin{equation}\label{7per-L2-limit}
        \|(u^a,\tau^a)-(u^0,\tau^0)\|_{L^{\infty}(0,T;L^2)}
        \leq C_Ta,
        \end{equation}
        and
        \begin{equation}\label{7per-strong-limit}
        \lim_{a\rightarrow0}\|(u^a,\tau^a)-(u^0,\tau^0)\|_{
        L^{\infty}(0,T;\dot H^1\cap B^1_{\infty,1})
        \times L^{\infty}(0,T;\dot H^1\cap B^0_{\infty,1})}=0.
        \end{equation}
  \end{theo}

We only present the details of the proof of Theorem \ref{thm1t} since the proof of Theorem \ref{thmlocsol} follows from the same fashion.

\begin{proof}[Proof of Theorem \ref{thm1t}]
We divide the proof into four steps.  All constants below are uniform for
$a\in[0,1]$.

\medskip
\noindent\textbf{Step 1. Local theory and the continuation criterion.}
Let $(\Delta_j)_{j\geq-1}$ be the inhomogeneous periodic dyadic blocks.  The
periodic versions of Lemmas~\ref{CI}, \ref{PL}, \ref{H}, \ref{CZ},
\ref{TP}, \ref{P} and Corollary~\ref{CR3} follow either by transference or
directly by summing over $k\in\mathbb Z^d$.  In particular, their constants
do not depend on $a$.  We normalize the pressure by $\overline P=0$; all
inverse Laplacians below act on nonzero Fourier modes.
Applying the Fourier truncation scheme of Proposition~\ref{3prop1}, with
$D(u)$ there replaced by $2D(u)$, produces a solution on $[0,T_0]$, where
$T_0$ depends only on the initial norm appearing in the theorem, and
\begin{align*}
 u&\in C([0,T_0];H^1\cap B^1_{\infty,1}),\\
 \tau&\in C([0,T_0];H^1\cap B^0_{\infty,1})
 \cap L^1(0,T_0;B^2_{\infty,1}).
\end{align*}
The same transport and parabolic estimates applied to the difference of two
approximations give convergence, uniqueness and continuous dependence.  The
corresponding blow-up alternative says that the solution can be continued
past $T<\infty$ provided
\begin{equation}\label{7per-continuation}
 \sup_{0\leq t<T}\big(\|(u,\tau)(t)\|_{H^1}
 +\|(\nabla u,\tau)(t)\|_{B^0_{\infty,1}}\big)<\infty.
\end{equation}
Indeed, on a finite interval the bound in \eqref{7per-continuation} inserted
in the periodic parabolic estimate gives
$\tau\in L^1(0,T;B^2_{\infty,1})$, so the local construction can be
restarted with a lifespan depending only on the displayed bound.
This also gives continuous dependence on the initial data on every interval
on which the quantity in \eqref{7per-continuation} is bounded.

\medskip
\noindent\textbf{Step 2. Uniform $H^1$ estimate and the zero Fourier mode.}
For a periodic function $f$, write
\begin{equation*}
 \overline f=\frac1{|\T^d|}\int_{\T^d}f(x)\,dx,
\qquad
 m=\overline\tau,\qquad \sigma=\tau-m.
\end{equation*}
Taking spatial averages in \eqref{7.1} gives
\begin{equation}\label{7per-mean-equations}
 \overline u(t)=\overline {u_0}=0,\qquad
 m'(t)+am(t)=-\overline{Q(\nabla u,\sigma)}.
\end{equation}
Indeed, $\overline{D(u)}=\overline{\Omega(u)}=0$, and hence
$\overline{Q(\nabla u,m)}=0$.  In particular, $u$ and $\sigma$ have zero
mean, while $\nabla\sigma=\nabla\tau$.
The velocity normalization now permits the direct Poincar\'e estimate
\begin{equation}\label{7per-u-poincare}
 \|u\|_{H^1}\leq C\|\nabla u\|_{L^2}.
\end{equation}
The stress requires a separate treatment because only $\sigma$, rather than
$\tau$, has zero mean.

Fix a small number $\delta>0$ and make the bootstrap assumption
\begin{equation}\label{7per-bootstrap}
 \sup_{0\leq t\leq T}\big(\|(u,\tau)(t)\|_{H^1}
 +\|(\nabla u,\tau)(t)\|_{B^0_{\infty,1}}\big)\leq\delta.
\end{equation}
Testing the velocity equation by $u$ and the stress equation by
$\frac12\tau$ cancels the linear coupling terms.  The same cancellation holds
after testing by $-\Delta u$ and $-\frac12\Delta\tau$, respectively. More precisely, if
\begin{equation*}
 \mathscr E_{\rm b}
 =\frac12\|u\|_{H^1}^2+\frac14\|\tau\|_{H^1}^2,
\end{equation*}
then the two coupling terms cancel at both derivative levels and
\begin{align}\label{7per-basic-H1-identity}
 \frac{d}{dt}\mathscr E_{\rm b}
 +\frac12\|\nabla\tau\|_{H^1}^2+\frac a2\|\tau\|_{H^1}^2
={}&-\frac12\langle Q(\nabla u,\tau),\tau\rangle
 +\langle u\cdot\nabla u,\Delta u\rangle\notag\\
 &+\frac12\langle u\cdot\nabla\tau,\Delta\tau\rangle
 +\frac12\langle Q(\nabla u,\tau),\Delta\tau\rangle .
\end{align} 
Since $\sigma$ has zero mean, Poincar\'e's
inequality and $d\leq3$ give
\begin{equation}\label{7per-sigma-poincare}
 \|\sigma\|_{H^1}+\|\sigma\|_{L^4}
 \leq C\|\nabla\sigma\|_{L^2}.
\end{equation}
The part of $Q(\nabla u,m)$ paired with $m$ has zero integral.  Expanding all
other terms and using \eqref{7per-sigma-poincare}, we find
\begin{align}\label{7per-Q-zero-mode}
 \left|\int_{\T^d}Q(\nabla u,m+\sigma):(m+\sigma)\,dx\right|
 &\leq C\big(|m|+\|\sigma\|_{H^1}\big)
       \|\nabla u\|_{L^2}\|\nabla\sigma\|_{L^2}\notag\\
 &\leq C\delta\big(\|\nabla u\|_{L^2}^2
       +\|\nabla\sigma\|_{L^2}^2\big).
\end{align}
At the differentiated level we use, without differentiating $Q$,
\begin{equation}\label{7per-Q-H1}
 |\langle Q(\nabla u,\tau),\Delta\tau\rangle|
 \leq C\|\tau\|_{L^\infty}\|\nabla u\|_{L^2}
       \|\Delta\tau\|_{L^2}.
\end{equation}
The differentiated transport terms are bounded by
$C\|\nabla u\|_{L^\infty}
(\|\nabla u\|_{L^2}^2+\|\nabla\tau\|_{L^2}^2)$.
Here, in dimension two, one may also use
$\langle u\cdot\nabla u,\Delta u\rangle=0$; in dimension three its absolute
value is at most
$C\|\nabla u\|_{B^0_{\infty,1}}\|\nabla u\|_{L^2}^2$.
Equations \eqref{7per-Q-zero-mode}--\eqref{7per-Q-H1}, the transport
estimate above, and the bootstrap assumption therefore imply
\begin{equation}\label{7per-basic-H1-bound}
 \left|\text{the right-hand side of
 \eqref{7per-basic-H1-identity}}\right|
 \leq C\delta\big(\|\nabla u\|_{L^2}^2
 +\|\nabla\sigma\|_{H^1}^2\big).
\end{equation}

To recover the velocity dissipation, choose $0<\eta\ll1$ and introduce
\begin{equation*}
 \mathscr E_1(t)=\frac12\|u(t)\|_{H^1}^2
 +\frac14\|\tau(t)\|_{H^1}^2
 -\eta\langle\sigma(t),\nabla u(t)\rangle.
\end{equation*}
Since $\int_{\T^d}m:\nabla u\,dx=0$, the mixed term is also
$-\eta\langle\tau,\nabla u\rangle$.  Let $\mathbb P$ be the periodic Leray
projector, with its zero mode fixed in the usual way.  Since
$u_t=\mathbb P(\operatorname{div}\tau-u\cdot\nabla u)$, direct
differentiation gives the exact identity
\begin{align}\label{7per-mixed-identity}
 &\frac{d}{dt}\big(-\eta\langle\tau,\nabla u\rangle\big)
 +\eta\|\nabla u\|_{L^2}^2\notag\\
 &\quad=\eta\langle\mathbb P(\operatorname{div}\sigma-u\cdot\nabla u),
              \operatorname{div}\sigma\rangle\notag\\
 &\qquad\quad+\eta\langle u\cdot\nabla\sigma+a\sigma
 +Q(\nabla u,m+\sigma)-\Delta\sigma,\nabla u\rangle .
\end{align}
By
\eqref{7per-bootstrap}, Poincar\'e's inequality and $a\leq1$,
\begin{align}
 &\eta\left|\langle\mathbb P(\operatorname{div}\sigma-u\cdot\nabla u),
              \operatorname{div}\sigma\rangle\right|\notag\\
 &\qquad\leq C\eta\|\nabla\sigma\|_{L^2}^2
 +C\eta\delta\big(\|\nabla u\|_{L^2}^2
 +\|\nabla\sigma\|_{L^2}^2\big),\label{7per-mixed-first}\\
 &\eta\left|\langle u\cdot\nabla\sigma+a\sigma
 +Q(\nabla u,m+\sigma)-\Delta\sigma,\nabla u\rangle\right|\notag\\
 &\qquad\leq \frac{\eta}{4}\|\nabla u\|_{L^2}^2
 +C\eta\|\nabla\sigma\|_{H^1}^2+\frac a8\|\tau\|_{L^2}^2
 +C(\eta\delta+\eta^2a)\|\nabla u\|_{L^2}^2 .
 \label{7per-mixed-second}
\end{align}
Here we used $|m|\leq \|\tau\|_{L^\infty}$ and $\|u\|_{L^\infty}
+\|\tau\|_{L^\infty}\leq C\delta$, which follows from
\eqref{7per-bootstrap}, the velocity mean bound, and
$B^0_{\infty,1}\hookrightarrow L^\infty$.  Combining
\eqref{7per-basic-H1-identity}--\eqref{7per-mixed-second}, and then choosing
first $\eta$ and next $\delta$ small, we obtain
\begin{align}\label{7per-energy}
 \frac{d}{dt}\mathscr E_1(t)
 +c\big(\|\nabla u\|_{L^2}^2+\|\nabla\tau\|_{H^1}^2
 +a\|\tau\|_{H^1}^2\big)\leq0.
\end{align}
Moreover, $\mathscr E_1\simeq\|(u,\tau)\|_{H^1}^2$. Then we have
\begin{align}\label{7per-H1-bound}
 &\sup_{0\leq t\leq T}\|(u,\tau)(t)\|_{H^1}
 +\left(\int_0^T\big(\|\nabla u\|_{L^2}^2
 +\|\nabla\tau\|_{H^1}^2+a\|\tau\|_{H^1}^2\big)dt\right)^{1/2}
 \leq C\mathcal E_0.
\end{align}

\medskip
\noindent\textbf{Step 3. Uniform critical estimate and global continuation.}
We first make the periodic heat estimate explicit.  The lowest dyadic block
contains only finitely many Fourier modes, so
\begin{equation}\label{7per-low-block}
 \|\Delta_{-1}\tau(t)\|_{L^\infty}\leq C\|\tau(t)\|_{L^2}\leq C\mathcal E_0.
\end{equation}
For $j\geq0$, Duhamel's formula and the spectral gap on the support of
$\Delta_j$ give
\begin{align}\label{7per-Duhamel-block}
 \|\Delta_j\tau(t)\|_{L^\infty}
 &\leq Ce^{-c2^{2j}t}\|\Delta_j\tau_0\|_{L^\infty}\notag\\
 &\quad+C\int_0^t e^{-c2^{2j}(t-s)}
 \|\Delta_j(2D(u)-u\cdot\nabla\tau-Q(\nabla u,\tau))(s)\|_{L^\infty}\,ds.
\end{align}
In two dimensions, use $H^1(\T^2)\hookrightarrow L^8(\T^2)$ and Bernstein's
inequality yields
\begin{align}\label{7per-heat-2d}
 \|\Delta_j(2D(u)-u\cdot\nabla\tau-Q)(s)\|_{L^\infty}
 \leq C2^{\frac32j}\big(&\|\nabla u\|_{L^2}
 +\|u\|_{L^8}\|\tau\|_{L^8}
 +\|\nabla u\|_{L^2}\|\tau\|_{L^4}\big)(s).
\end{align}
In three dimensions, H\"older's and Bernstein's inequalities with
$p=12/7$ instead give
\begin{align}\label{7per-heat-3d}
 \|\Delta_j(2D(u)-u\cdot\nabla\tau-Q)(s)\|_{L^\infty}
 \leq C2^{\frac74j}\big(&\|\nabla u\|_{L^2}
 +\|u\|_{L^{12}}\|\nabla\tau\|_{L^2}
 +\|\nabla u\|_{L^2}\|\tau\|_{L^{12}}\big)(s).
\end{align}
Under \eqref{7per-bootstrap}, the quantities in parentheses in
\eqref{7per-heat-2d} and \eqref{7per-heat-3d} are bounded by
$C\mathcal E_0$.  Since
\begin{equation*}
 \int_0^t e^{-c2^{2j}(t-s)}ds\leq C2^{-2j},
\end{equation*}
the high-frequency remainders are bounded respectively by
$C2^{-j/2}\mathcal E_0$ and $C2^{-j/4}\mathcal E_0$.  Both sequences are
summable.  Together with \eqref{7per-low-block}, this proves
\begin{equation}\label{7per-tau-B0}
 \sup_{0\leq t\leq T}\|\tau(t)\|_{B^0_{\infty,1}}\leq C\mathcal E_0.
\end{equation}

Define on periodic distributions
\begin{equation*}
 \mathcal R=-(-\Delta)^{-1}{\rm curl}\,{\rm div},\qquad
 \widehat{\mathcal Rf}(0)=0,\qquad
 \Gamma=\omega-\mathcal R\tau.
\end{equation*}
For every nonzero Fourier mode, $\mathcal R(2D(u))=\omega$.  In dimension
two, $\Gamma$ satisfies
\begin{align}\label{7per-Gamma-2d}
 \partial_t\Gamma+u\cdot\nabla\Gamma+\Gamma
 =(a-1)\mathcal R\tau+\mathcal RQ(\nabla u,\tau)
 +[\mathcal R,u\cdot\nabla]\tau.
\end{align}
In dimension three, it satisfies instead
\begin{align}\label{7per-Gamma-3d}
 \partial_t\Gamma+u\cdot\nabla\Gamma+\Gamma
 =(a-1)\mathcal R\tau+\mathcal RQ(\nabla u,\tau)
 +[\mathcal R,u\cdot\nabla]\tau+\omega\cdot\nabla u.
\end{align}
Thus there is no dimension-dependent indicator hidden in the notation.  The
coefficient of $\Gamma$ is $1$ because the stress equation contains
$2D(u)$.

The periodic multiplier and commutator bounds take the form
\begin{align}\label{7per-R-commutator}
 \|\mathcal Rf\|_{B^0_{\infty,1}}
 &\leq C\big(\|f\|_{B^0_{\infty,1}}+\|f\|_{L^2}\big),\notag\\
 \|[\mathcal R,u\cdot\nabla]\tau\|_{B^0_{\infty,1}}
 &\leq C\big(\|\nabla u\|_{L^\infty}+\|u\|_{L^2}\big)
 \big(\|\tau\|_{B^0_{\infty,1}}+\|\tau\|_{L^2}\big).
\end{align}
We spell out the nonlinear estimates needed here.  Fix
$0<\varepsilon<1$ if $d=2$ and $0<\varepsilon<1/2$ if $d=3$.  Periodic
Sobolev embedding and the dyadic product law give
\begin{align}
 \|\tau\|_{B^\varepsilon_{\infty,\infty}}
 &\leq C\big(\|\tau\|_{H^1}+\|\nabla^2\tau\|_{L^2}\big),
 \label{7per-positive-Besov}\\
 \|\mathcal RQ(\nabla u,\tau)\|_{B^0_{\infty,1}}
 &\leq C\big(\|Q(\nabla u,\tau)\|_{L^2}
             +\|Q(\nabla u,\tau)\|_{B^0_{\infty,1}}\big)\notag\\
 &\leq C\|\nabla u\|_{L^2}\|\tau\|_{L^\infty}
 +C\|\nabla u\|_{L^\infty}\|\tau\|_{B^0_{\infty,1}}\notag\\
 &\quad+C\|\nabla u\|_{B^0_{\infty,1}}
       \|\tau\|_{B^\varepsilon_{\infty,\infty}}\notag\\
 &\leq C\mathcal E_0^2+C\delta\mathcal E_0
       +C\delta\|\nabla^2\tau\|_{L^2}.
 \label{7per-RQ-bound}
\end{align}
The last inequality uses \eqref{7per-H1-bound},
\eqref{7per-tau-B0} and the bootstrap bound.  It includes the contribution
$Q(\nabla u,m)$: although $\mathcal Rm=0$, this product need not vanish, and
it is controlled by $|m|\|\nabla u\|_{B^0_{\infty,1}}$ in
\eqref{7per-RQ-bound}.  Likewise, \eqref{7per-R-commutator} yields
\begin{equation}\label{7per-commutator-bound}
 \|[\mathcal R,u\cdot\nabla]\tau\|_{B^0_{\infty,1}}
 \leq C(\delta+\mathcal E_0)\mathcal E_0.
\end{equation}

In dimension three we use $\operatorname{div}\omega=0$ to write
$\omega\cdot\nabla u=\operatorname{div}(\omega\otimes u)$.  The periodic
Bony decomposition, with the derivative placed on the $B^1_{\infty,1}$
factor, gives
\begin{align}\label{7per-vortex-stretching}
 \|\omega\cdot\nabla u\|_{B^0_{\infty,1}}
 &\leq C\|\omega\|_{B^0_{\infty,1}}\|u\|_{B^1_{\infty,1}}\notag\\
 &\leq C\delta\|\Gamma\|_{B^0_{\infty,1}}
       +C\delta\mathcal E_0.
\end{align}
Here we used $\omega=\Gamma+\mathcal R\tau$ and
$\|u\|_{B^1_{\infty,1}}\leq
C(\|u\|_{L^2}+\|\nabla u\|_{B^0_{\infty,1}})$.
Therefore, if $F_2$ and $F_3$ denote the complete right-hand sides of
\eqref{7per-Gamma-2d} and \eqref{7per-Gamma-3d}, respectively, then
\begin{align}\label{7per-Gamma-forcing}
 \|F_2(t)\|_{B^0_{\infty,1}}
 &\leq C\mathcal E_0+C\delta\|\nabla^2\tau(t)\|_{L^2},\notag\\
 \|F_3(t)\|_{B^0_{\infty,1}}
 &\leq C\mathcal E_0+C\delta\|\nabla^2\tau(t)\|_{L^2}
 +C\delta\|\Gamma(t)\|_{B^0_{\infty,1}}.
\end{align}
The last term is present only in dimension three.  The damped transport
estimate, with $\delta$ small
enough that $C\|\nabla u\|_{B^0_{\infty,1}}\leq1/4$, gives
\begin{equation}\label{7per-Gamma-damped}
 \|\Gamma(t)\|_{B^0_{\infty,1}}
 \leq Ce^{-ct}\|\Gamma_0\|_{B^0_{\infty,1}}
 +C\int_0^t e^{-c(t-s)}\|F_d(s)\|_{B^0_{\infty,1}}\,ds.
\end{equation}
By Cauchy--Schwarz and \eqref{7per-H1-bound},
\begin{equation*}
 \sup_{t\leq T}\int_0^t e^{-c(t-s)}
 \|\nabla^2\tau(s)\|_{L^2}\,ds\leq C\mathcal E_0.
\end{equation*}
Consequently, after absorbing the last term of
\eqref{7per-Gamma-forcing} in dimension three, we obtain
\begin{equation}\label{7per-Gamma-bound}
 \sup_{0\leq t\leq T}\|\Gamma(t)\|_{B^0_{\infty,1}}
 \leq C\mathcal E_0.
\end{equation}
Finally, the zero mode of $\nabla u$ vanishes, and the periodic Biot--Savart
estimate, \eqref{7per-tau-B0} and \eqref{7per-Gamma-bound} imply
\begin{equation}\label{7per-grad-u-B0}
 \sup_{0\leq t\leq T}\|\nabla u(t)\|_{B^0_{\infty,1}}
 \leq C\sup_{0\leq t\leq T}
 \big(\|\Gamma(t)\|_{B^0_{\infty,1}}
 +\|\tau(t)\|_{B^0_{\infty,1}}\big)
 \leq C\mathcal E_0.
\end{equation}
Together with
$\|u\|_{B^1_{\infty,1}}\leq
C(\|u\|_{L^2}+\|\nabla u\|_{B^0_{\infty,1}})$, this also gives the full
$B^1_{\infty,1}$ bound for the velocity.
Choose $\mathcal E_0$ so small that $C\mathcal E_0\leq\delta/2$.
Equations \eqref{7per-H1-bound}, \eqref{7per-tau-B0} and
\eqref{7per-grad-u-B0} improve the bootstrap assumption.  A standard
continuity argument and \eqref{7per-continuation} therefore give the unique
global solution and \eqref{7per-global-bound}.  Applying the periodic
parabolic estimate to the stress equation on each finite interval also gives
$\tau\in L^1_{\rm loc}(0,\infty;B^2_{\infty,1})$.

\medskip
\noindent\textbf{Step 4. Vanishing damping limit.}
Put $U=u^a-u^0$, $\Sigma=\tau^a-\tau^0$ and
$\Pi=P^a-P^0$.  The difference satisfies
\begin{equation}\label{7per-difference-system}
\left\{
\begin{aligned}
 &\partial_tU+u^a\cdot\nabla U+U\cdot\nabla u^0+\nabla\Pi
   =\operatorname{div}\Sigma,\qquad \operatorname{div}U=0,\\
 &\partial_t\Sigma+u^a\cdot\nabla\Sigma+U\cdot\nabla\tau^0
   +a\tau^a+Q(\nabla u^a,\Sigma)+Q(\nabla U,\tau^0)\\
 &\hspace{7cm}=2D(U)+\Delta\Sigma .
\end{aligned}\right.
\end{equation}
The terms $Q(\nabla U,\tau^0)$ and
$U\cdot\nabla\tau^0$ require some care because the velocity difference has
no direct diffusion.  Using the tensor structure of $Q$ and integrating the
derivative of $U$ by parts, we obtain
\begin{align}\label{7per-difference-hard-terms}
 &|\langle Q(\nabla U,\tau^0),\Sigma\rangle|
  +|\langle U\cdot\nabla\tau^0,\Sigma\rangle|\notag\\
 &\quad\leq C\|U\|_{L^2}\big(
 \|\nabla\tau^0\|_{L^{r_d}}\|\Sigma\|_{L^{q_d}}
 +\|\tau^0\|_{L^\infty}\|\nabla\Sigma\|_{L^2}\big)\notag\\
 &\quad\leq \frac18\|\nabla\Sigma\|_{L^2}^2
 +C A_d(t)\big(\|U\|_{L^2}^2+\|\Sigma\|_{L^2}^2\big),
\end{align}
where
\begin{equation*}
 (r_2,q_2,\rho_2)=(4,4,4/3),\qquad
 (r_3,q_3,\rho_3)=(3,6,2),
\end{equation*}
and
\begin{equation*}
 A_d(t)=1+\|\nabla u^a(t)\|_{L^\infty}
 +\|\nabla u^0(t)\|_{L^\infty}
 +\|\tau^0(t)\|_{L^\infty}^2
 +\|\nabla\tau^0(t)\|_{L^{r_d}}^{\rho_d}.
\end{equation*}
The Gagliardo--Nirenberg inequalities and
\eqref{7per-global-bound} show that
\begin{equation}\label{7per-difference-coefficient}
 \int_0^T A_d(t)\,dt\leq C_T.
\end{equation}
For example,
\begin{align*}
 \|\nabla\tau^0\|_{L^4}^{4/3}
 &\lesssim \|\nabla\tau^0\|_{L^2}^{2/3}
 \|\nabla^2\tau^0\|_{L^2}^{2/3}\qquad &&(d=2),\\
 \|\nabla\tau^0\|_{L^3}^2
 &\lesssim\|\nabla\tau^0\|_{L^2}
 \|\nabla\tau^0\|_{L^6}\qquad &&(d=3).
\end{align*}
The remaining nonlinear terms in \eqref{7per-difference-system} satisfy
\begin{align*}
 |\langle U\cdot\nabla u^0,U\rangle|
 +|\langle Q(\nabla u^a,\Sigma),\Sigma\rangle|
 \leq C\big(\|\nabla u^0\|_{L^\infty}
 +\|\nabla u^a\|_{L^\infty}\big)
 \big(\|U\|_{L^2}^2+\|\Sigma\|_{L^2}^2\big).
\end{align*}
Testing the first equation in \eqref{7per-difference-system} by $U$ and the
second one by $\frac12\Sigma$ cancels the two linear coupling terms.
Together with
$a|\langle\tau^a,\Sigma\rangle|
\leq C a^2\|\tau^a\|_{L^2}^2+C\|\Sigma\|_{L^2}^2$, this yields
\begin{align}\label{7per-difference-energy}
 \frac{d}{dt}\left(\frac12\|U\|_{L^2}^2
 +\frac14\|\Sigma\|_{L^2}^2\right)
 +\frac14\|\nabla\Sigma\|_{L^2}^2
 \leq C A_d(t)\big(\|U\|_{L^2}^2+\|\Sigma\|_{L^2}^2\big)
 +C a^2\|\tau^a\|_{L^2}^2.
\end{align}
Since $(U,\Sigma)|_{t=0}=0$, \eqref{7per-difference-coefficient}, the
uniform $L^2$ bound and Gronwall's inequality prove
\eqref{7per-L2-limit}.

We finally give the Bona--Smith argument at the endpoint.  Let $S_N$ be the
periodic Fourier cut-off to $|k|\lesssim2^N$, let
$z^a=(u^a,\tau^a)$, and let $z_N^a=(u_N^a,\tau_N^a)$ be the solution with
initial data $(S_Nu_0,S_N\tau_0)$.  Set
\begin{equation*}
 W_N^a=u^a-u_N^a,\qquad \Theta_N^a=\tau^a-\tau_N^a,
\end{equation*}
and
\begin{align*}
 \varepsilon_N={}&
 \|(I-S_N)u_0\|_{H^1\cap B^1_{\infty,1}}
 +\|(I-S_N)\tau_0\|_{H^1\cap B^0_{\infty,1}}.
\end{align*}
Then $\varepsilon_N\to0$.  The high-frequency localization of the initial
tails gives
\begin{equation}\label{7per-Bona-Smith-initial-lower}
 \|(W_N^a,\Theta_N^a)(0)\|_{L^2}
 +\|W_N^a(0)\|_{B^0_{\infty,1}}
 +\|\Theta_N^a(0)\|_{B^{-1}_{\infty,1}}
 \leq C2^{-N}\varepsilon_N .
\end{equation}
The $L^2$ difference estimate used above, now with the nonzero initial
difference in \eqref{7per-Bona-Smith-initial-lower}, gives
\begin{equation}\label{7per-Bona-Smith-L2-dissipation}
 \sup_{a\in[0,1]}\left(
 \|(W_N^a,\Theta_N^a)\|_{L^\infty_TL^2}
 +\|\nabla\Theta_N^a\|_{L^2_TL^2}\right)
 \leq C_T2^{-N}\varepsilon_N.
\end{equation}

We next prove, rather than assume, the stability estimate one Besov index
below the endpoint.  We begin with the uniform maximal-regularity bound
needed in that argument.  Write
\begin{equation*}
 \partial_t\tau_N^a-\Delta\tau_N^a+a\tau_N^a=F_N^a,\qquad
 F_N^a=2D(u_N^a)-u_N^a\cdot\nabla\tau_N^a
       -Q(\nabla u_N^a,\tau_N^a).
\end{equation*}
Let $I=[t_0,t_1]\subset[0,T]$.  For every $j\geq0$, the blockwise heat
estimate, with $e^{-a(t-s)}\leq1$, gives
\begin{align}\label{7per-Bona-Smith-MR-high-block}
 &\|\Delta_j\tau_N^a\|_{L^\infty(I;L^\infty)}
 +2^{2j}\|\Delta_j\tau_N^a\|_{L^1(I;L^\infty)}\notag\\
 &\qquad\leq C\left(\|\Delta_j\tau_N^a(t_0)\|_{L^\infty}
 +\|\Delta_jF_N^a\|_{L^1(I;L^\infty)}\right),
\end{align}
where $C$ is independent of $a$ and $N$.  The zero Fourier mode contained
in $\Delta_{-1}$ has no heat smoothing.  Since its $B^2_{\infty,1}$ weight
is only the fixed number $2^{-2}$, it is instead estimated directly by
\begin{align}\label{7per-Bona-Smith-MR-low-block}
 &\|\Delta_{-1}\tau_N^a\|_{L^\infty(I;L^\infty)}
 +2^{-2}\|\Delta_{-1}\tau_N^a\|_{L^1(I;L^\infty)}\notag\\
 &\qquad\leq C(1+|I|)\left(
 \|\Delta_{-1}\tau_N^a(t_0)\|_{L^\infty}
 +\|\Delta_{-1}F_N^a\|_{L^1(I;L^\infty)}\right).
\end{align}
Thus the zero mode changes only the harmless factor $1+|I|$ and introduces
no factor depending on $a^{-1}$.

Set
\begin{equation*}
 X_I=\|\tau_N^a\|_{\widetilde L^\infty(I;B^0_{\infty,1})}
     +\|\tau_N^a\|_{L^1(I;B^2_{\infty,1})}.
\end{equation*}
The transport structure in Lemma~\ref{PL} and the product estimate there,
used with positive exponent $1/2$, imply
\begin{align}\label{7per-Bona-Smith-MR-forcing}
 \|F_N^a\|_{B^0_{\infty,1}}
 \leq{}&C\|u_N^a\|_{B^1_{\infty,1}}
 +C\|u_N^a\|_{B^1_{\infty,1}}
 \big(\|\tau_N^a\|_{B^0_{\infty,1}}
     +\|\tau_N^a\|_{B^{1/2}_{\infty,1}}
     +\|\tau_N^a\|_{B^1_{\infty,1}}\big).
\end{align}
Here $u_N^a\cdot\nabla\tau_N^a
=\operatorname{div}(u_N^a\otimes\tau_N^a)$ is estimated by the
transport inequality, rather than by treating two $B^0_{\infty,1}$ factors
as an algebra.  The two intermediate stress norms are controlled on $I$ by
the explicit interpolation inequalities
\begin{align}
 \|\tau_N^a\|_{L^1(I;B^1_{\infty,1})}
 &\leq |I|^{1/2}
 \|\tau_N^a\|_{\widetilde L^\infty(I;B^0_{\infty,1})}^{1/2}
 \|\tau_N^a\|_{L^1(I;B^2_{\infty,1})}^{1/2}
 \leq C|I|^{1/2}X_I,
 \label{7per-Bona-Smith-MR-interp-one}\\
 \|\tau_N^a\|_{L^1(I;B^{1/2}_{\infty,1})}
 &\leq |I|^{3/4}
 \|\tau_N^a\|_{\widetilde L^\infty(I;B^0_{\infty,1})}^{3/4}
 \|\tau_N^a\|_{L^1(I;B^2_{\infty,1})}^{1/4}
 \leq C|I|^{3/4}X_I.
 \label{7per-Bona-Smith-MR-interp-half}
\end{align}
Summing \eqref{7per-Bona-Smith-MR-high-block} and
\eqref{7per-Bona-Smith-MR-low-block}, then using
\eqref{7per-Bona-Smith-MR-forcing}--
\eqref{7per-Bona-Smith-MR-interp-half} and the uniform critical estimate,
we obtain for $|I|\leq1$
\begin{align}\label{7per-Bona-Smith-MR-short}
 X_I\leq{}&C(1+|I|)\left(
 \|\tau_N^a(t_0)\|_{B^0_{\infty,1}}
 +\int_I\|u_N^a(s)\|_{B^1_{\infty,1}}\,ds\right)\notag\\
 &+C\delta(1+|I|)
 \big(|I|+|I|^{1/2}+|I|^{3/4}\big)X_I.
\end{align}
Choose a fixed $h_0\in(0,1]$ so that the last coefficient is at most
$1/2$ whenever $|I|\leq h_0$.  This choice is independent of $a$ and $N$.
After absorption, iterate \eqref{7per-Bona-Smith-MR-short} over at most
$1+T/h_0$ intervals.  Since
$\sup_{a,N}\|u_N^a\|_{L^\infty_TB^1_{\infty,1}}\leq C\mathcal E_0$,
this proves, without assuming the parabolic norm in advance,
\begin{equation}\label{7per-Bona-Smith-uniform-MR}
 \sup_{a\in[0,1],\,N\geq1}\left(
 \|\tau_N^a\|_{\widetilde L^\infty_TB^0_{\infty,1}}
 +\|\tau_N^a\|_{L^1_TB^2_{\infty,1}}\right)
 \leq C_T\mathcal E_0.
\end{equation}

Fix now $0<\eta<1$.  Real interpolation of the two norms in
\eqref{7per-Bona-Smith-uniform-MR} gives
\begin{equation}\label{7per-Bona-Smith-positive-time}
 \sup_{a\in[0,1],\,N\geq1}
 \|\tau_N^a\|_{L^{q_\eta}_TB^{1+\eta}_{\infty,1}}
 \leq C_T\mathcal E_0,
 \qquad q_\eta=\frac{2}{1+\eta}>1.
\end{equation}
In particular, for every interval $I\subset[0,T]$,
\begin{equation}\label{7per-Bona-Smith-uniform-integrability}
 \sup_{a,N}\|\tau_N^a\|_{L^1(I;B^{1+\eta}_{\infty,1})}
 \leq C_T\mathcal E_0 |I|^{(1-\eta)/2}.
\end{equation}

We spell out the endpoint which otherwise causes a high--high to low
ambiguity.  Bony's decomposition and $\eta>0$ give
\begin{align}
 \|u^a\cdot\nabla\Theta_N^a\|_{B^{-1}_{\infty,1}}
 +\|Q(\nabla u^a,\Theta_N^a)\|_{B^{-1}_{\infty,1}}
 &\leq C\|u^a\|_{B^1_{\infty,1}}
       \|\Theta_N^a\|_{B^1_{\infty,1}},
 \label{7per-Bona-Smith-endpoint-1}\\
 \|W_N^a\cdot\nabla\tau_N^a\|_{B^{-1}_{\infty,1}}
 +\|Q(\nabla W_N^a,\tau_N^a)\|_{B^{-1}_{\infty,1}}
 &\leq C_\eta\|W_N^a\|_{B^0_{\infty,1}}
       \|\tau_N^a\|_{B^{1+\eta}_{\infty,1}}.
 \label{7per-Bona-Smith-endpoint-2}
\end{align}

Let $I=[t_0,t_1]\subset[0,T]$ and set
\begin{align*}
 V_I&=\|W_N^a\|_{L^\infty(I;B^0_{\infty,1})},\\
 S_I&=\|\Theta_N^a\|_{\widetilde L^\infty(I;B^{-1}_{\infty,1})}
      +\|\Theta_N^a\|_{L^1(I;B^1_{\infty,1})}.
\end{align*}
For clarity, the velocity estimate at index zero does not use an algebra
property of $B^0_{\infty,1}$.  Since $\operatorname{div}W_N^a=0$, the term
$W_N^a\cdot\nabla u_N^a$ has the corresponding divergence form
$\operatorname{div}(W_N^a\otimes u_N^a)$ (with the tensor convention of the
transport term), and Lemma~\ref{PL} gives
\begin{equation}\label{7per-Bona-Smith-velocity-product}
 \|W_N^a\cdot\nabla u_N^a\|_{B^0_{\infty,1}}
 \leq C\|W_N^a\|_{B^0_{\infty,1}}
       \|u_N^a\|_{B^1_{\infty,1}}.
\end{equation}
Thus the required positive derivative is supplied by the
$B^1_{\infty,1}$ factor and retained through the divergence/transport
structure; one never estimates a product of two bare $B^0_{\infty,1}$
factors.  The periodic Leray projector and its pressure multiplier are
bounded at index zero, while
$\|\operatorname{div}\Theta_N^a\|_{B^0_{\infty,1}}
\leq C\|\Theta_N^a\|_{B^1_{\infty,1}}$.
Applying the transport estimate to this velocity equation and the heat
maximal-regularity estimate to the stress equation, and then using
\eqref{7per-Bona-Smith-endpoint-1}--
\eqref{7per-Bona-Smith-endpoint-2}, yields
\begin{align}
 V_I&\leq C_T\big(\|W_N^a(t_0)\|_{B^0_{\infty,1}}+S_I\big),
 \label{7per-Bona-Smith-short-velocity}\\
 S_I&\leq C\|\Theta_N^a(t_0)\|_{B^{-1}_{\infty,1}}
 +C\left(|I|+\|\tau_N^a\|_{L^1(I;B^{1+\eta}_{\infty,1})}\right)V_I
 +C\delta S_I .
 \label{7per-Bona-Smith-short-stress}
\end{align}
Here $2D(W_N^a)$ contributes $C|I|V_I$, the damping term has the
favorable sign, and the terms transported by $u^a$ are included in
$C\delta S_I$.  Choose first the universal smallness constant so that
$C\delta\leq1/4$.  By
\eqref{7per-Bona-Smith-uniform-integrability}, one can then choose a length
$0<h_T\leq\min\{1,h_0\}$, independent of $a$ and $N$, such that the coefficient of $V_I$ in
\eqref{7per-Bona-Smith-short-stress} is small whenever $|I|\leq h_T$.
Substitution into \eqref{7per-Bona-Smith-short-velocity} and absorption give
\begin{equation*}
 V_I+S_I\leq C_T\big(
 \|W_N^a(t_0)\|_{B^0_{\infty,1}}
 +\|\Theta_N^a(t_0)\|_{B^{-1}_{\infty,1}}\big).
\end{equation*}
Iterating over at most $1+T/h_T$ intervals and using
\eqref{7per-Bona-Smith-initial-lower} proves
\begin{align}\label{7per-Bona-Smith-lower}
 \sup_{a\in[0,1]}\big(&
 \|W_N^a\|_{L^\infty_TB^0_{\infty,1}}
 +\|\Theta_N^a\|_{L^\infty_TB^{-1}_{\infty,1}}
 +\|\Theta_N^a\|_{L^1_TB^1_{\infty,1}}\big)
 \leq C_T2^{-N}\varepsilon_N .
\end{align}
All constants are independent of $a$ and $N$.

At the $H^1$ level, the only apparently nonuniform factors are smooth norms
of $z_N^a$.  We first establish their precise $2^N$ growth.  Applying the
dyadic transport estimate at velocity index $2$ and the parabolic estimate
at stress index $1$ gives, with
\begin{equation*}
 Z_N^a(t)=\|u_N^a\|_{L^\infty(0,t;B^2_{\infty,1})}
 +\|\tau_N^a\|_{L^\infty(0,t;B^1_{\infty,1})}
 +\|\tau_N^a\|_{L^1(0,t;B^3_{\infty,1})},
\end{equation*}
the tame inequality
\begin{align}\label{7per-Bona-Smith-tame-inequality}
 Z_N^a(t)&\leq C\big(
 \|S_Nu_0\|_{B^2_{\infty,1}}
 +\|S_N\tau_0\|_{B^1_{\infty,1}}\big)\notag\\
 &\quad+C\int_0^t
 \big(1+\|u_N^a\|_{B^1_{\infty,1}}
       +\|\tau_N^a\|_{B^0_{\infty,1}}\big)
 \big(\|u_N^a\|_{B^2_{\infty,1}}
       +\|\tau_N^a\|_{B^1_{\infty,1}}\big)\,ds.
\end{align}
There is no hidden quadratic dependence on $Z_N^a$.  Indeed, the
commutators and the differentiated $Q$ term obey
\begin{align*}
 \sum_j2^{2j}\|[\Delta_j,u_N^a\cdot\nabla]u_N^a\|_{L^\infty}
 &\leq C\|u_N^a\|_{B^1_{\infty,1}}
          \|u_N^a\|_{B^2_{\infty,1}},\\
 \sum_j2^j\|[\Delta_j,u_N^a\cdot\nabla]\tau_N^a\|_{L^\infty}
 +\|Q(\nabla u_N^a,\tau_N^a)\|_{B^1_{\infty,1}}
 &\leq C\|u_N^a\|_{B^1_{\infty,1}}
          \|\tau_N^a\|_{B^1_{\infty,1}}
   +C\|\tau_N^a\|_{B^0_{\infty,1}}
          \|u_N^a\|_{B^2_{\infty,1}}.
\end{align*}
The pressure causes no additional loss.  Taking the divergence of the
velocity equation and using $\operatorname{div}u_N^a=0$ gives
\begin{equation*}
 -\Delta P_N^a=\operatorname{div}\operatorname{div}
 (u_N^a\otimes u_N^a-\tau_N^a),
\end{equation*}
and the periodic multiplier and tame product estimates imply
\begin{equation*}
 \|\nabla P_N^a\|_{B^2_{\infty,1}}
 \leq C\|\tau_N^a\|_{B^3_{\infty,1}}
 +C\|u_N^a\|_{B^1_{\infty,1}}
       \|u_N^a\|_{B^2_{\infty,1}}.
\end{equation*}
Thus the pressure also contains only the parabolic high norm or a controlled
critical factor times one velocity high norm.
In dimension three, the equivalent estimate for the good unknown contains
the vortex stretching and the multiplier commutator; their tame forms are
\begin{align*}
 \|\omega_N^a\cdot\nabla u_N^a\|_{B^1_{\infty,1}}
 &\leq C\|\omega_N^a\|_{B^0_{\infty,1}}
          \|u_N^a\|_{B^2_{\infty,1}}
   +C\|u_N^a\|_{B^1_{\infty,1}}
          \|\omega_N^a\|_{B^1_{\infty,1}},\\
 \|[\mathcal R,u_N^a\cdot\nabla]\tau_N^a\|_{B^1_{\infty,1}}
 &\leq C\|u_N^a\|_{B^1_{\infty,1}}
          \|\tau_N^a\|_{B^1_{\infty,1}}
   +C\|\tau_N^a\|_{B^0_{\infty,1}}
          \|u_N^a\|_{B^2_{\infty,1}}.
\end{align*}
Thus the three-dimensional terms also consist of a controlled critical
factor times one high norm.  Since the cut-off initial data satisfy
\begin{equation*}
 \|S_Nu_0\|_{B^2_{\infty,1}}
 +\|S_N\tau_0\|_{B^1_{\infty,1}}
 \leq C2^N\mathcal E_0,
\end{equation*}
the uniform critical estimate and Gronwall's inequality give
\begin{equation}\label{7per-Bona-Smith-tame}
 \sup_{a\in[0,1]}Z_N^a(T)\leq C_T2^N\mathcal E_0,
\end{equation}
where $C_T$ is independent of $a$ and $N$.

We now subtract the systems for $z^a$ and $z_N^a$ and test the velocity
equation by $-\Delta W_N^a$ and the stress equation by
$-\frac12\Delta\Theta_N^a$.  The differentiated linear coupling terms
cancel.  The transport terms with advecting velocity $u^a$ are
bounded by
\begin{equation*}
 C\|\nabla u^a\|_{L^\infty}
 \big(\|\nabla W_N^a\|_{L^2}^2
     +\|\nabla\Theta_N^a\|_{L^2}^2\big).
\end{equation*}
For the terms in which the difference is the advecting velocity, integration
by parts and Young's inequality give, for example,
\begin{align*}
 |\langle W_N^a\cdot\nabla u_N^a,\Delta W_N^a\rangle|
 &\leq C\|\nabla u_N^a\|_{L^\infty}\|\nabla W_N^a\|_{L^2}^2\notag\\
 &\quad+C\|W_N^a\|_{L^2}\|\nabla^2u_N^a\|_{L^\infty}
       \|\nabla W_N^a\|_{L^2},
\end{align*}
whereas the three non-transport terms in the stress equation satisfy
\begin{align*}
 &\big|\langle W_N^a\cdot\nabla\tau_N^a
   +Q(\nabla u^a,\Theta_N^a)+Q(\nabla W_N^a,\tau_N^a),
          \Delta\Theta_N^a\rangle\big|\notag\\
 &\quad\leq \frac14\|\nabla^2\Theta_N^a\|_{L^2}^2
 +C\|W_N^a\|_{L^2}^2\|\nabla\tau_N^a\|_{L^\infty}^2
 +C\|\Theta_N^a\|_{L^2}^2\|\nabla u^a\|_{L^\infty}^2\notag\\
 &\qquad+C\|\nabla W_N^a\|_{L^2}^2
       \|\tau_N^a\|_{L^\infty}^2.
\end{align*}
In particular, \eqref{7per-Bona-Smith-tame} and the Besov embeddings yield
\begin{equation*}
 \sup_{a\in[0,1]}
 \big(\|\nabla^2u_N^a\|_{L^\infty_TL^\infty}
     +\|\nabla\tau_N^a\|_{L^\infty_TL^\infty}\big)
 \leq C_T2^N\mathcal E_0.
\end{equation*}
All the other coefficients in the preceding inequalities are bounded
uniformly by \eqref{7per-global-bound}.  Consequently, if
\begin{equation*}
 Y_N^a=\|\nabla W_N^a\|_{L^2}^2
       +\frac12\|\nabla\Theta_N^a\|_{L^2}^2,
\end{equation*}
then
\begin{equation*}
 \frac{d}{dt}Y_N^a+c\|\nabla^2\Theta_N^a\|_{L^2}^2
 \leq B_N^a(t)Y_N^a
 +C_T2^{2N}\big(\|W_N^a\|_{L^2}^2
                 +\|\Theta_N^a\|_{L^2}^2\big),
 \qquad \int_0^TB_N^a(t)\,dt\leq C_T,
\end{equation*}
with both bounds uniform in $a$ and $N$.  Since
$Y_N^a(0)\leq C\varepsilon_N^2$, the $2^{-N}$ gain in
\eqref{7per-Bona-Smith-L2-dissipation} and Gronwall's inequality imply
\begin{align}\label{7per-Bona-Smith-H1}
 \sup_{a\in[0,1]}\big(&
 \|z^a-z_N^a\|_{L^\infty_T(\dot H^1\times\dot H^1)}
 +\|\nabla^2\Theta_N^a\|_{L^2_TL^2}\big)
 \leq C_T\varepsilon_N.
\end{align}
Combining this with \eqref{7per-Bona-Smith-L2-dissipation}, we shall use
below the uniform estimate
\begin{equation}\label{7per-Bona-Smith-full-dissipation}
 \sup_{a\in[0,1]}
 \|\nabla\Theta_N^a\|_{L^2(0,T;H^1)}
 \leq C_T\varepsilon_N.
\end{equation}

We next recover the endpoint stress norm by an explicit frequency split.
Put $\beta_2=1/2$, $\beta_3=1/4$.  The blockwise heat estimates
\eqref{7per-heat-2d}--\eqref{7per-heat-3d}, applied to both stresses, imply
that for every $J<N-C$,
\begin{align}\label{7per-Bona-Smith-stress}
 \sup_{a\in[0,1]}\|\Theta_N^a\|_{L^\infty_TB^0_{\infty,1}}
 &\leq C_T2^J\sup_{a\in[0,1]}
 \|\Theta_N^a\|_{L^\infty_TB^{-1}_{\infty,1}}\notag\\
 &\quad+C\sum_{j>J-C}\|\Delta_j\tau_0\|_{L^\infty}
 +C\mathcal E_0\,2^{-\beta_dJ}.
\end{align}
The first term is the low-frequency part of the difference.  The second is
the high-frequency tail of the original initial stress; the truncated
initial stress has no modes above $N$.  The last term is the summable
high-frequency Duhamel remainder.  Taking $J=\lfloor N/2\rfloor$ and using
\eqref{7per-Bona-Smith-lower}, define
\begin{equation}\label{7per-Bona-Smith-rN}
 r_N=C_T2^{J-N}\varepsilon_N
 +C\sum_{j>J-C}\|\Delta_j\tau_0\|_{L^\infty}
 +C\mathcal E_0\,2^{-\beta_dJ}.
\end{equation}
Then $r_N\to0$ and
$\sup_a\|\Theta_N^a\|_{L^\infty_TB^0_{\infty,1}}\leq r_N$.

Finally set $H_N^a=\Gamma^a-\Gamma_N^a$ and
$G_N=\sup_{a\in[0,1]}\|H_N^a\|_{L^\infty_TB^0_{\infty,1}}$.
The difference equation used below is
\begin{align}\label{7per-Bona-Smith-Gamma-equation}
 \partial_tH_N^a+u^a\cdot\nabla H_N^a+H_N^a
 ={}&-W_N^a\cdot\nabla\Gamma_N^a+(a-1)\mathcal R\Theta_N^a\notag\\
 &+[\mathcal R,u^a\cdot\nabla]\Theta_N^a
   +[\mathcal R,W_N^a\cdot\nabla]\tau_N^a\notag\\
 &+\mathcal RQ(\nabla u^a,\Theta_N^a)
   +\mathcal RQ(\nabla W_N^a,\tau_N^a)+\mathcal S_N^a,
\end{align}
where the stretching difference is written without an implicit
dimension-dependent indicator:
\begin{equation*}
 \mathcal S_N^a=
 \begin{cases}
  0,&d=2,\\
  (\omega^a-\omega_N^a)\cdot\nabla u^a
  +\omega_N^a\cdot\nabla W_N^a,&d=3.
 \end{cases}
\end{equation*}
Notice also that
$\|H_N^a(0)\|_{B^0_{\infty,1}}\leq C\varepsilon_N$.  To estimate the
right-hand side of \eqref{7per-Bona-Smith-Gamma-equation}, define
\begin{equation*}
 \Lambda_N^a(t)=\|H_N^a(t)\|_{B^0_{\infty,1}}
 +\|\Theta_N^a(t)\|_{B^0_{\infty,1}}
 +\|(W_N^a,\Theta_N^a)(t)\|_{L^2}.
\end{equation*}
The periodic Biot--Savart estimate and
$\omega^a-\omega_N^a=H_N^a+\mathcal R\Theta_N^a$ imply
\begin{equation}\label{7per-Bona-Smith-Biot}
 \|\nabla W_N^a\|_{B^0_{\infty,1}}
 +\|\omega^a-\omega_N^a\|_{B^0_{\infty,1}}
 +\|W_N^a\|_{B^1_{\infty,1}}
 \leq C\Lambda_N^a.
\end{equation}

We now estimate the sources separately.  Write
$K_t(s)=e^{-c(t-s)}\mathbf 1_{[0,t]}(s)$, where $c>0$ is the constant in
the damped transport estimate.  First, the transport structure and the
one-higher-order estimate for the truncated solution give
\begin{align}\label{7per-Bona-Smith-Gamma-transport}
 \sup_{t\leq T}\int_0^tK_t(s)
 \|W_N^a\cdot\nabla\Gamma_N^a\|_{B^0_{\infty,1}}\,ds
 &\leq C_T2^N\|W_N^a\|_{L^\infty_TB^0_{\infty,1}}.
\end{align}
Next, \eqref{7per-R-commutator}, $|a-1|\leq1$, and the uniform small
critical bounds yield, pointwise in time,
\begin{align}
 &\|\mathcal R\Theta_N^a\|_{B^0_{\infty,1}}
 +\|[\mathcal R,u^a\cdot\nabla]\Theta_N^a\|_{B^0_{\infty,1}}\notag\\
 &\qquad\leq C\big(\|\Theta_N^a\|_{B^0_{\infty,1}}
                   +\|\Theta_N^a\|_{L^2}\big),
 \label{7per-Bona-Smith-Gamma-commutator-1}\\
 &\|[\mathcal R,W_N^a\cdot\nabla]\tau_N^a\|_{B^0_{\infty,1}}
 \leq C\delta\Lambda_N^a.
 \label{7per-Bona-Smith-Gamma-commutator-2}
\end{align}

It remains important not to use an algebra estimate at the endpoint
$B^0_{\infty,1}$.  Fix $0<\kappa<2-d/2$.  The product law in
Lemma~\ref{PL}, the low-frequency part of the multiplier estimate, and
\eqref{7per-positive-Besov} give
\begin{align}
 \|\mathcal RQ(\nabla u^a,\Theta_N^a)\|_{B^0_{\infty,1}}
 \leq C\delta\big(&\|\Theta_N^a\|_{B^0_{\infty,1}}
 +\|\Theta_N^a\|_{L^2}
 +\|\Theta_N^a\|_{B^\kappa_{\infty,\infty}}\big).
 \label{7per-Bona-Smith-Q-u-Theta}
\end{align}
Here the positive Besov norm is precisely where the dissipative part of the
$H^1$ difference estimate is needed.  Indeed,
\begin{align}\label{7per-Bona-Smith-positive-Theta}
 &\sup_{a\in[0,1]}\sup_{t\leq T}\int_0^tK_t(s)
 \|\Theta_N^a(s)\|_{B^\kappa_{\infty,\infty}}\,ds\notag\\
 &\quad\leq C\sup_{a\in[0,1]}\left(
 \|\Theta_N^a\|_{L^\infty_TH^1}
 +\|\nabla^2\Theta_N^a\|_{L^2_TL^2}\right)
 \leq C_T\varepsilon_N,
\end{align}
where Cauchy--Schwarz in time was used in the second line.

For the other quadratic term, the same product law gives
\begin{align}\label{7per-Bona-Smith-Q-W-tau}
 \|\mathcal RQ(\nabla W_N^a,\tau_N^a)\|_{B^0_{\infty,1}}
 \leq{}&C\|\nabla W_N^a\|_{L^2}\|\tau_N^a\|_{L^\infty}\notag\\
 &+C\|\nabla W_N^a\|_{B^0_{\infty,1}}
 \big(\|\tau_N^a\|_{B^0_{\infty,1}}
     +\|\tau_N^a\|_{B^\kappa_{\infty,\infty}}\big).
\end{align}
The global energy estimate for the truncated solutions and
\eqref{7per-positive-Besov} imply, uniformly in $a$ and $N$,
\begin{equation}\label{7per-Bona-Smith-positive-tauN}
 \sup_{t\leq T}\int_0^tK_t(s)
 \|\tau_N^a(s)\|_{B^\kappa_{\infty,\infty}}\,ds
 \leq C\mathcal E_0.
\end{equation}
Combining \eqref{7per-Bona-Smith-Biot},
\eqref{7per-Bona-Smith-H1}, and
\eqref{7per-Bona-Smith-positive-tauN} in
\eqref{7per-Bona-Smith-Q-W-tau}, we obtain
\begin{align}\label{7per-Bona-Smith-Q-W-tau-integrated}
 &\sup_{a\in[0,1]}\sup_{t\leq T}\int_0^tK_t(s)
 \|\mathcal RQ(\nabla W_N^a,\tau_N^a)(s)\|_{B^0_{\infty,1}}\,ds\notag\\
 &\qquad\leq C_T\delta\varepsilon_N+C\delta(G_N+r_N).
\end{align}
Finally, in dimension three, the divergence-free product estimate used in
\eqref{7per-vortex-stretching} and \eqref{7per-Bona-Smith-Biot} give
\begin{align}\label{7per-Bona-Smith-stretching-difference}
 \|\mathcal S_N^a\|_{B^0_{\infty,1}}
 &\leq C\|\omega^a-\omega_N^a\|_{B^0_{\infty,1}}
             \|u^a\|_{B^1_{\infty,1}}\notag\\
 &\quad+C\|\omega_N^a\|_{B^0_{\infty,1}}
             \|W_N^a\|_{B^1_{\infty,1}}
 \leq C\delta\Lambda_N^a.
\end{align}

Applying the damped transport estimate to
\eqref{7per-Bona-Smith-Gamma-equation} and using
\eqref{7per-Bona-Smith-Gamma-transport}--
\eqref{7per-Bona-Smith-stretching-difference} therefore yields
\begin{align}\label{7per-Bona-Smith-Gamma}
 G_N\leq C_T\Big(&\varepsilon_N
 +2^N\sup_{a\in[0,1]}
   \|u^a-u_N^a\|_{L^\infty_TB^0_{\infty,1}}+r_N\Big)
 +C\delta G_N.
\end{align}
The factor $2^N$ is canceled by the $2^{-N}$ in
\eqref{7per-Bona-Smith-lower}.  Choosing the smallness constant so that
$C\delta<1/2$ proves
\begin{equation}\label{7per-Bona-Smith-Gamma-final}
 G_N\leq C_T(\varepsilon_N+r_N).
\end{equation}
Using the periodic Biot--Savart estimate
\begin{equation*}
 \|\nabla(u^a-u_N^a)\|_{B^0_{\infty,1}}
 \leq C\big(\|H_N^a\|_{B^0_{\infty,1}}
 +\|\mathcal R(\tau^a-\tau_N^a)\|_{B^0_{\infty,1}}\big),
\end{equation*}
we conclude that, for
$c_N=C_T(\varepsilon_N+r_N)\to0$,
\begin{equation}\label{7per-Bona-Smith-tail}
 \sup_{a\in[0,1]}\|(u^a-u_N^a,\tau^a-\tau_N^a)\|_{
 L^\infty_T(\dot H^1\cap B^1_{\infty,1})
 \times L^\infty_T(\dot H^1\cap B^0_{\infty,1})}\leq c_N.
\end{equation}
For fixed $N$, the smooth norms of $z_N^a$ on $[0,T]$ are bounded by a
constant $C_{T,N}$ uniformly in $a$.  We include the derivative-loss
argument in the smooth middle difference.  Put
$\mathcal U_N=u_N^a-u_N^0$ and
$\mathcal T_N=\tau_N^a-\tau_N^0$, and choose an integer
$s>d/2+1$.  After applying derivatives of order at most $s$, the linear
coupling still cancels between the velocity energy and one half of the
stress energy.  The apparently dangerous term is
$Q(\nabla\mathcal U_N,\tau_N^0)$.  Since $Q$ has constant coefficients and
is bilinear, for suitable bilinear maps $\mathcal B_\ell$ one has
\begin{equation*}
 Q(\nabla\mathcal U_N,\tau_N^0)
 =\sum_\ell\partial_\ell
   \mathcal B_\ell(\mathcal U_N,\tau_N^0)
 -\sum_\ell\mathcal B_\ell
   (\mathcal U_N,\partial_\ell\tau_N^0).
\end{equation*}
Integrating the first sum by parts places its derivative on
$\mathcal T_N$; Moser estimates for the remaining commutators then give
\begin{equation*}
 \big|\langle Q(\nabla\mathcal U_N,\tau_N^0),
                 \mathcal T_N\rangle_{H^s}\big|
 \leq \frac14\|\nabla\mathcal T_N\|_{H^s}^2
 +C_{T,N}\big(\|\mathcal U_N\|_{H^s}^2
               +\|\mathcal T_N\|_{H^s}^2\big).
\end{equation*}
All other transport, commutator and $Q$ terms have the same or lower order.
The inhomogeneous damping term satisfies
\begin{equation*}
 a|\langle\tau_N^a,\mathcal T_N\rangle_{H^s}|
 \leq C_{T,N}a^2+C\|\mathcal T_N\|_{H^s}^2.
\end{equation*}
Consequently,
\begin{equation*}
 \frac{d}{dt}\big(\|\mathcal U_N\|_{H^s}^2
             +\tfrac12\|\mathcal T_N\|_{H^s}^2\big)
 +\|\nabla\mathcal T_N\|_{H^s}^2
 \leq C_{T,N}\big(\|\mathcal U_N\|_{H^s}^2
             +\|\mathcal T_N\|_{H^s}^2+a^2\big).
\end{equation*}
The initial difference is zero, so Gronwall's inequality, Sobolev embedding,
and the corresponding Besov embeddings give
\begin{equation*}
 \|(u_N^a-u_N^0,\tau_N^a-\tau_N^0)\|_{
 L^\infty_T(\dot H^1\cap B^1_{\infty,1})
 \times L^\infty_T(\dot H^1\cap B^0_{\infty,1})}
 \leq C_{T,N}a.
\end{equation*}
Writing the full difference as the sum of the two tails in
\eqref{7per-Bona-Smith-tail} and this smooth middle difference, and then
letting first $a\to0$ and next $N\to\infty$, proves
\eqref{7per-strong-limit}.  The same difference estimates also complete the
uniqueness and continuous-dependence assertions at the stated regularity.
\end{proof}
\begin{rema}[Persistence of higher Besov regularity]\label{rem:periodic-higher}
Let $s>0$ and retain the divergence, symmetry and velocity-mean assumptions
of Theorem~\ref{thmlocsol}.  Suppose additionally that
\begin{equation*}
 \mathcal E_{0,s}:=\|(u_0,\tau_0)\|_{H^1}
 +\|(\nabla u_0,\tau_0)\|_{B^s_{\infty,1}}<\infty.
\end{equation*}
Higher regularity persists on the lifespan of the critical solution:
for every finite interval $[0,T]$ on which that solution exists,
\begin{align}\label{7per-higher-finite-time}
 &\|u^a\|_{L^\infty_TB^{1+s}_{\infty,1}}
 +\|\tau^a\|_{\widetilde L^\infty_TB^s_{\infty,1}}
 +\|\tau^a\|_{L^1_TB^{2+s}_{\infty,1}}
 \leq C_{s,T}\mathcal E_{0,s}.
\end{align}
The constant depends on the critical bounds on $[0,T]$, but not on $a$.
In particular, if the \emph{basic} norm $\mathcal E_0$ satisfies the
smallness condition of Theorem~\ref{thm1t}, this holds for every $T<\infty$;
smallness of the additional high norm is not required.  Finiteness of
$\mathcal E_{0,s}$ alone does not imply global existence.

Here is the tame estimate underlying this assertion.  The positive-index
product law gives
\begin{align*}
 \|u\cdot\nabla\tau\|_{B^s_{\infty,1}}
 &\leq C_s\big(\|u\|_{L^\infty}\|\tau\|_{B^{1+s}_{\infty,1}}
       +\|\nabla\tau\|_{L^\infty}\|u\|_{B^{1+s}_{\infty,1}}\big),\\
 \|Q(\nabla u,\tau)\|_{B^s_{\infty,1}}
 &\leq C_s\big(\|\nabla u\|_{L^\infty}\|\tau\|_{B^s_{\infty,1}}
       +\|\tau\|_{L^\infty}\|u\|_{B^{1+s}_{\infty,1}}\big).
\end{align*}
In the heat estimate the first term is bounded, for any $\epsilon>0$, by
$\epsilon\|\tau\|_{B^{2+s}_{\infty,1}}
+C_{s,\epsilon}\|u\|_{L^\infty}^2\|\tau\|_{B^s_{\infty,1}}$.
The velocity transport and pressure estimates are tame at index $1+s$,
with stress forcing controlled by $\|\tau\|_{B^{2+s}_{\infty,1}}$.
After substituting the heat estimate and absorbing its $\epsilon$ term,
Gronwall's inequality proves \eqref{7per-higher-finite-time}.  Its
coefficients are integrable on $[0,T]$: in particular,
\[
\|\nabla\tau\|_{L^1_TL^\infty}
\leq C T^{1/2}\|\tau\|_{L^\infty_TB^0_{\infty,1}}^{1/2}
\|\tau\|_{L^1_TB^2_{\infty,1}}^{1/2}.
\]
This is a finite-time persistence estimate, not a uniform-in-time bound
in higher Besov spaces or an additional same-topology damping-limit claim.
\end{rema}
\begin{rema}\label{rem:periodic-finite-time}
The restriction to a finite time interval in \eqref{7per-L2-limit} is
essential unless additional conditions remove the zero-mode obstruction.  For
example, let $u_0=0$ and let $\tau_0\neq0$ be a sufficiently small constant
symmetric matrix.  Then
\begin{equation*}
(u^a(t),\tau^a(t))=(0,e^{-at}\tau_0),\qquad
(u^0(t),\tau^0(t))=(0,\tau_0).
\end{equation*}
Thus the difference is $O(aT)$ on every fixed interval $[0,T]$, consistently
with Theorem~\ref{thm1t}, whereas
\begin{equation*}
\sup_{t\geq0}\|\tau^a(t)-\tau^0(t)\|_{L^2(\T^d)}
=\|\tau_0\|_{L^2(\T^d)}
\end{equation*}
for every $a>0$.  This also explains why the numerical comparisons below are
formulated on bounded time intervals.
\end{rema}

\begin{prop}[Exponential relaxation on the torus]\label{prop:periodic-exp}
Under the assumptions of Theorem~\ref{thm1t}, define the spatial average by
\begin{equation*}
\overline f:=\frac{1}{|\T^d|}\int_{\T^d}f(x)\,dx.
\end{equation*}
There exist constants $C,c,\lambda>0$, independent of $a\in[0,1]$, such that
\begin{equation}\label{7per-nonzero-decay}
\|u^a(t)\|_{H^1}
+\|\tau^a(t)-\overline{\tau^a(t)}\|_{H^1}
\leq C\mathcal E_0e^{-\lambda t},\qquad t\geq0.
\end{equation}
The behavior of the stress zero mode is as follows.
\begin{enumerate}[(i)]
\item If $a\in(0,1]$, then
\begin{equation}\label{7per-positive-a-decay}
\|u^a(t)\|_{H^1}+\|\tau^a(t)\|_{H^1}
\leq C\mathcal E_0e^{-ca t},\qquad t\geq0.
\end{equation}
\item If $a=0$, there exists a constant symmetric matrix $\tau_\infty$ such that
\begin{align}\label{7per-zero-a-decay}
&\|u^0(t)\|_{H^1}
+\|\tau^0(t)-\tau_\infty\|_{H^1}
\leq C\mathcal E_0e^{-\lambda t},\qquad t\geq0,\\
&\tau_\infty=\overline{\tau_0}-
\int_0^\infty\overline{Q(\nabla u^0(s),\tau^0(s))}\,ds,
\qquad
|\tau_\infty-\overline{\tau_0}|\leq C\mathcal E_0^2.\notag
\end{align}
\end{enumerate}
\end{prop}

\begin{proof}
The velocity average is conserved.  Indeed, periodicity and $\div u^a=0$
give
\begin{equation}\label{7per-mean-u}
\frac{d}{dt}\overline{u^a}=0.
\end{equation}
Set
\begin{equation*}
v=u^a,\qquad
m=\overline{\tau^a},\qquad
\sigma=\tau^a-m.
\end{equation*}
Thus $v$ and $\sigma$ have zero spatial mean.  Apply the projection
$f\mapsto f-\overline f$ to \eqref{7.1}.  The fluctuation system is
\begin{equation}\label{7per-fluctuation-system}
\left\{
\begin{aligned}
 &\partial_tv+v\cdot\nabla v+\nabla P
   =\operatorname{div}\sigma,\qquad \operatorname{div}v=0,\\
 &\partial_t\sigma+v\cdot\nabla\sigma+a\sigma
   +Q(\nabla v,\sigma)+Q(\nabla v,m)
   -\overline{Q(\nabla v,\sigma)}
   =2D(v)+\Delta\sigma .
\end{aligned}\right.
\end{equation}
The constant term $\overline{Q(\nabla v,\sigma)}$ is orthogonal to
$\sigma$, $\Delta\sigma$ and $\nabla v$.  The terms containing $m$ obey
\begin{align}\label{7per-fluctuation-mean-terms}
 |\langle Q(\nabla v,m),\sigma\rangle|
 &\leq C|m|\|\nabla v\|_{L^2}\|\sigma\|_{L^2},\notag\\
 |\langle Q(\nabla v,m),\Delta\sigma\rangle|
 &\leq C|m|\|\nabla v\|_{L^2}\|\Delta\sigma\|_{L^2},\notag\\
 |\langle Q(\nabla v,m),\nabla v\rangle|
 &\leq C|m|\|\nabla v\|_{L^2}^2.
\end{align}
Since $|m|\leq C\|\tau\|_{L^2}\leq C\mathcal E_0$, these terms are
absorbed by the fluctuation dissipation.  Repeating
\eqref{7per-basic-H1-identity} and \eqref{7per-mixed-identity} for
\eqref{7per-fluctuation-system} therefore gives an energy
$\mathscr F_1$ satisfying
\begin{equation*}
\mathscr F_1\simeq\|(v,\sigma)\|_{H^1}^2
\end{equation*}
and
\begin{equation}\label{7per-fluctuation-energy}
\frac{d}{dt}\mathscr F_1
+c_0\big(\|\nabla v\|_{L^2}^2
+\|\nabla\sigma\|_{H^1}^2+a\|\sigma\|_{H^1}^2\big)\leq0.
\end{equation}
More explicitly, $\mathscr F_1$ is obtained by adding a sufficiently small
multiple of the mixed term $-\langle\sigma,\nabla v\rangle$ to the basic
$H^1$ energy.  In dimension two,
$\langle v\cdot\nabla v,\Delta v\rangle=0$; in dimension three this term is
bounded by
$C\|\nabla v\|_{L^\infty}\|\nabla v\|_{L^2}^2$ and is absorbed by the
small global bound from Theorem~\ref{thm1t}.

By \eqref{7per-u-poincare} and the Poincar\'e inequality for $\sigma$, the
dissipation in \eqref{7per-fluctuation-energy} controls $\mathscr F_1$ with a
constant independent of $a$.  Consequently,
\begin{equation*}
\mathscr F_1(t)\leq \mathscr F_1(0)e^{-2\lambda t},
\end{equation*}
which proves \eqref{7per-nonzero-decay}.

It remains to examine $m$.  Averaging the stress equation yields
\begin{equation}\label{7per-mean-tau}
m'(t)+am(t)=-\overline{Q(\nabla v,\sigma)}.
\end{equation}
Here we used the bilinearity of $Q$, the fact that $m$ is spatially constant,
and $\overline{\nabla v}=0$, which imply
$\overline{Q(\nabla v,m)}=0$.  By \eqref{7per-nonzero-decay},
\begin{equation}\label{7per-Q-mean-decay}
|\overline{Q(\nabla v,\sigma)}|
\leq C\|\nabla v\|_{L^2}\|\sigma\|_{L^2}
\leq C\mathcal E_0^2e^{-2\lambda t}.
\end{equation}
For $a>0$, the variation-of-constants formula applied to
\eqref{7per-mean-tau} gives
\begin{equation*}
m(t)=e^{-at}\overline{\tau_0}-
\int_0^t e^{-a(t-s)}\overline{Q(\nabla v(s),\sigma(s))}\,ds.
\end{equation*}
Choose $0<c<\min\{1,\lambda\}$.  Since $0<a\leq1$,
\begin{equation*}
\int_0^t e^{-a(t-s)}e^{-2\lambda s}\,ds\leq Ce^{-cat},
\end{equation*}
with $C$ independent of $a$.  Combining this inequality with
\eqref{7per-Q-mean-decay} gives $|m(t)|\leq C\mathcal E_0e^{-cat}$ uniformly for
$a\in(0,1]$.  This and \eqref{7per-nonzero-decay} prove
\eqref{7per-positive-a-decay}.

When $a=0$, the right-hand side of \eqref{7per-mean-tau} is integrable.
Therefore $m(t)$ converges to the matrix $\tau_\infty$ in
\eqref{7per-zero-a-decay}, and \eqref{7per-Q-mean-decay} also gives
\begin{equation*}
|m(t)-\tau_\infty|\leq C\mathcal E_0^2e^{-2\lambda t},
\qquad
|\tau_\infty-\overline{\tau_0}|\leq C\mathcal E_0^2.
\end{equation*}
Together with \eqref{7per-nonzero-decay}, this completes the proof.
\end{proof}

\begin{rema}\label{rem:periodic-spectral-gap}
Estimate \eqref{7per-nonzero-decay} is a spectral-gap effect.  For the
linearization of \eqref{7.1}, every nonzero Fourier mode of the velocity obeys
\begin{equation*}
\partial_{tt}\widehat u(k)+(a+|k|^2)\partial_t\widehat u(k)
+|k|^2\widehat u(k)=0,\qquad k\neq0.
\end{equation*}
Thus the nonzero modes have a decay rate bounded away from zero uniformly in
$a\in[0,1]$.  For the homogeneous linearized zero mode (and, in particular,
for a spatially constant stress solution), the stress decays exactly on the
time scale $a^{-1}$ when $a>0$ and persists when $a=0$.  In the nonlinear
system its mean instead satisfies the forced ODE \eqref{7per-mean-tau}, and
the same slow homogeneous component may remain; the spatially constant
example in Remark~\ref{rem:periodic-finite-time} shows that this rate is
attained.  Hence the factor $a$ in
\eqref{7per-positive-a-decay} is unavoidable, and the full decay estimate
cannot be uniform as $a\to0$.
\end{rema}

\begingroup
\setlength{\textfloatsep}{10pt plus 2pt minus 2pt}
\setlength{\floatsep}{8pt plus 2pt minus 2pt}
\setlength{\intextsep}{8pt plus 2pt minus 2pt}
\renewcommand{\topfraction}{0.92}
\renewcommand{\bottomfraction}{0.85}
\renewcommand{\textfraction}{0.08}
\renewcommand{\floatpagefraction}{0.72}
\subsection{Toy models in $\mathbb{T}^2$ and corresponding semi-implicit numerical schemes}
We distinguish the exact co-rotational reduction of \eqref{7.1} from the
modified two-component benchmark used in the Gaussian-vortex computations.
Write $u=(u^1,u^2)^\intercal$ and
$\omega=\partial_1u^2-\partial_2u^1$.  For a two-component stress
$\tau=(\tau^1,\tau^2)^\intercal$, define its matrix representation and the
coupling operators by
\begin{equation}\label{7per-toy-operators}
 \mathsf T(\tau)=\begin{pmatrix}\tau^1&\tau^2\\\tau^2&-\tau^1\end{pmatrix},
 \quad
 \mathsf G\tau=\begin{pmatrix}\partial_1\tau^1+\partial_2\tau^2\\
 \partial_1\tau^2-\partial_2\tau^1\end{pmatrix},
 \quad
 \mathsf S u=\begin{pmatrix}2\partial_1u^1\\
 \partial_1u^2+\partial_2u^1\end{pmatrix}.
\end{equation}
Thus $\mathsf G\tau=\operatorname{div}\mathsf T(\tau)$ and,
when $\operatorname{div}u=0$, $\mathsf T(\mathsf S u)=2D(u)$.
For $b=0$, the trace of the matrix stress solves the homogeneous
transport--diffusion--damping equation, since
$\operatorname{tr}(\tau\Omega-\Omega\tau)=\operatorname{tr}D(u)=0$.
Consequently trace-free initial stress remains trace free.  In this invariant
class, direct computation gives
\begin{equation}\label{7per-toy-closures}
 Q(\nabla u,\mathsf T(\tau))=\mathsf T(\mathsf C(u,\tau)),
 \qquad
 \mathsf C(u,\tau)=\omega\begin{pmatrix}\tau^2\\-\tau^1\end{pmatrix}.
\end{equation}
The exact nonlinear co-rotational toy model is therefore
\begin{equation}\label{eq:toy1}
    \begin{cases}
        \pa_t u + (u\cdot \na)u+\na p = \begin{pmatrix}
            \pa_1 \tau^1+\pa_2 \tau^2\\
            \pa_1 \tau^2 -\pa_2 \tau^1
        \end{pmatrix},\qquad \div u =0,\\
        \pa_t \tau +(u\cdot \na)\tau -\Delta \tau +a\tau +\mathsf C(u,\tau)
        = \begin{pmatrix}
            2\pa_1u^1\\
            \pa_1 u^2+\pa_2 u^1
        \end{pmatrix},
    \end{cases}
\end{equation}
The nonlinear vanishing-damping table and temporal-convergence plots below
use \eqref{eq:toy1}.  The Gaussian nonlinear benchmark instead replaces
$\mathsf C$ by the bilinear closure implemented in the MATLAB driver,
\begin{equation}\label{eq:toy-modified}
 \mathsf M(u,\tau)=\omega
 \begin{pmatrix}-2\tau^2\\\tau^1-\tau^2\end{pmatrix},
 \qquad
 \partial_t\tau+(u\cdot\nabla)\tau-\Delta\tau+a\tau
 +\mathsf M(u,\tau)=\mathsf S u,
\end{equation}
with the same velocity equation as in \eqref{eq:toy1}.  This defines a
different nonlinear toy system, not the co-rotational Oldroyd--B reduction:
$\mathsf M\ne\mathsf C$, and $\mathsf M$ is not a skew rotation of the two
stress components.  Theorem~\ref{thm1t} is not asserted for this modified
closure.  Keeping these definitions separate permits an unambiguous reading
of the existing numerical data without changing the computations.
Deleting the transport and bilinear closure terms from either nonlinear
system gives the common linear toy model
\begin{equation}\label{eq:toy2}
    \begin{cases}
        \pa_t u +\na p = \begin{pmatrix}
            \pa_1 \tau^1+\pa_2 \tau^2\\
            \pa_1 \tau^2 -\pa_2 \tau^1
        \end{pmatrix},\qquad \div u =0,\\
        \pa_t \tau -\Delta \tau +a\tau = \begin{pmatrix}
            2\pa_1u^1\\
            \pa_1 u^2+\pa_2 u^1
        \end{pmatrix}.
    \end{cases}
\end{equation}
We will compute \eqref{eq:toy1} and \eqref{eq:toy2} models above by the well-known semi-implicit Fourier spectral method. More precisely speaking, \eqref{eq:toy1} and \eqref{eq:toy2} will be computed by the following schemes respectively:
\begin{equation}\label{scheme:toy1}
    \begin{cases}
        \displaystyle\frac{u_{n+1}-u_n}{\Delta t} + \Pi_K((u_n\cdot \na)u_{n+1})+\na p_n = \begin{pmatrix}
            \pa_1 \tau_{n+1}^1+\pa_2 \tau_{n+1}^2\\
            \pa_1 \tau_{n+1}^2 -\pa_2 \tau_{n+1}^1
        \end{pmatrix},\qquad \div u_{n+1} =0,\\
        \displaystyle\frac{\tau_{n+1}-\tau_n}{\Delta t} +\Pi_K((u_n\cdot \na)\tau_n) -\Delta \tau_{n+1} +a\tau_{n+1} +\Pi_K\mathsf C(u_n,\tau_n)
        = \begin{pmatrix}
            2\pa_1u^1_n\\
            \pa_1 u^2_n+\pa_2 u^1_n
        \end{pmatrix},
    \end{cases}
\end{equation}
and 
\begin{equation}\label{scheme:toy2}
    \begin{cases}
        \displaystyle\frac{u_{n+1}-u_n}{\Delta t} +\na p_n = \begin{pmatrix}
            \pa_1 \tau_{n+1}^1+\pa_2 \tau_{n+1}^2\\
            \pa_1 \tau_{n+1}^2 -\pa_2 \tau_{n+1}^1
        \end{pmatrix},\qquad \div u_{n+1} =0,\\
        \displaystyle\frac{\tau_{n+1}-\tau_n}{\Delta t} -\Delta \tau_{n+1} +a\tau_{n+1} = \begin{pmatrix}
            2\pa_1u^1_n\\
            \pa_1 u^2_n+\pa_2 u^1_n
        \end{pmatrix},
    \end{cases}
\end{equation}
subject to the initial data $(\Pi_K u_0,\Pi_K \tau_0)$. Here $(u_{n},\tau_n)$ is the discrete solution at $n$-th time step and $\Delta t$ is the time step.  For $K\geq 2$, we introduce the space
$$X_K=\text{span}\left\{\cos(k\cdot x)\ ,\ \sin(k\cdot x):\ \ k=(k_1,k_2)\in\Z^2\ ,\ |k|_\infty=\max\{|k_1|,|k_2|\}\leq K \right\} .$$
Here $K$ denotes the Fourier cutoff, whereas the quoted $N_x,N_y$
(and the shorthand $N$ in the plots) denote the numbers of grid points.
Then we define $\Pi_K$ to be the truncation operator of Fourier modes $|k|_\infty\leq K$. $\Pi_Ku_0,\Pi_K \tau_0\in X_K$ and by induction, we have $u_n,\tau_n\in X_K\ ,\forall n\geq 0$. Indeed it 
 is clear that in order to solve $\tau_{n+1}$ in \eqref{scheme:toy1} and \eqref{scheme:toy2} the viscosity terms $ \Delta \tau_{n+1}$ is treated unknown and the nonlinear terms are known from the previous time step. To solve $u_{n+1}$ we put $\tau_{n+1}$ instead of $\tau_n$ in the sourcing term; such choice guarantees the regularity balance.

In addition we can further derive that $\div u_n=0$ for all $n$ by induction. To solve $u_{n+1}$, we recall the Leray projection $\LP$, namely the $L^2$-orthogonal projection onto the divergence-free subspace: for any $v\in L^2$ we have $\LP v=v-\na q$, where $q\in H^1$ solves the following Poisson equation under periodic boundary conditions:
\begin{equation*}
        \De q=\na \cdot v.
\end{equation*}
Then we apply the Leray projection $\LP$ to derive that 
\begin{equation}\label{scheme:toy1b}
\frac{u_{n+1}-u_n}{\De t} +\LP\Pi_K(u_n \cdot \na u_{n+1})=\LP\begin{pmatrix}
            \pa_1 \tau_{n+1}^1+\pa_2 \tau_{n+1}^2\\
            \pa_1 \tau_{n+1}^2 -\pa_2 \tau_{n+1}^1
        \end{pmatrix},
\end{equation}
or 
\begin{equation}\label{scheme:toy2b}
\frac{u_{n+1}-u_n}{\De t} =\LP\begin{pmatrix}
            \pa_1 \tau_{n+1}^1+\pa_2 \tau_{n+1}^2\\
            \pa_1 \tau_{n+1}^2 -\pa_2 \tau_{n+1}^1
        \end{pmatrix},
\end{equation}
for \eqref{scheme:toy1} and \eqref{scheme:toy2} respectively.  The modified
Gaussian benchmark uses the same stress-first scheme with
$\mathsf C(u_n,\tau_n)$ replaced by $\mathsf M(u_n,\tau_n)$ in
\eqref{scheme:toy1}; the projected velocity step is unchanged.
The Leray projection is evaluated by Fourier multipliers.  The damping-rate
computations for the exact co-rotational reduction are consistent with the
finite-time behavior in Theorem~\ref{thm1t}; the modified Gaussian benchmark
is only an illustrative toy-model experiment.  A rigorous
time-discretization error estimate, including uniformity as $a\to0$,
requires a separate stability analysis and is not asserted here.

\begin{rema}
    The velocity convection in \eqref{scheme:toy1} and
    \eqref{scheme:toy1b} is treated semi-implicitly as
    $(u_n\cdot\nabla)u_{n+1}$, cf. \cite{CLW24}.  Both $\mathsf C$ and
    $\mathsf M$ are evaluated explicitly in physical space at $t_n$.
    The damping-rate and temporal-convergence runs use dealiased products;
    the GPU Gaussian benchmark also applies $2/3$ Fourier filtering and
    solves the semi-implicit velocity equation by GMRES.  It inherits the
    modified closure $\mathsf M$ from the MATLAB benchmark, not the
    unfiltered MATLAB implementation.  These choices are part of the
    numerical specification, not a time-discretization stability theorem.
\end{rema}
\begin{rema}
    Indeed we will postpone the discussion of the error analysis in a forthcoming work and only present the dynamics here.
\end{rema}

 \subsection{Rate of vanishing damping}

In this section we present numerical evidence for the rate of vanishing
damping of the schemes \eqref{scheme:toy1} and \eqref{scheme:toy2}.
Here the nonlinear calculations use the exact co-rotational model
\eqref{eq:toy1}, not the modified closure \eqref{eq:toy-modified}.
Theorem~\ref{thm1t} proves an $L^2$ estimate of order $a$ on fixed time
intervals for sufficiently small matrix data.  The $L^\infty$ slopes reported below
are a stronger numerical observation, and the amplitudes used in these tests
have not been checked against the nonquantitative smallness threshold
$\varepsilon_0$ in that theorem.  Throughout this subsection we choose
initial data suggested by \cite{CLW24}.  Since $m=3.8$ is not an integer, the ordinary expression $(\cos z)^m$ is not real-valued when $\cos z<0$.  We therefore use the real signed power
\begin{equation*}
c_m(z):=\sgn(\cos z)|\cos z|^m,
\qquad d_m(z):=|\cos z|^{m-1},
\end{equation*}
and set
\begin{equation}\label{7per-numerical-initial-data}
u_0=\tau_0=\left(-\frac{m}{2}c_m(x)d_m(y)\sin y,
\frac{m}{2}d_m(x)c_m(y)\sin x\right),\qquad m=3.8.
\end{equation}
In particular, $u_0=(\partial_y\psi_m,-\partial_x\psi_m)$ with
$\psi_m(x,y)=\frac12c_m(x)c_m(y)$, so the velocity is real-valued, periodic,
and exactly divergence free.

We generate a same-step reference solution
$(u_R,v_R,\tau^1_R,\tau^2_R)$ using the same schemes with
$a=0$ and $\Delta t=10^{-4}$, and compare it with solutions at different
$a$-values.  Thus the displayed discrepancies isolate the damping effect at
fixed temporal resolution.  A finer-step $a=0$ reference would mix damping
with additional time-discretization error; the first-order-in-time check is
recorded separately in Figure~\ref{fig:temporal-conv}.

For reproducibility, let $e=(u^1-u_R,u^2-v_R,\tau^1-\tau^1_R,
\tau^2-\tau^2_R)$ on the uniform mesh.  The reported errors use
\begin{equation}\label{7per-discrete-error-norms}
 \|e\|_{\infty,h}=\max_{j,\ell}|e^\ell(x_j)|,
 \qquad
 \|e\|_{2,h}=\left(\frac{h_xh_y}{4}
       \sum_j\sum_{\ell=1}^4|e^\ell(x_j)|^2\right)^{1/2}.
\end{equation}
The table and plot labels $L^\infty$ and $L^2$ refer to these discrete
norms.  Thus $\|e\|_{2,h}$ is one half of the usual Euclidean joint
four-component $L^2$ quadrature norm.  This component normalization does
not change the observed convergence order; it is not the Frobenius norm
of the lifted matrix stress $\mathsf T(\tau)$.

More clearly in the first numerical experiment we compute the linear toy model \eqref{eq:toy2} using \eqref{scheme:toy2} together with the projected \eqref{scheme:toy2b}. We fix $\De t= 0.0001, N_x=N_y=128$ and we vary $a=\frac{1}{2^k}$. The $L^\infty$-difference at $T= 1$ is presented in Table~\ref{fig1:linear}.  Over the displayed range, the $L^\infty$-differences exhibit approximately first-order behavior in $a$.

\begin{table}[htbp]
\begin{minipage}[c]{0.42\textwidth}
    \centering\small
    \setlength{\tabcolsep}{4pt}
\begin{tabular}{c c c}
 \toprule
 $a=1$  & $L^\infty$-difference & $L^2$-difference  \\ [0.2ex]
 \hline
 $a$ & 0.0433 & 0.0959 \\[0.2ex]
 $a/2$ & 0.0234 & 0.0540
 \\[0.2ex]
$a/4$ &  0.0122 & 0.0288
 \\[0.2ex]
 $a/8$ & 0.00625 & 0.0149
 \\[0.2ex]
 $a/16$ & 0.00316 & 0.00758
 \\[0.2ex]
 $a/32$ & 0.00159 & 0.00382\\[0.2ex]
 $a/64$ & 0.000797 & 0.00192\\[0.2ex]
 $a/128$ & $3.993\times10^{-4}$ & $9.623\times10^{-4}$\\[0.2ex]
 \bottomrule
\end{tabular}

\end{minipage}
\hfill
\begin{minipage}[c]{0.48\textwidth}
\centering
\includegraphics[width=1\textwidth]{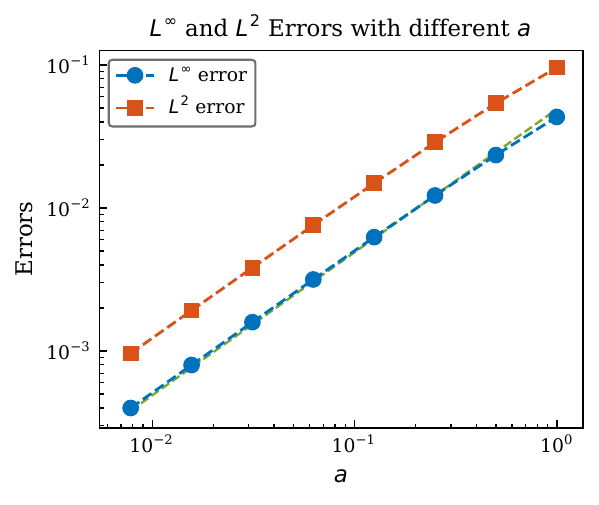}
\end{minipage}
\caption{$L^\infty$ and $L^2$-differences with fixed $\De t$ and varying damping $a$ for the linear toy model \eqref{eq:toy2}.}\label{fig1:linear}
\end{table}

In the second experiment, we compute the nonlinear toy model \eqref{eq:toy1} using \eqref{scheme:toy1} together with the projected \eqref{scheme:toy1b}. We fix $\Delta t=0.0001$, $N_x=N_y=128$, and vary $a=\frac{1}{2^k}$. The $L^\infty$-differences at $T=1$ can be found below in Table~\ref{fig2:nonlinear}.  They likewise show approximately first-order behavior over the tested values of $a$; this is an empirical observation rather than a consequence of the $L^2$ estimate in Theorem~\ref{thm1t}.

\begin{table}[htbp]
\begin{minipage}[c]{0.42\textwidth}
    \centering\small
    \setlength{\tabcolsep}{4pt}
\begin{tabular}{c  c c}
 \toprule
 $a=1$  & $L^\infty$-difference & $L^2$-difference  \\ [0.2ex]
 \hline
 $a$ & 0.0430 & 0.0957 \\[0.2ex]
 $a/2$ & 0.0233 & 0.0538
 \\[0.2ex]
$a/4$ &  0.0122 & 0.0287
 \\[0.2ex]
 $a/8$ & 0.00622 & 0.0149
 \\[0.2ex]
 $a/16$ & 0.00314 & 0.00756
 \\[0.2ex]
 $a/32$ & 0.00158 & 0.00381 \\[0.2ex]
 $a/64$ & 0.000793 & 0.00192\\[0.2ex]
 $a/128$ & $3.970\times10^{-4}$ & $9.598\times10^{-4}$\\[0.2ex]
 \bottomrule
\end{tabular}
\end{minipage}
\hfill
\begin{minipage}[c]{0.48\textwidth}
\centering
\includegraphics[width=1\textwidth]{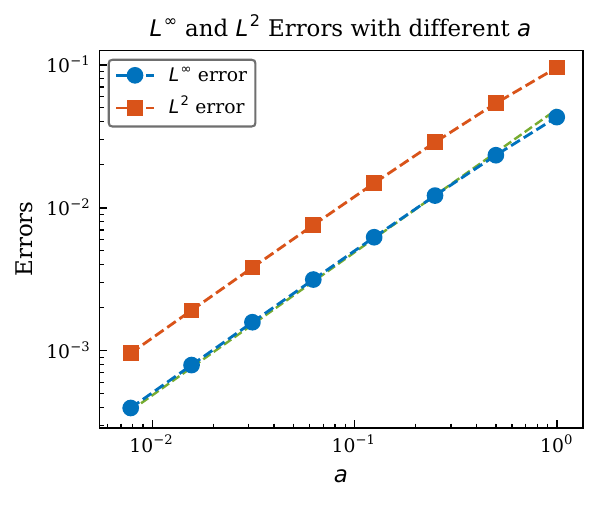}
\end{minipage}
\caption{$L^\infty$ and $L^2$-differences with fixed $\De t$ and varying damping $a$ for the nonlinear toy model \eqref{eq:toy1}.}\label{fig2:nonlinear}
\end{table}

As a consistency check on the time discretization, we refine $\Delta t$ at
fixed $N=128$ and $a=0.1$ up to $T=0.5$, comparing against a same-grid
reference with a substantially smaller step.  Both toy schemes exhibit
first-order decay of the $L^\infty$ and $L^2$ errors; see
Figure~\ref{fig:temporal-conv}.

\begin{figure}[htbp]
\centering
\includegraphics[width=0.88\textwidth]{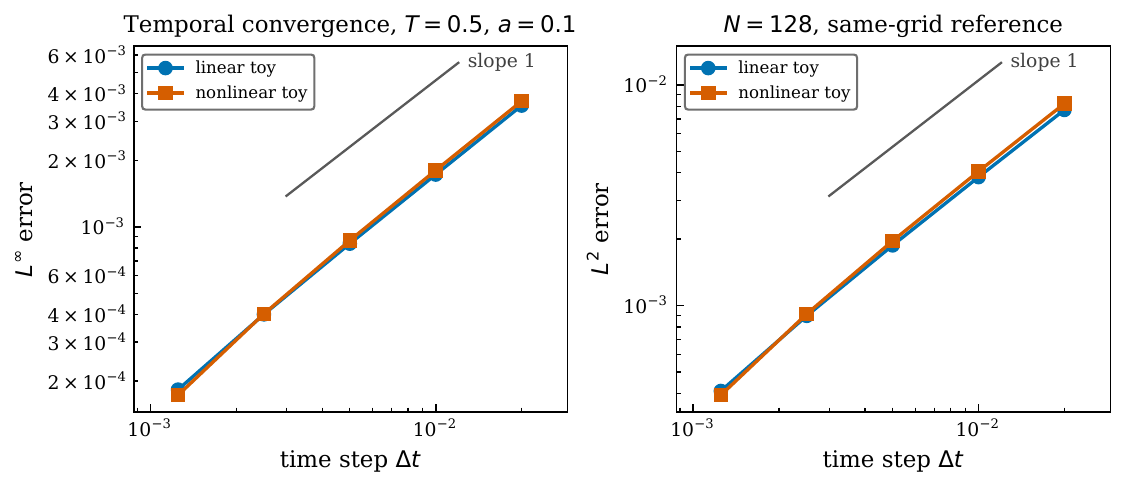}
\caption{Temporal convergence of the linear and nonlinear toy schemes at $a=0.1$, $T=0.5$, and $N=128$.}\label{fig:temporal-conv}
\end{figure}

\FloatBarrier
\subsection{Dynamics of the linear and modified toy models: Two Gaussian vortices}

In this subsection we consider the following initial data of two Gaussian vortices.  To make the vorticity genuinely periodic and compatible with
$w_0=\na^\perp\cdot u_0$, we periodize each Gaussian and remove the zero
Fourier mode.  Namely, set
\begin{equation*}
G_{(c_1,c_2)}(x,y):=\sum_{(k_1,k_2)\in\Z^2}
\exp\left(-5\big((x-c_1+2\pi k_1)^2+(y-c_2+2\pi k_2)^2\big)\right)
\end{equation*}
and take the vertical pair used in the MATLAB plates,
\begin{equation}\label{7.13}
    w_0=\na^\perp\cdot u_0
    =G_{-}+G_{+}-\overline{\bigl(G_{-}+G_{+}\bigr)},
    \qquad
    G_{\pm}:=G_{(0,\pm\pi/4)}.
\end{equation}
We define $u_0$ by the zero-mean periodic Biot--Savart inverse of
\eqref{7.13}, with the sign fixed by
$\na^\perp\cdot u_0=w_0$.  Thus $\operatorname{div}u_0=0$ and
$\overline{u_0}=0$, as required by the normalization used in the periodic
theory.
Moreover we choose the initial $\tau_0$ as follows:
  $$\tau_0 = \left(1-\frac{m}{2}(\cos(x))^m(\cos(y))^{m-1}\sin(y), 1+\frac{m}{2}(\cos(x))^{m-1}(\cos(y))^{m}\sin(x)\right),$$
where $m=4$.  The constant $1$ adds a nonzero background to both reduced
stress components; no small-data threshold has been verified for these
data.  This background should not be interpreted as positive definiteness:
the matrix $\mathsf T(\tau_0)$ is trace free and is indefinite wherever it
is nonzero.  Moreover, the modified closure used below is not covered by
Theorem~\ref{thm1t}.

We use $N_x=N_y=512$ and $\Delta t=0.001$ with the initial data in
\eqref{7.13}.  Figures~\ref{fig1} and \ref{fig2} show the vorticity of the
linear model \eqref{eq:toy2} and the modified nonlinear benchmark
\eqref{eq:toy-modified}, respectively.  In both plates, rows correspond to
$a=0.001,0.1,1$, and columns to $t=0.001,2,10,15$; each panel has its own
color scale.  These benchmark examples are motivated by \cite{CLW24}.
The modified nonlinear run uses \eqref{scheme:toy1} with
$\mathsf C$ replaced by $\mathsf M$.  The vortices rotate and shear in this
benchmark, whereas the linear snapshots show more diffusive behavior.
Figure~\ref{fig2b} records the kinetic enstrophy
\begin{equation*}
 E(t)=\frac12\int_{\T^2}|\nabla u(t,x)|^2\,dx
     =\frac12\int_{\T^2}\omega(t,x)^2\,dx,
\end{equation*}
not the total velocity--stress energy.  For $a=0.001$, it rises from about
$0.30$ to a peak of about $107$ near $t=5.1$, then decreases to about $62$
at $t=15$; for $a=0.1$ the peak is much smaller, and for $a=1$ it is already
very small at the tested solver resolution by $t=15$.  These finite-resolution observations concern
the modified toy benchmark only and do not establish large-data global
regularity for that benchmark or for the Oldroyd--B system.

\begin{figure}[!htp]
\centering
\begingroup\setlength{\tabcolsep}{1pt}
\begin{tabular}{@{}cccc@{}}
\includegraphics[width=0.225\textwidth]{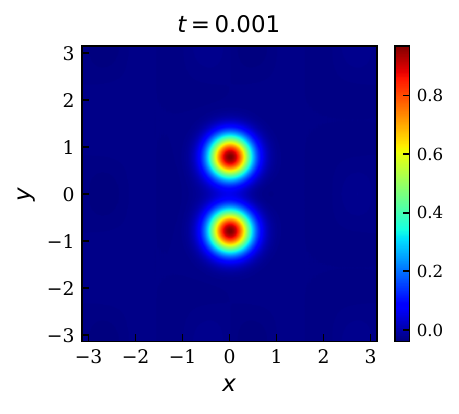} &
\includegraphics[width=0.225\textwidth]{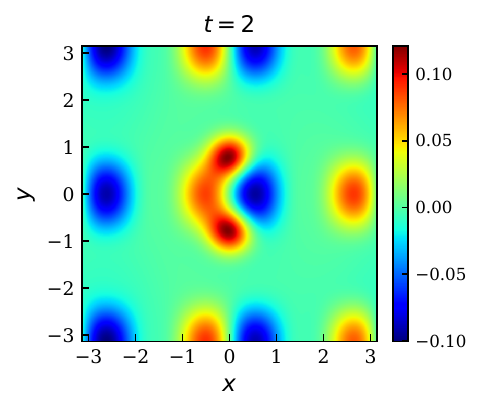} &
\includegraphics[width=0.225\textwidth]{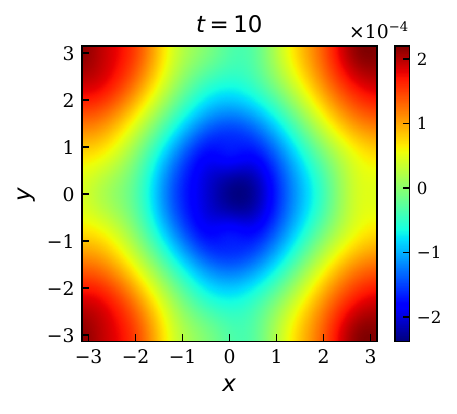} &
\includegraphics[width=0.225\textwidth]{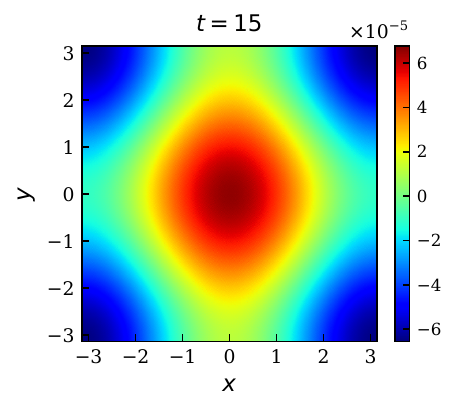} \\[-0.8mm]
\includegraphics[width=0.225\textwidth]{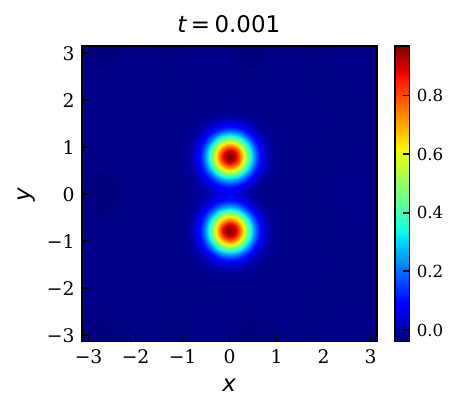} &
\includegraphics[width=0.225\textwidth]{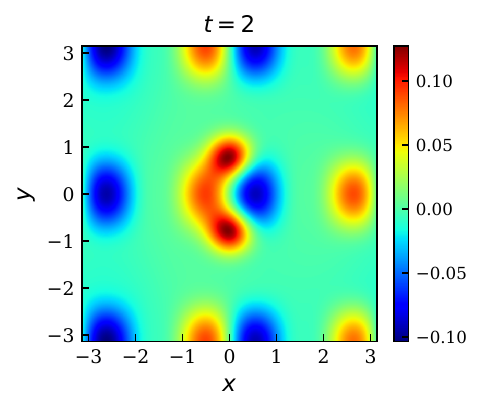} &
\includegraphics[width=0.225\textwidth]{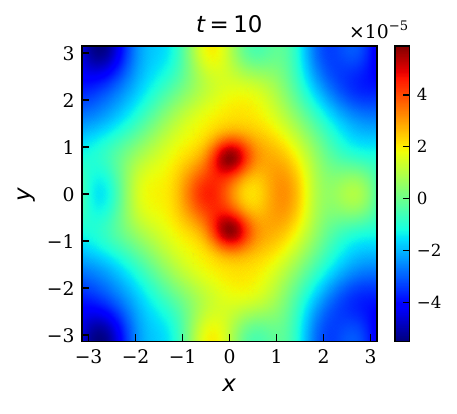} &
\includegraphics[width=0.225\textwidth]{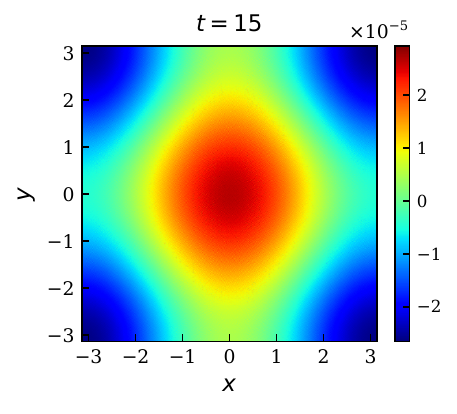} \\[-0.8mm]
\includegraphics[width=0.225\textwidth]{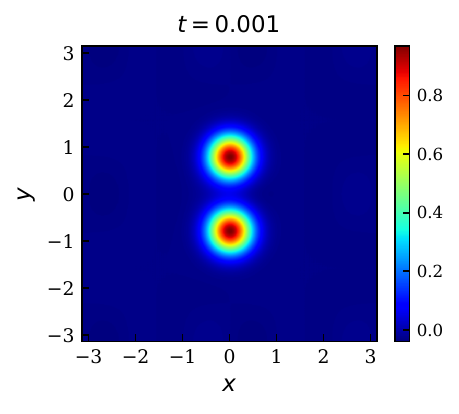} &
\includegraphics[width=0.225\textwidth]{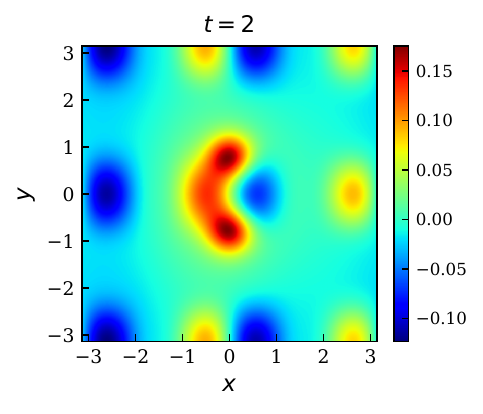} &
\includegraphics[width=0.225\textwidth]{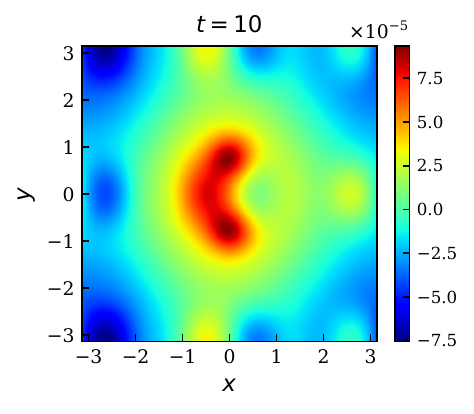} &
\includegraphics[width=0.225\textwidth]{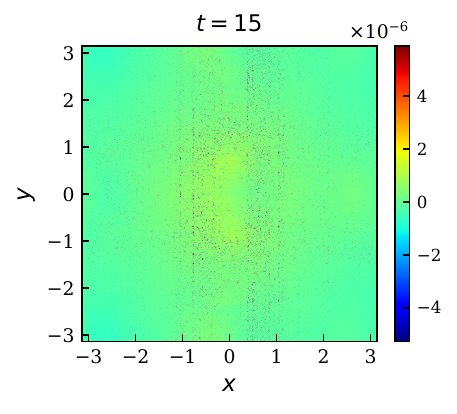}
\end{tabular}
\endgroup
\caption{\small Vorticity of the linear model \eqref{eq:toy2}; $N_x=N_y=512$, $\Delta t=0.001$. Rows: $a=0.001,0.1,1$; columns: $t=0.001,2,10,15$.}\label{fig1}
\medskip
\begingroup\setlength{\tabcolsep}{1pt}
\begin{tabular}{@{}cccc@{}}
\includegraphics[width=0.225\textwidth]{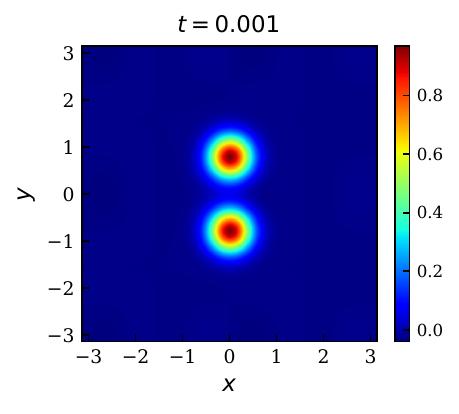} &
\includegraphics[width=0.225\textwidth]{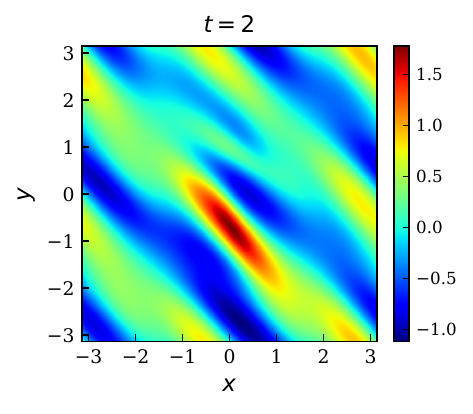} &
\includegraphics[width=0.225\textwidth]{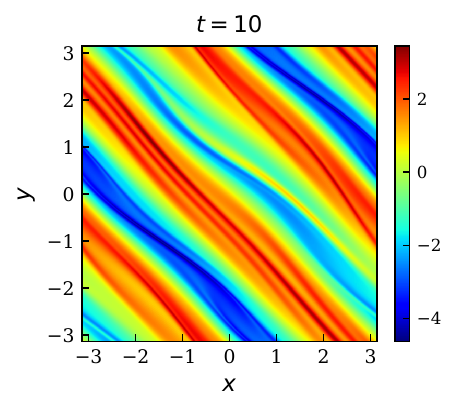} &
\includegraphics[width=0.225\textwidth]{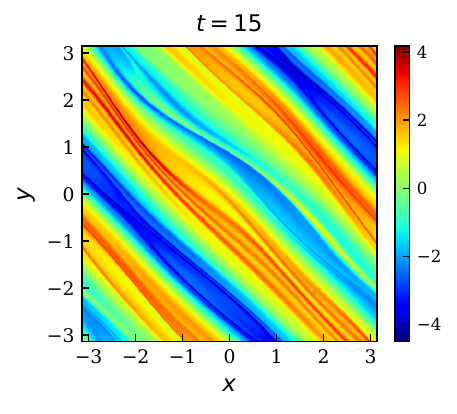} \\[-0.8mm]
\includegraphics[width=0.225\textwidth]{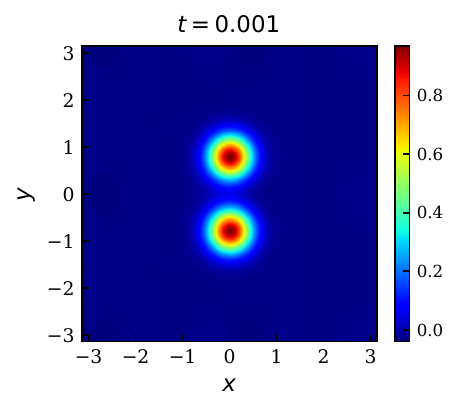} &
\includegraphics[width=0.225\textwidth]{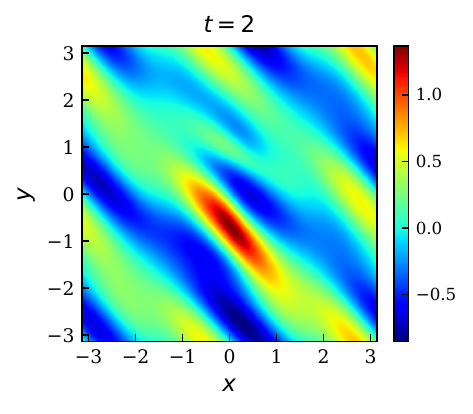} &
\includegraphics[width=0.225\textwidth]{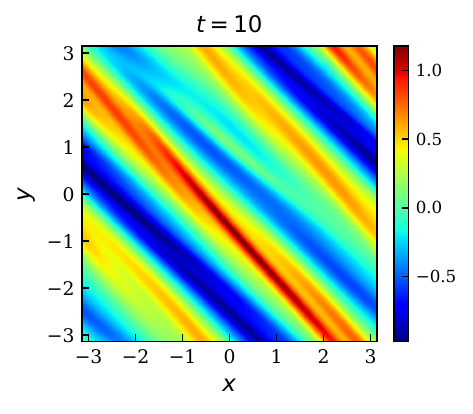} &
\includegraphics[width=0.225\textwidth]{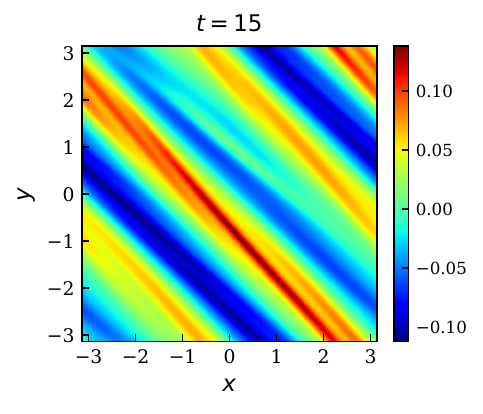} \\[-0.8mm]
\includegraphics[width=0.225\textwidth]{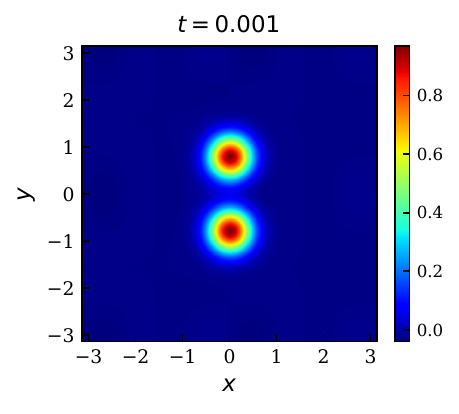} &
\includegraphics[width=0.225\textwidth]{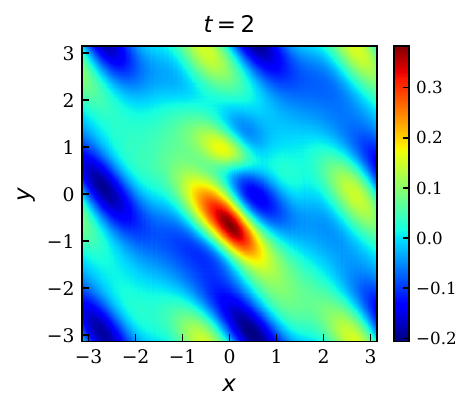} &
\includegraphics[width=0.225\textwidth]{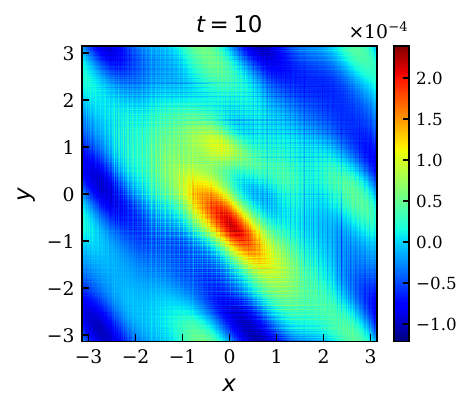} &
\includegraphics[width=0.225\textwidth]{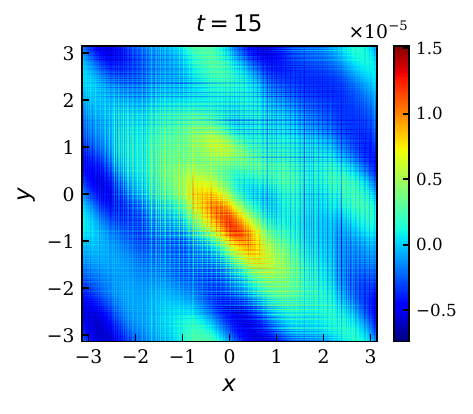}
\end{tabular}
\endgroup
\caption{\small Vorticity of the modified nonlinear benchmark \eqref{eq:toy-modified}, with the same mesh, time step, row and column ordering as Figure~\ref{fig1}.}\label{fig2}
\end{figure}

\FloatBarrier
\begin{figure}[H]
\centering
\includegraphics[width=0.43\textwidth]{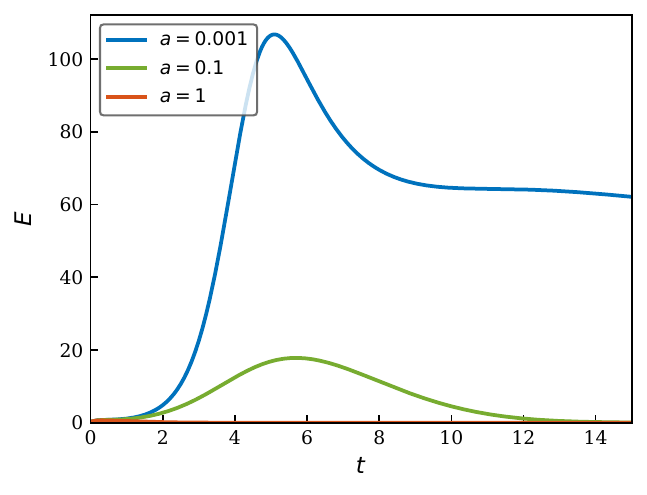}
\caption{\small Kinetic enstrophy $E=\frac12\int_{\T^2}|\nabla u|^2$ of the modified nonlinear benchmark \eqref{eq:toy-modified}, with $N_x=N_y=512$ and $\Delta t=0.001$.}\label{fig2b}
\end{figure}

\endgroup
\section{Conclusion}
To conclude, we have established the local well-posedness of the inviscid Oldroyd-B model in the sense of Hadamard with critical regularity. Then we have established the global existence of solutions for $d=2$ with damping in the low regularity class $(L^2\cap B^1_{\infty,1})\times (L^2\cap B^0_{\infty,1})$, which is novel in the literature. Furthermore, in both the 2D and 3D cases, we have proved global existence for the inviscid Oldroyd-B model uniformly in the damping parameter. In addition, we have investigated the optimal temporal decay rates and time integrability by improving the existing Fourier splitting method and developing a novel decomposition strategy. One of the major contributions of the present paper is the uniform-in-time vanishing damping limit and the resulting correlation between the sharp vanishing damping rate and the temporal decay rate. On periodic domains, we have further shown exponential relaxation toward spatially homogeneous equilibria. The nonzero Fourier modes decay uniformly in $a\in[0,1]$; for the full stress, the decay rate is of order $a$ when $a>0$, whereas an undamped constant stress generally remains when $a=0$. Finally, we have supported our findings with numerical evidence for the vanishing damping limit in the periodic setting.

Several related questions are remaining open after this work. 
Our work in the article relies on the smallness of solutions in the critical space. The problem of the global existence of smooth solutions for the 2D inviscid Oldroyd-B model without any smallness assumption is still open. It is also unknown whether there exists a global solution with the critical regularity of the inviscid Oldroyd-B model equation with fractional dissipation $(-\Delta)^{\frac{\beta}{2}}\tau$, where $\beta\in(0,1)$. We also would like to point out that the global well-posedness problem of the 2D classical Oldroyd-B model for the case $b\neq0$ remains open, as discussed in \cite{M13}. Indeed one possible approach to lower the regularity is to introduce the propagation of randomness, cf. \cite{BDNY24,DNY24,DNY22}. It is worth pointing out that the critical regularity may become ``subcritical'' under the probabilistic scaling as suggested by Deng, Nahmod and Yue in \cite{DNY22}. However, such framework needs delicate adaption in the fluid setting and we shall leave it for the future discussion.

\section*{Acknowledgement}
X. Cheng is supported in part by NSFC (Grant No. 12401270 and No. 42450192), Natural Science Foundation of Shanghai (Grant No. 24ZR1404200), Shanghai Magnolia Talent Plan Pujiang Project (Grant No. 24PJA007). Z. Luo is partially supported by the China Postdoctoral Science Foundation (No. 2022TQ0077 and No. 2023M730699) and Shanghai Post-doctoral Excellence Program (No. 2022062). C. Yuan is supported by NSFC (No. 123B2008).

\bibliographystyle{abbrv}

\end{document}